\documentclass[12pt, reqno]{amsart}
\usepackage{amsmath, amsthm, amssymb}
\usepackage{enumitem}

\usepackage{fullpage}
\usepackage[many]{tcolorbox}
\usepackage{xcolor}
\usepackage{wasysym}
\usepackage{mathtools}
\usepackage{textcomp}
\usepackage{stmaryrd}

\usepackage{tikz}
\usetikzlibrary{cd}
\usetikzlibrary{calc}
\usetikzlibrary{arrows,backgrounds,patterns.meta}
\usetikzlibrary{positioning,shadings}
\usetikzlibrary{shapes}
\usetikzlibrary{backgrounds}
\usetikzlibrary{decorations,decorations.pathreplacing,decorations.markings,decorations.pathmorphing}
\tikzstyle{snake}=[decorate, decoration={snake, segment length=1mm, amplitude=.3mm}]
\tikzstyle{saw}=[decorate, decoration={saw, segment length=.7mm, amplitude=.25mm}]
\newcommand{\tikzmath}[2][]
{\vcenter{\hbox{\begin{tikzpicture}[#1]#2\end{tikzpicture}}}
}
\newcommand{\roundNbox}[6]{
	\draw[rounded corners=5pt, very thick, #1] ($#2+(-#3,-#3)+(-#4,0)$) rectangle ($#2+(#3,#3)+(#5,0)$);
	\coordinate (ZZa) at ($#2+(-#4,0)$);
	\coordinate (ZZb) at ($#2+(#5,0)$);
	\node at ($1/2*(ZZa)+1/2*(ZZb)$) {#6};
}
\tikzset{super thick/.style={line width=3pt}}
\tikzstyle{mid>}=[decoration={markings, mark=at position 0.55 with {\arrow{>}}}, postaction={decorate}]
\tikzstyle{mid<}=[decoration={markings, mark=at position 0.55 with {\arrow{<}}}, postaction={decorate}]

\tikzstyle{frameR}=[preaction={draw=#1,super thick,opacity=.2,decorate,decoration={curveto,amplitude=0,raise=-1.5pt}},thick,#1]
\tikzstyle{frameL}=[preaction={draw=#1,super thick,opacity=.2,decorate,decoration={curveto,amplitude=0,raise=1.5pt}},thick,#1]

\newcommand{\xz}{\otimes}
\newcommand{\xo}{\circ}

\def\semicolon{;}
\def\applytolist#1{
    \expandafter\def\csname multi#1\endcsname##1{
        \def\multiack{##1}\ifx\multiack\semicolon
            \def\next{\relax}
        \else
            \csname #1\endcsname{##1}
            \def\next{\csname multi#1\endcsname}
        \fi
        \next}
    \csname multi#1\endcsname}

\def\calc#1{\expandafter\def\csname c#1\endcsname{{\mathcal #1}}}
\applytolist{calc}QWERTYUIOPLKJHGFDSAZXCVBNM;
\def\bbc#1{\expandafter\def\csname bb#1\endcsname{{\mathbb #1}}}
\applytolist{bbc}QWERTYUIOPLKJHGFDSAZXCVBNM;
\def\bfc#1{\expandafter\def\csname bf#1\endcsname{{\mathbf #1}}}
\applytolist{bfc}QWERTYUIOPLKJHGFDSAZXCVBNM;
\def\sfc#1{\expandafter\def\csname s#1\endcsname{{\sf #1}}}
\applytolist{sfc}QWERTYUIOPLKJHGFDSAZXCVBNM;
\def\fc#1{\expandafter\def\csname f#1\endcsname{{\mathfrak #1}}}
\applytolist{fc}QWERTYUIOPLKJHGFDSAZXCVBNM;
\def\rmc#1{\expandafter\def\csname rm#1\endcsname{{\mathrm #1}}}
\applytolist{rmc}QWERTYUIOPLKJHGFDSAZXCVBNM;

\tikzstyle{shaded}=[fill=red!10!blue!20!gray!30!white]
\tikzstyle{unshaded}=[fill=white]
\tikzstyle{over}=[double, draw=white, super thick, double=]
\tikzstyle{snake}=[decorate, decoration={snake, segment length=1mm, amplitude=.3mm}]
\tikzstyle{saw}=[decorate, decoration={saw, segment length=.7mm, amplitude=.25mm}]

\tikzstyle{coupon}=[draw, very thick, rectangle, rounded corners=5pt]
\tikzset{Rightarrow/.style={double equal sign distance,>={Implies},->},
triplecd/.style={-,preaction={draw,Rightarrow}},
quadruplecd/.style={preaction={draw,Rightarrow,
shorten >=0pt
},
shorten >=1pt,
-,double,double
distance=0.2pt}}
\tikzset{
    tripleline/.style args={[#1] in [#2] in [#3]}{
        #1,preaction={preaction={draw,#3},draw,#2}
    }
}
\tikzstyle{triple}=[tripleline={[line width=.15mm,black] in
      [line width=.7mm,white] in
      [line width=1mm,black]}] 
\tikzset{
    quadrupleline/.style args={[#1] in [#2] in [#3] in [#4]}{
        #1,preaction={preaction={preaction={draw,#4},draw,#3}, draw,#2}
    }
}
\tikzstyle{quadruple}=[quadrupleline={[line width=.3mm,white] in
      [line width=.6mm,black] in
      [line width=1.2mm,white] in
      [line width=1.5mm,black]}]

\tikzdeclarepattern{
  name=primeddots,
  type=uncolored,
  bounding box={(-.6pt,-.6pt) and (.6pt,.6pt)},
  tile size={(5pt,5pt)},
  tile transformation={rotate=60},
  parameters={none}, 
  code={
    \fill(0pt,0pt) circle (.35pt);
  }
}

\tikzdeclarepattern{
  name=primedbox,
  type=uncolored,
  bounding box={(-1pt,-1pt) and (1pt,1pt)},
  tile size={(5pt,5pt)},
  tile transformation={rotate=60},
  parameters={none}, 
  code={
    \draw[very thin] (-.9pt,-.9pt) -- (-.9pt,.9pt) -- (.9pt,.9pt) -- (.9pt,-.9pt) -- (-.9pt,-.9pt);
  }
}

\tikzdeclarepattern{
  name=primedplus,
  type=uncolored,
  bounding box={(-1pt,-1pt) and (1pt,1pt)},
  tile size={(5pt,5pt)},
  tile transformation={rotate=30},
  parameters={none}, 
  code={
    \draw[very thin] (-.9pt,-.9pt) -- (.9pt,.9pt);
    \draw[very thin] (.9pt,-.9pt) -- (-.9pt,.9pt);
  }
}

\tikzdeclarepattern{
  name=primedstar,
  type=uncolored,
  bounding box={(-1.2pt,-1.2pt) and (1.2pt,1.2pt)},
  tile size={(5pt,5pt)},
  tile transformation={rotate=30},
  parameters={none}, 
  code={
    \draw[very thin] (-1.2pt,0pt) -- (1.2pt,0pt);
    \draw[very thin] (-.6pt,-1.04pt) -- (.6pt,1.04pt);
    \draw[very thin] (-.6pt,1.04pt) -- (.6pt,-1.04pt);
  }
}

\tikzdeclarepattern{
  name=primeddots2,
  type=uncolored,
  bounding box={(-.3pt,-.3pt) and (.3pt,.3pt)},
  tile size={(2.5pt,2.5pt)},
  tile transformation={rotate=60},
  parameters={none}, 
  code={
    \fill(0pt,0pt) circle (.35pt);
  }
}

\tikzstyle{primedregion}[none]=[
	preaction={fill=#1},
	pattern=primeddots,
  draw=#1,
]

\tikzstyle{boxregion}[none]=[
	preaction={fill=#1},
	pattern=primedbox,
  draw=#1,
]

\tikzstyle{plusregion}[none]=[
	preaction={fill=#1},
	pattern=primedplus,
  draw=#1,
]

\tikzstyle{starregion}[none]=[
	preaction={fill=#1},
	pattern=primedstar,
  draw=#1,
]

\tikzstyle{primedregion2}[none]=[
	preaction={fill=#1},
	pattern=primeddots2,
  draw=#1,
]

\newcommand{\rhoColor}{black}
\newcommand{\sigmaColor}{snake}

\newcommand{\MsColor}{black}

\newcommand{\XColor}{black}

\newcommand{\QsColor}{black}

\newcommand{\AsColor}{red}
\newcommand{\BsColor}{blue}

\newcommand{\AColor}{gray!30}
\newcommand{\BColor}{white}

\newcommand{\ArColor}{red!20}
\newcommand{\BrColor}{blue!20}
\newcommand{\arColor}{gray!20}
\newcommand{\brColor}{gray!55}

\newcommand{\aColor}{white}
\newcommand{\bColor}{green!60}
\newcommand{\cColor}{blue!60}

\makeatletter
\newcommand{\xMapsto}[2][]{\ext@arrow 0599{\Mapstofill@}{#1}{#2}}
\def\Mapstofill@{\arrowfill@{\Mapstochar\Relbar}\Relbar\Rightarrow}
\makeatother

\definecolor{violet}{RGB}{148,0,211}
\definecolor{DarkGreen}{RGB}{0,150,0}

\usepackage[pdftex,plainpages=false,hypertexnames=false,pdfpagelabels]{hyperref}
\definecolor{medium-blue}{rgb}{0,0,.8}
\hypersetup{colorlinks, linkcolor={purple}, citecolor={medium-blue}, urlcolor={medium-blue}}
\newcommand{\arxiv}[1]{\href{http://arxiv.org/abs/#1}{\tt arXiv:\nolinkurl{#1}}}
\newcommand{\arXiv}[1]{\href{http://arxiv.org/abs/#1}{\tt arXiv:\nolinkurl{#1}}}

\newcommand{\doi}[1]{\href{http://dx.doi.org/#1}{{\tt DOI:#1}}}

\DeclareMathOperator{\coev}{coev}
\DeclareMathOperator{\End}{End}
\DeclareMathOperator{\ev}{ev}

\DeclareMathOperator{\Forget}{Forget}

\DeclareMathOperator{\Hom}{Hom}
\DeclareMathOperator{\id}{id}

\DeclareMathOperator{\Irr}{Irr}
\DeclareMathOperator{\mate}{mate}
\DeclareMathOperator{\ONB}{ONB}
\DeclareMathOperator{\op}{op}

\DeclareMathOperator{\Tr}{Tr}
\DeclareMathOperator{\tr}{tr}

\newcommand{\set}[2]{\left\{#1 \middle| #2\right\}}

\renewcommand{\cent}{\textup{\textcent}}
\newcommand{\condense}{\mathrel{\,\hspace{.75ex}\joinrel\rhook\joinrel\hspace{-.75ex}\joinrel\rightarrow}}

\newcommand{\Isom}{\mathsf{Isom}}

\newcommand{\Mod}{\mathsf{Mod}}
\newcommand{\Fun}{\mathsf{Fun}}
\newcommand{\Vect}{\mathsf{Vec}}
\newcommand{\Hilb}{\mathsf{Hilb}}

\newcommand{\Alg}{\mathsf{Alg}}
\newcommand{\QSys}{\mathsf{QSys}}
\newcommand{\HstarAlg}{\mathsf{H^*Alg}}
\newcommand{\Gray}{\mathsf{Gray}}

\newcommand{\mFC}{\mathsf{mFC}}
\newcommand{\HilbPlus}{\Hilb_\boxplus}

\def\semicolon{;}
\def\applytolist#1{
    \expandafter\def\csname multi#1\endcsname##1{
        \def\multiack{##1}\ifx\multiack\semicolon
            \def\next{\relax}
        \else
            \csname #1\endcsname{##1}
            \def\next{\csname multi#1\endcsname}
        \fi
        \next}
    \csname multi#1\endcsname}

\def\calc#1{\expandafter\def\csname c#1\endcsname{{\mathcal #1}}}
\applytolist{calc}QWERTYUIOPLKJHGFDSAZXCVBNM;
\def\bbc#1{\expandafter\def\csname bb#1\endcsname{{\mathbb #1}}}
\applytolist{bbc}QWERTYUIOPLKJHGFDSAZXCVBNM;
\def\bfc#1{\expandafter\def\csname bf#1\endcsname{{\mathbf #1}}}
\applytolist{bfc}QWERTYUIOPLKJHGFDSAZXCVBNM;
\def\sfc#1{\expandafter\def\csname s#1\endcsname{{\sf #1}}}
\applytolist{sfc}QWERTYUIOPLKJHGFDSAZXCVBNM;
\def\fc#1{\expandafter\def\csname f#1\endcsname{{\mathfrak #1}}}
\applytolist{fc}QWERTYUIOPLKJHGFDSAZXCVBNM;
\def\rmc#1{\expandafter\def\csname rm#1\endcsname{{\mathrm #1}}}
\applytolist{rmc}QWERTYUIOPLKJHGFDSAZXCVBNM;

\theoremstyle{plain}
\newtheorem{thm}{Theorem}[section]
\newtheorem*{thm*}{Theorem}
\newtheorem{thmalpha}{Theorem}

\newtheorem{cor}[thm]{Corollary}
\newtheorem{coralpha}[thmalpha]{Corollary}
\newtheorem*{cor*}{Corollary}

\newtheorem*{conj*}{Conjecture}
\newtheorem{lem}[thm]{Lemma}
\newtheorem*{lem*}{Lemma}

\newtheorem{prop}[thm]{Proposition}
\newtheorem{quest}[thm]{Question}
\newtheorem*{quest*}{Question}
\newtheorem*{claim*}{Claim}

\theoremstyle{definition}
\newtheorem{defn}[thm]{Definition}
\newtheorem{defnalpha}[thmalpha]{Definition}

\newtheorem{facts}[thm]{Facts}
\newtheorem{construction}[thm]{Construction}

\newtheorem{nota}[thm]{Notation}

\newtheorem{exs}[thm]{Examples}
\newtheorem{ex}[thm]{Example}
\newtheorem{sub-ex}[thm]{Sub-Example}
\newtheorem{counter-ex}[thm]{Counter-Example}
\newtheorem{rem}[thm]{Remark}
\newtheorem*{rem*}{Remark}
\newtheorem{remark}[thm]{Remark}

\newtheorem{warn}[thm]{Warning}

\newcommand{\bigboxplus}{\boxplus}

\title{Manifestly unitary higher Hilbert spaces}
\author{Quan Chen, Giovanni Ferrer, Brett Hungar, David Penneys, and Sean Sanford}
\date{\today}

\begin{document}

\begin{abstract}
Higher idempotent completion gives a formal inductive construction of the $n$-category of finite dimensional $n$-vector spaces starting with the complex numbers.
We propose a manifestly unitary construction of low dimensional higher Hilbert spaces, formally constructing the $\mathrm{C}^*$-3-category of 3-Hilbert spaces from Baez's 2-Hilbert spaces, which itself forms a 3-Hilbert space.
We prove that the forgetful functor from 3-Hilbert spaces to 3-vector spaces is fully faithful.
\end{abstract}

\maketitle

\tableofcontents

\section{Introduction}

An exciting recent development in higher categories is the notion of \emph{higher idempotent completion} for $n$-categories of \cite{1905.09566} (see \cite{1812.11933} for $n=2$), and the higher \emph{suspension} operation $\Sigma$.\footnote{
As we are interested in finite dimensional higher Hilbert spaces, we work here with a fully unital version of higher idempotent and suspension, which is currently not known to be equivalent to the non-unital version.}
Starting with $\bbC$, $\Sigma$ inductively constructs higher categories of finite dimensional \emph{higher vector spaces} in a completely formal way.
In turn, these higher vector spaces are receptacles for fully extended topological quantum field theories (TQFTs) by the Cobordism Hypothesis \cite{MR1355899,MR2555928}.
For $n=1$, $\Sigma \bbC=\Vect$, the category of finite dimensional vector spaces.
For $n\geq 2$, we define $n\Vect := \Sigma^n \bbC$.
For $n=2$,
we obtain $\Sigma^2 \bbC = \Alg$, the 2-category of finite dimensional multimatrix algebras, bimodules, and intertwiners,
which is equivalent to $2\Vect$, the 2-category of finite semisimple linear categories, linear functors, and natural tranformations.
One level higher,
we get the 3-category $\mFC$ of multifusion categories \cite{MR4254952,MR4444089},
which is equivalent to $3\Vect$, the 3-category of finite semisimple 2-categories \cite{1905.09566,MR4372801}.
\begin{equation}
\label{eq:HigherVectorSpaces}
\bbC
\xrightarrow{\Sigma}
\Vect
\xrightarrow{\Sigma}
\Alg
\simeq
2\Vect
\xrightarrow[\text{\cite{MR4444089}}]{\Sigma}
\mFC
\underset{\text{\cite{MR4372801}}}{\simeq}
3\Vect
\end{equation}

One may expect that a fully extended $n$D unitary TQFT should correspond to a notion of finite dimensional $n$-Hilbert space.
In theoretical condensed matter physics, it is widely accepted that higher idempotents describe phase changes between topologically ordered phases of matter \cite{MR3246855,MR4444089,MR4640433,2403.07813}.
However, these topological orders are inherently unitary theories, and thus require a unitary version of higher idempotent completion and the suspension operation.
In this article, we propose such an answer up to $n=3$, giving a manifestly unitary construction of finite\footnote{We will  investigate infinite dimensional higher Hilbert spaces in future joint work.} higher Hilbert spaces.
Here, the term \emph{unitary} incorporates both the presence of dagger structures at multiple levels, together with positivity conditions.
The former structures were investigated coherently in \cite{2403.01651} building on the coherent dagger 1-category treatment in \cite{2304.02928}.
In the present paper, we study multiple dagger structures and positivity simultaneously.

\begin{figure}[!ht]
$$
\begin{tikzcd}[column sep=3em]
&&&
\rmB2\Hilb
\arrow[r,mapsto,"\cent^\dag"]
&
\sH^*\mFC
\arrow[r, "\Mod^\dag", "\simeq"']
&
3\Hilb
\\
&
\rmB\Hilb
\arrow[r,mapsto, "\cent^\dag"]
&
\sH^*\Alg
\arrow[r,"\Mod^\dag", "\simeq"']
&
2\Hilb
\arrow[u,mapsto,"\rmB"]
\arrow[ur,mapsto,"\Sigma^\dag"]
\\
\rmB\bbC
\arrow[r,mapsto,"\cent^\dag"]
&
\Hilb
\arrow[u, mapsto,"\rmB"]
\arrow[ur,mapsto,"\Sigma^\dag"]
\\
\bbC
\arrow[u, mapsto,"\rmB"]
\arrow[ur, mapsto,"\Sigma^\dag"]
\end{tikzcd}
$$
\caption[]{
\label{fig:UnitaryStaircase}
Manifestly unitary construction of $n\Hilb$ from $(n-1)\Hilb$
}
\end{figure}
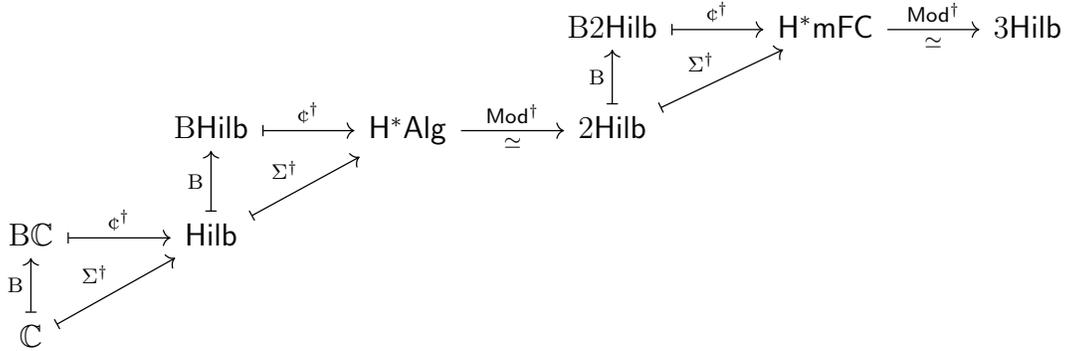
In Figure \ref{fig:UnitaryStaircase} above:
\begin{itemize}
\item
$\rmB$ means take the \emph{delooping} \cite[\S5.6]{MR2664619}, i.e., consider the unitary monoidal $k$-category as a unitary  $(k+1)$-category with one object.
\item
$\cent^\dag$ means take a \emph{unital higher unitary Cauchy completion} (unitary and unital version of \cite{1905.09566}, similar to \cite{1812.11933}).
\item
$\Sigma^\dag$ is the composite $\cent^\dag\circ \rmB$, called the \emph{unitary suspension}.
\item 
$\sH^*\Alg$ is the Morita $\rmC^*$ 2-category of finite dimensional $\rmH^*$-\emph{algebras} in the sense of \cite{MR13235}.
\item 
$2\Hilb$ is the $\rmC^*$ 2-category of Baez's \emph{2-Hilbert spaces} \cite{MR1448713}
\item $\sH^*\mFC$ is the Morita $\rmC^*$-3-category of \emph{$\rmH^*$-multifusion categories}, a notion defined below on the next page
\item $3\Hilb$ is the $\rmC^*$-3-category of \emph{3-Hilbert spaces}, which are also defined on the next page.
\item
$\Mod^\dag$ is the unitary equivalence given by taking the 1- or 2-category of unitary modules for the algebra/multifusion category respectively.
\end{itemize}

In $\rmB\Hilb$, there are several versions of higher idempotent that one might use.
A natural choice which was recently studied in \cite{MR4419534,MR4482713}
is the notion of \emph{special standard Q-systems},
originally defined by \cite{MR1257245} in the context of subfactor theory \cite{MR1027496},
which are unital separable algebras $(A,\mu,\iota)$ such that the adjoint $\mu^\dag$ of the multiplication map is an $A-A$ bimodule map,
$\mu^\dag\mu=\id_A$ (this is the \emph{special} condition),
and a standardness condition for the maps $1\to A\otimes A$ and $A\otimes A\to 1$ built from $\mu,\iota$ and their adjoints.
In this article, we use the notion of not necessarily special Q-systems, which are exactly the $\rmH^*$-algebras in the sense of \cite{MR13235} (see also \cite{MR1448713,MR3971584}).
We thus call them $\rmH^*$-\emph{algebras} to avoid any confusion with the higher idempotents studied in \cite{MR4419534,MR4482713}. 

Recall that a \emph{2-Hilbert space} in the sense of \cite{MR1448713} is a finite semisimple $\rmC^*$-category $\cA$ equipped with a faithful trace $\Tr^\cA_a: \cA(a\to a)\to \bbC$ for all $a\in \cA$ satisfying some compatibility requirements (see Definition \ref{defn:pre2Hilb} below).
The notion of 2-Hilbert space
affords many advantages 
and
satisfies certain desiderata 
for higher Hilbert spaces.
First, we get a unitary Yoneda embedding
$$
\cA\hookrightarrow \Fun^\dag(\cA^{\op}\to \Hilb),
$$
which was first used in \cite{MR3687214,MR3948170} and made explicit in \cite{MR4598730}.
Second, unitary functors between 2-Hilbert spaces admit unique \emph{unitary adjoints} \cite{MR4750417} in the sense that the adjunction natural isomorphisms are unitary:
$$
\cB(F(a)\to b) \cong^\dag \cA(a\to F^*(b))
\qquad\qquad\qquad
F\dashv^\dag F^*.
$$

In this article, we give a definition of a (finite) 3-Hilbert space.
To present our definition, we introduce the following notions.
First, given a linear 2-category $\fX$ with additive hom categories and objects $a_1,\dots, a_n\in\fX$,
we define the \emph{linking algebra}\footnote{Here, linking algebra means linking $E_1$-algebra in the 2-category $2\Vect$ of semisimple categories.}
$$
\cL(a_1,\dots, a_n)
:=
\bigoplus_{i,j=1}^n
\fX(a_j\to a_i),
$$
with tensor product defined via matrix multiplication and component-wise composition.
When $\fX$ is $\rmC^*$, so is every linking algebra.
Second, for ordinary 2-categories, rigidity (existence of both left and right adjoint 1-morphisms) is a property; any two choices of adjoints are canonically isomorphic.
However, for $\rm C^*$ 2-categories, choosing compatible unitary adjoints over the entire 2-category is a structure, namely that of a \emph{unitary adjoint 2-functor} (which is the identity on objects).
The main difference here is that the canonical isomorphism between distinct choices of adjoints need not be unitary \cite{MR4133163}.

Equipped with the notions of linking algebra and unitary adjoint 2-functor, we make the following definition.

\begin{defnalpha}\label{def:3Hilbspace}
A \emph{pre-3-Hilbert space} is a linear $\rmC^*$-2-category $\mathfrak{X}$ such that every linking algebra 
is unitary multifusion, equipped with a unitary adjoint 2-functor $\vee$ and a \emph{spherical weight} $\Psi$, which consists of positive linear maps $\Psi_a: \End_\fX(1_a)\to \bbC$ for all $a\in\fX$ 
satisfying
$$
\Psi_b\left(\,
\tikzmath{
\fill[rounded corners = 5pt, fill=\brColor] (-1.2,-.9) rectangle (.6,.9);
\fill[\arColor] (0,.3) arc (0:180:.3cm) -- (-.6,-.3) arc (-180:0:.3cm);
\draw[thick] (0,.3) node[right,yshift=.2cm]{$\scriptstyle X$} arc (0:180:.3cm) --node[left]{$\scriptstyle X^\vee$} (-.6,-.3) arc (-180:0:.3cm) node[right,yshift=-.2cm]{$\scriptstyle X$};
\roundNbox{fill=white}{(0,0)}{.3}{0}{0}{$f$}
}
\,\right)
=
\Psi_a\left(\,
\tikzmath{
\fill[rounded corners = 5pt, fill=\arColor] (1.2,-.9) rectangle (-.6,.9);
\fill[\brColor]  (0,.3) arc (180:0:.3cm) -- (.6,-.3) arc (0:-180:.3cm);
\draw[thick] (0,.3) node[left,yshift=.2cm]{$\scriptstyle X$} arc (180:0:.3cm) --node[right]{$\scriptstyle X^\vee$} (.6,-.3) arc (0:-180:.3cm) node[left,yshift=-.2cm]{$\scriptstyle X$};
\roundNbox{fill=white}{(0,0)}{.3}{0}{0}{$f$}
}
\,\right)
\qquad\qquad
\begin{aligned}
&\forall\,a,b\in\fX
\\
&\forall\, {}_aX_b\in\fX(a\to b)
\\
&\forall\, f\in\End_\fX({}_aX_b)
\end{aligned}
\qquad\qquad
\begin{aligned}
\tikzmath{\fill[fill=\arColor, rounded corners=5] (-.3,-.3) rectangle (.3,.3);}&=a
\\
\tikzmath{\fill[fill=\brColor, rounded corners=5] (-.3,-.3) rectangle (.3,.3);}&=b.
\end{aligned}
$$

A 3-Hilbert space $\fX$ is a pre-3-Hilbert space that is suitably Cauchy complete, i.e.~$\fX$ admits all \emph{Hilbert direct sums} (a notion defined in \ref{sec:HilbDirectSum}) and all $\rmH^*$-monads in $\fX$ (defined in \ref{sec:H*Monad}) split.
\end{defnalpha}

The prototypical example of a pre-3-Hilbert space is the delooping $\rmB\cC$ of a unitary multifusion category $\cC$ equipped with a unitary dual functor and a spherical weight \cite{MR4133163}.
We call such a pre-3-Hilbert space an $\rmH^*$-\emph{multifusion category}.
Our first examples of 3-Hilbert spaces are:
\begin{itemize}
    \item the Morita $\rmC^*$-2-category $\HstarAlg$ of $\rmH^*$-algebras \cite{MR13235,MR1448713} (see also \cite{MR3971584})
    \item the $\rmC^*$-2-category $2\Hilb$ of 2-Hilbert spaces in the sense of \cite{MR1448713},
\end{itemize}
As mentioned in Definition \ref{def:3Hilbspace}, we generalize the notion of $\rmH^*$-algebra to an $\rmH^*$-monad in a pre-3-Hilbert space, which gives more examples of 3-Hilbert spaces: 
\begin{itemize}
    \item the Morita $\rmC^*$-2-category $\HstarAlg(\fX)$ of $\rmH^*$-algebras in a pre-3-Hilbert space $\fX$ that is additively complete in some unitary sense, and
    \item the $\rmC^*$-2-category $\Mod^\dag(\cC)$ of unitary modules for a $\rmH^*$-multifusion category $\cC$.
\end{itemize}
We have the following theorem.

\begin{thmalpha}
\label{thmalpha:All3Hilbs}
Every 3-Hilbert space arises as $\Mod^\dag(\cC)$ for some $\rmH^*$-multifusion category $\cC$.
In particular, $\HstarAlg(\rmB\cC)\simeq \Mod^\dag(\cC)$ isometrically as 3-Hilbert spaces.
\end{thmalpha}

As an immediate corollary, $\HstarAlg\simeq 2\Hilb$ isometrically as 3-Hilbert spaces.
Thus the notion of $\HstarAlg(-)$ can be viewed as a higher idempotent completion \cite{1812.11933,1905.09566,MR4419534,MR4482713} which is \emph{manifestly unitary}, fitting in the staircase diagram Figure \ref{fig:UnitaryStaircase}.

Our second main result is as follows.

\begin{thmalpha}
\label{thm:3Hilb=H*mFC}
Taking $\Mod^\dag$ is a unitary equivalence between the $\rmC^*$-3-categories $\sH^*\mFC$ and $3\Hilb$ which preserves unitary adjoints and spherical weights.\footnote{\label{Footnote:EquivalenceOf4Hilbs}We expect this equivalence is an isometric equivalence of 4-Hilbert spaces (a notion yet to be defined).
}
\end{thmalpha}

As an application, we get the following corollary combining Theorem \ref{thm:3Hilb=H*mFC} with \cite{MR4538281} and \cite{MR4616673}.

\begin{coralpha}
\label{coralpha:ForgetFF}
The forgetful 3-functor $3\Hilb \to 3\Vect$ is fully faithful.\footnote{The forgetful functor is well-known to not be essentially surjective from the existence of non-unitarizable fusion categories \cite{MR2180373,MR4327964}.}
\end{coralpha}

Indeed, the above result follows from commutativity of the following diagram
$$
\begin{tikzcd}
\sH^*\mFC
\arrow[r,"\Forget"']
\arrow[d, "\simeq"']
&
\mFC
\arrow[d, "\simeq"]
\\
3\Hilb
\arrow[r,"\Forget"]
&
3\Vect
\end{tikzcd}
$$
where \cite{MR4538281} and \cite{MR4616673} can be used to prove the 3-functor on the top line is fully faithful.

In Figure \ref{fig:UnitaryStaircase} above, we implicitly claimed that $\Sigma^\dag2\Hilb$ is equivalent to $\sH^*\mFC$.
(This would be the unitary analog of the equivalence
$\Sigma2\Vect \cong \mFC$ 
by \cite[Rem.~3.2.9]{MR4600461} using
that multifusion categories in characteristic zero are separable by \cite{MR4254952}.)
In \S\ref{sec:H*algInB2Hilb} below, we define a generalizable notion of $\rmH^*$-algebra in $\rmB2\Hilb$.
Our final result in this article is as follows.

\begin{thmalpha}
\label{thmalpha:H*AlgebrasIn2Hilb}
There is an equivalence of $\rmC^*$-3-categories $U:\sH^*\mFC \cong \HstarAlg(\rmB2\Hilb):|-|$
such that $|U(\cA)|$ is isometrically equivalent to $\cA$ for every $\cA\in \sH^*\mFC$
and
${}_{|U(\cA)|}|U(\cM)|_{|U(\cB)|}$ is isometrically equivalent to ${}_\cA\cM_\cB$ for every $\rmH^*$ bimodule category.\textsuperscript{\textup{\ref{Footnote:EquivalenceOf4Hilbs}}}
\end{thmalpha}

Theorems \ref{thm:3Hilb=H*mFC} and \ref{thmalpha:H*AlgebrasIn2Hilb} together tell us that we may view $3\Hilb$ as the unitary suspension of $2\Hilb$.

\subsection*{Acknowledgements}
The authors would like to thank
Andr\'e Henriques,
Theo Johnson-Freyd, and
David Reutter
for helpful conversations, along with the participants of the June 2023 virtual meeting on Dagger Higher Categories.
The authors were supported by NSF DMS 2154389.

\section{Background}

We note that our treatment of $n$-Hilbert spaces consists only of those which are suitably finite-dimensional. 

\begin{nota}
In this article, we typically use Roman letters $A,B,C$ for sets (possibly with extra structure), caligraphic letters $\cA,\cB,\cC$ for 1-categories (possibly with extra structure), and mathfrak letters $\fA,\fB,\fC$ for 2-categories (possibly with extra structure).
We usually use a sans-serif font for categories and 2-categories with distinct names like $\Vect,\Hilb,2\Vect,2\Hilb$, etc.
\end{nota}

\subsection{\texorpdfstring{$\rmH^*$}{H*}-algebras and 2-Hilbert spaces}
We begin with a rapid treatment of unitary algebras and 2-Hilbert spaces in the sense of \cite{MR1448713}.

A \emph{unitary algebra} is a finite dimensional $\rm C^*$-algebra.
A \emph{trace} on a unitary algebra $A$ is a linear functional $\Tr:A\to \bbC$ satisfying
$\Tr(ab)=\Tr(ba)$ for all $a,b\in A$ 
and
$\Tr(a^*a)\geq 0$ 
with equality if and only if $a=0$.
A unitary algebra $A$ equipped with a trace $\Tr$ gives a Hilbert space $L^2(A,\Tr)$ via the GNS construction:
$$
\langle a | b\rangle_{L^2(A,\Tr)}
:=
\Tr(a^*b).
$$
Since multiplication is bounded on $L^2(A,\Tr)$, we get an $\rmH^*$-\emph{algebra} in the sense of \cite{MR13235}.

\begin{defn}
\label{defn:1H*Alg}
An \emph{$\rmH^*$-algebra} is a unitary algebra $A$ equipped with a trace $\Tr_A:A \to \bbC$.
We identify $A$ with the GNS Hilbert space $L^2(A,\Tr_A)$.
\end{defn}

Given an $\rmH^*$-algebra $(A,\Tr_A)$ and a right action of $A$ on a Hilbert space $H$, every $\xi \in H$ gives a left creation operator $|\xi\rangle:A_A \to H_A$ by $a \mapsto \xi a$ which is manifestly a right $A$-module map.
We denote $|\xi\rangle^\dag$ by $\langle\xi|$.
Observe that $\langle \eta|\xi\rangle_A \in\End(A_A) \cong A$
gives an $A$-valued inner product.

\begin{defn}\label{def:hstar-module}
    An \emph{$\rmH^*$-module} for an $\rmH^*$-algebra $A$ is a Hilbert space $H$ equipped with a right $A$-action and a trace $\Tr_H \colon \End(H_A) \to \bbC$ satisfying 
    \begin{equation}
    \label{eq:H*ModuleTrace}
    \Tr_H( | \xi \rangle \langle \eta | ) = \Tr_A ( \langle \eta  | \xi \rangle_A ) 
    \qquad\qquad\forall\, \eta,\xi\in H.
    \end{equation}
We write $\Mod^\dag(A,\Tr_A)$ for the category of $\rmH^*$-modules over $(A,\Tr_A)$ with $A$-module maps.
\end{defn}

Since $H$ is finite dimensional, $\End(H_A)$ is the $A$-compact operators, i.e., the span of the $A$-rank one operators 
$\set{|\eta\rangle\langle\xi|}{\eta,\xi \in H}$.
Thus $\Tr_H$ is completely determined by \eqref{eq:H*ModuleTrace}.
Moreover, we have the identity
$$
\Tr_A ( \langle \eta  | \xi \rangle_A ) 
=
\langle 1_A \,|\, \langle \eta  | \xi \rangle_A \rangle_{L^2(A,\Tr_A)}
=
\langle\, | \eta \rangle_A 1_A \, | \,| \xi \rangle_A 1_A \rangle_{H}
=
\langle \eta | \xi\rangle_H.
$$

\begin{defn}
Given $a,b\in\cC$ in a linear category $\cC$, the \emph{linking algebra} $L(a,b)$ is given by formal matrices 
$$
\begin{pmatrix} \cC(a\to a) & \cC(b\to a) \\ \cC(a\to b) & \cC(b\to b) \end{pmatrix},
$$ 
with product given by matrix multiplication.
(We remind the reader that the columns correspond to the domain, and the rows correspond to the codomain, similar to how a matrix acts as a linear operator.)
When $\cC$ has a dagger structure, $L(a,b)$ is a $*$-algebra with $*$ being the $\dag$-transpose.
Similarly, we can define the linking algebra 
$L(a_1,\dots, a_n)$
for $n$ objects $a_1,\dots, a_n\in \cC$.
\end{defn}

\begin{defn}
A linear category is called \emph{pre-semisimple} if all $n$-fold linking algebras are finite dimensional semisimple.
A pre-semisimple category is then \emph{semisimple} if and only if it is Cauchy complete, i.e., it admits finite direct sums and all idempotents split.
The \emph{Cauchy completion} of a linear category is the idempotent completion of the additive envelope.
\end{defn}

\begin{defn}
A \emph{unitary (1-)category} is a linear dagger category $\cC$ such that every $n$-fold linking algebra is a unitary algebra.
A unitary category is called \emph{finite} if there is a global bound on the dimensions of the centers of all linking algebras.
That is, there is a $K>0$ such that for any linking algebra $L=L(a_1,\dots, a_n)$, $\dim(Z(L))<K$.
\end{defn}

\begin{rem}\label{rem:unitary-implies-cstar}
    The requirement that all the linking algebras be unitary algebras automatically means all unitary categories are $\mathrm{C}^*$ and $\mathrm{W}^*$ categories \cite{MR808930}.
    Moreover,  a unitary category is semisimple if and only if it is Cauchy complete.
    We note that a unitary category admits direct sums if and only if it admits orthogonal direct sums and is idempotent complete if and only if it is projection complete, i.e., all orthogonal projections split orthogonally.
    We thus define the \emph{Cauchy completion} of a unitary category as the projection completion of the orthogonal additive envelope.
\end{rem}

\begin{exs}\label{exs:PrototypicalUnitaryCategories}
The prototypical unitary category is $\Hilb$.
For any unitary algebra $A$, the \emph{delooping} $\rmB A$ (the category with one object $\star$ and $\End_{\rmB A}(\star) = A$) is a unitary category.  
The $n$-fold linking algebras in $\rmB A$ are canonically isomorphic to $M_n(A)$, and for any $M=(a_{i,j})\in M_n(A)$, $M^*=(a_{j,i}^*)$. 
\end{exs}

\begin{defn}
\label{defn:pre2Hilb}
A \emph{(unitary) trace} on a unitary category $\cC$ is a collection of linear maps $\Tr^\cC_c:\End_\cC(c) \to \bbC$ for each $c\in\cC$ such that:
    \begin{enumerate}
        \item $\Tr^\cC_c(gf) = \Tr^\cC_d(fg)$ for all $f:c \to d$ and $g:d \to c$.
        \item The sesquilinear form $\langle f |g\rangle_{c\to d}:=\Tr^\cC_c(f^\dag g)$ on $\cC(c\to d)$ is positive definite.
\end{enumerate}
A \emph{pre-2-Hilbert space} is a finite unitary category equipped with a unitary trace.
Observe that the inner products $\langle f |g\rangle_{a\to b}$ on the Hilbert spaces $\cC(a\to b)$ satisfy
\begin{equation}
\label{eq:InnerProductOnPre2Hilb}
\langle g|hf^\dag \rangle_{b\to c}
=
\langle gf|h\rangle_{a\to c}
=
\langle f|g^\dag h\rangle_{a\to b}
\qquad\qquad
\forall\,f:a\to b,\, g:b\to c,\, h:a\to c.
\end{equation}

An \emph{isometry} between two pre-2-Hilbert spaces is a fully faithful $\dag$-functor which preserves the trace.

We call two pre-2-Hilbert spaces \emph{isometrically equivalent} if there is an isometry $F \colon (\cC,\Tr^\cC) \to (\cD,\Tr^\cD)$ which is also essentially surjective.
Note that this definition is equivalent to having isometries both ways together with unitary natural isomorphisms from their composites to the appropriate identity functors.
\end{defn}

\begin{exs}\label{exs:BAIsAPre2Hilb}
    Given an $\rmH^*$-algebra $(A,\Tr_A)$, the unitary category $\rmB A$ from Example \ref{exs:PrototypicalUnitaryCategories} admits a canonical trace defined by the formula $\Tr^{\rmB A}_\star\big(\star\mathop{\longrightarrow}\limits^a \star\big):=\Tr_A(a)$.  With this definition, $(\rmB A,\Tr^{\rmB A})$ is a pre-2-Hilbert space.
\end{exs}

\begin{defn}[cf.~{\cite[Def 17]{MR1448713}}]
We call a pre-2-Hilbert space $(\cC,\Tr^\cC)$ a \emph{2-Hilbert space} if it is Cauchy complete, equivalently semisimple.
In this case, for $s\in\cC$ simple, we define its \emph{quantum dimension} $d_s :=\Tr^\cC(\id_s)$.
\end{defn}

\begin{defn}
There is a 2-category $2\Hilb$ whose objects are 2-Hilbert spaces.  The 1-morphisms $\cA\to\cB$ in $2\Hilb$ are $\dag$-functors, and the 2-morphisms are natural transformations.
This 2-category $2\Hilb$ is a $\dag$-2-category in the sense of \cite{MR1444286} \cite[Def 2.2]{MR4419534}, with the dagger structure given by $(\eta^\dag)_a=(\eta_a)^\dag$.

We say two 2-Hilbert spaces are \emph{isometrically equivalent} if they are isometrically equivalent as pre-2-Hilbert spaces.
Note that this notion of equivalence is stronger than the notion of equivalence in $2\Hilb$.
Indeed, it is straightforward to show that a $\dag$-equivalence $F:(\cC,\Tr^\cC)\to (\cD,\Tr^\cD)$ between two 2-Hilbert spaces is isometric if and only if it preserves quantum dimensions, i.e., for all $c\in \Irr(\cC)$,
$d_{F(c)}=d_c$.
\end{defn}

\begin{ex}
If $(\cC,\Tr^\cC)$ is a pre-2-Hilbert space, then the Cauchy completion of $\cC$
is a 2-Hilbert space where if $\bigoplus c_i$ is an orthogonal direct sum and $p\in \End(\bigoplus c_i)$ is an orthogonal projection, we define
$$
\Tr_{(\bigoplus c_i,p)}(f) = \Tr_{\bigoplus c_i}(pfp)=\Tr_{\bigoplus c_i}(f) := \sum_i \Tr^\cC_{c_i} (p_i f p_i^\dag)
\,,
\qquad
\forall\, f=pfp \in \End\left(\bigoplus c_i\right).
$$
\end{ex}

\begin{ex}
Given an $\rmH^*$-algebra $(A,\Tr_A)$, $\Mod^\dag(A,\Tr_A)$ from Definition \ref{def:hstar-module} is a 2-Hilbert space
which is isometrically equivalent, as a 2-Hilbert space, to the Cauchy completion of $(\rmB A,\Tr^{\rmB A})$.

Moreover, every 2-Hilbert space is of this form.
Let $\Irr(\cC)$ denote a set of representatives for the simple objects of $\cC$, and 
set $X:=\bigoplus_{c\in\Irr(\cC)} c$.
Then $(A:=\End(X), \Tr_A:=\Tr^\cC_X)$ is an $\rmH^*$-algebra such that $\cC\cong \Mod^\dag(A,\Tr_A)$ isometrically.
\end{ex}

\begin{ex}
\label{ex:FunCatsArePre2Hilbs}
Suppose $(\cC,\Tr^\cC)$ is a 2-Hilbert space and $(\cD,\Tr^\cD)$ is a pre-2-Hilbert space.
Then $\Fun^\dag(\cC\to \cD)$ is again a pre-2-Hilbert space when equipped with the trace
$$
\Tr_F(\rho:F\Rightarrow F)\coloneqq \sum_{s\in\Irr(\cC)} d_s\cdot \Tr^\cD_{F(s)}(\rho_s).
$$
Indeed, if $F,G: \cC\to \cD$ and $\rho:F\Rightarrow G$ and $\sigma:G\Rightarrow F$, then
\begin{align*}
\Tr_F(\sigma\cdot\rho)
&=
\sum_{s\in\Irr(\cC)} d_s\cdot\Tr^\cD_{F(s)}((\sigma\cdot\rho)_s)
=
\sum_{s\in\Irr(\cC)} d_s\cdot\Tr^\cD_{F(s)}(\sigma_s\circ \rho_s)
\\&=
\sum_{s\in\Irr(\cC)} d_s\cdot\Tr^\cD_{G(s)}(\rho_s\circ \sigma_s)
=
\sum_{s\in\Irr(\cC)} d_s\cdot\Tr^\cD_{G(s)}((\rho\cdot\sigma)_s)
=
\Tr_G(\rho\cdot \sigma).
\end{align*}
The scalars $d_s$ are inserted into the formula above
for a unitary Yoneda result \eqref{eq:UnitaryYoneda} below.
When $\cD$ is Cauchy complete, i.e.~a 2-Hilbert space, so is $\Fun^\dag(\cC \to \cD)$.
\end{ex}

\begin{rem}
\label{rem:A==Hom(Hilb->A)Isometric}
The equivalence $(\cA,\Tr^\cA) \cong \Fun^\dag(\Hilb \to \cA)$ given by $a \mapsto (H\mapsto H\otimes a)$ is isometric for $(\cA,\Tr^\cA)$ a pre-2-Hilbert space.
Indeed, for $f\in\cA(a\to a)$, by Example \ref{ex:FunCatsArePre2Hilbs}, we have 
\[\Tr^{\Fun^\dag(\Hilb \to \cA)}_{-\otimes a}(-\otimes f) = d_{\bbC} \cdot \Tr^\cA_{\bbC\otimes a}(\bbC\otimes f) = \Tr^{\cA}_a(f).
\]
\end{rem}

We now quickly recall the unitary Yoneda results for 2-Hilbert spaces from \cite[\S2.1]{MR4750417}.
As noted in \cite[Rem.~3.61 and footnote]{MR4598730}, and implicitly used in \cite{MR3687214,MR3948170}, a unitary trace $\Tr^\cC$
and the associated inner products and their compatibility with $\dag$ in \eqref{eq:InnerProductOnPre2Hilb}
makes the Yoneda embedding 
$$
\cC\hookrightarrow \Fun^\dag(\cC^{\op}\to \Hilb)
\qquad\text{given by}\qquad
c\mapsto \cC(-\to c)
$$ 
a dagger functor.
When $(\cC,\Tr^\cC)$ is a 2-Hilbert space and $\Irr(\cC)$ is a set of representatives of the simple objects of $\cC$, 
we have a canonical isometric equivalence 
\begin{equation}
\label{eq:UnitaryYoneda}
(\cC,\Tr^\cC)
\cong 
(\Fun^\dag(\cC^{\op}\to \Hilb),\Tr^{\Fun})
\end{equation}
where for $F: \cC^{\op}\to \Hilb$,
$$
\Tr^{\Fun}_F(\rho:F\Rightarrow F)
:=
\sum_{s\in\Irr(\cC)} d_s\cdot \Tr^{\Hilb}_{F(s)}(\rho_s)
$$
as in Example \ref{ex:FunCatsArePre2Hilbs}.
We get a canonical unitary natural isomorphism
\begin{equation}
\label{eq:GeneralizedElements}
\upsilon_c: c\longmapsto \bigoplus_{s\in\Irr(\cC)}d_s^{-1}\cC(s\to c)\otimes s,
\end{equation}
where the direct sum on the right hand side is the orthogonal direct sum.
The scalar $d_s^{-1}$ indicates that the inner product on $\cC(s\to c)$ is scaled by $d_s^{-1}$, i.e., 
$$
\langle f|g\rangle_{d_s^{-1}\cC(s\to c)}
=
d_s^{-1}
\langle f|g\rangle_{\cC(s\to c)}
=
d_s^{-1}
\Tr^\cC(f^\dag\circ g).
$$
We call the elements of $d_s^{-1}\cC(s\to c)$ on the right hand side of \eqref{eq:GeneralizedElements} the \emph{generalized elements} of $c\in\cC$.

\subsection{Unitary dual functors, spherical weights, and unitary adjunction} 

We refer the reader to \cite{MR4261588} for the basics of 2-category theory.
In this article, we make heavy use of the graphical calculus for 2-categories; we refer the reader to \cite[\S8.1.2]{MR3971584} for the basics of the graphical calculus.
Our conventions are as follows:
\begin{itemize}
    \item We write 2-composition vertically from \emph{bottom-to-top}.
    \item We write 1-composition horizontally from \emph{left-to-right}.    
\end{itemize}

\begin{defn}\label{def:udf}
A \emph{unitary adjoint functor} (UAF) for a rigid $\dag$-2-category $\fX$ 
is a choice of $\ev_X,\coev_X$ for every 1-morphism $X:a\to b$ such that the induced functor $\vee:\fX\to \fX^{1,2-\op}$ given by
\[
(f\colon X \Rightarrow Y)
\,\longmapsto\,
\tikzmath{
\begin{scope}
\clip[rounded corners = 5pt] (-1.3,-.9) rectangle (1.3,.9);
\fill[fill=\brColor] (-1.3,-.9) rectangle (1.3,.9);
\fill[\arColor] (-.6,-.9) -- (-.6,.3) arc (180:0:.3) -- (0,-.3) arc (-180:0:.3) -- (.6,.9) -- (1.3,.9) -- (1.3,-.9);
\end{scope}
\draw[thick] (0,.3) node[right, yshift=.2cm]{$\scriptstyle Y$} arc (0:180:.3cm) --node[left]{$\scriptstyle Y^\vee$} (-.6,-.9);
\draw[thick] (0,-.3) node[left, yshift=-.2cm]{$\scriptstyle X$} arc (-180:0:.3cm) --node[right]{$\scriptstyle X^\vee$} (.6,.9);
\roundNbox{unshaded}{(0,0)}{.3}{0}{0}{$f$}
}
\qquad\qquad
\tikzmath{\fill[fill=\arColor, rounded corners=5] (-.3,-.3) rectangle (.3,.3);}=a
\qquad
\tikzmath{\fill[fill=\brColor, rounded corners=5] (-.3,-.3) rectangle (.3,.3);}=b
\]
is dagger and its canonical tensorators are unitary.
Here, we denote $\ev_a^\dag$ and $\coev_a^\dag$ as usual in the graphical calculus by cups and caps.
We can also denote their adjoints as their vertical reflections using the usual graphical calculus for daggers, which allows us to avoid drawing coupons labelled by $\ev^\dag$ and $\coev^\dag$.
\[
\coev_X
=
\tikzmath{
\fill[\arColor, rounded corners = 5pt]  (-.2,-.5) rectangle (.8,0);
\filldraw[fill = \brColor, thick]  (0,0) node[above]{$\scriptstyle X$} arc (-180:0:.3) node[above]{$\scriptstyle X^\vee$};
}
\qquad\qquad
\coev^\dag_X
=
\tikzmath{
\fill[\arColor, rounded corners = 5pt]  (-.2,-.5) rectangle (.8,0);
\filldraw[fill = \brColor, thick] (0,-.5) node[below]{$\scriptstyle X$}  arc (180:0:.3) node[below]{$\scriptstyle X^\vee$};
}
\qquad\qquad
\ev^\dag_X
=
\tikzmath{
\fill[\brColor, rounded corners = 5pt]  (-.2,-.5) rectangle (.8,0);
\filldraw[fill = \arColor, thick]  (0,0) node[above]{$\scriptstyle X^\vee$} arc (-180:0:.3) node[above]{$\scriptstyle X$};
}
\qquad\qquad
\ev_X
=
\tikzmath{
\fill[\brColor, rounded corners = 5pt]  (-.2,-.5) rectangle (.8,0);
\filldraw[fill = \arColor, thick] (0,-.5) node[below]{$\scriptstyle X^\vee$} arc (180:0:.3) node[below]{$\scriptstyle X$};
}
\]
\end{defn}

\begin{rem}
Although all adjoint functors are canonically equivalent to each other in a contractible way, this equivalence need not be unitary.
Thus a UAF is structure, not property, as in \cite{MR4133163}.
\end{rem}

\begin{rem}
The choice of a UAF on a rigid $\dag$-2-category determines a canonical iconic unitary pivotal structure given by
\[
\varphi_X:=
\tikzmath{
\begin{scope}
\clip[rounded corners = 5pt] (-1.3,-.9) rectangle (1.3,.9);
\fill[fill=\arColor] (-1.3,-.9) rectangle (1.3,.9);
\fill[\brColor] (-.6,-.9) -- (-.6,.3) arc (180:0:.3) -- (0,-.3) arc (-180:0:.3) -- (.6,.9) -- (1.3,.9) -- (1.3,-.9);
\end{scope}
\draw[thick] (0,.3) arc (0:180:.3cm) -- node[left, yshift=-.2cm,xshift=.1cm]{$\scriptstyle X$} (-.6,-.9);
\draw[thick] (0,.3) -- (0,-.3) arc (-180:0:.3cm) -- node[right,xshift=-.1cm]{$\scriptstyle X^{\vee\vee}$} ((.6,.9);
\roundNbox{fill=white}{(-.3,.4)}{.3}{.3}{.3}{$\scriptstyle\coev^\dag_X$};
\roundNbox{fill=white}{(.3,-.4)}{.3}{.3}{.3}{$\scriptstyle\coev_{X^\vee}$};
}\,.
\]
Thus the following maps are tracial.
\begin{align*}
\tr^\vee_L \colon \End({}_a X_b) &\to \End(1_a) &&& \tr^\vee_R \colon \End({}_a X_b) &\to \End(1_b)\\
(f \colon X \Rightarrow X) &\mapsto 
\tikzmath{
\fill[rounded corners = 5pt, \arColor] (1.2,-.9) rectangle (-.6,.9);
\fill[\brColor]  (0,.3) arc (180:0:.3cm) -- (.6,-.3) arc (0:-180:.3cm);
\draw[thick] (0,.3) node[left,yshift=.2cm]{$\scriptstyle X$} arc (180:0:.3cm) --node[right]{$\scriptstyle X^\vee$} (.6,-.3) arc (0:-180:.3cm) node[left,yshift=-.2cm]{$\scriptstyle X$};
\roundNbox{fill=white}{(0,0)}{.3}{0}{0}{$f$}
}
&&&
(f \colon X \Rightarrow X) &\mapsto 
\tikzmath{
\fill[rounded corners = 5pt, \brColor] (-1.2,-.9) rectangle (.6,.9);
\fill[\arColor] (0,.3) arc (0:180:.3cm) -- (-.6,-.3) arc (-180:0:.3cm);
\draw[thick] (0,.3) node[right,yshift=.2cm]{$\scriptstyle X$} arc (0:180:.3cm) --node[left]{$\scriptstyle X^\vee$} (-.6,-.3) arc (-180:0:.3cm) node[right,yshift=-.2cm]{$\scriptstyle X$};
\roundNbox{fill=white}{(0,0)}{.3}{0}{0}{$f$}
}
\end{align*}
\end{rem}

\begin{ex}
Recall that a \emph{unitary multifusion category} is a multifusion category $\cC$ equipped with a dagger structure $\dag$ such that the coherence isomorphisms are all unitary.
We may view $\cC$ as a $\rmC^*$-2-category $\rmB \cC$ with one object.
A UAF on $\rmB \cC$ is exactly the notion of a unitary dual functor (UDF) on $\cC$ from \cite{MR4133163}.

By \cite[Theorem A]{MR4133163}, UDFs on $\cC$ are classified by groupoid homomorphisms from the so-called \emph{universal grading groupoid} $\cU_\cC$ of $\cC$ to $\bbR_{>0}$.
Since $\bbR_{>0}$ has no torsion and $\cU_\cC$ is finite, we see that 
$$
\Hom(\cU_\cC\to \bbR_{>0})
\cong
\Hom(\cM_\cC\to \bbR_{>0})
$$
where $\cM_\cC$ is the `matrix groupoid' of $\cC$ which has one object for every 
simple summand $1_i$ of $1_\cC$ and a unique morphism from $1_i\to 1_j$ whenever $1_i\otimes \cC\otimes 1_j\neq 0$.
\end{ex}

\begin{ex}[{\cite[Def 2.2]{MR4750417}}]\label{ex:UnitaryAdjunction}
Suppose $\cA, \cB\in2\Hilb$  and $F: \cA\to \cB,\ G: \cB\to \cA$ are functors.
We say that $F,G$ are \emph{unitary adjoints}, denoted 
$F\dashv^\dag G$, if there is a family of natural unitary isomorphisms
$$
\cB(F(a)\to b)
\cong
\cA(a\to G(b)).
$$
Observe that this automatically implies that $F,G$ are linear dagger functors, and we also have $G\dashv^\dag F$.
Diagrammatically, the above unitary isomorphism is denoted by

$$
\tikzmath{
\begin{scope}
\clip[rounded corners = 5pt] (-1,-.6) rectangle (.5,.6);
\fill[\brColor]  (-1,-.6) -- (.1,-.6) -- (.1,0) -- (0,0) -- (0,.6) -- (-1,.6) -- cycle;
\fill[\arColor] (-1,-.6) -- (-.1,-.6) -- (-.1,0) -- (0,0) -- (0,.6) -- (-1,.6) -- cycle;
\end{scope}
\draw[thick] (-.1,-.6) -- node[left]{$\scriptstyle F$} (-.1,-.3);
\draw[thick] (.1,-.6) -- node[right]{$\scriptstyle a$} (.1,-.3);
\draw[thick] (0,.3) -- node[right]{$\scriptstyle b$} (0,.6);
\roundNbox{unshaded}{(0,0)}{.3}{0}{0}{$f$}
}
\mapsto 
\tikzmath{
\begin{scope}
\clip[rounded corners = 5pt] (-1,-.6) rectangle (.5,.6);
\fill[\arColor] (-1,-.6) -- (-.1,-.6) -- (-.1,0) -- (0,0) -- (0,.6) -- (-1,.6) -- cycle;
\fill[\brColor] (-1,-.6) -- (.1,-.6) -- (.1,0) -- (0,0) -- (0,.6) -- (-.1,-.3) arc (0:-180:.2) -- (-.5,.6) -- (-1,.6) -- cycle;
\end{scope}
\draw[thick] (-.1,-.3) arc (0:-180:.2) -- node[left]{$\scriptstyle G$} (-.5,.6);
\draw[thick] (.1,-.6) -- node[right]{$\scriptstyle a$} (.1,-.3);
\draw[thick] (0,.3) -- node[right]{$\scriptstyle b$} (0,.6);
\roundNbox{unshaded}{(0,0)}{.3}{0}{0}{$f$}
},
\qquad 
\tikzmath{
\fill[rounded corners=5pt, fill=gray!50] (0,0) rectangle (.6,.6);
}
=
\cA
\qquad
\tikzmath{
\fill[rounded corners=5pt, fill=gray!20] (0,0) rectangle (.6,.6);
}
=
\cB
\qquad
\tikzmath{
\draw[rounded corners=5pt,dotted] (0,0) rectangle (.6,.6);
}
=
\Hilb
$$
where we view $a\in\cA$ as a functor $\Hilb\to \cA$, and similarly $b:\Hilb\to \cB$.\footnote{
While we usually denote 1-composition in a 2-category from left-to-right, for the 2-category $2\Hilb$, we read composition of functors from right-to-left.}
Unitarity is equivalent to the condition that
$$
\Tr_{F(a)}^\cB 
\left(
\tikzmath{
\begin{scope}
\clip[rounded corners = 5pt] (-1,-.6) rectangle (.5,1.5);
\fill[\brColor]  (-1,-.6) -- (.1,-.6) -- (.1,0) -- (0,0) -- (0,.9) -- (.1,.9) -- (.1,1.5) -- (-1,1.5) -- cycle;
\fill[\arColor] (-1,-.6) -- (-.1,-.6) -- (-.1,0) -- (0,0) -- (0,.9) -- (-.1,.9) -- (-.1,1.5) -- (-1,1.5) -- cycle;
\end{scope}
\draw[thick] (-.1,-.6) -- node[left]{$\scriptstyle F$} (-.1,-.3);
\draw[thick] (.1,-.6) -- node[right]{$\scriptstyle a$} (.1,-.3);
\draw[thick] (0,.3) -- node[right]{$\scriptstyle b$} (0,.6);
\roundNbox{unshaded}{(0,0)}{.3}{0}{0}{$f$}
\roundNbox{unshaded}{(0,.9)}{.3}{0}{0}{$f^\dag$}
\draw[thick] (-.1,1.5) -- node[left]{$\scriptstyle F$} (-.1,1.2);
\draw[thick] (.1,1.5) -- node[right]{$\scriptstyle a$} (.1,1.2);
}
\right)
=
\Tr_a^\cA
\left(
\tikzmath{
\begin{scope}
\clip[rounded corners = 5pt] (-1,-.6) rectangle (.5,1.5);
\fill[\brColor]  (-1,-.6) -- (.1,-.6) -- (.1,0) -- (0,0) -- (0,.9) -- (.1,.9) -- (.1,1.5) -- (-1,1.5) -- cycle;
\fill[\arColor] (-.1,-.3) arc (0:-180:.2) -- (-.5,1.2) arc (180:0:.2) -- (0,.9) -- (0,0) -- cycle;
\end{scope}
\draw[thick] (-.1,-.3) arc (0:-180:.2) -- node[left]{$\scriptstyle G$} (-.5,1.2) arc (180:0:.2);
\draw[thick] (.1,-.6) -- node[right]{$\scriptstyle a$} (.1,-.3);
\draw[thick] (0,.3) -- node[right]{$\scriptstyle b$} (0,.6);
\roundNbox{unshaded}{(0,0)}{.3}{0}{0}{$f$}
\roundNbox{unshaded}{(0,.9)}{.3}{0}{0}{$f^\dag$}
\draw[thick] (.1,1.5) -- node[right]{$\scriptstyle a$} (.1,1.2);
}
\right)
\qquad
\forall a \in \cA,\; 
\forall b \in \cB,\; 
\forall f \in \cB(F(a) \to b).
$$
Unitary adjoints in $2\Hilb$ exist and are unique up to unique unitary natural isomorphism.
Indeed, for $F: \cA\to \cB$, $G$ is determined by unitary Yoneda \eqref{eq:UnitaryYoneda} by
\begin{equation}
\label{eq:UnitaryAdjointFormulaViaYoneda}    
G(b):= \bigoplus_{a\in\Irr(\cA)} d_a^{-1}\cB(F(a)\to b)\otimes a.
\end{equation}
By \cite[Prop 2.8]{MR4750417}, unitary adjunction is a UAF on $2\Hilb$.
\end{ex}

\begin{sub-ex}
\label{sub-ex:UnitaryAdjointOf-otimesA}
Suppose $(A,\Tr^\cA)$ is a 2-Hilbert space and
consider the isometric equivalence $\cA\cong \Fun^\dag(\Hilb\to \cA)$ from Remark \ref{rem:A==Hom(Hilb->A)Isometric}.
By \eqref{eq:UnitaryAdjointFormulaViaYoneda}, the unitary adjoint of $-\otimes a:\Hilb\to \cA$ is given by $a'\mapsto \cA(a\to a')$.
Moreover, $\eta_\bbC$, the leg of the unit of the adjunction at $\bbC\in\Hilb$, is the mate of $\id_{a}$ under the unitary isomorphism
$$
\cA(\bbC\otimes a\to a) \cong \Hilb(\bbC\to \cA(a\to a)).
$$
\end{sub-ex}

\begin{defn}
A \emph{spherical weight} for a $\dag$-2-category $\fX$ equipped with a unitary adjoint functor $\vee$ is a family of faithful positive linear functionals $\Psi_a: \End_\fX(1_a)\to \bbC$ for all $a\in\fX$ 
satisfying
$$
\Psi_b\left(\,
\tikzmath{
\fill[rounded corners = 5pt, \brColor] (-1.2,-.9) rectangle (.6,.9);
\fill[\arColor] (0,.3) arc (0:180:.3cm) -- (-.6,-.3) arc (-180:0:.3cm);
\draw[thick] (0,.3) node[right,yshift=.2cm]{$\scriptstyle X$} arc (0:180:.3cm) --node[left]{$\scriptstyle X^\vee$} (-.6,-.3) arc (-180:0:.3cm) node[right,yshift=-.2cm]{$\scriptstyle X$};
\roundNbox{fill=white}{(0,0)}{.3}{0}{0}{$f$}
}
\,\right)
=
\Psi_a\left(\,
\tikzmath{
\fill[rounded corners = 5pt, \arColor] (1.2,-.9) rectangle (-.6,.9);
\fill[\brColor]  (0,.3) arc (180:0:.3cm) -- (.6,-.3) arc (0:-180:.3cm);
\draw[thick] (0,.3) node[left,yshift=.2cm]{$\scriptstyle X$} arc (180:0:.3cm) --node[right]{$\scriptstyle X^\vee$} (.6,-.3) arc (0:-180:.3cm) node[left,yshift=-.2cm]{$\scriptstyle X$};
\roundNbox{fill=white}{(0,0)}{.3}{0}{0}{$f$}
}
\,\right)
\qquad\qquad
\begin{aligned}
&\forall\,a,b\in\fX
\\
&\forall\, {}_aX_b\in\fX(a\to b)
\\
&\forall\, f\in\End_\fX({}_aX_b)
\end{aligned}
\qquad\qquad
\begin{aligned}
\tikzmath{\fill[fill=\arColor, rounded corners=5] (-.3,-.3) rectangle (.3,.3);}&=a
\\
\tikzmath{\fill[fill=\brColor, rounded corners=5] (-.3,-.3) rectangle (.3,.3);}&=b.
\end{aligned}
$$
\end{defn}

\begin{ex}
\label{ex:WeightDeterminesUDF}
Suppose $\cC$ is a unitary multifusion category.
By \cite{MR4133163} (see \cite[Lem.~2.13]{MR4750417}),
\begin{itemize}
\item 
If $\cC$ is indecomposable, each UDF on $\cC$ admits a unique spherical faithful state
$\psi: \cC(1_\cC\to 1_\cC)\to \bbC$ normalized so that $\psi(\id_{1_\cC})=1$.
Delooping, $\Psi_\star:=\psi$ is a spherical weight on $\rmB\cC$.
\item 
When $\cC$ is arbitrary, for every faithful state $\psi$ on $\End_\cC(1_{\cC})$, there is a unique UDF for which $\psi$ is spherical.
Again, delooping, every choice of weight $\Psi_\star$ on $\rmB\cC$ admits a unique UAF for which $\Psi$ is spherical.
\end{itemize}
\end{ex}

\subsection{Semisimple 2-categories}

\begin{defn}\label{def:linkingE1algebra}
Given a 2-category $\fX$ and objects $a,b\in\fX$, the \emph{linking ($E_1$-)algebra} $\cL(a,b)$ is the monoidal category
$$
\cL(a,b):=
\begin{pmatrix}
\fX(a\to a) & \fX(b\to a)
\\
\fX(a\to b) & \fX(b\to b)
\end{pmatrix}
$$
where composition is component-wise, and tensor product is defined via a matrix multiplication formula.
Similarly, we can define the $n$-fold linking algebra $\cL(a_1,\dots, a_n)$ for every $a_1,\dots, a_n\in\cC$.

When $\fX$ is $\rmC^*$, so is every linking algebra with its obvious $\dag$-transpose dagger structure.
\end{defn}

\begin{defn}
A linear 2-category is called \emph{pre-semisimple} if all $n$-fold linking algebras are semisimple multitensor categories.
A pre-semisimple 2-category is called \emph{finite} if both
\begin{itemize}
\item 
all $n$-fold linking algebras are multifusion, and 
\item 
there is a global bound on the dimensions of $\End(1)$ for the centers of all linking algebras.
That is, there is a $K>0$ such that for any linking algebra $\cL=\cL(a_1,\dots, a_n)$, $\dim(\End(1_{Z(\cL)}))<K$.
\end{itemize}
\end{defn}

\begin{rem}
Given a pre-semisimple $2$-category $\fX$,
we say $a,b\in \fX$ are \emph{connected}
when $\fX(a\to b)\neq 0$. 
Since all 1-morphisms are dualizable and the 1-composite of non-zero 1-morphisms is non-zero,
connectedness is an equivalence relation.
The equivalence classes are called the \emph{components} of $\fX$.

If $a,b\in\fX$ are in distinct components, then 
the linking algebra $\cL=\cL(a,b)$ splits as a direct sum $\cL(a)\oplus\cL(b)$, i.e., it is a decomposable multitensor category.
In turn, the center $Z(\cL)$ is decomposable as a braided multifusion category, and 
the summands of $1_{Z(\cL)}$ correspond to the indecomposable fusion blocks of $\cL$.

If $\fX$ has $n$ connected components, then we can choose representatives $\{a_i\}$ of these components, and $\cL=\cL(a_1,\cdots,a_n)$ will decompose into $n$ separate multifusion blocks.  
It follows that $\dim\End(1_{Z(\cL)})=n$.  
In this way, the second condition on a finite pre-semisimple 2-category $\fX$ can be interpreted as saying that $\fX$ has no more than $K$ components.

Finally, we note that the semisimple completion of a component is connected, and thus $K$ is also an upper bound for the number of components of the semisimple completion of $\fX$.
\end{rem}

\begin{rem}
There should be a natural coinductive generalization of (finite) pre-semsimple $n$-category for all $n$.
For example, a pre-semisimple $3$-category is a 3-category whose $n$-fold linking algebras are all semisimple 2-categories.
A pre-semisimple 3-category is finite if all $n$-fold linking algebras are finite semisimple 2-categories and there is a global bound on the dimensions of $\End_{Z(\cL)}(\id_{1_\cL})$ for the centers of all linking algebras $\cL$.
\end{rem}

By \cite[\S 1.4.2]{1812.11933}, the (unital) condensation completion of the additive completion of a (finite) pre-semisimple 2-category is (finite) semisimple.
We now briefly recall these notions.
We begin with direct sums in 2-categories.

\begin{defn}\label{defn:directsumin2category}
A direct sum of objects $a_1,a_2$ in a 2-category consists of an object $a_1 \boxplus a_2$ and 1-morphisms $I_i \colon a_i \to a_1 \boxplus a_2$, $P_i \colon a_1 \boxplus a_2 \to a_i$ for $i=1,2$ such that 
$$I_i \otimes P_i \cong 1_{a_i} \quad\text{and}\quad (P_1 \otimes I_1) \oplus (P_2 \otimes I_2) \cong 1_{a_1 \boxplus a_2}.$$
We note that admitting these 2-isomorphisms is a property rather than structure; i.e. the space of choices of 1-morphisms and 2-morphisms satisfying these conditions is contractible.
\end{defn}

\begin{rem}
\label{rem:UpgradeDirectSums}
By \cite[Prop 1.1.3]{1812.11933}, the 2-isomorphisms in Definition \ref{defn:directsumin2category} can be chosen such that they witness the fact that $I_c$ is both a left and right adjoint to $P_c$. 
\end{rem}

We now discuss (unital) condensation algebras in 2-categories.
We warn the reader that condensations in this article are always assumed to be unital.

\begin{defn}
An algebra $(A,\mu\colon A\otimes A\to A,\iota\colon 1\to A)$ in a monoidal 1-category is said to be \emph{separable Frobenius} if is equipped with an $A$-$A$ bimodule map $\Delta:A\to A\otimes A$ which splits $\mu$ and $\Delta$ admits a counit $\epsilon$.
Diagrammatically, we represent
$$
\tikzmath{
\fill[\BColor, rounded corners=5pt ] (0,0) rectangle (.6,.6);
\draw[\QsColor,thick] (.3,0) -- (.3,.6);
}= A
\qquad
\tikzmath{
\fill[\BColor, rounded corners=5pt] (-.3,0) rectangle (.9,.6);
\draw[\QsColor,thick] (0,0) arc (180:0:.3cm);
\draw[\QsColor,thick] (.3,.3) -- (.3,.6);
}=\mu
\qquad
\tikzmath{
\fill[\BColor, rounded corners=5pt] (-.3,0) rectangle (.9,-.6);
\draw[\QsColor,thick] (0,0) arc (-180:0:.3cm);
\draw[\QsColor,thick] (.3,-.3) -- (.3,-.6);
}=\Delta
\qquad
\tikzmath{
\fill[\BColor, rounded corners=5pt] (-.3,-.3) rectangle (.3,.3);
\draw[\QsColor,thick] (0,0) -- (0,.3);
\filldraw[\QsColor] (0,0) circle (.05cm);
}=\iota
\qquad
\tikzmath{
\fill[\BColor, rounded corners=5pt] (-.3,-.3) rectangle (.3,.3);
\draw[\QsColor,thick] (0,0) -- (0,-.3);
\filldraw[\QsColor] (0,0) circle (.05cm);
}=\epsilon
$$
and we represent several axioms of a separable Frobenius algebra graphically as follows.
$$
\underbrace{
\tikzmath{
\fill[\BColor, rounded corners=5pt] (-.3,-.6) rectangle (1.5,.6);
\draw[\QsColor,thick] (0,-.6) -- (0,0) arc (180:0:.3cm) arc (-180:0:.3cm) -- (1.2,.6);
\draw[\QsColor,thick] (.3,.3) -- (.3,.6);
\draw[\QsColor,thick] (.9,-.3) -- (.9,-.6);
}
=
\tikzmath{
\fill[\BColor, rounded corners=5pt] (-.3,0) rectangle (.9,1.2);
\draw[\QsColor,thick] (0,0) arc (180:0:.3cm);
\draw[\QsColor,thick] (0,1.2) arc (-180:0:.3cm);
\draw[\QsColor,thick] (.3,.3) -- (.3,.9);
}
=
\tikzmath{
\fill[\BColor, rounded corners=5pt] (-.3,.6) rectangle (1.5,-.6);
\draw[\QsColor,thick] (0,.6) -- (0,0) arc (-180:0:.3cm) arc (180:0:.3cm) -- (1.2,-.6);
\draw[\QsColor,thick] (.3,-.3) -- (.3,-.6);
\draw[\QsColor,thick] (.9,.3) -- (.9,.6);
}
}_{\text{(Frobenius)}}
\qquad\qquad\qquad
\underbrace{
\tikzmath{
\draw[\QsColor,thick] (0,-.6) -- (0,-.3);
\draw[\QsColor,thick] (0,.6) -- (0,.3);
\draw[\QsColor,thick] (0,0) circle (.3cm);
}
=
\tikzmath{
\draw[\QsColor,thick] (0,-.6) -- (0,.6);
}
}_{\text{(Separable)}}
$$

    A (unital) \emph{condensation algebra} $({}_aA_a, \mu:A\otimes A\to A, \iota:1_a\to A)$ in a 2-category $\fX$ is a separable Frobenius algebra in the monoidal 1-category $\End_{\fX}(a)$.
\end{defn}

\begin{ex}
The monad associated to a separable adjunction is a condensation algebra.
In more detail,
given $a, b\in \fX$ and an adjunction ${}_aX_b\dashv{}_bY_a$, the monad associated to this adjunction is the algebra $({}_aX\otimes Y_a,\mu = X\otimes\epsilon\otimes Y,\eta\colon 1_a\to X\otimes Y)$ in $\End_{\fX}(a)$, where $\eta$ and $\epsilon$ are the unit and the counit of the adjunction.
The adjunction ${}_aX_b\dashv{}_bY_a$ is called \emph{separable} if $\epsilon$ admits a right inverse. 
\end{ex}

\begin{defn}
A \emph{splitting} for a condensation algebra $A=({}_aA_a, \mu:A\otimes A\to A, \iota:\id_a\to A)$ is separable adjunction ${}_aX_b\dashv{}_bY_a$ whose associated monad is isomorphic to $A$ as an algebra in $\End_{\fX}(a)$.
Whenever $X\dashv Y$ is a splitting for $A$, we say 
$X\dashv Y$ \emph{splits} $A$, or simply that $A$ \emph{splits}.

A 2-category $\fX$ is said to be \emph{condensation complete} if all condensation algebras split.
\end{defn}

\begin{defn}[{\cite[Def 1.4.1]{1812.11933}}]
    A semisimple 2-category is a pre-semisimple 2-category which is additive and condensation complete.
\end{defn}

\begin{ex}
Suppose $\cC$ is a multifusion category.
The \emph{delooping} $\rmB \cC$, which is the 2-category with one object $\star$ whose endomorphisms is $\cC$, is finite pre-semisimple.
The condensation completion of the additive completion of $\rmB \cC$ is equivalent to the semisimple 2-category $\Mod(\cC)$.
\end{ex}

\begin{defn}
The 3-category $3\Vect$ has objects semisimple 2-categories, 1-morphisms linear 2-functors, 2-morphisms 2-natural transformations, and 3-morphisms modifications.
By \cite{MR4372801}, 
the map $\cC\mapsto \Mod(\cC)$ is an equivalence $\mFC\to 3\Vect$, where $\mFC$ is the Morita 3-category of multifusion categories, with 1-morphisms bimodules categories, 2-morphisms bimodule functors, and 3-morphisms bimodule natural transformations.
\end{defn}

\section{\texorpdfstring{$\rmH^*$}{H*}-multifusion categories and \texorpdfstring{$\rmH^*$}{H*}-algebras}

In this section, we endow a unitary multifusion category $\cC$ with the extra structure of a UDF and a spherical weight, which gives it a canonical structure of a 2-Hilbert space.
We call such a triple an $\rmH^*$-multifusion category.
Unitary module categories are also equipped with unitary module traces in the sense of \cite{MR3019263,MR4598730} to endow them with the structure of 2-Hilbert spaces.
We then define the notion of an $\rmH^*$-algebra internal to $\cC$, and we prove each module category equipped with a trace arises as $\cC_A$ for an $\rmH^*$-algebra $A\in\cC$.
Here, we use unitary adjunction to endow the unitary internal hom with the structure of an $\rmH^*$-algebra.
An essential point here is to choose the correct module trace on $\cC_A$ for an $\rmH^*$-algebra 
and the correct UAF on the unitary 2-category $\HstarAlg(\cC)$ of $\rmH^*$-algebras and their bimodules
so that:
\begin{itemize}
\item
the unitary internal end $[A,A]$ is again unitarily isomorphic to $A$, and
\item
the unitary adjoint of $-\otimes_AM_B : \cC_A\to \cC_B$ is unitarily isomorphic to $-\otimes_BM^\vee_A$.
\end{itemize}

\subsection{\texorpdfstring{$\rmH^*$}{H*}-multifusion categories and their module categories}

\begin{defn}
    An \emph{$\rmH^*$-multifusion category} is a 
    triple $(\cC,\vee,\psi)$ where 
    $\cC$ is a unitary multifusion category, 
    $\vee$ is a UDF, and 
    $\psi$ is a spherical weight on $\rmB\cC$.
\end{defn}

\begin{rem}
An $\rmH^*$-multifusion category has a canonical 2-Hilbert space structure given by
$$
\Tr^\cC(f:c\to c)
:=
\psi\left(
\tikzmath{
\draw[thick] (0,.3) node[right,yshift=.2cm]{$\scriptstyle c$} arc (0:180:.3cm) --node[left]{$\scriptstyle c^\vee$} (-.6,-.3) arc (-180:0:.3cm) node[right,yshift=-.2cm]{$\scriptstyle c$};
\roundNbox{fill=white}{(0,0)}{.3}{0}{0}{$f$}
}
\,\right)
=
\psi\left(\,
\tikzmath{
\draw[thick] (0,.3) node[left,yshift=.2cm]{$\scriptstyle c$} arc (180:0:.3cm) --node[right]{$\scriptstyle c^\vee$} (.6,-.3) arc (0:-180:.3cm) node[left,yshift=-.2cm]{$\scriptstyle c$};
\roundNbox{fill=white}{(0,0)}{.3}{0}{0}{$f$}
}
\right).
$$
In fact, the data of an H*-multifusion category is precisely that of a unitary multifusion category $(\cC,\vee)$ equipped with a 2-Hilbert space structure such that 
$$
- \otimes c \dashv^\dag - \otimes c^\vee \quad\text{and}\quad c \otimes - \dashv^\dag c^\vee \otimes - \qquad\text{for all } c \in \cC.
$$
\end{rem}

\begin{nota}
We write $1_\cC=\bigoplus_{i=1}^{k} 1_i$ for a decomposition into simples, $p_i$ for the orthogonal projection onto $1_i$ in $\End_\cC(1_\cC)$
As usual, we write $d_c=\Tr^\cC(\id_c)$ for $c\in\Irr(\cC)$
, and we abbreviate $d_{1_i} =: d_i$.
For $c\in\Irr(\cC)$, we write 
$s(c)=i$ and $t(c)=j$ for the unique $1\leq i,j\leq k$ such that $c =1_i\otimes c\otimes 1_j$.
Observe that the coevaluation $1_\cC\to c\otimes c^\vee$ factors through $1_{s(c)}$ and the evaluation $c^\vee\otimes c\to 1_\cC$ factors through $1_{t(c)}$.
This implies that
$$
d_{s(c)}\dim_L^\vee(c)
=
(\psi\circ \tr_L^\vee)(\id_c)
=
d_c
=
(\psi\circ \tr_R^\vee)(\id_c)
=
d_{t(c)}\dim_R^\vee(c).
$$
In diagrams, we have
\begin{equation}
\label{eq:PopBubbles}
\tikzmath{
\fill[fill=\brColor, rounded corners = 5pt] (-.7,-.5) rectangle (.7,.5);
\filldraw[blue, thick, fill=\arColor] (0,0) circle (.3cm);
\node[blue] at (-.5,0) {$\scriptstyle c$};
\node[blue] at (.5,0) {$\scriptstyle c^\vee$};
}
=
\frac{d_c}{d_{s(c)}}
\cdot
\tikzmath{
\fill[fill=\brColor, rounded corners = 5pt] (-.5,-.5) rectangle (.5,.5);
}
\qquad\text{and}\qquad
\tikzmath{
\fill[fill=\arColor, rounded corners = 5pt] (-.7,-.5) rectangle (.7,.5);
\filldraw[blue, thick, fill=\brColor] (0,0) circle (.3cm);
\node[blue] at (-.5,0) {$\scriptstyle c^\vee$};
\node[blue] at (.5,0) {$\scriptstyle c$};
}
=
\frac{d_c}{d_{t(c)}}
\cdot
\tikzmath{
\fill[fill=\arColor, rounded corners = 5pt] (-.5,-.5) rectangle (.5,.5);
}
\qquad\qquad
\begin{aligned}
\tikzmath{\fill[fill=\brColor, rounded corners=5] (-.3,-.3) rectangle (.3,.3);}&=1_{s(c)}
\\
\tikzmath{\fill[fill=\arColor, rounded corners=5] (-.3,-.3) rectangle (.3,.3);}&=1_{t(c)}.
\end{aligned}
\end{equation}
\end{nota}

We now discuss module categories for $\rmH^*$-multifusion categories.
In this section, we discuss right modules, and we leave the notions for left modules to the reader.
We will, however, use left $\cC$-modules frequently later in this article, and results proved for right $\cC$-modules will also be used for left $\cC$-modules without further comment.

\begin{defn}[\cite{MR3019263,MR4598730}]
Let $\cC$ be a unitary multifusion category equipped with a UDF $\vee$.
Let $\cM$ be a unitary module category for a $\cC$.
A \emph{(unitary) module trace} for $\cM$ is a unitary trace $\Tr^\cM$ on $\cM$ such that
        for all $c \in \cC$ (represented with a blue string), $m \in \cM$ (represented with a black string), and maps $f:m\lhd c \to m\lhd c$, we have
        \[ 
        \Tr^\cM_m\left(
        \tikzmath{
        \draw[thick, blue] (.15,.3) arc(180:0:.2cm) --node[right]{$\scriptstyle c^\vee$} (.55,-.3) arc (0:-180:0.2);
        \draw[thick, black] (-.15,-.7) --node[left]{$\scriptstyle m$} (-.15,-.3);
        \draw[thick, black] (-.15,.7) --node[left]{$\scriptstyle m$} (-.15,.3);
        \roundNbox{unshaded}{(0,0)}{.3}{0}{0}{$f$};
        }
        \,\right) =
        \Tr^\cM_{m\lhd c}\left(
        \tikzmath{
        \draw[thick, black] (-.15,-.7) --node[left]{$\scriptstyle m$} (-.15,-.3);
        \draw[thick, black] (-.15,.7) --node[left]{$\scriptstyle m$} (-.15,.3);
        \draw[thick, blue] (.15,-.7) --node[right]{$\scriptstyle c$} (.15,-.3);
        \draw[thick, blue] (.15,.7) --node[right]{$\scriptstyle c$} (.15,.3);
        \roundNbox{unshaded}{(0,0)}{.3}{0}{0}{$f$};
        }
        \,\right).
        \]
\end{defn}

\begin{rem}
The data of a unitary $\cC$-module category $\cM$ equipped with a unitary module trace $\Tr^\cM$ is precisely the data of a 2-Hilbert space $(\cM,\Tr^\cM)$ equipped with a $\cC$-action on $\cM$ such that 
$$- \triangleleft c \dashv_\dag - \triangleleft c^\vee \qquad \forall c \in \cC.$$
\end{rem}

\begin{defn}
Given an $\rmH^*$-multifusion category $(\cC,\vee,\psi)$,
the unitary 2-category $\Mod^\dag(\cC)$ has:
\begin{itemize}
\item
objects are pairs $(\cM,\Tr^\cM)$ where $\cM$ is a unitary $\cC$-module category and $\Tr^\cM$ is a unitary $\cC$-module trace
\item
1-morphisms are unitary $\cC$-module functors, and
\item
2-morphisms are natural transformations.
\end{itemize}
\end{defn}

The next result is similar to results in \cite[\S2.3]{MR4750417}.

\begin{lem}
\label{lem:UAFonModC}
Unitary adjunction on $\Mod^\dag(\cC)$ is a well-defined UAF.
\end{lem}
\begin{proof}
Under the 2-Hilbert space structure on a $\cC$-module $(\cM,\Tr^\cM)$, the functors $-\lhd c: \cM\to \cM$ and $-\lhd c^\vee:\cM\to \cM$ are unitary adjoints. 
Given a unitary $\cC$-module functor $F\colon \cM_\cC \to \cN_\cC$, we can view the unitary modulator $\omega_{m,c}\colon F(m \lhd c) \to F(m) \lhd c$ as a unitary crossing like a `half-brading' for $F$ with every $-\lhd c: \cM\to \cM$ and $-\lhd c: \cN\to \cN$:

\[
\omega^F_{-,c}
=
\tikzmath{
\begin{scope}
\clip[rounded corners=5pt] (-.7,-.4) rectangle (.7,.4);
\fill[gray!30] (-.4,-.4) .. controls ++(90:.4cm) and ++(270:.4cm) .. (.4,.4) -- (.7,.4) -- (.7,-.4);
\fill[gray!50] (-.4,-.4) .. controls ++(90:.4cm) and ++(270:.4cm) .. (.4,.4) -- (-.7,.4) -- (-.7,-.4);
\end{scope}
\draw[thick, blue] (.4,-.4) node[below]{$\scriptstyle c$} .. controls ++(90:.4cm) and ++(270:.4cm) .. (-.4,.4) node[above]{$\scriptstyle c$};
\draw (-.4,-.4) node[below]{$\scriptstyle F$} .. controls ++(90:.4cm) and ++(270:.4cm) .. (.4,.4) node[above]{$\scriptstyle F$};
}
\qquad\qquad
(\omega^F_{-,c})^{-1}
=
\tikzmath{
\begin{scope}
\clip[rounded corners=5pt] (-.7,-.4) rectangle (.7,.4);
\fill[gray!30] (.4,-.4) .. controls ++(90:.4cm) and ++(270:.4cm) .. (-.4,.4) -- (.7,.4) -- (.7,-.4);
\fill[gray!50] (.4,-.4) .. controls ++(90:.4cm) and ++(270:.4cm) .. (-.4,.4) -- (-.7,.4) -- (-.7,-.4);
\end{scope}
\draw[thick, blue] (-.4,-.4) node[below]{$\scriptstyle c$} .. controls ++(90:.4cm) and ++(270:.4cm) .. (.4,.4) node[above]{$\scriptstyle c$};
\draw (.4,-.4) node[below]{$\scriptstyle F$} .. controls ++(90:.4cm) and ++(270:.4cm) .. (-.4,.4) node[above]{$\scriptstyle F$};
}
\]
The unitary adjoint $F^*\colon \cN \to \cM$ has unitary modulator
\[
\tikzmath{
\begin{scope}
\clip[rounded corners=5pt] (-.7,-.4) rectangle (.7,.4);
\fill[gray!50] (-.4,-.4) .. controls ++(90:.4cm) and ++(270:.4cm) .. (.4,.4) -- (.7,.4) -- (.7,-.4);
\fill[gray!30] (-.4,-.4) .. controls ++(90:.4cm) and ++(270:.4cm) .. (.4,.4) -- (-.7,.4) -- (-.7,-.4);
\end{scope}
\draw[thick, blue] (.4,-.4) node[below]{$\scriptstyle c$} .. controls ++(90:.4cm) and ++(270:.4cm) .. (-.4,.4) node[above]{$\scriptstyle c$};
\draw (-.4,-.4) node[below]{$\scriptstyle F^*$} .. controls ++(90:.4cm) and ++(270:.4cm) .. (.4,.4) node[above]{$\scriptstyle F^*$};
}
:=
\omega^{F^*}_{-,c} :=
\tikzmath{
\begin{scope}
\clip[rounded corners=5pt] (-1.9,-1.2) rectangle (1.9,1.2);
\fill[gray!50] (-1,1.2) -- (-1,-.4) arc(-180:0:.3cm) .. controls ++(90:.4cm) and ++(270:.4cm) .. (.4,.4) arc (180:0:.3cm) -- (1,-1.2)  -- (1.9,-1.2) -- (1.9,1.2);
\fill[gray!30] (-1,1.2) -- (-1,-.4) arc(-180:0:.3cm) .. controls ++(90:.4cm) and ++(270:.4cm) .. (.4,.4) arc (180:0:.3cm) -- (1,-1.2) -- (-1.9,-1.2) -- (-1.9,1.2);
\end{scope}
\draw[thick, blue] (-1.6,1.2) node[above]{$\scriptstyle c$} -- (-1.6,-.4) .. controls ++(270:.8cm) and ++(270:.8cm) .. (.4,-.4) .. controls ++(90:.4cm) and ++(270:.4cm) .. (-.4,.4) .. controls ++(90:.8cm) and ++(90:.8cm) .. (1.6,.4) -- (1.6,-1.2) node[below]{$\scriptstyle c$};
\draw (-1,1.2) node[above]{$\scriptstyle F^*$} -- (-1,-.4) arc(-180:0:.3cm) node[right]{$\scriptstyle F$} .. controls ++(90:.4cm) and ++(270:.4cm) .. (.4,.4) node[left]{$\scriptstyle F$} arc (180:0:.3cm) -- (1,-1.2) node[below]{$\scriptstyle F^*$};
}
\qquad\text{and}\qquad
\tikzmath[yscale=-1]{
\begin{scope}
\clip[rounded corners=5pt] (-.7,-.4) rectangle (.7,.4);
\fill[gray!50] (-.4,-.4) .. controls ++(90:.4cm) and ++(270:.4cm) .. (.4,.4) -- (.7,.4) -- (.7,-.4);
\fill[gray!30] (-.4,-.4) .. controls ++(90:.4cm) and ++(270:.4cm) .. (.4,.4) -- (-.7,.4) -- (-.7,-.4);
\end{scope}
\draw[thick, blue] (.4,-.4) node[above]{$\scriptstyle c$} .. controls ++(90:.4cm) and ++(270:.4cm) .. (-.4,.4) node[below]{$\scriptstyle c$};
\draw (-.4,-.4) node[above]{$\scriptstyle F^*$} .. controls ++(90:.4cm) and ++(270:.4cm) .. (.4,.4) node[below]{$\scriptstyle F^*$};
}
:=
(\omega^{F^*}_{-,c})^{-1}=
\tikzmath{
\begin{scope}
\clip[rounded corners=5pt] (-1.3,-1) rectangle (1.3,1);
\fill[gray!50] (-1,1) -- (-1,-.4) arc(-180:0:.3cm) .. controls ++(90:.4cm) and ++(270:.4cm) .. (.4,.4) arc (180:0:.3cm) -- (1,-1)  -- (1.3,-1) -- (1.3,1);
\fill[gray!30] (-1,1) -- (-1,-.4) arc(-180:0:.3cm) .. controls ++(90:.4cm) and ++(270:.4cm) .. (.4,.4) arc (180:0:.3cm) -- (1,-1) -- (-1.3,-1) -- (-1.3,1);
\end{scope}
\draw[thick, blue] (.4,-1) node[below]{$\scriptstyle c$} -- (.4,-.4) .. controls ++(90:.4cm) and ++(270:.4cm) .. (-.4,.4) -- (-.4,1) node[above]{$\scriptstyle c$};
\draw (-1,1) node[above]{$\scriptstyle F^*$} -- (-1,-.4) arc(-180:0:.3cm) node[right]{$\scriptstyle F$} .. controls ++(90:.4cm) and ++(270:.4cm) .. (.4,.4) node[left]{$\scriptstyle F$} arc (180:0:.3cm) -- (1,-1) node[below]{$\scriptstyle F^*$};
}
\,.
\]
Unitarity of the map $\omega^{F^*}$ follows as it is a $\pi$-rotation of a unitary under the UAF of unitary adjunction in $2\Hilb$.
The above definition automatically makes the unit and counit of the unitary adjunction $F\dashv^\dag F^*$ into $\cC$-module natural transformations.
\end{proof}

We now list some general facts about module categories for $\rmH^*$-multifusion categories.

\begin{facts}
Suppose $(\cM,\Tr^\cM)$ is a unitary module for an $\rmH^*$-multifusion category $(\cC,\vee,\psi)$.
\begin{enumerate}[label=(M\arabic*)]
\item 
For every $c\in\cC$, $-\lhd c: \cM\to \cM$ is a unitary functor with unitary adjoint $-\lhd c^\vee$.
\item
\label{H*Mod:UnitaryInternalHom}
For every $m\in\cM$, $m\lhd - : \cC\to \cM$ is a unitary $\cC$-module functor with unitary adjoint given by the unitary internal hom $[m,-]_\cC \in\cC$ which is determined by \eqref{eq:UnitaryAdjointFormulaViaYoneda}:
$$
\cM(m\lhd c\to n) \cong \cC(c\to [m,n]_\cC)
\qquad\qquad
[m,n]_\cC := \bigoplus_{c\in\Irr(\cC)} d_c^{-1} \cM(m\lhd c \to n) \otimes c.
$$
(Note that $[m,-]_\cC$ is also a unitary $\cC$-module functor by Lemma \ref{lem:UAFonModC} above.)
\end{enumerate}
\end{facts}

Suppose $\cC$ is an $\rmH^*$-multifusion category and $(\cM_\cC,\Tr^\cM),(\cN_\cC,\Tr^\cN)$ are unitary right $\cC$-module categories.
Since $\Mod^\dag(\cC)$ admits a forgetful functor to $2\Hilb$, 
$$
\Tr^{\Fun^\dag_\cC(\cM_\cC\to \cN_\cC)}_F(\rho:F\Rightarrow F)
:=
\sum_{m\in\Irr(\cM)}
d_m \Tr^\cN_{F(m)}(\rho_{m})
$$
is a unitary trace on $\Fun^\dag_\cC(\cM_\cC\to \cN_\cC)$.
When $\cM=\cC$, we would like that the map $n\mapsto n\lhd -$ is an isometric equivalence
$\cN\to \Fun_\cC^\dag(\cC_\cC\to \cN_\cC)$.
Unfortunately, the trace above is off by a scalar for each indecomposable summand of $\cC$ as we calculate below.

Suppose for simplicity that $\cC$ is indecomposable and let $1_\cC=\bigoplus_{i=1}^d 1_k$ be a decomposition into simples.
For $1\leq i,j\leq k$, define $\cN_i := \cN\lhd 1_i$
and $\cC_{ij}:= 1_i\otimes \cC\otimes 1_j$.
Graphically, we represent $1_i,1_j$, $\id_n$ for $n\in\cN_i$, and $\id_c$ for $c\in\cC_{ij}$ by
$$
\tikzmath{\fill[fill=\brColor, rounded corners=5] (-.3,-.3) rectangle (.3,.3);}=1_i
\qquad\qquad
\tikzmath{\fill[fill=\arColor, rounded corners=5] (-.3,-.3) rectangle (.3,.3);}=1_j
\qquad\qquad
\tikzmath{
\begin{scope}
\clip[rounded corners=5] (-.3,-.3) rectangle (.3,.3);
\fill[fill=\brColor] (0,-.3) rectangle (.3,.3);
\end{scope}
\draw[dotted, rounded corners=5pt] (0,-.3) -- (-.3,-.3) -- (-.3,.3) -- (0,.3);
\draw (0,-.3) --node[left, xshift=.1cm]{$\scriptstyle n$} (0,.3);
}=\id_n,\, n\in\cN_i
\qquad\qquad
\tikzmath{
\begin{scope}
\clip[rounded corners=5] (-.3,-.3) rectangle (.3,.3);
\fill[fill=\brColor] (-.3,-.3) rectangle (0,.3);
\fill[fill=\arColor] (0,-.3) rectangle (.3,.3);
\end{scope}
\draw[thick, blue] (0,-.3) --node[left, xshift=.1cm]{$\scriptstyle c$} (0,.3);
}=\id_c,\, c\in\cC_{ij}.
$$
Fixing an $n\in\cN_i$, we see that
\begin{align*}
\Tr^{\Fun^\dag_\cC(\cC_\cC\to \cN_\cC)}_{n\lhd -}(\id_{n\lhd -})
&=
\sum_{j=1}^k
\sum_{c\in\Irr(\cC_{ij})}
d_c
\Tr^\cN_{n\lhd c}(\id_{n\lhd c})
\\&=
\sum_{j=1}^k
\sum_{c\in\Irr(\cC_{ij})}
d_c
\Tr^\cN_{n}\left(
\tikzmath{
\begin{scope}
\clip[rounded corners = 5pt] (-1,-.5) rectangle (.7,.5);
\fill[fill=\brColor] (-.8,-.5) rectangle (.7,.5);
\end{scope}
\filldraw[blue, thick, fill=\arColor] (0,0) circle (.3cm);
\node[blue] at (-.5,0) {$\scriptstyle c$};
\node[blue] at (.5,0) {$\scriptstyle c^\vee$};
\node at (0,0) {$\scriptstyle j$};
\draw (-.8,-.5) --node[left]{$\scriptstyle n$} (-.8,.5);
}
\right)
\\&\underset{\text{\eqref{eq:PopBubbles}}}{=}
\left(
\sum_{j=1}^k
\sum_{c\in\Irr(\cC_{ij})}
\frac{d_c^2}{d_i}
\right)
d_n
\end{align*}
Analyzing the above scalar, $\sum_{c\in\Irr(\cC_{ij})} d_c^2$ could reasonably be called $\dim(\cC_{ij})$.
Taking a sum over $j$ gives $\dim(1_i\otimes \cC)$, so abreviating $1_i\otimes \cC=: \cC_i$, we are off by the scalar $\dim(\cC_i)/d_i$, which is not immediately obvious to be independent of $i$.
However, we calculate
\begin{align*}
\sum_{j=1}^k
\sum_{c\in\Irr(\cC_{ij})}
\frac{d_c^2}{d_i}
&=
\sum_{j=1}^k
\sum_{c\in\Irr(\cC_{ij})}
d_i^{-1}
\psi_\cC\left(
\tikzmath{
\fill[rounded corners = 5pt, fill=\brColor] (-.7,-.5) rectangle (.7,.5);
\filldraw[thick, blue, fill=\arColor] (0,0) circle (.3cm);
\node at (-.5,0) {$\scriptstyle c$};
\node at (.5,0) {$\scriptstyle c^\vee$};
\node at (0,0) {$\scriptstyle j$};
}
\right)^2
\\&=
\sum_{j=1}^k
\sum_{c\in \Irr(\cC_{ij})} 
d_i^{-1}
\psi_\cC\left(
\tikzmath{
\fill[rounded corners = 5pt, fill=\brColor] (-.7,-.5) rectangle (1.7,.5);
\filldraw[thick, blue, fill=\arColor] (0,0) circle (.3cm);
\filldraw[thick, blue, fill=\arColor] (1,0) circle (.3cm);
\node at (-.5,0) {$\scriptstyle c$};
\node at (1.5,0) {$\scriptstyle c^\vee$};
\node at (0,0) {$\scriptstyle j$};
\node at (1,0) {$\scriptstyle j$};
}
\right)
\cdot
\psi_\cC\left(
\tikzmath{
\fill[rounded corners = 5pt, fill=\brColor] (-.5,-.5) rectangle (.5,.5);
\node at (0,0) {$\scriptstyle i$};
}
\right)
\\&=
\sum_{j=1}^k
\sum_{c\in \Irr(\cC_{ij})} 
\psi_\cC\left(
\tikzmath{
\fill[rounded corners = 5pt, fill=\arColor] (-.9,-.9) rectangle (.9,.9);
\filldraw[thick, blue, fill=\brColor] (0,0) circle (.7cm);
\filldraw[thick, blue, fill=\arColor] (0,0) circle (.3cm);
\node at (-.5,0) {$\scriptstyle c$};
\node at (.5,0) {$\scriptstyle c^\vee$};
\node at (0,0) {$\scriptstyle j$};
\node at (.7,-.7) {$\scriptstyle j$};
}
\right)
\\&=
\sum_{j=1}^k
\left(
\sum_{c\in \Irr(\cC_{ij})} 
\operatorname{FPdim}(c)^2
\right)
\cdot
\psi_\cC\left(
\tikzmath{
\fill[rounded corners = 5pt, fill=\arColor] (-.5,-.5) rectangle (.5,.5);
\node at (0,0) {$\scriptstyle j$};
}
\right)
\\&=
\operatorname{FPdim}(\cC_{11})
\cdot
\sum_{j=1}^k 
\psi_\cC\left(
\tikzmath{
\fill[rounded corners = 5pt, fill=\arColor] (-.5,-.5) rectangle (.5,.5);
\node at (0,0) {$\scriptstyle j$};
}
\right)
\\&=
\operatorname{FPdim}(\cC_{11})
\cdot
\psi_\cC(\id_{1_\cC})
\\&=
\frac{
\operatorname{FPdim}(\cC)
\cdot
\psi_\cC(\id_{1_\cC})
}{k^2},
\end{align*}
which is clearly independent of the choice of $i$.

We now renormalize our initial guess for the trace on $\Fun^\dag_\cC(\cM_\cC\to \cN_\cC)$ by this constant:
\begin{equation}
\label{eq:RenormalizedTraceOnFun}
\Tr^{\Fun^\dag_\cC(\cM_\cC\to \cN_\cC)}_F(\rho:F\Rightarrow F)
:=
\frac{\dim(\End_\cC(1_\cC))^2}{\operatorname{FPdim}(\cC)\cdot \psi_\cC(\id_{1_\cC})}\cdot
\sum_{m\in\Irr(\cM)}
d_m \Tr^\cN_{F(m)}(\rho_{m}).
\end{equation}

With this renormalized trace, we have the following result.

\begin{prop}
\label{prop:RenormalizedTraceOnFun}
Suppose $\cC$ is an indecomposable $\rmH^*$-multifusion category and $\cN_\cC\in \Mod^\dag(\cC)$.
The map $n\mapsto n\lhd -$ is an isometric equivalence
$\cN\to \Fun^\dag_\cC(\cC_\cC\to \cN_\cC)$
when the latter is equipped with the unitary trace \eqref{eq:RenormalizedTraceOnFun} in the case $\cM_\cC=\cC_\cC$.
\end{prop}

\subsection{Bimodule categories and the relative Deligne product}\label{sec:bimod-cats-relative-deligne}

For this section, $(\cC,\vee,\psi)$ and $(\cD,\vee,\psi)$ are $\rmH^*$-multifusion categories.

Similar to a module category over an $\rmH^*$-multifusion category, a bimodule category must also be equipped with a bimodule trace.
In more detail, a bimodule trace for ${}_\cC\cM_\cD$ is a unitary trace on $\cM$ such that for all $c\in\cC$, $d\in \cD$, $m\in\cM$, and $f: c\rhd m\lhd d \to c\rhd m \lhd d$,
        \[ 
        \Tr^\cM_{c\rhd m}\left(
        \tikzmath{
        \draw[thick, red] (.3,.3) arc(180:0:.2cm) --node[right]{$\scriptstyle d^\vee$} (.7,-.3) arc (0:-180:0.2);
        \draw[thick, blue] (-.3,-.7) node[below]{$\scriptstyle c$} --(-.3,.7) node[above]{$\scriptstyle c$};
        \draw[thick] (0,-.7) node[below]{$\scriptstyle m$} --(0,.7) node[above]{$\scriptstyle m$};
        \roundNbox{unshaded}{(0,0)}{.3}{.2}{.2}{$f$};
        }
        \,\right) =
        \Tr^\cM_{c\rhd m\lhd d}\left(
        \tikzmath{
        \draw[thick, blue] (-.3,-.7) node[below]{$\scriptstyle c$} --(-.3,.7) node[above]{$\scriptstyle c$};
        \draw[thick] (0,-.7) node[below]{$\scriptstyle m$} --(0,.7) node[above]{$\scriptstyle m$};
        \draw[thick, red] (.3,-.7) node[below]{$\scriptstyle d$} --(.3,.7) node[above]{$\scriptstyle d$};
        \roundNbox{unshaded}{(0,0)}{.3}{.2}{.2}{$f$};
        }
        \,\right)
        =
        \Tr^\cM_{m\lhd d}\left(
        \tikzmath{
        \draw[thick, blue] (-.3,.3) arc(0:180:.2cm) --node[left]{$\scriptstyle c^\vee$} (-.7,-.3) arc (-180:0:0.2);
        \draw[thick] (0,-.7) node[below]{$\scriptstyle m$} --(0,.7) node[above]{$\scriptstyle m$};
        \draw[thick, red] (.3,-.7) node[below]{$\scriptstyle d$} --(.3,.7) node[above]{$\scriptstyle d$};        \roundNbox{unshaded}{(0,0)}{.3}{.2}{.2}{$f$};
        }
        \,\right).
        \]
As with module categories, the data of a $\cC$-$\cD$ bimodule equipped with a $\cC$-$\cD$ bimodule  trace $\Tr^\cM$ is precisely the same data as a 2-Hilbert space $(\cM,\Tr^\cM)$ equipped with a $\cC -\cD$ bimodule action on $\cM$ such that
\begin{align*}
c\rhd - &\dashv^\dag  c^\vee\rhd - \qquad \forall c \in \cC,\\
-\lhd d &\dashv^\dag  -\lhd d ^\vee \qquad \forall d \in \cD.
\end{align*}

Given a right $\cC$-module category $(\cM_\cC,\Tr^\cM)$ and a $\cC$-$\cD$ bimodule category $({}_\cC\cN_\cD,\Tr^\cN)$, we can form the unitary Deligne product
$\cM\boxtimes_\cC \cN_\cD$, which satisfies a universal property for $\cC$-balanced $\dag$-functors $\cM \times \cN \to \cP$ for every unitary category $\cP$ \cite{MR2677836}.
Similar to \cite{MR2677836}, there are several concrete models for the relative Deligne product; we look at two of them here.
\begin{itemize}
\item 
The first is the category of left $\cC$-module functors
$$
\Fun^\dag_{\cC}({}_\cC\cM^{\op}\to {}_\cC\cN)
$$
with the trace from \eqref{eq:RenormalizedTraceOnFun}, which is manifestly a right $\cD$-module trace.
\item 
The second is the unitary Cauchy completion of the \emph{ladder category} \cite{MR3975865} whose objects are formal tensors $m\boxtimes n$ for $m\in\cM$ and $n\in\cN$ and whose morphism spaces are given by 
$$
\Hom(m_1\boxtimes n_1 \to m_2\boxtimes n_2) := \bigoplus_{c\in\Irr(\cC)} \cM(m_1\to m_2\lhd c)\otimes \cN(c\rhd n_1 \to n_2).
$$
Graphically, we represent simple tensors in a single summand of the above hom space by `ladders':
$$
\tikzmath{
\draw[thick] (0,-.8) --node[left]{$\scriptstyle m_1$} (0,-1.2);
\draw[thick] (0,-.2) --node[left]{$\scriptstyle m_2$} (0,1.2);
\draw[thick] (1,.8) --node[right]{$\scriptstyle n_2$} (1,1.2);
\draw[thick] (1,.2) --node[right]{$\scriptstyle n_1$} (1,-1.2);
\draw[thick, blue] (0,-.5) --node[above]{$\scriptstyle c$} (1,.5);
\roundNbox{fill=white}{(0,-.5)}{0.3}{.1}{.1}{$f$};
\roundNbox{fill=white}{(1,.5)}{0.3}{.1}{.1}{$g$};
}
\in
\cM(m_1\to m_2\lhd c)\otimes \cN(c\rhd n_1\to n_2).
$$
Composition is given by `stacking ladders' and decomposing using the fusion rules in $\cC$.
The dagger is given by taking dagger in each tensorand in each summand and rotating the $\cC$-strings using the UDF.
We refer the reader to \cite{MR3975865} for more details.

The right $\cD$-action here is given by
$$
(m\boxtimes n) \lhd d := m\boxtimes n\lhd d.
$$
We equip the ladder category with a right $\cD$-module trace in Lemma \ref{lem:RelativeDeligneTrace} below.

\end{itemize}
We prove these two models for the relative Deligne product are isometrically equivalent in Proposition \ref{prop:IsometricEquivalenceOfRelativeDeligne} below.
For notational convenience, we will denote the unitary Cauchy completion of the ladder category by $\cM\boxtimes_\cC \cN$ and the functor category by $\Fun^\dag_\cC({}_\cC\cM^{\op}\to {}_\cC\cN)$

To state the next lemma, we introduce the following notation.

\begin{nota}
We write $1_\cC=\bigoplus_{i=1}^{k} 1_i$ for a decomposition into simples.
We write $\cM_i := \cM\lhd 1_i$ and $\cN_i:= 1_i\rhd \cN$; 
for $m\in\cM$, we write $m_i:= m\lhd 1_i$ and for $n\in\cN$, we write $n_i:= 1_i\rhd n$.
\end{nota}

\begin{lem}
\label{lem:RelativeDeligneTrace}
The formula
\begin{align*}
\bigoplus_{c\in\Irr(\cC)} \cM(m\to m\lhd c)\otimes \cN(c\rhd n\to n)
&\xrightarrow{\bigoplus\delta_{c=1_{j}}}
\bigoplus\cM_j( m_j\to m_j) \otimes \cN_i(n_j\to n_j) 
\\&\xrightarrow{\sum d_j^{-1} \Tr^{\cM}_{m_j}\otimes \Tr^\cN_{n_j}} \bbC
\end{align*}
defines a right $\cD$-module trace on $\cM\boxtimes_\cC\cN_\cD$.
\end{lem}
\begin{proof}
For a morphism
$$
\sum_{c\in\Irr(\cC)} \sum_{r=1}^{k_c} f^c_r\otimes g^c_r
\in
\bigoplus_{c\in\Irr(\cC)} \cM(m\to m'\lhd c)\otimes \cN(c\rhd n\to n'),
$$
we have

\begin{align*}
\sum_{\substack{i,j \\ b,c\in\Irr(\cC_{ij}) \\ r,l}}
\Tr^{\cM\boxtimes_\cC\cN}_{m_j\boxtimes n_j}\left(
\tikzmath{
\begin{scope}
\clip (0,-.2) rectangle (1,4.2);
\fill[\arColor] (0,-.2) rectangle (1,4.2);
\fill[\brColor] (0,0.5) -- (1,1.5) -- (1,2.5) -- (0,3.5);
\end{scope}
\draw[thick] (0,-0.2) node[left,yshift=.1cm]{$\scriptstyle m_j$} -- node[left]{$\scriptstyle m'_i$} (0,4.2) node[left,yshift=-.1cm]{$\scriptstyle m_j$};
\draw[thick] (1,-0.2) node[right,yshift=.1cm]{$\scriptstyle n_j$} -- node[right]{$\scriptstyle n'_i$} (1,4.2) node[right,yshift=-.1cm]{$\scriptstyle n_j$};
\draw[thick, DarkGreen] (0,0.5) --node[above]{$\scriptstyle b$} (1,1.5);
\draw[thick, blue] (0,3.5) --node[above]{$\scriptstyle c$} (1,2.5);
\roundNbox{fill=white}{(0,0.5)}{0.3}{.1}{.1}{$f^b_r$};
\roundNbox{fill=white}{(0,3.5)}{0.3}{.1}{.1}{\scriptsize{$(f^c_l)^\dag$}};
\roundNbox{fill=white}{(1,1.5)}{0.3}{.1}{.1}{$g^b_r$};
\roundNbox{fill=white}{(1,2.5)}{0.3}{.1}{.1}{\scriptsize{$(g^c_l)^\dag$}};
}
\right) 
&=
\sum_{\substack{i,j \\ c\in\Irr(\cC_{ij})\\ r,l}}
\dim^\vee_R(c)^{-1}
\Tr^{\cM\boxtimes_\cC\cN}_{m_j\boxtimes n_j}\left(
\tikzmath{
\begin{scope}
\clip (-.15,-.2) rectangle (1.55,4.2);
\fill[\arColor] (0,-.2) -- (0,4.2) -- (1.55,4.2) -- (1.55,2.5) -- (1.4,2.5) -- (1.4,1.5) -- (1.55,1.5) -- (1.55,-.2);
\fill[\brColor] (-.15,.5) rectangle (.15,3.5);
\fill[\brColor] (1.25,2.8) arc (0:180:.3) -- (.65,1.2) arc (-180:0:.3) -- (1.4,1.5) -- (1.4,2.5);
\end{scope}
\draw[thick] (0,-0.2) node[left,yshift=.1cm]{$\scriptstyle m_j$} -- (0,.5);
\draw[thick] (0,3.5) -- (0,4.2) node[left,yshift=-.1cm]{$\scriptstyle m_j$};
\draw[thick] (-.15,.5) -- node[left]{$\scriptstyle m'_i$} (-.15,3.5);
\draw[thick, blue] (.15,.5) -- (.15,3.5);
\node[blue] at (.28,2) {$\scriptstyle c$};
\draw[thick] (1.55,-0.2) node[right,yshift=.1cm]{$\scriptstyle n_j$} -- (1.55,1.5);
\draw[thick] (1.4,1.5) -- node[right]{$\scriptstyle n'_i$} (1.4,2.5);
\draw[thick] (1.55,2.5) -- (1.55,4.2) node[right,yshift=-.1cm]{$\scriptstyle n_j$};
\draw[thick, blue] (1.25,2.8) arc (0:180:.3) --  (.65,1.2) arc (-180:0:.3);
\node[blue] at (.85,2) {$\scriptstyle c^\vee$};
\roundNbox{fill=white}{(0,0.5)}{0.3}{.1}{.1}{$f^c_r$};
\roundNbox{fill=white}{(0,3.5)}{0.3}{.1}{.1}{\scriptsize{$(f^c_l)^\dag$}};
\roundNbox{fill=white}{(1.4,1.5)}{0.3}{.1}{.1}{$g^c_r$};
\roundNbox{fill=white}{(1.4,2.5)}{0.3}{.1}{.1}{\scriptsize{$(g^c_l)^\dag$}};
}\right) 
\displaybreak[1]\\&=
\sum_{\substack{i,j \\ c\in\Irr(\cC_{ij}) \\ r,l}}
\underbrace{d_{t(c)}^{-1}\dim^\vee_R(c)^{-1}}_{=d_c^{-1}}
\Tr_{m_j}^{\cM}\left(
\tikzmath{
\begin{scope}
\clip[rounded corners = 5pt] (-.2,-.7) rectangle (.7,1.7);
\fill[\arColor] (0,-.7) rectangle (.7,1.7);
\fill[\brColor] (-.15,0) rectangle (.15,1);
\end{scope}
\draw[thick] (0,1.3) --node[left]{$\scriptstyle m_j$} (0,1.7);
\draw[thick] (0,-.7) --node[left]{$\scriptstyle m_j$} (0,-.3);
\draw[thick] (-.15,.3) --node[left]{$\scriptstyle m'_i$} (-.15,.7);
\draw[thick, blue] (.15,.3) --node[right]{$\scriptstyle c$} (.15,.7);
\roundNbox{fill=white}{(0,0)}{0.3}{0.1}{0.1}{$f^c_r$};
\roundNbox{fill=white}{(0,1)}{0.3}{0.1}{0.1}{\scriptsize{$(f^c_l)^\dag$}};
}
\right)
\Tr^\cN_{n_j}\left(
\tikzmath{
\begin{scope}
\clip[rounded corners = 5pt] (-1.15,-.7) rectangle (.2,1.7);
\fill[\arColor] (-1.15,-.7) -- (.15,-.7) -- (.15,0) -- (0,0) -- (0,1) -- (.15,1) -- (.15,1.7) -- (-1.15,1.7);
\fill[\brColor] (-.15,1.3) arc (0:180:.3cm) -- (-.75,-.3) arc (-180:0:.3cm) -- (0,0) -- (0,1);
\end{scope}
\draw[thick, blue] (-.15,1.3) arc (0:180:.3cm) --node[right,xshift=-.1cm]{$\scriptstyle c^\vee$} (-.75,-.3) arc (-180:0:.3cm);
\draw[thick] (0,.3) --node[right]{$\scriptstyle n'_i$} (0,.7);
\draw[thick] (.15,1.3) --node[right]{$\scriptstyle n_j$} (.15,1.7);
\draw[thick] (.15,-.3) --node[right]{$\scriptstyle n_j$} (.15,-.7);
\roundNbox{fill=white}{(0,0)}{0.3}{0.1}{0.1}{$g^c_r$};
\roundNbox{fill=white}{(0,1)}{0.3}{0.1}{0.1}{\scriptsize{$(g^c_l)^\dag$}};
}
\right)
\displaybreak[1]\\
\sum_{\substack{i,j \\ b,c\in\Irr(\cC_{ij}) \\ r,l}}
\Tr^{\cM\boxtimes_\cC\cN}_{m'_i\boxtimes n'_i}\left(
\tikzmath{
\begin{scope}
\clip (0,-.2) rectangle (1,4.2);
\fill[\brColor] (0,-.2) rectangle (1,4.2);
\fill[\arColor] (1,.5) -- (0,1.5) -- (0,2.5) -- (1,3.5);
\end{scope}
\draw[thick] (0,-0.2) node[left,yshift=.1cm]{$\scriptstyle m'_i$} -- node[left]{$\scriptstyle m_j$} (0,4.2) node[left,yshift=-.1cm]{$\scriptstyle m'_i$};
\draw[thick] (1,-0.2) node[right,yshift=.1cm]{$\scriptstyle n'_i$} -- node[right]{$\scriptstyle n_j$} (1,4.2) node[right,yshift=-.1cm]{$\scriptstyle n'_i$};
\draw[thick, blue] (0,1.5) --node[above]{$\scriptstyle c$} (1,0.5);
\draw[thick, DarkGreen] (0,2.5) --node[above]{$\scriptstyle b$} (1,3.5);
\roundNbox{fill=white}{(0,2.5)}{0.3}{.1}{.1}{$f^b_r$};
\roundNbox{fill=white}{(0,1.5)}{0.3}{.1}{.1}{\scriptsize{$(f^c_l)^\dag$}};
\roundNbox{fill=white}{(1,3.5)}{0.3}{.1}{.1}{$g^b_r$};
\roundNbox{fill=white}{(1,.5)}{0.3}{.1}{.1}{\scriptsize{$(g^c_l)^\dag$}};
}
\right) 
&=
\sum_{\substack{i,j \\ c\in\Irr(\cC_{ij}) \\ r,l}}
\underbrace{d_{s(c)}^{-1}\dim^\vee_L(c)^{-1}}_{=d_c^{-1}}
\Tr_{m'_i}^{\cM}\left(
\tikzmath{
\begin{scope}
\clip[rounded corners = 5pt] (-.35,-.7) rectangle (1.15,1.7);
\fill[\brColor] (-.15,-.7) -- (-.15,0) -- (0,0) -- (0,1) -- (-.15,1) -- (-.15,1.7) -- (1.15,1.7) -- (1.15,-.7);
\fill[\arColor] (.15,1.3) arc (180:0:.3cm) -- (.75,-.3) arc (0:-180:.3cm) -- (0,0) -- (0,1);
\end{scope}
\draw[thick, blue] (.15,1.3) arc (180:0:.3cm) --node[left,xshift=.1cm]{$\scriptstyle c^\vee$} (.75,-.3) arc (0:-180:.3cm);
\draw[thick] (0,.3) --node[left]{$\scriptstyle m_j$} (0,.7);
\draw[thick] (-.15,1.3) --node[left]{$\scriptstyle m'_i$} (-.15,1.7);
\draw[thick] (-.15,-.3) --node[left]{$\scriptstyle m'_i$} (-.15,-.7);
\roundNbox{fill=white}{(0,1)}{0.3}{0.1}{0.1}{$f^c_r$};
\roundNbox{fill=white}{(0,0)}{0.3}{0.1}{0.1}{\scriptsize{$(f^c_l)^\dag$}};
}
\right)
\Tr^\cN_{n'_i}\left(
\tikzmath{
\begin{scope}
\clip[rounded corners = 5pt] (-.7,-.7) rectangle (.2,1.7);
\fill[\brColor] (0,-.7) -- (0,0) -- (.15,0) -- (.15,1) -- (0,1) -- (0,1.7) -- (-.7,1.7) -- (-.7,-.7);
\fill[\arColor] (-.15,0) rectangle (.15,1);
\end{scope}
\draw[thick] (0,1.3) --node[right]{$\scriptstyle n'_i$} (0,1.7);
\draw[thick] (0,-.7) --node[right]{$\scriptstyle n'_i$} (0,-.3);
\draw[thick, blue] (-.15,.3) --node[left]{$\scriptstyle c$} (-.15,.7);
\draw[thick] (.15,.3) --node[right]{$\scriptstyle n_j$} (.15,.7);
\roundNbox{fill=white}{(0,1)}{0.3}{0.1}{0.1}{$g^c_r$};
\roundNbox{fill=white}{(0,0)}{0.3}{0.1}{0.1}{\scriptsize{$(g^c_l)^\dag$}};
}
\right).
\end{align*}
By inspection of the above formulas, it follows from the traciality and faithfulness of $\cC$-module traces $\Tr^{\cM}$ and $\Tr^\cN$ that $\Tr^{\cM\boxtimes_\cC\cN}$ is tracial and faithful.
For example,
\[
\Tr_{m_j}^{\cM}\left(
\tikzmath{
\begin{scope}
\clip[rounded corners = 5pt] (-.2,-.7) rectangle (.7,1.7);
\fill[\arColor] (0,-.7) rectangle (.7,1.7);
\fill[\brColor] (-.15,0) rectangle (.15,1);
\end{scope}
\draw[thick] (0,1.3) --node[left]{$\scriptstyle m_j$} (0,1.7);
\draw[thick] (0,-.7) --node[left]{$\scriptstyle m_j$} (0,-.3);
\draw[thick] (-.15,.3) --node[left]{$\scriptstyle m'_i$} (-.15,.7);
\draw[thick, blue] (.15,.3) --node[right]{$\scriptstyle c$} (.15,.7);
\roundNbox{fill=white}{(0,0)}{0.3}{0.1}{0.1}{$f^c_r$};
\roundNbox{fill=white}{(0,1)}{0.3}{0.1}{0.1}{\scriptsize{$(f^c_l)^\dag$}};
}
\right)
=
\Tr_{m'_i\lhd c}^{\cM}\left(
\tikzmath{
\begin{scope}
\clip[rounded corners = 5pt] (-.35,-.7) rectangle (.7,1.7);
\fill[\arColor] (-.15,-.7) -- (-.15,0) -- (0,0) -- (0,1) -- (-.15,1) -- (-.15,1.7) -- (.7,1.7) -- (.7,-.7);
\fill[\brColor] (-.15,-.7) rectangle (.15,0);
\fill[\brColor] (-.15,1) rectangle (.15,1.7);
\end{scope}
\draw[thick, blue] (.15,1.3) -- node[right]{$\scriptstyle c$} (.15,1.7);
\draw[thick, blue] (.15,-.7) -- node[right]{$\scriptstyle c$} (.15,-.3);
\draw[thick] (0,.3) --node[left]{$\scriptstyle m_j$} (0,.7);
\draw[thick] (-.15,1.3) --node[left]{$\scriptstyle m'_i$} (-.15,1.7);
\draw[thick] (-.15,-.3) --node[left]{$\scriptstyle m'_i$} (-.15,-.7);
\roundNbox{fill=white}{(0,1)}{0.3}{0.1}{0.1}{$f^c_r$};
\roundNbox{fill=white}{(0,0)}{0.3}{0.1}{0.1}{\scriptsize{$(f^c_l)^\dag$}};
}
\right)
=
\Tr_{m'_i}^{\cM}\left(
\tikzmath{
\begin{scope}
\clip[rounded corners = 5pt] (-.35,-.7) rectangle (1.15,1.7);
\fill[\brColor] (-.15,-.7) -- (-.15,0) -- (0,0) -- (0,1) -- (-.15,1) -- (-.15,1.7) -- (1.15,1.7) -- (1.15,-.7);
\fill[\arColor] (.15,1.3) arc (180:0:.3cm) -- (.75,-.3) arc (0:-180:.3cm) -- (0,0) -- (0,1);
\end{scope}
\draw[thick, blue] (.15,1.3) arc (180:0:.3cm) --node[left,xshift=.1cm]{$\scriptstyle c^\vee$} (.75,-.3) arc (0:-180:.3cm);
\draw[thick] (0,.3) --node[left]{$\scriptstyle m_j$} (0,.7);
\draw[thick] (-.15,1.3) --node[left]{$\scriptstyle m'_i$} (-.15,1.7);
\draw[thick] (-.15,-.3) --node[left]{$\scriptstyle m'_i$} (-.15,-.7);
\roundNbox{fill=white}{(0,1)}{0.3}{0.1}{0.1}{$f^c_r$};
\roundNbox{fill=white}{(0,0)}{0.3}{0.1}{0.1}{\scriptsize{$(f^c_l)^\dag$}};
}
\right).
\]
Moreover,
$\Tr^{\cM\boxtimes_\cC\cN}$ is visibly a right $\cD$-module trace because $\Tr^\cN$ is a right $\cD$-module trace.
\end{proof}

\begin{cor}
\label{cor:RightCActionIsometric}
The right $\cC$-action map
$\cM\boxtimes_\cC \cC \to \cM$
is an isometric equivalence.
\end{cor}
\begin{proof}
This map is well known to be a unitary equivalence, so we show it is isometric.
To do so, it suffices to look at $m\boxtimes c$ for $m\in\Irr(\cM)$ and $c\in\Irr(\cC)$.
Consider a morphism
$\sum_{d,r} f^d_r\otimes g^d_r:m\boxtimes c \to m\boxtimes c$, 
where $f^d_r:m \to m\lhd d$ and $g^d_r:d\otimes c \to c$ with $d$ ranging over $\Irr(\cC)$.
By the formula in Lemma \ref{lem:RelativeDeligneTrace}, we have 
$$
\Tr^{\cM\boxtimes_\cC\cC}_{m\boxtimes c}(f) 
= 
\sum_{r} \Tr^\cM_m(f^{1_{s(c)}}_r)\Tr^\cC_c(g^{1_{s(c)}}_r)d_{s(c)}^{-1} 
= 
\sum_{r} \Tr^\cM_m(f^{1_{s(c)}}_r)(\psi\circ\tr^\vee_L)(g^{1_{s(c)}}_r)d_{s(c)}^{-1}.
$$
On the other hand, the image of each $f^d_r\otimes g^d_r$ under the action functor is $(\id_m\lhd g^d_r)\circ(f^d_r\lhd \id_c)$.
Denoting the image of $f$ in $\cM$ by $\tilde f$, we have
\begin{align*}
\Tr^{\cM}_{m\lhd c}\left(\tilde f\right) 
&= 
\sum_{d,r} \Tr^{\cM}_{m\lhd c}
\left(
\tikzmath{
\begin{scope}
\clip[rounded corners = 5pt] (-.2,-1.2) rectangle (1.7,1.2);
\fill[\brColor] (0,-1.2) rectangle (1,1.2);
\fill[\arColor] (1,-1.2) rectangle (1.8,1.2);
\end{scope}
\draw[thick] (0,-.8) --node[left]{$\scriptstyle m$} (0,-1.2);
\draw[thick] (0,-.2) --node[left]{$\scriptstyle m$} (0,1.2);
\draw[thick, blue] (1,.8) --node[right]{$\scriptstyle c$} (1,1.2);
\draw[thick, blue] (1,.2) --node[right]{$\scriptstyle c$} (1,-1.2);
\draw[thick, DarkGreen] (0,-.5) --node[above]{$\scriptstyle d$} (1,.5);
\roundNbox{fill=white}{(0,-.5)}{0.3}{.1}{.1}{$f^d_r$};
\roundNbox{fill=white}{(1,.5)}{0.3}{.1}{.1}{$g^d_r$};
}
\right) 
=
\sum_{d,r}\Tr^{\cM}_{m}
\left(
\tikzmath{
\begin{scope}
\clip[rounded corners = 5pt] (-.2,-1.2) rectangle (2,1.2);
\fill[\brColor] (0,-1.2) rectangle (2,1.2);
\fill[\arColor] (1,.8) arc(180:0:.3cm) -- (1.6,.2) arc(0:-180:.3cm);
\end{scope}
\draw[thick] (0,-.8) --node[left]{$\scriptstyle m$} (0,-1.2);
\draw[thick] (0,-.2) --node[left]{$\scriptstyle m$} (0,1.2);
\draw[thick, blue] (1,.8) arc(180:0:.3cm) --node[right,xshift=-1]{$\scriptstyle c^\vee$} (1.6,.2) arc(0:-180:.3cm);
\draw[thick, DarkGreen] (0,-.5) --node[above]{$\scriptstyle d$} (1,.5);
\roundNbox{fill=white}{(0,-.5)}{0.3}{.1}{.1}{$f^d_r$};
\roundNbox{fill=white}{(1,.5)}{0.3}{.1}{.1}{$g^d_r$};
}
\right) 
\\&=
\sum_{r}\Tr^{\cM}_{m}
\left(
\tikzmath{
\begin{scope}
\clip[rounded corners = 5pt] (-.2,-.2) rectangle (2.2,1.2);
\fill[\brColor] (0,-1.2) rectangle (2.2,1.2);
\fill[\arColor] (1.2,.8) arc(180:0:.3cm) -- (1.8,.2) arc(0:-180:.3cm);
\end{scope}
\draw[thick] (0,-.2) --node[left]{$\scriptstyle m$} (0,.2);
\draw[thick] (0,.8) --node[left]{$\scriptstyle m$} (0,1.2);
\draw[thick, blue] (1.2,.8) arc(180:0:.3cm) --node[right,xshift=-1]{$\scriptstyle c^\vee$} (1.8,.2) arc(0:-180:.3cm);
\roundNbox{fill=white}{(0,.5)}{0.3}{.15}{.15}{$\scriptstyle f^{1_{s(c)}}_i$};
\roundNbox{fill=white}{(1.2,.5)}{0.3}{.15}{.15}{$\scriptstyle g^{1_{s(c)}}_i$};
}
\right).
\end{align*}
Now since $1_{s(c)}$ is simple, a straightforward calculation shows that
$$
\tr^\vee_L(g^{1_{s(c)}}_i)
=
\frac{(\psi\circ\tr^\vee_L)(g^{1_{s(c)}}_i)}{d_{s(c)}}
\id_{1_{s(c)}}.
$$
The result follows.
\end{proof}

The next corollary shows that every unitary module category over an $\rmH^*$-multifusion category admits a unitary module trace.
One should view this as the multifusion analog to the existence result \cite{MR3019263} for module traces for modules over a pseudounitary fusion category equipped with its canonical spherical structure.

\begin{cor}
\label{cor:ExistsModuleTrace}
Suppose $(\cC,\vee,\psi_\cC)$ is an $\rmH^*$-multifusion category.
If $\cM$ is a unitary $\cC$-module, then there exists a faithful $\cC$-module trace.
If $\cC$ and $\cM$ are indecomposable, this trace is unique up to positive scalar.
\end{cor}
\begin{proof}
We may assume $\cC$ and $\cM$ are indecomposable.
Since $\cM$ is indecomposable, $\cM$ is unitarily equivalent as a right $\cC$-module category to the unitary relative Deligne product $\cM_1\boxtimes_{\cC_{11}}\cC$ by \cite[Prop.~3.12]{MR4598730}
where $\cC_{11}:= 1_1\otimes \cC\otimes 1_1$.
By \cite{MR3019263,MR4598730}, there is a unique faithful unitary right $\cC_{11}$-module trace $\Tr^{\cM_1}$ on $\cM_1$ up to positive scalar.
We also know $\cC$ admits a $\cC_{11}$-$\cC$ bimodule trace given by $\psi\circ \tr_L^\vee=\psi\circ\tr_R^\vee$.
Hence applying Lemma \ref{lem:RelativeDeligneTrace}, we get a right $\cC$-module trace on $\cM_1\boxtimes_{\cC_{11}}\cC\cong \cM$.
Finally, by indecomposability, there is at most one faithful unitary right $\cC$-module trace $\Tr^{\cM}$ on $\cM$ up to positive scalar.
\end{proof}

Now to identify $\cM\boxtimes_\cC \cN$ with $\Fun^\dag_\cC(\cM^{\op}\to \cN)$, we use \ref{H*Mod:UnitaryInternalHom}.
Every $m^{\op}\in\cM^{\op}$ can be viewed as a functor $-\rhd m^{\op}: \cC\to \cM^{\op}$ which has a unitary adjoint, and similarly for $n\in\cN$.
Thus given $m\boxtimes n \in\cM\boxtimes_\cC\cN$, we get a unitary $\cC$-module functor ${}_\cC\cM^{\op}\to {}_\cC\cN$ by
$$
(-\rhd n)\circ (-\rhd m^{\op})^*
=
\left(
p^{\op}
\mapsto 
\underbrace{[m^{\op},p^{\op}]_\cC}_{=[p,m]_\cC} \rhd n
=
\bigoplus_{c\in\Irr(\cC)} d_c^{-1} 
\underbrace{\cM^{\op}(c\rhd m^{\op}\to p^{\op})}_{=\cM(p\to m\lhd c^\vee)} 
\otimes c \rhd n
\right).
$$
We explain how this functor acts on a `simple tensor' morphism $f\otimes g \in \cM(m\lhd a \to m' ) \otimes \cN(n \to a\rhd n')$ for $a\in\Irr(\cC)$.
First, observe that morphisms of $[p,m]_\cC\rhd n\to [p,m']_\cC\rhd n'$,
i.e., maps
$$
\bigoplus_{c\in\Irr(\cC)} d_c^{-1} 
\underbrace{\cM^{\op}(c\rhd m^{\op}\to p^{\op})}_{=\cM(p\to m\lhd c^\vee)} 
\otimes c \rhd n
\longrightarrow
\bigoplus_{b\in\Irr(\cC)} d_b^{-1} 
\underbrace{\cM^{\op}(b\rhd {m'}^{\op}\to p^{\op})}_{=\cM(p\to m'\lhd b^\vee)} 
\otimes b \rhd n',
$$
can be viewed as block matrices of maps;
the $b,c$-block for $b,c\in\Irr(\cC)$ 
is a block of size
$\dim(\cM(p\to m'\lhd b^\vee))\times \dim(\cM(p\to m\lhd c^\vee))$
whose entries are morphisms $c\rhd n \to b\rhd n'$.
We can think of the rows of this block as being indexed over an ONB $\{\gamma\}$ for $d_b^{-1}\cM(p\to m'\lhd b^\vee)$,
and the columns as being indexed over an ONB $\{\delta\}$ for 
$d_c^{-1}\cM(p\to m\lhd c^\vee)$.
The $b,c$-block of the morphism of $[p,m]_\cC\rhd n \to [p,m']_\cC\rhd n'$ corresponding to our simple tensor $f\otimes g$ has
$\gamma,\delta$-entry equal to 
\begin{equation}
\label{eq:NaturalTransformationFromSimpleTensor}
\sum_{\beta\in\operatorname{ONB}(a\otimes c^\vee\to b^\vee)}
\Tr^\cM_p
\left(
\tikzmath{
\draw[thick, blue] (0,0) to[out=45,in=-90] node[right]{$\scriptstyle c^\vee$} (.6,1.5);
\draw[thick, blue] (0,1) to[out=45,in=-135] (.6,1.5);
\filldraw[blue] (.6,1.5) node[right]{$\scriptstyle \beta$} circle (.05cm);
\draw[thick, blue] (.6,1.5) to[out=90, in=-45] node[right, xshift=.1cm]{$\scriptstyle b^\vee$} (0,2);
\node[blue] at (.3,1.4) {$\scriptstyle a$};
\draw (0,2.7) --node[left]{$\scriptstyle p$} (0,2.3);
\draw (0,1.7) --node[left]{$\scriptstyle m'$} (0,1.3);
\draw (0,.7) --node[left]{$\scriptstyle m$} (0,.3);
\draw (0,-.7) --node[left]{$\scriptstyle p$} (0,-.3);
\roundNbox{fill=white}{(0,0)}{.3}{0}{0}{$\delta$}
\roundNbox{fill=white}{(0,1)}{.3}{0}{0}{$f$}
\roundNbox{fill=white}{(0,2)}{.3}{0}{.2}{$\gamma^\dag$}
}
\right)
\cdot
\tikzmath{
\draw (.1,-1.6) --node[right]{$\scriptstyle n$} (.1,-.3);
\draw (.1,.3) --node[right]{$\scriptstyle n'$} (.1,.7);
\filldraw[blue] (-.6,-1) node[left]{$\scriptstyle \beta^\dag$} circle (.05cm);
\draw[thick, blue] (0,0) to[out=-135, in=135] node[left]{$\scriptstyle a$} (-.6,-1);
\draw[thick, blue] (-.6,-1) to[out=45,in=90] (-.3,-1) -- node[right]{$\scriptstyle c$} (-.3,-1.6);
\draw[thick, blue] (-.6,-1) arc(0:-180:.3cm) --node[left]{$\scriptstyle b$} (-1.2,.7);
\roundNbox{fill=white}{(0,0)}{.3}{0}{.2}{$g$}
}
\end{equation}
where the sum is independent of the choice of ONB for $\cC(a\otimes c^\vee\to b^\vee)$
(which is isometrically isomorphic to $\cC(b\otimes a\to c)$ using $\Tr^\cC=\psi\circ\tr_L$).\footnote{The meticulous reader may notice that the scalar $d_b^{-1}$ from the inner product on $d_b^{-1}\cM(p\to m\lhd b^\vee)$
cancels with the normalization $\{\sqrt{d_b}\beta\}$ needed to turn an ONB for $\cC(a\otimes c^\vee\to b^\vee)$ with the inner product from $\Tr^\cC$ into an ONB for the closely related \emph{isometry inner product} (e.g., see \cite[\S2.5]{MR4079745})
needed for functoriality.}
We thus get a unitary $\cD$-module functor 
$$
\cM\boxtimes_\cC\cN_\cD \to \Fun^\dag({}_\cC\cM^{\op}\to {}_\cC\cN)
$$
by the universal property of Cauchy completion.
It is a straightforward but tedious exercise to prove this unitary functor is an equivalence; we leave the details to the interested reader.

\begin{prop}
\label{prop:IsometricEquivalenceOfRelativeDeligne}
We have an isometric equivalence of unitary $\cD$-modules
$$
\cM\boxtimes_\cC \cN_\cD
\cong
\Fun^\dag_{\cC}(\cM^{\op}\to \cN)
$$
where the latter has the trace from  \eqref{eq:RenormalizedTraceOnFun}.
\end{prop}

\begin{proof}
It remains to to prove the above unitary equivalence is isometric.
It is enough to show that it is trace-preserving for endomorphisms of objects of the form $m\boxtimes n$, since all objects are direct sums of summands of objects in this form. 
We can assume $m,n$ are simple with $s(m) = t(n) = 1_i$.
Let 
$$
(f:m \to m \lhd a)\otimes (g:a\rhd n \to n)
\in
\bigoplus_{c\in\Irr(\cC)} \cM(m\to m\lhd c)\otimes \cN(c\rhd n\to n)
$$
be a `simple tensor' morphism in the ladder category.
Note we must have $a \in \cC_i$ for this to be nonzero.
By the formula in Lemma \ref{lem:RelativeDeligneTrace}, we have
\[
\Tr^{\cM\boxtimes_\cC \cN}_{m\boxtimes n}(f\otimes g) = \sum_i \Tr^\cM(f^{1_{t(a)}}_i) \Tr^\cN(g^{1_{t(a)}}_i) d_{t(a)}^{-1} = \delta_{a=1_i} d_{1_i}^{-1} \Tr^\cM(f)\Tr^\cN(g).
\]
On the other hand, we can compute the trace in
$\Fun^\dag_\cC({}_\cC\cM^{\op}\to {}_\cC\cN)$ of
the block matrix associated to our `simple tensor' morphism $f\otimes g$.
Clearly the trace only sees entries in the block diagonal of this matrix corresponding to the blocks where $b=c$.
Analyzing \eqref{eq:NaturalTransformationFromSimpleTensor},
since $\Tr^\cN$ is a left $\cC$-module trace, the trace of the image of $f\otimes g$ will only be non-zero if $a=1_i$;
the reasoning here is similar to that used in the proof of Corollary \ref{cor:RightCActionIsometric}.
In this case, since $m,n$ were assumed simple, $f,g$ are scalars, so it suffices to consider the case $f=\id_m$ and $g=\id_n$.
By the formula in \eqref{eq:RenormalizedTraceOnFun}, we have
\begin{align*}
\Tr^{\Fun^\dag_\cC({}_\cC\cM^{\op} \to {}_\cC\cN)}_{(-\rhd n)\circ (-\rhd m^{\op})^*}(\id) 
&= 
\frac{d_{1_i}}{\dim(\cC_i)} 
\sum_{p\in\Irr(\cM)} 
d_p \Tr^{\cN}_{[p,m]_\cC\rhd n}(\id_{[p,m]_\cC\rhd n})
\\&=
\frac{d_{1_i}}{\dim(\cC_i)} 
\sum_{p\in\Irr(\cM)} 
d_p \sum_{c \in \Irr(\cC)} \dim(\cM(p \to m\lhd c^\vee)) \Tr^\cN_{c\rhd n}(\id_{c\rhd n})
\\&=
\frac{d_{1_i}}{\dim(\cC_i)} 
\sum_{c \in \Irr(\cC_i)} \Tr^\cM_{m\lhd c^\vee}(\id_{m\lhd c^\vee}) \Tr^\cN_{c\rhd n}(\id_{c\rhd n})
\\&\underset{\text{\eqref{eq:PopBubbles}}}{=}
\frac{d_{1_i}}{\dim(\cC_i)} 
\sum_{c \in \Irr(\cC_i)} d_m \frac{d_c}{d_{1_i}} d_n \frac{d_c}{d_{1_i}}
\\&=
\frac{1}{d_{1_i}}d_md_n 
\\&= 
\Tr^{\cM \boxtimes_\cC \cN}_{m\boxtimes n}(\id_{m\boxtimes n}).
\end{align*}
The result follows.
\end{proof}

In what follows, we briefly describe the universal property that $\cM \boxtimes_{\cC} \cN$ enjoys.

\begin{defn}
For a left $\cC$-module $\cM_{\cC}$ and right $\cC$-module $_{\cC}\cN$, a $\cC$-balanced $\dag$-functor into a 2-Hilbert space $\cP$ consists of a functor $F \in \Fun^\dag(\cM \boxtimes \cN \to \cP)$ together with a unitary natural isomorphism
$$
\tikzmath[scale=.8]{
\draw[thick] (-.5,0) node[below]{$\scriptstyle\cM$} -- (-.5,1);
\draw[thick] (0,0) node[below]{$\scriptstyle\cC$} arc (0:90:.5);
\draw[thick] (.5,0) node[below]{$\scriptstyle\cN$} -- (.5,1);
\draw[thick] (0,1) -- (0,2) node[above]{$\scriptstyle\cP$};
\roundNbox{fill=white}{(0,1)}{.3}{.4}{.4}{$\scriptstyle F$};
}
\;\xrightarrow{F^\alpha}\;
\tikzmath[scale=.8]{
\draw[thick] (-.5,0) node[below]{$\scriptstyle\cM$} -- (-.5,1);
\draw[thick] (0,0) node[below]{$\scriptstyle\cC$} arc (180:90:.5);
\draw[thick] (.5,0) node[below]{$\scriptstyle\cN$} -- (.5,1);
\draw[thick] (0,1) -- (0,2) node[above]{$\scriptstyle\cP$};
\roundNbox{fill=white}{(0,1)}{.3}{.4}{.4}{$\scriptstyle F$};
}
$$
satisfying the obvious associativity and unitality axioms.
A $\cC$-balanced natural transformation $\eta \colon F \Rightarrow G$ between $\cC$-balanced functors is a natural transformation between the underlying functors which intertwines the actions of the unitary constraints $F^\alpha$ and $G^\alpha$. 
We denote the unitary category of $\cC$-balanced functors and $\cC$-balanced natural transformations by $\mathsf{Bal}_\cC(\cM,\cN;\cP)$.

It is easy to see that $\mathsf{Bal}$ is functorial in $\cM$, $\cN$, and $\cP$.
\end{defn}

\begin{rem}
    There exists a universal $\cC$-balanced functor 
    $$\iota \in \mathsf{Bal}_\cC(\cM,\cN;\cM \boxtimes_\cC \cN)$$
    such that for every 2-Hilbert space $\cP$, precomposition with $\iota$ is an equivalence of unitary categories
    $$
    \Fun^\dag(\cM \boxtimes_\cC \cN \to \cP) \xrightarrow{\iota^*} \mathsf{Bal}_\cC(\cM,\cN;\cP).
    $$
We note that this universal property only determines $\cM \boxtimes_\cC \cN$ up to unitary equivalence.
Once a unitary 2-categorical Yoneda lemma has been established, we expect that $\cM \boxtimes_\cC \cN$ can be determined up to isometric equivalence as the object representing the functor
$$
\mathsf{Bal}_\cC(\cM,\cN; -) \colon 2\Hilb \to 2\Hilb.
$$ 
\end{rem}

\begin{rem}\label{rem:unitary-assoc-constraint}
For H*-multifusion categories $\cC$ and $\cD$,
one may easily generalize the notion of a balanced functor from a left $\cC$-module $\cM$, a $\cC$-$\cD$ bimodule $\cN$, and a right $\cD$-module $\cO$.
For a 2-Hilbert space $\cP$, this yields a unitary category
$$\mathsf{Bal}_{\cC,\cD}(\cM,\cN,\cO;\cP).$$
It is then easy to show that both $\cM \boxtimes_\cC (\cN \boxtimes_\cD \cO)$ and $(\cM \boxtimes_\cC \cN) \boxtimes_\cD \cO$ represent the functor $\mathsf{Bal}_{\cC,\cD}(\cM,\cN,\cO; - )$. This provides a unitary associativity constraint, which we conjecture to be isometric.

In general, one can define balanced functors for arbitrarily many bimodules, which will provide unitary higher associativity constraints.
\end{rem}

\subsection{\texorpdfstring{$\rmH^*$}{H*}-algebras}
\label{sec:H*Algs}
For this section, $(\cC,\vee,\psi)$ is an $\rmH^*$-multifusion category.
We now generalize the notion of $\rmH^*$-algebra to a $\rmC^*$-Frobenius algebra which is separable (multiplication splits as a bimodule map) and standard (a compatiblity condition with $\vee$).
The similar notion of special standard Q-system or standard separable $\rmC^*$-Frobenius algebra (SSFA) was studied in 
\cite{MR4419534} and \cite[\S3.1]{MR4482713}.
In both of those articles, the separator was required to be the adjoint of multiplication.
Moreover, in the former article, duality was completely ignored, and in the latter article, only the standard spherical UAF was analyzed.
To distinguish the setting we study here, we will call our objects of study $\rmH^*$-algebras.
(Notice also the notion of $\rmH^*$-algebra agrees exactly with that of a standard Q-system from \cite{MR3308880}, so we avoid using the term `Q-system' here to avoid confusion.)

\begin{defn}
An $\rmH^*$-\emph{algebra} in a 
$\rmH^*$-multifusion category $(\cC,\vee,\psi)$ 
consists of an (associative, unital) algebra $(A,\mu,\iota)$ which satisfies certain axioms.
We represent $A$, $\mu$, $\iota$ and their adjoints $\mu^\dag$, $\iota^\dag$ graphically as follows:
$$
\tikzmath{
\fill[\BColor, rounded corners=5pt ] (0,0) rectangle (.6,.6);
\draw[\QsColor,thick] (.3,0) -- (.3,.6);
}= A
\qquad
\tikzmath{
\fill[\BColor, rounded corners=5pt] (-.3,0) rectangle (.9,.6);
\draw[\QsColor,thick] (0,0) arc (180:0:.3cm);
\draw[\QsColor,thick] (.3,.3) -- (.3,.6);
}=\mu
\qquad
\tikzmath{
\fill[\BColor, rounded corners=5pt] (-.3,0) rectangle (.9,-.6);
\draw[\QsColor,thick] (0,0) arc (-180:0:.3cm);
\draw[\QsColor,thick] (.3,-.3) -- (.3,-.6);
}=\mu^\dag
\qquad
\tikzmath{
\fill[\BColor, rounded corners=5pt] (-.3,-.3) rectangle (.3,.3);
\draw[\QsColor,thick] (0,0) -- (0,.3);
\filldraw[\QsColor] (0,0) circle (.05cm);
}=\iota
\qquad
\tikzmath{
\fill[\BColor, rounded corners=5pt] (-.3,-.3) rectangle (.3,.3);
\draw[\QsColor,thick] (0,0) -- (0,-.3);
\filldraw[\QsColor] (0,0) circle (.05cm);
}=\iota^\dag.
$$
The multiplication and unit satisfy the following axioms:
\begin{enumerate}[label=($\rmH^*$\arabic*)]
\item
\label{H:Frobenius}
($\rmC^*$-Frobenius)
$\mu^\dag$ is an $A$-$A$ bimodule map, i.e.,
$
\tikzmath{
\fill[\BColor, rounded corners=5pt] (-.3,-.6) rectangle (1.5,.6);
\draw[\QsColor,thick] (0,-.6) -- (0,0) arc (180:0:.3cm) arc (-180:0:.3cm) -- (1.2,.6);
\draw[\QsColor,thick] (.3,.3) -- (.3,.6);
\draw[\QsColor,thick] (.9,-.3) -- (.9,-.6);
}
=
\tikzmath{
\fill[\BColor, rounded corners=5pt] (-.3,0) rectangle (.9,1.2);
\draw[\QsColor,thick] (0,0) arc (180:0:.3cm);
\draw[\QsColor,thick] (0,1.2) arc (-180:0:.3cm);
\draw[\QsColor,thick] (.3,.3) -- (.3,.9);
}
=
\tikzmath{
\fill[\BColor, rounded corners=5pt] (-.3,.6) rectangle (1.5,-.6);
\draw[\QsColor,thick] (0,.6) -- (0,0) arc (-180:0:.3cm) arc (180:0:.3cm) -- (1.2,-.6);
\draw[\QsColor,thick] (.3,-.3) -- (.3,-.6);
\draw[\QsColor,thick] (.9,.3) -- (.9,.6);
}
$\,.
\item\label{H:Separable} (Separable) The endomorphism $\mu\mu^\dag$ of $A$ is invertible.
\item\label{H:Standard} (Standard)
For all endomorphisms $f$ of $A$, 
$
\psi\left(
\tikzmath{
\draw[\QsColor,thick] (.4,.7) -- (.4,1);
\filldraw (.4,1) circle (.05cm);
\draw[\QsColor,thick] (.8,.3) arc (0:180:.4cm) -- +(0,-.6) arc (-180:-90:.4) coordinate (x1) arc (-90:0:.4) -- +(0,.6);
\roundNbox{fill=white}{(0,0)}{.3}{0}{0}{$f$}
\draw[\QsColor,thick] (x1) -- +(0,-.3) coordinate (x2);
\filldraw (x2) circle (.05cm);
}
\right)
=
\psi\left(
\tikzmath{
\draw[\QsColor,thick] (.4,.7) -- (.4,1);
\filldraw (.4,1) circle (.05cm);
\draw[\QsColor,thick] (.8,.3) arc (0:180:.4cm) -- +(0,-.6) arc (-180:-90:.4) coordinate (x1) arc (-90:0:.4) -- +(0,.6);
\roundNbox{fill=white}{(0.8,0)}{.3}{0}{0}{$f$}
\draw[\QsColor,thick] (x1) -- +(0,-.3) coordinate (x2);
\filldraw (x2) circle (.05cm);
}
\right).
$
\end{enumerate}
An $\rmH^*$-algebra is called a \emph{standard special Q-system} or \emph{standard separable $\rmC^*$-Frobenius algebra} if $\mu\mu^\dag=\id_A$.
\end{defn}

\begin{facts}
We now list some basic facts about $\rmH^*$-algebras.
\begin{enumerate}[label=(A\arabic*)]
\item
The $\rmH^*$-algebras in $\Hilb$ equipped with its standard UDF and trace $\Tr_H^{\rmB\Hilb}:=\Tr_{B(H)}$ are exactly $\rmH^*$-algebras in the sense of Definition \ref{defn:1H*Alg}.
First, by the GNS construction,
every such $\rmH^*$-algebra $(A,\Tr_A)$ gives a Hilbert space $L^2(A,\Tr_A)$
which is $A$ again with inner product
$$
\langle a| b\rangle_{L^2(A,\Tr_A)} := \Tr(a^*b).
$$
It is straightforward to verify that $L^2(A,\Tr_A)$ with its multiplication and unit is an $\rmH^*$-algebra.
Conversely, given an $\rmH^*$-algebra $(A,\mu,\iota)\in\Hilb$,
the \emph{realization} $|A|:=\Hom(\bbC\to A)$ \cite{MR4419534} gives a unitary algebra with a unitary trace given by
$$
\Tr_{|A|} := 
\Psi_\star\left(\,
\tikzmath{
\draw[\QsColor,thick] (0,.3) -- (0,.6);
\filldraw (0,.6) circle (.05cm);
\roundNbox{fill=white}{(0,0)}{.3}{0}{0}{$f$}
}
\,\right)
=
\Tr_{\bbC}\left(\,
\tikzmath{
\draw[\QsColor,thick] (0,.3) -- (0,.6);
\filldraw (0,.6) circle (.05cm);
\roundNbox{fill=white}{(0,0)}{.3}{0}{0}{$f$}
}
\,\right),
$$
which is clearly tracial by \ref{H:Standard} and faithful.
One verifies these two constructions are mutually inverse.
\item
\label{rem:HStarAlgsQSystems}
Every $\rmH^*$-algebra is (non-unitarily) equivalent to a standard Q-system.
Indeed, for $\rm H^*$-algebra $(A,\mu,\iota)$, 
setting $x:=(\mu\mu^\dag)^{\frac{1}{2}}$, 
$(A,x^{-1}\mu, x\iota)$ is a standard Q-system,
where the equivalence is given by $x$.
\item
By combining \ref{rem:HStarAlgsQSystems} and \cite[Facts~3.4]{MR4419534},
\begin{itemize}
\item 
$\rmC^*$-Frobenius follows automatically from unitality, associativity, and separability.
\item 
Associativity follows automatically from $\rmC^*$-Frobenius and separable.
\end{itemize}
\item
\label{rem:RightInvIffBubbleInv}
Given an algebra $A$ satisfying \ref{H:Frobenius}, then \ref{H:Separable} is equivalent to ordinary separability.
Indeed, \ref{H:Frobenius} implies that the category of $A$-$A$ bimodules ${}_A\cC_A$ is unitary and semisimple, and thus equivalent to $\Hilb^{\oplus n}$ for some finite $n$.
In $\Hilb^{\oplus n}$, a morphism $f:x\to y$ is right invertible if and only if $ff^\dag$ is invertible.
\item 
The standardness condition \ref{H:Standard} holds if and only if
for all $c\in\cC$, $f:c\Rightarrow A$ and $g:c^\vee\Rightarrow A$,
$$
\psi\left(
\tikzmath{
\draw[\QsColor,thick] (.4,.7) -- (.4,1);
\filldraw (.4,1) circle (.05cm);
\draw[thick, blue] (0,-.3) node[left,yshift=-.2cm]{$\scriptstyle c$} arc (-180:0:.4cm) node[right,yshift=-.2cm]{$\scriptstyle c^\vee$};
\draw[\QsColor,thick] (.8,.3) arc (0:180:.4cm);
\roundNbox{fill=white}{(0,0)}{.3}{0}{0}{$f$}
\roundNbox{fill=white}{(.8,0)}{.3}{0}{0}{$g$}
}
\right)
=
\psi\left(
\tikzmath{
\draw[\QsColor,thick] (.4,.7) -- (.4,1);
\filldraw (.4,1) circle (.05cm);
\draw[thick, blue] (0,-.3) node[left,yshift=-.2cm]{$\scriptstyle c^\vee$} arc (-180:0:.4cm) node[right,yshift=-.2cm]{$\scriptstyle c$};
\draw[\QsColor,thick] (.8,.3) arc (0:180:.4cm);
\roundNbox{fill=white}{(0,0)}{.3}{0}{0}{$g$}
\roundNbox{fill=white}{(.8,0)}{.3}{0}{0}{$f$}
}
\right),
$$
which is equivalent to either of the following morphisms (and thus both) being unitary:
$$
\tikzmath{
\draw[\QsColor,thick] (0,-.3) --node[left]{$\scriptstyle A$} (0,.3) arc (180:0:.3cm) arc (-180:0:.3cm) --node[right]{$\scriptstyle A^\vee$} (1.2,1.2);
\draw[\QsColor,thick] (.3,.9) -- (.3,.6);
\filldraw (.3,.9) circle (.05cm);
}
\qquad\qquad\text{or}\qquad\qquad
\tikzmath{
\draw[\QsColor,thick] (0,-.3) --node[right]{$\scriptstyle A$} (0,.3) arc (0:180:.3cm) arc (0:-180:.3cm) --node[left]{$\scriptstyle A^\vee$} (-1.2,1.2);
\draw[\QsColor,thick] (-.3,.9) -- (-.3,.6);
\filldraw (-.3,.9) circle (.05cm);
}
\,.
$$
Observe that the above morphisms are unitary if and only if they are equal; the dagger of the first map is the inverse of the second. 
It follows immediately that \ref{H:Standard} implies that without loss of generality, the evaluation and coevaluation morphisms for $A$ assigned by the UDF are given by $\ev_A=\iota^\dag\circ \mu$ and $\coev_A=\mu^\dag\circ \iota$.
\end{enumerate}
\end{facts}

\begin{ex}
\label{ex:separabledualyieldsH*algebra}
 Suppose $c\in \cC$, 
and recall that $\id_c\otimes \ev_c\ev_c^\dag\otimes \id_{c^\vee}$ is invertible in $\End(c\otimes c^\vee)$ by \cite[Cor.~1.19]{MR2298822}, see also \cite[Facts 3.5(Z3)]{MR4419534}.
Then
$c\otimes c^\vee$ has a canonical structure of an $\rmH^*$-algebra
with multiplication and unit given by
$$
\mu:=\id_c\otimes \ev_c\otimes \id_{c^\vee}
\qquad\text{and}\qquad 
\iota:= \coev_c.
$$
The only interesting thing to verify is  \ref{H:Standard}, which follows as $\vee$ is a UDF.
Indeed, $c\otimes c^\vee$ is canonically symmetrically self-dual, so identifying $(c\otimes c^\vee)^\vee\cong c\otimes c^\vee$,
the morphisms in the equivalent condition listed under \ref{H:Standard} are both identities.
(Here, we use that $\coev_{c^\vee}=\ev^\dag_c$ and $\ev_{c^\vee}=\coev_c^\dag$.)
\end{ex}

A particular case of the above example gives a more conceptual approach to the construction in \cite[Appendix A]{MR3933035} that the internal endomorphisms of any object in unitary module category for a unitary tensor category can be equipped with the structure of a Q-system.

\begin{sub-ex}
\label{ex:H*AlgOnEnd(m)}
Given a right unitary $\cC$-module category $\cM$ equipped with a unitary $\cC$-module trace $\Tr^\cM$,
we get a `linking' unitary multifusion category
$$
\cL(\cM_{\cC}) \coloneqq
\begin{pmatrix}
\cC & \cM^{\op}
\\
\cM & \End(\cM_\cC)
\end{pmatrix}.
$$
The tensor products in this category are given by
$$
\begin{pmatrix}
0 & n^{\op} 
\\
0 & 0
\end{pmatrix}
\otimes
\begin{pmatrix}
0 & 0
\\
m & 0
\end{pmatrix}
:=
\begin{pmatrix}
[n,m] & 0 \\
0 & 0
\end{pmatrix} 
\qquad\text{and}\qquad
\begin{pmatrix}
0 & 0
\\
m & 0
\end{pmatrix}
\otimes
\begin{pmatrix}
0 & n^{\op}
\\
0 & 0
\end{pmatrix}
:=
\begin{pmatrix}
0 & 0 \\
0 &  m \lhd [n, -]
\end{pmatrix}
$$
where the internal hom $[n,m]$ is taken under the unitary adjunction
$$
\cM(n \lhd c \to m) \cong \cC(c \to [n,m]),
$$
i.e., $[n,-]$ is the unitary adjoint of $n \lhd -$.
Observe that $\cL(\cM_\cC)$ is equivalent to the unitary multifusion category 
$$
\End_\cC(\cC_\cC\boxplus \cM_{\cC})
=
\begin{pmatrix}
\End(\cC_\cC) & \Hom(\cM_\cC \to \cC_\cC) \\
\Hom(\cC_\cC \to \cM_\cC)
& \End(\cM_\cC)
\end{pmatrix},
$$
where
$m\in\cM$ corresponds to $m\lhd -\in \Hom(\cC_\cC \to \cM_\cC)$
and
$n^{\op}\in\cM^{\op}$ corresponds to $[n,-]\in \Hom(\cM_\cC \to \cC_\cC)$.
Unitary adjunction gives a canonical UDF on this unitary multifusion category such that the inclusion $\cC\hookrightarrow \cL(\cM_\cC)$ is unitary pivotal.
Unpacking,
the dual for $m$ is $m^{\op}$.
The evaluation map $\ev_m:m\otimes m^{\op} \Rightarrow \id_\cM$ is the mate of the identity natural transformation under the unitary adjunction 
$$
\cM(m \lhd [m,-] \to -) \cong \cC([m,-] \to [m,-]).
$$
The map $\coev_m:1_\cC \to m^{\op} \otimes m$ is given by the mate of $\id_m$ under the unitary adjunction
\[ \cM(m,m) \cong \cC(1_\cC,[m,m]). \]
Note that the conventions for duals are opposite to our usual conventions as composition of functors is read right-to-left, which is opposite the left-to-right convention for usual tensor products.

We conclude that by the above Example \ref{ex:separabledualyieldsH*algebra}, for any $m\in \cM$, the unitary internal end $\underline{\End}_\cC(m)=[m,m]=m^{\op}\otimes m$ has a canonical $\rmH^*$-algebra structure.
\end{sub-ex}

\begin{defn}
\label{defn:HStarAlgC}
The $\rmH^*$-algebras in $\cC$ are the objects in a unitary 2-category $\HstarAlg(\cC)$ whose 1-morphisms are unital, associative bimodules and whose 2-morphisms are intertwiners.
The structure of this 2-category is closely related to the \emph{Q-system completion} $\QSys(\rmB\cC)$ studied in \cite{MR4419534} and the category $\mathsf{Bimod}(\cC)$ studied in \cite{MR4482713}; the main difference is that we must now correct for the bubble term $\mu\mu^\dag$ which no longer pops.
We explain where these differences occur below.

A \emph{bimodule} for two $\rmH^*$-algebras $A,B$ in $\cC$ is an object $M\in\cC$ together with maps $\lambda:A\otimes M \to M$ and $\rho:M\otimes B\to M$ which are unital and associative.
As above, we represent $\lambda,\rho$ by trivalent vertices whose vertical reflections are their adjoints.
$$
\tikzmath{
\begin{scope}
\clip[rounded corners = 5pt] (-.9,-.6) rectangle (.3,.5);
\filldraw[\BColor] (-.9,-.6) rectangle (0,.5);
\filldraw[\BColor] (0,-.6) rectangle (.3,.5);
\end{scope}
\draw[\XColor,thick] (0,-.6) -- (0,.5);
\draw[\AsColor,thick] (-.6,-.6) -- (-.6,-.4) arc (180:90:.6cm);
\draw[\AsColor,thick] (-.3,-.6) -- (-.3,-.4) arc (180:90:.3cm);
}
=
\tikzmath{
\begin{scope}
\clip[rounded corners = 5pt] (-.9,-.6) rectangle (.3,.5);
\filldraw[\BColor] (-.9,-.6) rectangle (0,.5);
\filldraw[\BColor] (0,-.6) rectangle (.3,.5);
\end{scope}
\draw[\XColor,thick] (0,-.6) -- (0,.5);
\draw[\AsColor,thick] (-.4,-.2) arc (180:90:.4cm);
\draw[\AsColor,thick] (-.6,-.6) -- (-.6,-.4)  arc (180:0:.2cm) -- (-.2,-.6);
}
\qquad
\tikzmath{
\begin{scope}
\clip[rounded corners = 5pt] (-.3,-.6) rectangle (.9,.5);
\filldraw[\BColor] (-.3,-.6) rectangle (0,.5);
\filldraw[\BColor] (0,-.6) rectangle (.9,.5);
\end{scope}
\draw[\XColor,thick] (0,-.6) -- (0,.5);
\draw[\BsColor,thick] (.6,-.6) -- (.6,-.4) arc (0:90:.6cm);
\draw[\BsColor,thick] (.3,-.6) -- (.3,-.4) arc (0:90:.3cm);
}
=
\tikzmath{
\begin{scope}
\clip[rounded corners = 5pt] (-.3,-.6) rectangle (.9,.5);
\filldraw[\BColor] (-.3,-.6) rectangle (0,.5);
\filldraw[\BColor] (0,-.6) rectangle (.9,.5);
\end{scope}
\draw[\XColor,thick] (0,-.6) -- (0,.5);
\draw[\BsColor,thick] (.4,-.2) arc (0:90:.4cm);
\draw[\BsColor,thick] (.6,-.6) -- (.6,-.4)  arc (0:180:.2cm) -- (.2,-.6);
}
\qquad
\tikzmath{
\begin{scope}
\clip[rounded corners = 5pt] (-.7,-.5) rectangle (.7,.5);
\filldraw[\BColor] (-.7,-.5) rectangle (0,.5);
\filldraw[\BColor] (0,-.5) rectangle (.7,.5);
\end{scope}
\draw[\XColor,thick] (0,-.5) -- (0,.5);
\draw[\AsColor,thick] (-.4,-.5) -- (-.4,-.2) arc (180:90:.4cm);
\draw[\BsColor,thick] (.4,-.5) arc (0:90:.4cm);
}
=
\tikzmath{
\begin{scope}
\clip[rounded corners = 5pt] (-.7,-.5) rectangle (.7,.5);
\filldraw[\BColor] (-.7,-.5) rectangle (0,.5);
\filldraw[\BColor] (0,-.5) rectangle (.7,.5);
\end{scope}
\draw[\XColor,thick] (0,-.5) -- (0,.5);
\draw[\AsColor,thick] (-.4,-.5) arc (180:90:.4cm);
\draw[\BsColor,thick] (.4,-.5) -- (.4,-.2) arc (0:90:.4cm);
}
\qquad
\tikzmath{
\begin{scope}
\clip[rounded corners = 5pt] (-.7,-.5) rectangle (.3,.5);
\filldraw[\BColor] (-.7,-.5) rectangle (0,.5);
\filldraw[\BColor] (0,-.5) rectangle (.3,.5);
\end{scope}
\draw[\XColor,thick] (0,-.5) -- (0,.5);
\draw[\AsColor,thick] (-.4,-.2) arc (180:90:.4cm);
\filldraw[\AsColor] (-.4,-.2) circle (.05cm);
}
=
\tikzmath{
\begin{scope}
\clip[rounded corners = 5pt] (-.3,-.5) rectangle (.3,.5);
\filldraw[\BColor] (-.3,-.5) rectangle (0,.5);
\filldraw[\BColor] (0,-.5) rectangle (.3,.5);
\end{scope}
\draw[\XColor,thick] (0,-.5) -- (0,.5);
}
\qquad
\tikzmath{
\begin{scope}
\clip[rounded corners = 5pt] (-.3,-.5) rectangle (.7,.5);
\filldraw[\BColor] (-.3,-.5) rectangle (0,.5);
\filldraw[\BColor] (0,-.5) rectangle (.7,.5);
\end{scope}
\draw[\XColor,thick] (0,-.5) -- (0,.5);
\draw[\BsColor,thick] (.4,-.2) arc (0:90:.4cm);
\filldraw[\BsColor] (.4,-.2) circle (.05cm);
}
=
\tikzmath{
\begin{scope}
\clip[rounded corners = 5pt] (-.3,-.5) rectangle (.3,.5);
\filldraw[\BColor] (-.3,-.5) rectangle (0,.5);
\filldraw[\BColor] (0,-.5) rectangle (.3,.5);
\end{scope}
\draw[\XColor,thick] (0,-.5) -- (0,.5);
}
$$
It follows that $M$ is separable and satisfies the $\rmC^*$-Frobenius axiom by \cite[Facts~3.16]{MR4419534}.
$$
\tikzmath{
\begin{scope}
\clip[rounded corners = 5pt] (-.5,-.5) rectangle (.3,.5);
\filldraw[\BColor] (-.6,-.5) rectangle (0,.5);
\filldraw[\BColor] (0,-.5) rectangle (.3,.5);
\end{scope}
\draw[\XColor,thick] (0,-.5) -- (0,.5);
\draw[\AsColor, thick] (0,-.3) arc (270:90:.3cm);
}
\,,\,
\tikzmath{
\begin{scope}
\clip[rounded corners = 5pt] (-.3,-.5) rectangle (.5,.5);
\filldraw[\BColor] (-.3,-.5) rectangle (0,.5);
\filldraw[\BColor] (0,-.5) rectangle (.6,.5);
\end{scope}
\draw[\XColor,thick] (0,-.5) -- (0,.5);
\draw[\BsColor, thick] (0,-.3) arc (-90:90:.3cm);
}
\in\End(M)^\times
\qquad
\tikzmath{
\begin{scope}
\clip[rounded corners = 5pt] (-1.3,-.5) rectangle (.3,.8);
\filldraw[\BColor] (-1.3,-.5) rectangle (0,.8);
\filldraw[\BColor] (0,-.5) rectangle (.3,.8);
\end{scope}
\draw[\XColor,thick] (0,-.5) -- (0,.8);
\draw[\AsColor, thick] (-1,-.5) -- (-1,.2) arc (180:0:.3cm) arc (180:270:.4cm);
\draw[\AsColor, thick] (-.7,.5) -- (-.7,.8);
}
=
\tikzmath{
\begin{scope}
\clip[rounded corners = 5pt] (-.7,-.2) rectangle (.3,1.1);
\filldraw[\BColor] (-.7,-.2) rectangle (0,1.1);
\filldraw[\BColor] (0,-.2) rectangle (.3,1.1);
\end{scope}
\draw[\XColor,thick] (0,-.2) -- (0,1.1);
\draw[\AsColor, thick] (-.4,-.2) arc (180:90:.4cm);
\draw[\AsColor, thick] (-.4,1.1) arc (180:270:.4cm);
}
=
\tikzmath{
\begin{scope}
\clip[rounded corners = 5pt] (-1.3,-.8) rectangle (.3,.5);
\filldraw[\BColor] (-1.3,-.8) rectangle (0,.5);
\filldraw[\BColor] (0,-.8) rectangle (.3,.5);
\end{scope}
\draw[\XColor,thick] (0,-.8) -- (0,.5);
\draw[\AsColor, thick] (-1,.5) -- (-1,-.2) arc (-180:0:.3cm) arc (180:90:.4cm);
\draw[\AsColor, thick] (-.7,-.5) -- (-.7,-.8);
}
\qquad
\tikzmath{
\begin{scope}
\clip[rounded corners = 5pt] (-.3,-.5) rectangle (1.3,.8);
\filldraw[\BColor] (-.3,-.5) rectangle (0,.8);
\filldraw[\BColor] (0,-.5) rectangle (1.3,.8);
\end{scope}
\draw[\XColor,thick] (0,-.5) -- (0,.8);
\draw[\BsColor, thick] (1,-.5) -- (1,.2) arc (0:180:.3cm) arc (0:-90:.4cm);
\draw[\BsColor, thick] (.7,.5) -- (.7,.8);
}
=
\tikzmath{
\begin{scope}
\clip[rounded corners = 5pt] (-.3,-.2) rectangle (.7,1.1);
\filldraw[\BColor] (-.3,-.2) rectangle (0,1.1);
\filldraw[\BColor] (0,-.2) rectangle (.7,1.1);
\end{scope}
\draw[\XColor,thick] (0,-.2) -- (0,1.1);
\draw[\BsColor, thick] (.4,-.2) arc (0:90:.4cm);
\draw[\BsColor, thick] (.4,1.1) arc (0:-90:.4cm);
}
=
\tikzmath{
\begin{scope}
\clip[rounded corners = 5pt] (-.3,-.8) rectangle (1.3,.5);
\filldraw[\BColor] (-.3,-.8) rectangle (0,.5);
\filldraw[\BColor] (0,-.8) rectangle (1.3,.5);
\end{scope}
\draw[\XColor,thick] (0,-.8) -- (0,.5);
\draw[\BsColor, thick] (1,.5) -- (1,-.2) arc (0:-180:.3cm) arc (0:90:.4cm);
\draw[\BsColor, thick] (.7,-.5) -- (.7,-.8);
}
$$
An intertwiner $f: {}_AM_B\to {}_AN_B$ of $A$-$B$ bimodules is a map $f: M\to N$ such that
$$
\tikzmath{
\begin{scope}
\clip[rounded corners = 5pt] (-.7,-1.4) rectangle (.7,.5);
\filldraw[\BColor] (-.7,-.5) rectangle (0,.5);
\filldraw[\BColor] (0,-.5) rectangle (.7,.5);
\end{scope}
\draw[\XColor,thick] (0,-.5) -- (0,.5);
\draw[\AsColor,thick] (-.6,-1.4) -- (-.6,-.4) arc (180:90:.6cm);
\draw[\BsColor,thick] (.6,-1.4) -- (.6,-.7) arc (0:90:.6cm);
\roundNbox{fill=white}{(0,-.8)}{.3}{0}{0}{$f$}
\draw[\XColor,thick] (0,-1.1) -- (0,-1.4);
}
=
\tikzmath{
\begin{scope}
\clip[rounded corners = 5pt] (-.7,-.5) rectangle (.7,1.4);
\filldraw[\BColor] (-.7,-.5) rectangle (0,.5);
\filldraw[\BColor] (0,-.5) rectangle (.7,.5);
\end{scope}
\draw[\XColor,thick] (0,-.5) -- (0,.5);
\draw[\AsColor,thick] (-.4,-.5) -- (-.4,-.2) arc (180:90:.4cm);
\draw[\BsColor,thick] (.4,-.5) arc (0:90:.4cm);
\roundNbox{fill=white}{(0,.8)}{.3}{0}{0}{$f$}
\draw[\XColor,thick] (0,1.1) -- (0,1.4);
}
$$
which also implies $f^\dag:N\to M$ is an $A$-$B$ bimodule map by \cite[Facts~3.15]{MR4419534}.

Given an $A$-$B$ bimodule $M$ and a $B$-$C$ bimodule $N$,
1-composition in $\HstarAlg(\cC)$ is achieved by splitting the \emph{separability idempotent}
$$
p_{M,N}
=
\tikzmath{
\draw[thick, black] (-.3,-.6) -- (-.3,.6);
\draw[thick, black] (.3,-.6) -- (.3,.6);
\draw[thick, \BsColor] (-.3,0) -- (.3,0);
}
:=
\tikzmath{
\draw[thick, black] (-.5,-.7) -- (-.5,.7);
\draw[thick, black] (.5,-.7) -- (.5,.7);
\draw[thick, \BsColor] (0,0) circle (.1cm);
\node[\BsColor] at (.3,.2) {$\scriptstyle -1$};
\draw[thick, \BsColor] (0,.1) to[out=90,in=-135] (.5,.5);
\draw[thick, \BsColor] (0,-.1) to[out=-90,in=45] (-.5,-.5);
}
=
\tikzmath{
\draw[thick, black] (-.5,-.7) -- (-.5,.7);
\draw[thick, black] (.5,-.7) -- (.5,.7);
\draw[thick, \BsColor] (0,0) circle (.1cm);
\node[\BsColor] at (.3,.2) {$\scriptstyle -1$};
\draw[thick, \BsColor] (0,.1) to[out=90,in=-45] (-.5,.5);
\draw[thick, \BsColor] (0,-.1) to[out=-90,in=135] (.5,-.5);
}
\in 
\End_{\fX}(M\otimes N)
\qquad\text{where}\qquad
\tikzmath{
\draw[thick, \BsColor] (0,0) circle (.1cm);
\node[\BsColor] at (.2,.2) {$\scriptstyle r$};
\draw[thick, \BsColor] (0,.1) -- (0,.3);
\draw[thick, \BsColor] (0,-.1) -- (0,-.3);
}
:=
\left(
\tikzmath{
\draw[thick, \BsColor] (0,0) circle (.2cm);
\filldraw[thick, \BsColor] (0,.2) -- (0,.4);
\filldraw[thick, \BsColor] (0,-.2) -- (0,-.4);
}
\right)^{r},\, r\in\bbR.
$$
We call the image of this idempotent $M\boxtimes_A N$, which is unique up to a contractible choice.
The unitors $A\boxtimes_A M \to M$ and $M\boxtimes_B B \to M$ are given by
$$
\tikzmath{
\draw[thick, black] (0,-1.3) -- node[right]{$\scriptstyle M$} (0,.3);
\draw[thick, \AsColor] (-.6,-.8)  to[out=90,in=-135] node[above]{$\scriptstyle A$}(0,0);
\draw[thick,\AsColor] (-.6,-.9) circle (.1cm);
\node[\AsColor] at (-.3,-.7) {$\scriptstyle -\frac{1}{2}$};
\draw[thick,\AsColor] (-.6,-1) -- (-.6,-1.3);
}
\qquad\text{and}\qquad
\tikzmath{
\draw[thick, black] (0,-1.3) -- node[left]{$\scriptstyle M$} (0,.3);
\draw[thick, \BsColor] (.6,-.8)  to[out=90,in=-45] node[right]{$\scriptstyle B$}(0,0);
\draw[thick,\BsColor] (.6,-.9) circle (.1cm);
\node[\BsColor] at (.9,-.7) {$\scriptstyle -\frac{1}{2}$};
\draw[thick,\BsColor] (.6,-1) -- (.6,-1.3);
}\,.
$$

Given an $A$-$B$ bimodule $M$, its dual $M^\vee$ carries a $B$-$A$ bimodule structure given on the right hand side in the diagram below. 
Since $M$ is separable as both a left and right module, 
$\ev_M^\cC: M^\vee\otimes M \to 1_\cC$
and 
$\coev_M^\cC: 1_\cC\to M\otimes M^\vee$
both absorb the corresponding separability idemopotents, 
as in the right hand side of the diagrams below.
$$
\tikzmath{
\draw[thick, black] (-.6,-.8) node[below]{$\scriptstyle M^\vee$} -- (-.6,.3) arc(180:0:.3cm) --node[left,xshift=.1cm]{$\scriptstyle M$} (0,0) -- (0,-.3) arc(-180:0:.3cm) --node[right]{$\scriptstyle M^\vee$} (.6,1.5);
\draw[thick, \AsColor] (-.2,-.8) node[below]{$\scriptstyle A$} to[out=90,in=-135] (0,0);
\draw[thick, \BsColor] (0,0) to[out=45,in=-90] (.4,.4) arc(0:180:.7cm) -- (-1,-.8) node[below]{$\scriptstyle B$};
\filldraw[thick, \BsColor] (-.3,1.3) circle (.05cm) -- (-.3,1.1);
}
\qquad\qquad
\tikzmath{
\draw[thick, black] (0,.3) node[above]{$\scriptstyle M$} -- (0,-.3) arc(-180:0:.3cm) -- (.6,.3) node[above]{$\scriptstyle M^\vee$};
}
=
\tikzmath{
\draw[thick, black] (0,.3) node[above]{$\scriptstyle M$} -- (0,-.3) arc(-180:0:.3cm) -- (.6,.3) node[above]{$\scriptstyle M^\vee$};
\draw[thick, \BsColor] (0,0) -- (.6,0);
}
\qquad\qquad
\tikzmath{
\draw[thick, black] (0,-.3) node[below]{$\scriptstyle M^\vee$} -- (0,.3) arc(180:0:.3cm) -- (.6,-.3) node[below]{$\scriptstyle M$};
}
=
\tikzmath{
\draw[thick, black] (0,-.3) node[below]{$\scriptstyle M^\vee$} -- (0,.3) arc(180:0:.3cm) -- (.6,-.3) node[below]{$\scriptstyle M$};
\draw[thick, \AsColor] (0,0) -- (.6,0);
}
$$
We thus get a 1-parameter family of UAFs on $\HstarAlg(\cC)$
indexed by $\delta\in \bbR$
given by
\begin{equation}
\label{eq:1ParameterUAFOnH*Alg}
\left.
\begin{aligned}
\ev_M^\delta
&:=
\tikzmath{
\draw[thick,\XColor] (.4,-.5) -- (.4,.4) arc (180:0:.5 and .4) -- (1.4,-.5);
\node at (.4,-.8) {$\scriptstyle M^\vee$};
\node at (1.6,-.8) {$\scriptstyle M$};
\draw[thick,\BsColor] (1.4,0) to[out=45,in=-90] (1.8,.6);
\draw[thick,\AsColor] (1.4,.3) -- (.4,-.3);
\filldraw[thick,draw=\AsColor,fill=white] (.9,0) circle (.1);
\node[\AsColor] at (.9,.45) {$\scriptstyle -\delta-\frac{3}{2}$};
\draw[thick,\BsColor] (1.8,.7) circle (.1cm);
\draw[thick,\BsColor] (1.8,.8) -- (1.8,1.1);
\node[\BsColor] at (2,.9) {$\scriptstyle \delta$};
} 
=
\tikzmath{
\draw[thick,\XColor] (.4,-.5) -- (.4,.4) arc (180:0:.5 and .4) -- (1.4,-.5);
\node at (.4,-.8) {$\scriptstyle M^\vee$};
\node at (1.6,-.8) {$\scriptstyle M$};
\draw[thick,\BsColor] (1.4,0) to[out=45,in=-90] (1.8,.6);
\draw[thick,\AsColor] (1.4,.3) -- (.9,0);
\filldraw[thick,draw=\AsColor,fill=white] (.9,0) circle (.1);
\node[\AsColor] at (.9,.45) {$\scriptstyle -\delta-\frac{1}{2}$};
\draw[thick,\BsColor] (1.8,.7) circle (.1cm);
\draw[thick,\BsColor] (1.8,.8) -- (1.8,1.1);
\node[\BsColor] at (2,.9) {$\scriptstyle \delta$};
} 
\qquad\text{where}\qquad
\tikzmath{
\draw[thick,\AsColor] (0,.1) -- (0,.4);
\draw[thick,white] (0,-.1) -- (0,-.3);
\filldraw[thick,draw=\AsColor,fill=white] (0,0) circle (.1);
\filldraw[white] (0,-.3) circle (.05cm);
\node[\AsColor] at (.2,.2) {$\scriptstyle r$};
}
:=
\tikzmath{
\draw[thick,\AsColor] (0,.1) -- (0,.4);
\draw[thick,\AsColor] (0,-.1) -- (0,-.3);
\filldraw[thick,draw=\AsColor,fill=white] (0,0) circle (.1);
\filldraw[\AsColor] (0,-.3) circle (.05cm);
\node[\AsColor] at (.2,.2) {$\scriptstyle r$};
},\, r\in\bbR
\\
\coev_M^\delta
&:=
\tikzmath{
\draw[thick,\XColor] (.4,.5) -- (.4,-.4) arc (-180:0:.5 and .4) -- (1.4,.5);
\node at (.4,.8) {$\scriptstyle M$};
\node at (1.6,.8) {$\scriptstyle M^\vee$};
\draw[thick,\AsColor] (.4,0) to[out=-135,in=90] (0,-.6);
\draw[thick,\BsColor] (.4,.3) -- (1.4,-.3);
\filldraw[thick,draw=\BsColor,fill=white] (.9,0) circle (.1);
\node[\BsColor] at (1,.3) {$\scriptstyle -\delta-\frac{3}{2}$};
\draw[thick,\AsColor] (0,-.7) circle (.1cm);
\draw[thick,\AsColor] (0,-.8) -- (0,-1.1);
\node[\AsColor] at (.2,-.5) {$\scriptstyle \delta$};
}
=
\tikzmath{
\draw[thick,\XColor] (.4,.5) -- (.4,-.4) arc (-180:0:.5 and .4) -- (1.4,.5);
\node at (.4,.8) {$\scriptstyle M$};
\node at (1.6,.8) {$\scriptstyle M^\vee$};
\draw[thick,\AsColor] (.4,0) to[out=-135,in=90] (0,-.6);
\draw[thick,\BsColor] (.4,.3) -- (.9,0);
\filldraw[thick,draw=\BsColor,fill=white] (.9,0) circle (.1);
\node[\BsColor] at (1,.3) {$\scriptstyle -\delta-\frac{1}{2}$};
\draw[thick,\AsColor] (0,-.7) circle (.1cm);
\draw[thick,\AsColor] (0,-.8) -- (0,-1.1);
\node[\AsColor] at (.2,-.5) {$\scriptstyle \delta$};
}
\end{aligned}
\right\}
\end{equation}
which all yield the identical coherence isomorphisms $\nu:a^\vee\otimes b^\vee\to (b\otimes a)^\vee$ and $\varphi: a\to a^{\vee\vee}$.
Moreover, for each of these UAFs on $\HstarAlg(\cC)$, the organic inclusion $\cC\hookrightarrow \HstarAlg(\cC)$ by $1_\cC\mapsto 1_\cC$ is unitary pivotal.
\end{defn}

\begin{rem}
It was posited in \cite{MR4419534} that the unitary version of a condensation 
monad in the spirit of \cite{1905.09566} in a $\rmC^*$-2-category is a not necessarily unital special Q-system.
It was pointed out to us by David Reutter that if we require a not necessarily unital special Q-system to be compatible with the ambient UDF in an $\rmH^*$-multifusion category
\begin{equation}
\label{eq:UAFCompatible}
\tikzmath{
\fill[\BColor, rounded corners=5pt] (-.5,-.6) rectangle (1.8,1.6);
\draw[\QsColor,thick] (0,-.6+1) --node[left]{$\scriptstyle A^\vee$} (0,0+1) arc (180:0:.3cm) arc (-180:0:.3cm) --node[right]{$\scriptstyle A^\vee$} (1.2,.6+1);
\draw[\QsColor,thick] (.3,.3+1) -- (.3,.6+1);
\draw[\QsColor,thick] (0,.6) -- (0,0) arc (-180:0:.3cm) arc (180:0:.3cm) --node[right]{$\scriptstyle A^\vee$} (1.2,-.6);
\draw[\QsColor,thick] (.3,-.3) -- (.3,-.6);
}
=
\tikzmath{
\fill[\BColor, rounded corners=5pt] (-.5,-.6) rectangle (1.8,1.6);
\draw[\QsColor,thick] (0,-.6+1) --node[left]{$\scriptstyle A^\vee$} (0,0+1) arc (180:0:.3cm) arc (-180:0:.3cm) --node[right]{$\scriptstyle A^\vee$} (1.2,.6+1);
\draw[\QsColor,thick] (.6,+1) .. controls ++(.3,.15) .. (.9,.6+1);
\draw[\QsColor,thick] (0,.6) -- (0,0) arc (-180:0:.3cm) arc (180:0:.3cm) --node[right]{$\scriptstyle A^\vee$} (1.2,-.6);
\draw[\QsColor,thick] (.6,0) .. controls ++(.3,-.15) .. (.9,-.6);
}
\qquad\text{and}\qquad
\tikzmath{
\fill[\BColor, rounded corners=5pt] (.5,-.6) rectangle (-1.8,1.6);
\draw[\QsColor,thick] (0,-.6+1) --node[right]{$\scriptstyle A^\vee$} (0,0+1) arc (0:180:.3cm) arc (0:-180:.3cm) --node[left]{$\scriptstyle A^\vee$} (-1.2,.6+1);
\draw[\QsColor,thick] (-.3,.3+1) -- (-.3,.6+1);
\draw[\QsColor,thick] (0,.6) -- (0,0) arc (0:-180:.3cm) arc (0:180:.3cm) --node[left]{$\scriptstyle A^\vee$} (-1.2,-.6);
\draw[\QsColor,thick] (-.3,-.3) -- (-.3,-.6);
}
=
\tikzmath{
\fill[\BColor, rounded corners=5pt] (.5,-.6) rectangle (-1.8,1.6);
\draw[\QsColor,thick] (0,-.6+1) --node[right]{$\scriptstyle A^\vee$} (0,0+1) arc (0:180:.3cm) arc (0:-180:.3cm) --node[left]{$\scriptstyle A^\vee$} (-1.2,.6+1);
\draw[\QsColor,thick] (-.6,+1) .. controls ++(-.3,.15) .. (-.9,.6+1);
\draw[\QsColor,thick] (0,.6) -- (0,0) arc (0:-180:.3cm) arc (0:180:.3cm) --node[left]{$\scriptstyle A^\vee$} (-1.2,-.6);
\draw[\QsColor,thick] (-.6,0) .. controls ++(-.3,-.15) .. (-.9,-.6);
}
\end{equation}
so that any morphism between 1-composite powers of $A$'s only depends on the connectivity of the graph, then $A$ is automatically unital.
Indeed, using \ref{H:Frobenius} and capping off \eqref{eq:UAFCompatible}, one shows $A$ has both a right and a left unit given by 
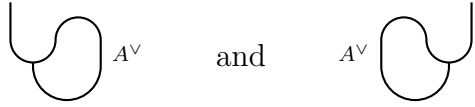
\[
\tikzmath{
\fill[\BColor, rounded corners=5pt] (1.8,-1) rectangle (-.3,.5);
\draw[thick] (0,.5) -- (0,0) arc (-180:0:.3) arc (180:0:.3) --node[right]{$\scriptstyle A^\vee$} (1.2,-.35) arc (0:-180:.45) -- (.3,-.3);
}
\qquad\text{and}\qquad
\tikzmath{
\fill[\BColor, rounded corners=5pt] (-1.8,-1) rectangle (.3,.5);
\draw[thick] (0,.5) -- (0,0) arc (0:-180:.3) arc (0:180:.3) --node[left]{$\scriptstyle A^\vee$} (-1.2,-.35) arc (-180:0:.45) -- (-.3,-.3);
}
\]
respectively,
and thus they are automatically equal.
The above can also be adjusted for $\rmH^*$-algebras by including factors of the bubble inverse.
This fact justifies working solely with unital algebras.
\end{rem}

\subsection{Module categories come from \texorpdfstring{$\rmH^*$}{H*}-algebras}

For this section, $(\cC,\vee,\psi)$ is an $\rmH^*$-multifusion category.

Suppose $A\in\cC$ is an $\rmH^*$-algebra.
Observe that 
$$
\tau_M(f):= 
\tikzmath{
\draw[thick, black] (0,.3) -- (0,.9) arc(0:180:.3cm) --node[left]{$\scriptstyle M^\vee$} (-.6,-.9) arc (-180:0:.3cm) -- (0,-.3);
\draw[thick, \AsColor] (.6,-1.4)  to[out=90,in=-45] node[right]{$\scriptstyle A$}(0,-.6);
\draw[thick, \AsColor] (.6,1.4)  to[out=270,in=45] node[right]{$\scriptstyle A$}(0,.6);
\roundNbox{fill=white}{(0,0)}{.3}{0}{0}{$f$}
}
$$
defines a unitary operator-valued 
$\cC$-module trace
$\tau_M: \cC_A(M\to M) \to \End({}_AA_A)$, which takes values in an abelian unitary algebra.
Post-composing with any faithful weight $\phi: \End({}_AA_A)\to \bbC$ yields a unitary $\cC$-module trace on $\cC_A$.
By Corollary \ref{cor:ExistsModuleTrace}, every unitary module trace on $\cC_A$ is of this form, as the indecomposable summands of $\cC_A$ correspond to the simple algebra summands of $A$.

We now determine the faithful weight $\phi: \End({}_AA_A)\to \bbC$ (up to uniform scalar) such that the unitary internal hom $[A,A]$ is unitarily isomorphic to $A$.

\begin{lem}
\label{lem:module_trace}
If $A\in \cC$ is an $\rmH^*$-algebra,
then every unitary $\cC$-module trace on $\cC_A$ is completely determined by its values on $\cC_A(A\to A)$.
\end{lem}
\begin{proof}
Every $M\in \cC_A$ is a retract of the free module $M\otimes A$ where the retract coisometry is given by
$$
\rho_A
:=
\tikzmath{
\draw[thick, black] (0,-1.3) -- node[left]{$\scriptstyle M$} (0,.3);
\draw[thick, \AsColor] (.4,-.8)  to[out=90,in=-45] node[right]{$\scriptstyle A$}(0,0);
\draw[thick,\AsColor] (.4,-.9) circle (.1cm);
\node[\AsColor] at (.8,-.7) {$\scriptstyle -1/2$};
\draw[thick,\AsColor] (.4,-1) -- (.4,-1.3);
}\,.
$$
If $\Tr^{\cC_A}$ is a unitary $\cC$-module trace on $\cC_A$, then for every $f\in\cC_A(M\to M)$,
$$
\Tr^{\cC_A}_M\left(
\tikzmath{
\draw[thick] (0,-.7) -- (0,.7);
\roundNbox{fill=white}{(0,0)}{.3}{0}{0}{$f$}
}
\right)
=
\Tr^{\cC_A}_M\left(
\tikzmath{
\draw[thick, black] (0,.3) -- node[left]{$\scriptstyle M$} (0,2);
\draw[thick, black] (0,-.3) -- node[left]{$\scriptstyle M$} (0,-2);
\draw[thick, \AsColor] (.6,.8)  to[out=90,in=-45] node[right]{$\scriptstyle A$}(0,1.6);
\draw[thick, \AsColor] (.6,-.8)  to[out=270,in=45] node[right]{$\scriptstyle A$}(0,-1.6);
\draw[thick,\AsColor] (.6,.7) circle (.1cm);
\draw[thick,\AsColor] (.6,-.7) circle (.1cm);
\node[\AsColor] at (1,.9) {$\scriptstyle -1/2$};
\node[\AsColor] at (1,-.5) {$\scriptstyle -1/2$};
\draw[thick,\AsColor] (.6,.6) -- (.6,-.6);
\roundNbox{fill=white}{(0,0)}{.3}{0}{0}{$f$}
}
\right)
=
\Tr^{\cC_A}_{M\otimes A}\left(
\tikzmath{
\draw[thick, black] (0,.3) -- node[left]{$\scriptstyle M$} (0,2);
\draw[thick, black] (0,-.3) -- node[left]{$\scriptstyle M$} (0,-2);
\draw[thick, \AsColor] (.6,-1.4)  to[out=90,in=-45] node[right]{$\scriptstyle A$}(0,-.6);
\draw[thick, \AsColor] (.6,1.4)  to[out=270,in=45] node[right]{$\scriptstyle A$}(0,.6);
\draw[thick,\AsColor] (.6,1.5) circle (.1cm);
\draw[thick,\AsColor] (.6,-1.5) circle (.1cm);
\node[\AsColor] at (1,1.7) {$\scriptstyle -1/2$};
\node[\AsColor] at (1,-1.3) {$\scriptstyle -1/2$};
\draw[thick,\AsColor] (.6,1.6) -- (.6,2);
\draw[thick,\AsColor] (.6,-1.6) -- (.6,-2);
\roundNbox{fill=white}{(0,0)}{.3}{0}{0}{$f$}
}
\right)
=
\Tr^{\cC_A}_{A}\left(
\tikzmath{
\draw[thick, black] (0,.3) -- (0,.9) arc(0:180:.3cm) --node[left]{$\scriptstyle M^\vee$} (-.6,-.9) arc (-180:0:.3cm) -- (0,-.3);
\draw[thick, \AsColor] (.6,-1.4)  to[out=90,in=-45] node[right]{$\scriptstyle A$}(0,-.6);
\draw[thick, \AsColor] (.6,1.4)  to[out=270,in=45] node[right]{$\scriptstyle A$}(0,.6);
\draw[thick,\AsColor] (.6,1.5) circle (.1cm);
\draw[thick,\AsColor] (.6,-1.5) circle (.1cm);
\node[\AsColor] at (1,1.7) {$\scriptstyle -1/2$};
\node[\AsColor] at (1,-1.3) {$\scriptstyle -1/2$};
\draw[thick,\AsColor] (.6,1.6) -- (.6,2);
\draw[thick,\AsColor] (.6,-1.6) -- (.6,-2);
\roundNbox{fill=white}{(0,0)}{.3}{0}{0}{$f$}
}
\right).
$$
\end{proof}

\begin{prop}
\label{prop:InternalEndRecognition}
The canonical algebra isomorphism $A \cong [A,A]_\cC$ is unitary if the module trace on $\cC_A$ is given on $\End(A_A)$ by
\begin{equation}
\label{eq:TraceOnC_A}
\Tr^{\cC_A}_A(f) :=
\psi\left(
\tikzmath{
\draw[thick,\AsColor] (0,-.2) coordinate (a1) -- +(0,1.2) coordinate (a2);
\roundNbox{fill=white}{(0,.4)}{0.3}{0}{0}{$f$};
\filldraw[\AsColor] (a1) circle (.05);
\filldraw[\AsColor] (a2) circle (.05);
}
\right)
\qquad\qquad f \in \End_{\cC_A}(A). 
\end{equation}
In particular, for all $f\in \cC_A(M_A\to M_A)$,
$
\Tr^{\cC_A}_M(f)
=
\psi\left(
\tikzmath{
\draw[thick, black] (0,.3) -- (0,.6) arc(0:180:.3cm) --node[left]{$\scriptstyle M^\vee$} (-.6,-.3) arc (-180:0:.3cm);
\node[\AsColor] at (.85,.1) {$\scriptstyle -1$};
\draw[thick, \AsColor] (.6,-.1) circle (.1cm);
\draw[thick, \AsColor] (.6,0)  to[out=90,in=-45] node[right]{$\scriptstyle A$}(0,.6);
\roundNbox{fill=white}{(0,0)}{.3}{0}{0}{$f$}
}
\right)
$.
\end{prop}
\begin{proof}
To verify traciality, notice that
\[
\psi\left(
\tikzmath{
\draw[thick,\AsColor] (0,0) -- +(0,2.1);
\roundNbox{fill=white}{(0,.6)}{0.3}{0}{0}{$f$}
\roundNbox{fill=white}{(0,1.5)}{0.3}{0}{0}{$g$}
\filldraw[\AsColor] (0,0) circle (.05);
\filldraw[\AsColor] (0,2.1) circle (.05);
}
\right) =
\psi\left(
\tikzmath{
\draw[thick,\AsColor] (0,0) -- +(0,2.1) arc (-90:450:.2) -- +(0,.2) arc (-90:450:.1) -- +(0,.2) coordinate (a);
\roundNbox{fill=white}{(0,.6)}{0.3}{0}{0}{$f$}
\roundNbox{fill=white}{(0,1.5)}{0.3}{0}{0}{$g$}
\filldraw[\AsColor] (0,0) circle (.05);
\filldraw[\AsColor] (a) circle (.05);
\node[\AsColor] () at (.3,3) {$\scriptstyle -1$};
}
\right) =
\psi\left(
\tikzmath{
\draw[thick,\AsColor] (0,0.3) -- (0,2.1) arc (-90:450:.1) arc (180:90:.4) coordinate (c) -- +(0,.3) coordinate (a) -- +(0,0) arc (90:0:.4) -- +(0,-2) arc (0:-90:.4) coordinate (d) -- +(0,-.3) coordinate (b) -- +(0,0) arc (-90:-180:.4);
\roundNbox{fill=white}{(0,.6)}{0.3}{0}{0}{$f$}
\roundNbox{fill=white}{(0,1.5)}{0.3}{0}{0}{$g$}
\filldraw[\AsColor] (a) circle (.05);
\filldraw[\AsColor] (b) circle (.05);
\node[\AsColor] () at (.3,2.4) {$\scriptstyle -1$};
}
\right) =
\psi\left(
\tikzmath{
\draw[thick,\AsColor] (0,0.3) -- +(0,.9) arc (-90:450:.1) -- +(0,.9) arc (180:90:.4) coordinate (c) -- +(0,.3) coordinate (a) -- +(0,0) arc (90:0:.4) -- +(0,-2) arc (0:-90:.4) coordinate (d) -- +(0,-.3) coordinate (b) -- +(0,0) arc (-90:-180:.4);
\roundNbox{fill=white}{(0,.6)}{0.3}{0}{0}{$g$}
\roundNbox{fill=white}{(0,2)}{0.3}{0}{0}{$f$}
\filldraw[\AsColor] (a) circle (.05);
\filldraw[\AsColor] (b) circle (.05);
\node[\AsColor] () at (.3,1.5) {$\scriptstyle -1$};
}
\right)
=
\psi\left(
\tikzmath{
\draw[thick,\AsColor] (0,0.3) -- (0,2.1) arc (-90:450:.1) arc (180:90:.4) coordinate (c) -- +(0,.3) coordinate (a) -- +(0,0) arc (90:0:.4) -- +(0,-2) arc (0:-90:.4) coordinate (d) -- +(0,-.3) coordinate (b) -- +(0,0) arc (-90:-180:.4);
\roundNbox{fill=white}{(0,.6)}{0.3}{0}{0}{$g$}
\roundNbox{fill=white}{(0,1.5)}{0.3}{0}{0}{$f$}
\filldraw[\AsColor] (a) circle (.05);
\filldraw[\AsColor] (b) circle (.05);
\node[\AsColor] () at (.3,2.4) {$\scriptstyle -1$};
}
\right) 
=
\psi\left(
\tikzmath{
\draw[thick,\AsColor] (0,0) -- +(0,2.1);
\roundNbox{fill=white}{(0,.6)}{0.3}{0}{0}{$g$}
\roundNbox{fill=white}{(0,1.5)}{0.3}{0}{0}{$f$}
\filldraw[\AsColor] (0,0) circle (.05);
\filldraw[\AsColor] (0,2.1) circle (.05);
}
\right).
\]
Above, the second and fourth equalities use that the  bubble inverse is a bimodule map and $f,g$ are right $A$-module maps and 
the third equality follows from standardness of the algebra.

By \eqref{eq:GeneralizedElements}, we have canonical unitary isomorphisms
$$
A\cong \bigoplus_{s\in\Irr(\cC)}d_s^{-1}\cC(s\to A)\otimes s
\qquad\text{and}\qquad
[A,A]_\cC\cong \bigoplus_{s\in\Irr(\cC)}d_s^{-1}\cC_A(s\otimes A\to A)\otimes s,
$$
where the second isomorphism implicitly uses the unitary adjunction $\cC(s\to [A,A]_\cC)\cong \cC_A(s\otimes A\to A)$.
Under the above unitary isomorphisms, it is straightforward to verify that the canonical algebra isomorphism $A\to [A,A]_\cC$ is given on generalized elements by
$$
\tikzmath{
\roundNbox{fill=white}{(0,0)}{.3}{0}{0}{$f$}
\draw[thick, \AsColor] (0,.3) --node[right]{$\scriptstyle A$} (0,.7);
\draw[thick, \XColor] (0,-.3) --node[right]{$\scriptstyle s$} (0,-.7);
}
\,
\longmapsto
\tikzmath{
\roundNbox{fill=white}{(0,0)}{.3}{0}{0}{$f$}
\draw[thick, \AsColor] (0,.3) to[out=90,in=-135] node[above]{$\scriptstyle A$} (.6,.7);
\draw[thick, \AsColor] (.6,-.7) -- node[right]{$\scriptstyle A$} (.6,1);
\draw[thick, \XColor] (0,-.3) --node[right]{$\scriptstyle s$} (0,-.7);
}\,.
$$
Thus the canonical algebra map is unitary if and only if the above map is isometric, i.e.
$$
d_s^{-1}\cdot
\psi\left(
\tikzmath{
\roundNbox{fill=white}{(0,0)}{.3}{0}{0}{$f$}
\roundNbox{fill=white}{(0,1)}{.3}{0}{0}{$f^\dag$}
\draw[thick, \AsColor] (0,.3) --node[left]{$\scriptstyle A$} (0,.7);
\draw[thick, \XColor] (0,1.3) arc(180:0:.3cm) --node[right]{$\scriptstyle s^\vee$} (.6, -.3) arc (0:-180:.3cm);
}
\right)
\overset{?}{=}
d_s^{-1}\cdot
\Tr^{\cC_A}_A\left(
\tikzmath{
\roundNbox{fill=white}{(0,0)}{.3}{0}{0}{$f$}
\roundNbox{fill=white}{(0,-1)}{.3}{0}{0}{$f^\dag$}
\draw[thick, \AsColor] (0,.3) to[out=90,in=-135] node[above]{$\scriptstyle A$} (.6,.7);
\draw[thick, \AsColor] (.6,-2) -- node[right]{$\scriptstyle A$} (.6,1);
\draw[thick, \AsColor] (0,-1.3) to[out=-90,in=135] node[below]{$\scriptstyle A$} (.6,-1.7);
\draw[thick, \XColor] (0,-.3) --node[left]{$\scriptstyle s$} (0,-.7);
}
\right).
$$
This equality clearly holds under the formula \eqref{eq:TraceOnC_A} for $\Tr^{\cC_A}_A$.
\end{proof}

\begin{cor}
For $M,N\in\cC_A$ with the $\cC$-module trace determined by \eqref{eq:TraceOnC_A}, the unitary internal hom $[M,N]_\cC=\underline{\Hom}_\cC(M,N)$ is unitarily isomorphic to $N\boxtimes_A M^\vee$ via the unitary Yoneda isomorphism afforded by the unitary adjunction
$$
\cC_A(c\rhd M \to N) \cong \cC(c\to N\boxtimes_A M^\vee)
\qquad\text{given by}\qquad
\tikzmath{
\draw[thick] (.2,-.3) --node[right]{$\scriptstyle M$} (.2,-.7);
\draw[thick] (-.2,-.3) --node[left]{$\scriptstyle c$} (-.2,-.7);
\draw[thick] (0,.3) --node[right]{$\scriptstyle N$} (0,.7);
\roundNbox{}{(0,0)}{.3}{.1}{.1}{$f$}
}
\longmapsto
\tikzmath{
\draw[thick] (.2,-.3) -- (.2,-.8) arc(-180:0:.3cm) --node[right]{$\scriptstyle M^\vee$} (.8,.7);
\draw[thick] (-.2,-.3) --node[left]{$\scriptstyle c$} (-.2,-1.3);
\draw[thick] (0,.3) --node[left]{$\scriptstyle N$} (0,.7);
\draw[thick, \AsColor] (0,.5) -- (.8,.5);
\draw[thick, \AsColor] (.2,-.5) -- (.5,-.8);
\filldraw[thick, \AsColor, fill=white] (.5,-.8) node[above]{$\scriptstyle -\frac{1}{2}$} circle (.1cm);
\roundNbox{}{(0,0)}{.3}{.1}{.1}{$f$}
}\,.
$$
\end{cor}
\begin{proof}
It suffices to prove the above map is isometric:
\begin{align*}
\left\|
\tikzmath{
\draw[thick] (.2,-.3) --node[right]{$\scriptstyle M$} (.2,-.7);
\draw[thick] (-.2,-.3) --node[left]{$\scriptstyle c$} (-.2,-.7);
\draw[thick] (0,.3) --node[right]{$\scriptstyle N$} (0,.7);
\roundNbox{}{(0,0)}{.3}{.1}{.1}{$f$}
}
\right\|_{\cC_A}^2
&=
\Tr^{\cC_A}_N\left(
\tikzmath{
\draw[thick] (.2,-.3) --node[right]{$\scriptstyle M$} (.2,-.7);
\draw[thick] (-.2,-.3) --node[left]{$\scriptstyle c$} (-.2,-.7);
\draw[thick] (0,.3) --node[right]{$\scriptstyle N$} (0,.7);
\draw[thick] (0,-1.3) --node[right]{$\scriptstyle N$} (0,-1.7);
\roundNbox{}{(0,0)}{.3}{.1}{.1}{$f$}
\roundNbox{}{(0,-1)}{.3}{.1}{.1}{$f^\dag$}
}
\right)
=
\psi_\cC\left(
\tikzmath{
\draw[thick] (.2,-.3) --node[right]{$\scriptstyle M$} (.2,-.7);
\draw[thick] (-.2,-.3) --node[left]{$\scriptstyle c$} (-.2,-.7);
\draw[thick] (0,.3) --node[left]{$\scriptstyle N$} (0,.7) arc(0:180:.3cm) --node[left]{$\scriptstyle N^\vee$} (-.6,-1.3) arc(-180:0:.3cm);
\draw[thick, \AsColor] (0,.7) to[out=-45,in=90] (.6,.1);
\filldraw[thick, \AsColor, fill=white] (.6,0) node[above,xshift=.2cm]{$\scriptstyle -1$} circle (.1cm);
\roundNbox{}{(0,0)}{.3}{.1}{.1}{$f$}
\roundNbox{}{(0,-1)}{.3}{.1}{.1}{$f^\dag$}
}
\right)
=
\psi_\cC\left(
\tikzmath{
\draw[thick] (.2,-.3) -- (.2,-.8) arc(-180:0:.3cm) --node[right]{$\scriptstyle M^\vee$} (.8,.7) arc(0:180:.9 and .6) -- (-1,-3.3) arc(-180:0:.9 and .6) -- (.8,-1.8) arc(0:180:.3cm) -- (.2,-2.3);
\draw[thick] (-.2,-.3) --node[right]{$\scriptstyle c$} (-.2,-2.3);
\draw[thick] (0,.3) --node[left]{$\scriptstyle N$} (0,.7) arc(0:180:.3cm) -- (-.6,-3.3) arc(-180:0:.3cm) -- (0,-2.9);
\draw[thick, \AsColor] (0,.5) -- (.8,.5);
\draw[thick, \AsColor] (0,-3.1) -- (.8,-3.1);
\draw[thick, \AsColor] (.2,-.5) -- (.5,-.8);
\filldraw[thick, \AsColor, fill=white] (.5,-.8) node[above]{$\scriptstyle -\frac{1}{2}$} circle (.1cm);
\draw[thick, \AsColor] (.2,-1.8) -- (.5,-2.1);
\filldraw[thick, \AsColor, fill=white] (.5,-2.1) node[above]{$\scriptstyle -\frac{1}{2}$} circle (.1cm);
\roundNbox{}{(0,0)}{.3}{.1}{.1}{$f$}
\roundNbox{}{(0,-2.6)}{.3}{.1}{.1}{$f^\dag$}
}
\right)
\\&=
\left\|
\tikzmath{
\draw[thick] (.2,-.3) -- (.2,-.8) arc(-180:0:.3cm) --node[right]{$\scriptstyle M^\vee$} (.8,.7);
\draw[thick] (-.2,-.3) --node[left]{$\scriptstyle c$} (-.2,-1.3);
\draw[thick] (0,.3) --node[left]{$\scriptstyle N$} (0,.7);
\draw[thick, \AsColor] (0,.5) -- (.8,.5);
\draw[thick, \AsColor] (.2,-.5) -- (.5,-.8);
\filldraw[thick, \AsColor, fill=white] (.5,-.8) node[above]{$\scriptstyle -\frac{1}{2}$} circle (.1cm);
\roundNbox{}{(0,0)}{.3}{.1}{.1}{$f$}
}
\right\|^2_{\cC}
\qedhere
\end{align*}
\end{proof}

\begin{rem}
For an algebra $(A,\mu,\iota)\in\cC$,
there is a canonical isomorphism

\begin{align*}
\set{f\in\cC(1\to A)}{\,\,
\tikzmath{
\draw[thick,\AsColor] (.6,-.7) -- (.6,1);
\draw[thick, \AsColor] (0,.3) to[out=90,in=-135] (.6,.7);
\roundNbox{fill=white}{(0,0)}{.3}{0}{0}{$f$}
}
\,\,
=
\,\,
\tikzmath{
\draw[thick,\AsColor] (-.6,-.7) -- (-.6,1);
\draw[thick, \AsColor] (0,.3) to[out=90,in=-45] (-.6,.7);
\roundNbox{fill=white}{(0,0)}{.3}{0}{0}{$f$}
}
}
\cong
\End({}_AA_A)
\quad\text{where}\quad
\tikzmath{
\draw[thick, \AsColor] (0,.3) -- (0,.7);
\roundNbox{fill=white}{(0,0)}{.3}{0}{0}{$f$}
}
\,\,\mapsto \,\,
\tikzmath{
\draw[thick, \AsColor] (.6,-.7) -- (.6,1);
\draw[thick, \AsColor] (0,.3) to[out=90,in=-135] (.6,.7);
\roundNbox{fill=white}{(0,0)}{.3}{0}{0}{$f$}
}\,;
\quad
\tikzmath{
\draw[thick, \AsColor] (0,.3) -- (0,.7);
\draw[thick, \AsColor] (0,-.3) -- (0,-.6);
\filldraw[\AsColor] (0,-.6) circle (.05cm);
\roundNbox{fill=white}{(0,0)}{.3}{0}{0}{$g$}
}
\,\,\mapsfrom \,\,
\tikzmath{
\draw[thick, \AsColor] (0,.3) -- (0,.7);
\draw[thick, \AsColor] (0,-.3) -- (0,-.7);
\roundNbox{fill=white}{(0,0)}{.3}{0}{0}{$g$}
}\,.
\end{align*}
\end{rem}

We now combine the above results into the main result of this subsection.

\begin{thm}
\label{thm:EveryModuleComesFromH*Alg}
A unitary $\cC$-module category equipped with a unitary trace
$(\cM,\Tr^\cM)$ is isometrically equivalent to $(\cC_A,\Tr^{\cC_A})$ where $A=[m,m]_\cC$ for any generator $m\in\cM$.
Hence every $(\cM,\Tr^\cM)$ comes from an $\rmH^*$-algebra.
\end{thm}
\begin{proof}
We first recall the construction of the unitary $\cC$-module equivalence $\cM\to \cC_A$.
We first define it on the full subcategory $\cM_0\subset \cM$ whose objects are of the form $c\rhd m$ for $c\in\cC$.
We then extend it to an equivalence $\cM\to\cC_A$ by the universal property of Cauchy completion.
The functor $\cM_0\to \cC_A$ is given on objects by $c\rhd m\mapsto c\otimes A$ and on morphisms by
$$
\tikzmath{
\draw[thick, black] (-.15,-.7) --node[left]{$\scriptstyle c$} (-.15,-.3);
\draw[thick, black] (-.15,.7) --node[left]{$\scriptstyle d$} (-.15,.3);
\draw[thick, blue] (.15,-.7) --node[right]{$\scriptstyle m$} (.15,-.3);
\draw[thick, blue] (.15,.7) --node[right]{$\scriptstyle m$} (.15,.3);
\roundNbox{fill=white}{(0,0)}{.3}{0}{0}{$f$}
}
\longmapsto
\tikzmath{
\draw[thick, black] (0,-.7) --node[left]{$\scriptstyle c$} (0,-.3);
\draw[thick, black] (-.5,.7) --node[left]{$\scriptstyle d$} (-.5,.3);
\draw[thick, \AsColor] (.5,.3) to[out=90,in=-135] node[above]{$\scriptstyle A$} (1,.7);
\draw[thick, \AsColor] (1,-.7) --node[right]{$\scriptstyle A$} (1,1);
\roundNbox{fill=white}{(0,0)}{.3}{.5}{.5}{$\mate(f)$}
}
$$
where the above mate is taken under the unitary adjunction
$$
\cM(c\rhd m\to d\rhd m)
\cong 
\cC(c\to d\otimes A).
$$
We also recall from Sub-Example \ref{ex:H*AlgOnEnd(m)} (which is explained in terms of right $\cC$-module categories instead of left $\cC$-module categories) that $A$ has unit given by $\coev_m: 1\to [m,m]_\cC$ defined as the mate of the identity $\id_m\in\cM(m,m)$ under the unitary adjunction
$$
\cM(m\to m)\cong \cM(1_\cC\rhd m\to m) \cong \cC(1_\cC\to [m,m]_\cC).
$$
We now compute using Proposition \ref{prop:InternalEndRecognition} that for all $f\in \cM(c\rhd m\to c\rhd m)$,
\begin{align*}
\Tr^\cM_{c\rhd m}(f)
&=
\langle \id_{c\rhd m} | f \rangle_{\cM}
=
\langle \id_c\otimes \coev_m | \mate(f) \rangle_{\cC}
\\&=
\psi\left(
\tikzmath{
\draw[thick, black] (-.5,.3) arc(0:180:.3cm) --node[left]{$\scriptstyle c^\vee$} (-1.1,-.3) arc(-180:0:.3);
\filldraw[thick, \AsColor] (.5,.3) -- node[right]{$\scriptstyle A$} (.5,.7) circle (.05cm);
\roundNbox{fill=white}{(0,0)}{.3}{.5}{.5}{$\mate(f)$}
}
\right)
\underset{\text{\eqref{eq:TraceOnC_A}}}{=}
\Tr^{\cC_A}_{c\otimes A}\left(\,
\tikzmath{
\draw[thick, black] (0,-.7) --node[left]{$\scriptstyle c$} (0,-.3);
\draw[thick, black] (-.5,.7) --node[left]{$\scriptstyle c$} (-.5,.3);
\draw[thick, \AsColor] (.5,.3) to[out=90,in=-135] node[above]{$\scriptstyle A$} (1,.7);
\draw[thick, \AsColor] (1,-.7) --node[right]{$\scriptstyle A$} (1,1);
\roundNbox{fill=white}{(0,0)}{.3}{.5}{.5}{$\mate(f)$}
}\,
\right)
.
\end{align*}
The result now follows as $m$ generates $\cM$ and $A$ generates $\cC_A$ as left $\cC$-modules.
\end{proof}

\subsection{Equivalence of unitary 2-categories \texorpdfstring{ $\HstarAlg(\cC)\cong \Mod^\dag(\cC)$}{H*Alg(C)==Mod(C)}}
\label{sec:H*Alg==ModC}

By Theorem \ref{thm:EveryModuleComesFromH*Alg}, we have a unitary equivalence of unitary 2-categories $\HstarAlg(\cC)\to \Mod^\dag(\cC)$ given by
\begin{itemize}
    \item $A\mapsto \cC_A$,
    \item ${}_AM_B \mapsto -\boxtimes_A M_B$, and
    \item $(f:{}_AM_B\Rightarrow {}_AN_B) \mapsto -\boxtimes_A f$.
\end{itemize}
We saw in Lemma \ref{lem:UAFonModC} above that $\Mod^\dag(\cC)$ is equipped with a canonical UAF given by unitary adjunction, and in \eqref{eq:1ParameterUAFOnH*Alg}  above, we found a 1-parameter family of UAFs on $\HstarAlg(\cC)$.
We now choose the correct UAF on $\HstarAlg(\cC)$ making this unitary equivalence preserve the UAFs, i.e., the unitary adjoint of $-\boxtimes_A M_B$ is $-\boxtimes_B M^\vee_A$.

\begin{prop}
The choice $\delta = 0$ from \eqref{eq:1ParameterUAFOnH*Alg} 
for the UAF on $\HstarAlg(\cC)$
make the functors $-\boxtimes_A M_B:\cC_A \to \cC_B$ and $-\boxtimes_B M_A^\vee : \cC_B \to \cC_A$ unitary adjoints.
Hence the unitary equivalence $\HstarAlg(\cC)\to \Mod^\dag(\cC)$ is compatible with UAFs and thus unitary pivotal.
\end{prop}
\begin{proof}
The isomorphism $\cC_B(N_A\boxtimes_A M_B\to P_B) \cong \cC_A(N_A\to  P_B\boxtimes_B M_A^\vee)$ is given by composing with the coevaluation. 
$$
\tikzmath{
\roundNbox{fill=white}{(0,0)}{.3}{.2}{.2}{$f$}
\draw[thick] (-.3,-.7) --node[left]{$\scriptstyle N$} (-.3,-.3);
\draw[thick] (.3,-.7) --node[right]{$\scriptstyle M$} (.3,-.3);
\draw[thick] (0,.7) --node[left]{$\scriptstyle P$} (0,.3);
\draw[thick,\AsColor] (-.3,-.5) -- (.3,-.5);
}
\longmapsto
\tikzmath{
\roundNbox{fill=white}{(0,0)}{.3}{.2}{.2}{$f$}
\draw[thick] (-.3,-1.8) --node[left]{$\scriptstyle N$} (-.3,-.3);
\draw[thick] (.3,-.3) -- (.3,-1.2) arc(-180:0:.4cm) --node[right]{$\scriptstyle M^\vee$} (1.1,.7);
\draw[thick] (0,.7) --node[left]{$\scriptstyle P$} (0,.3);
\draw[thick,\AsColor] (-.3,-.5) -- (.3,-.5);
\draw[thick,\BsColor] (0,.5) -- (1.1,.5);
\node[\BsColor] at (.8,-.3) {$\scriptstyle -\frac{1}{2}$};
\draw[thick,\BsColor] (.6,-.5) circle (.1cm);
\draw[thick, \BsColor] (.3,-.9) to[out=45,in=-90] (.6,-.6);
\draw[thick, \AsColor] (.3,-1) -- (-.3,-1.4);
\filldraw[thick,\AsColor,fill=white] (0,-1.2) circle (.1cm);
\node[\AsColor] at (0,-.85) {$\scriptstyle -\frac{1}{2}$};
}
=
\tikzmath{
\roundNbox{fill=white}{(0,0)}{.3}{.2}{.2}{$f$}
\draw[thick] (-.3,-1.8) --node[left]{$\scriptstyle N$} (-.3,-.3);
\draw[thick] (.3,-.3) -- (.3,-1.2) arc(-180:0:.4cm) --node[right]{$\scriptstyle M^\vee$} (1.1,.7);
\draw[thick] (0,.7) --node[left]{$\scriptstyle P$} (0,.3);
\draw[thick,\AsColor] (-.3,-.5) -- (.3,-.5);
\node[\BsColor] at (.8,-.3) {$\scriptstyle -\frac{1}{2}$};
\draw[thick,\BsColor] (.6,-.5) circle (.1cm);
\draw[thick, \BsColor] (.3,-.9) to[out=45,in=-90] (.6,-.6);
\draw[thick, \AsColor] (.3,-1) -- (-.3,-1.4);
\filldraw[thick,\AsColor,fill=white] (0,-1.2) circle (.1cm);
\node[\AsColor] at (0,-.85) {$\scriptstyle -\frac{1}{2}$};
}
$$
Since this map is invertible, it is unitary if and only if it is isometric.
Using Proposition \ref{prop:InternalEndRecognition}, we compute
\begin{align*}
\left\|
\tikzmath{
\roundNbox{fill=white}{(0,0)}{.3}{.2}{.2}{$f$}
\draw[thick] (-.3,-1.8) --node[left]{$\scriptstyle N$} (-.3,-.3);
\draw[thick] (.3,-.3) -- (.3,-1.2) arc(-180:0:.4cm) --node[right]{$\scriptstyle M^\vee$} (1.1,.7);
\draw[thick] (0,.7) --node[left]{$\scriptstyle P$} (0,.3);
\draw[thick,\AsColor] (-.3,-.5) -- (.3,-.5);
\node[\BsColor] at (.8,-.3) {$\scriptstyle -\frac{1}{2}$};
\draw[thick,\BsColor] (.6,-.5) circle (.1cm);
\draw[thick, \BsColor] (.3,-.9) to[out=45,in=-90] (.6,-.6);
\draw[thick, \AsColor] (.3,-1) -- (-.3,-1.4);
\filldraw[thick,\AsColor,fill=white] (0,-1.2) circle (.1cm);
\node[\AsColor] at (0,-.85) {$\scriptstyle -\frac{1}{2}$};
}
\right\|^2_{\cC_A}
&=
\Tr^{\cC_A}\left(
\tikzmath{
\roundNbox{fill=white}{(0,0)}{.3}{.2}{.2}{$f$}
\draw[thick] (-.3,-3.6) --node[left]{$\scriptstyle N$} (-.3,-.3);
\draw[thick] (.3,-.3) -- (.3,-1.2) arc(-180:0:.4cm) --node[right]{$\scriptstyle M^\vee$} (1.1,.7);
\draw[thick] (0,.7) --node[left]{$\scriptstyle P$} (0,.3);
\draw[thick,\AsColor] (-.3,-.5) -- (.3,-.5);
\node[\BsColor] at (.8,-.3) {$\scriptstyle -\frac{1}{2}$};
\draw[thick,\BsColor] (.6,-.5) circle (.1cm);
\draw[thick, \BsColor] (.3,-.9) to[out=45,in=-90] (.6,-.6);
\draw[thick, \AsColor] (.3,-1) -- (-.3,-1.4);
\filldraw[thick,\AsColor,fill=white] (0,-1.2) circle (.1cm);
\node[\AsColor] at (0,-.85) {$\scriptstyle -\frac{1}{2}$};
\roundNbox{fill=white}{(0,-3.6)}{.3}{.2}{.2}{$f^\dag$}
\draw[thick] (0,-3.9) --node[left]{$\scriptstyle P$} (0,-4.3);
\draw[thick] (.3,-3.3) -- (.3,-2.4) arc(180:0:.4cm) --node[right]{$\scriptstyle M^\vee$} (1.1,-4.3);
\draw[thick,\AsColor] (-.3,-3.1) -- (.3,-3.1);
\draw[thick, \AsColor] (.3,-2.6) -- (-.3,-2.2);
\filldraw[thick,\AsColor,fill=white] (0,-2.4) circle (.1cm);
\node[\AsColor] at (0,-2.75) {$\scriptstyle -\frac{1}{2}$};
\draw[thick,\BsColor] (.6,-3.1) circle (.1cm);
\draw[thick, \BsColor] (.3,-2.7) to[out=315,in=90] (.6,-3);
\node[\BsColor] at (.8,-3.3) {$\scriptstyle -\frac{1}{2}$};
}
\right)
=
\psi\left(
\tikzmath{
\roundNbox{fill=white}{(0,0)}{.3}{.2}{.2}{$f$}
\draw[thick] (-.3,-3.6) --node[left]{$\scriptstyle N$} (-.3,-.3);
\draw[thick] (.3,-.3) -- (.3,-1.2) arc(-180:0:.4cm) --node[right]{$\scriptstyle M^\vee$} (1.1,.3) arc (0:180:1.1 and .8) -- (-1.1,-3.9) arc (-180:0:1.1 and .8); 
\draw[thick] (0,.3) arc (0:180:.4) --  (-.8,-3.9) arc (-180:0:.4); 
\node at (.2,.5) {$\scriptstyle P$};
\draw[thick,\AsColor] (-.3,-.5) -- (.3,-.5);
\node[\BsColor] at (.8,-.3) {$\scriptstyle -\frac{1}{2}$};
\draw[thick,\BsColor] (.6,-.5) circle (.1cm);
\draw[thick, \BsColor] (.3,-.9) to[out=45,in=-90] (.6,-.6);
\draw[thick, \AsColor] (.3,-1) -- (-.3,-1.4);
\filldraw[thick,\AsColor,fill=white] (0,-1.2) circle (.1cm);
\node[\AsColor] at (0,-.85) {$\scriptstyle -\frac{1}{2}$};
\roundNbox{fill=white}{(0,-3.6)}{.3}{.2}{.2}{$f^\dag$}
\draw[thick] (.3,-3.3) -- (.3,-2.4) arc(180:0:.4cm) -- (1.1,-3.9);
\draw[thick,\AsColor] (-.3,-3.1) -- (.3,-3.1);
\draw[thick, \AsColor] (.3,-2.6) -- (-.3,-2.2);
\filldraw[thick,\AsColor,fill=white] (0,-2.4) circle (.1cm);
\node[\AsColor] at (0,-2.75) {$\scriptstyle -\frac{1}{2}$};
\draw[thick,\BsColor] (.6,-3.1) circle (.1cm);
\draw[thick, \BsColor] (.3,-2.7) to[out=315,in=90] (.6,-3);
\node[\BsColor] at (.8,-3.3) {$\scriptstyle -\frac{1}{2}$};
\draw[thick, \AsColor] (1.1,-2.7) to[out=315,in=90] (1.4,-3);
\filldraw[thick,\AsColor,fill=white] (1.4,-3.1) circle (.1cm);
\node[\AsColor] at (1.5,-3.3) {$\scriptstyle -1$};
}
\right)
\displaybreak[1]\\
&=\psi\left(
\tikzmath{
\roundNbox{fill=white}{(0,0)}{.3}{.2}{.2}{$f$}
\draw[thick] (-.3,-3.6) --node[left]{$\scriptstyle N$} (-.3,-.3);
\draw[thick] (.3,-.3) -- (.3,-1) arc(-180:0:.4cm) --node[right]{$\scriptstyle M^\vee$} (1.1,.3) arc (0:180:1.1 and .8) -- (-1.1,-3.9) arc (-180:0:1.1 and .8); 
\draw[thick] (0,.3) arc (0:180:.4) --  (-.8,-3.9) arc (-180:0:.4); 
\node at (.2,.5) {$\scriptstyle P$};
\draw[thick,\AsColor] (-.3,-.5) -- (.3,-.5);
\draw[thick, \AsColor] (.3,-1) to[out=225,in=90] (.1,-1.5) -- (.1,-2.1) to[out=-90,in=135] (.3,-2.6);
\draw[thick,\AsColor] (.1,-1.8) -- (-.3,-1.8);
\roundNbox{fill=white}{(0,-3.6)}{.3}{.2}{.2}{$f^\dag$}
\draw[thick] (.3,-3.3) -- (.3,-2.6) arc(180:0:.4cm) -- (1.1,-3.9);
\draw[thick,\AsColor] (-.3,-3.1) -- (.3,-3.1);
\draw[thick,\BsColor] (.6,-3.1) circle (.1cm);
\draw[thick, \BsColor] (.3,-2.7) to[out=315,in=90] (.6,-3);
\node[\BsColor] at (.8,-3.3) {$\scriptstyle -1$};
\draw[thick, \AsColor] (1.1,-2.7) to[out=315,in=90] (1.4,-3);
\filldraw[thick,\AsColor,fill=white] (1.4,-3.1) circle (.1cm);
\node[\AsColor] at (1.5,-3.3) {$\scriptstyle -1$};
}
\right)
=\psi\left(
\tikzmath{
\roundNbox{fill=white}{(0,0)}{.3}{.2}{.2}{$f$}
\draw[thick] (-.3,-3.6) --node[left]{$\scriptstyle N$} (-.3,-.3);
\draw[thick] (.3,-.3) -- (.3,-1) arc(-180:0:.4cm) -- (1.1,.3) arc (180:0:.4) -- node[right]{$\scriptstyle M$} (1.9,-3.9) arc (0:-180:.4); 
\draw[thick] (0,.3) arc (0:180:.4) --  (-.8,-3.9) arc (-180:0:.4); 
\node at (.2,.5) {$\scriptstyle P$};
\draw[thick, \AsColor] (.3,-1) to[out=225,in=90] (.1,-1.5) -- (.1,-2.1) to[out=-90,in=135] (.3,-2.6);
\draw[thick,\AsColor] (.1,-1.8) -- (-.3,-1.8);
\roundNbox{fill=white}{(0,-3.6)}{.3}{.2}{.2}{$f^\dag$}
\draw[thick] (.3,-3.3) -- (.3,-2.6) arc(180:0:.4cm) -- (1.1,-3.9);
\draw[thick,\BsColor] (.6,-3.1) circle (.1cm);
\draw[thick, \BsColor] (.3,-2.7) to[out=315,in=90] (.6,-3);
\node[\BsColor] at (.8,-3.3) {$\scriptstyle -1$};
\draw[thick, \AsColor] (.1,-1.5) to[out=315,in=90] (.4,-1.8);
\filldraw[thick,\AsColor,fill=white] (.4,-1.9) circle (.1cm);
\node[\AsColor] at (.7,-2) {$\scriptstyle -1$};
}
\right) \\
& =\psi\left(
\tikzmath{
\roundNbox{fill=white}{(0,0)}{.3}{.2}{.2}{$f$}
\draw[thick] (-.3,-2.6) --node[left]{$\scriptstyle N$} (-.3,-.3);
\draw[thick] (.3,-.3) -- node[right]{$\scriptstyle M$} (.3,-2.3); 
\draw[thick] (0,.3) arc (0:180:.4) --  (-.8,-2.9) arc (-180:0:.4); 
\node at (.2,.5) {$\scriptstyle P$};
\draw[thick,\AsColor] (.3,-1.3) -- (-.3,-1.3);
\roundNbox{fill=white}{(0,-2.6)}{.3}{.2}{.2}{$f^\dag$}
\draw[thick,\BsColor] (.6,-2.1) circle (.1cm);
\draw[thick, \BsColor] (.3,-1.7) to[out=315,in=90] (.6,-2);
\node[\BsColor] at (.8,-2.3) {$\scriptstyle -1$};
}
\right) 
=
\Tr^{\cC_B}\left(
\tikzmath{
\roundNbox{fill=white}{(0,0)}{.3}{.2}{.2}{$f$}
\draw[thick] (-.3,-2.6) --node[left]{$\scriptstyle N$} (-.3,-.3);
\draw[thick] (.3,-.3) -- node[right]{$\scriptstyle M$} (.3,-2.3); 
\draw[thick] (0,.3) --node[left]{$\scriptstyle P$} (0,.7); 
\draw[thick] (0,-2.9) --node[left]{$\scriptstyle P$} (0,-3.3);
\draw[thick,\AsColor] (.3,-.5) -- (-.3,-.5);
\draw[thick,\AsColor] (.3,-2.1) -- (-.3,-2.1);
\roundNbox{fill=white}{(0,-2.6)}{.3}{.2}{.2}{$f^\dag$}
}
\right)
= \left\|
\tikzmath{
\roundNbox{fill=white}{(0,0)}{.3}{.2}{.2}{$f$}
\draw[thick] (-.3,-.7) --node[left]{$\scriptstyle N$} (-.3,-.3);
\draw[thick] (.3,-.7) --node[right]{$\scriptstyle M$} (.3,-.3);
\draw[thick] (0,.7) --node[left]{$\scriptstyle P$} (0,.3);
\draw[thick,\AsColor] (-.3,-.5) -- (.3,-.5);
}
\right\|_{\cC_B}^2
\end{align*}
The third equality uses the Frobenius property and the definition of the split idempotent.
\end{proof}

\section{3-Hilbert spaces}

Now that we have completed our analysis of $\rmH^*$-multifusion categories and their modules, we turn our attention to the notion of a 3-Hilbert space.
We begin with pre-3-Hilbert space, with the basic example begin an $\rmH^*$-multifusion category.
We then discuss Hilbert direct sums and $\rmH^*$-monads, together with their corresponding completion operations.
Finally, 3-Hilbert spaces are introduced as `completed' pre-3-Hilbert spaces, and we show they are all of the form $\Mod^\dag(\cC)\cong \HstarAlg(\cC)$ for an $\rmH^*$-multifusion category.

\subsection{pre-3-Hilbert spaces}
\label{sec:Pre3Hilb}

\begin{defn}
A \emph{unitary 2-category} is a 
$\dag$-2-category such that every $n$-fold linking algebra from Definition \ref{def:linkingE1algebra} with its obvious dagger structure is a unitary multitensor category. 
Thus unitary 2-categories are pre-semisimple.
A unitary 2-category is called \emph{finite} if it is finite as a pre-semisimple 2-category.
\end{defn}

\begin{ex}
For a unitary multitensor category $\cC$,
the 2-category of semisimple unitary $\cC$-module categories, dualizable unitary $\cC$-module functors, and uniformly bounded $\cC$-module natural transformations is a unitary 2-category.
When $\cC$ is unitary multifusion, the 2-category $\Mod^\dag(\cC)$ of finite semsimple unitary module categories is a finite unitary 2-category.
\end{ex}

\begin{defn}
A \emph{pre-3-Hilbert space} $(\fX, \vee,\Psi)$ is a finite unitary 2-category $\fX$ equipped with a UAF $\vee$ and a spherical tracial weight $\Psi$ on $\End(1_a)$ for all objects $a\in\fX$.
\end{defn}

Notice that given a pre-3-Hilbert space $(\fX,\vee,\Psi)$, we may uniformly rescale the spherical weight $\Psi$ on each component of $\fX$ independently to obtain a new pre-3-Hilbert space.

\begin{ex}\label{ex:BCIsAPre3Hilb}
The basic example of a pre-3-Hilbert space is $\rmB \cC$ where $(\cC,\vee,\psi)$ is an $\rmH^*$-multifusion category.
\end{ex}

\begin{ex}
\label{ex:LinkingH*Alg}
Given any pre-3-Hilbert space $(\fX,\vee,\Psi)$,
for every $a_1,\dots, a_n\in\fX$, 
the linking algebra $\cL(a_1,\dots, a_n)$ has an 
organic structure of an $\rmH^*$-multifusion category with UDF induced by the UAF of $\fX$
and
spherical weight $\psi$ induced by the spherical weight $\Psi^\fX$.
\end{ex}

\begin{remark}\label{rem:pre3Hilbislocally2Hilb}
A pre-3-Hilbert space structure on $\fX$ induces a 2-Hilbert space structure on each hom-category $\fX(a \to b)$. 
Indeed, for $X \in \fX(a \to b)$, we define $\Tr_X \colon \End(X) \to \bbC$ by  
$$
\Tr_X(f:X \Rightarrow X)
:=
\Psi_b\left(\,
\tikzmath{
\fill[rounded corners = 5pt, fill=\brColor] (-1.2,-.9) rectangle (.6,.9);
\fill[\arColor] (0,.3) arc (0:180:.3cm) -- (-.6,-.3) arc (-180:0:.3cm);
\draw[thick] (0,.3) node[right,yshift=.2cm]{$\scriptstyle X$} arc (0:180:.3cm) --node[left]{$\scriptstyle X^\vee$} (-.6,-.3) arc (-180:0:.3cm) node[right,yshift=-.2cm]{$\scriptstyle X$};
\roundNbox{fill=white}{(0,0)}{.3}{0}{0}{$f$}
}
\,\right)
=
\Psi_a\left(\,
\tikzmath{
\fill[rounded corners = 5pt, fill=\arColor] (1.2,-.9) rectangle (-.6,.9);
\fill[\brColor]  (0,.3) arc (180:0:.3cm) -- (.6,-.3) arc (0:-180:.3cm);
\draw[thick] (0,.3) node[left,yshift=.2cm]{$\scriptstyle X$} arc (180:0:.3cm) --node[right]{$\scriptstyle X^\vee$} (.6,-.3) arc (0:-180:.3cm) node[left,yshift=-.2cm]{$\scriptstyle X$};
\roundNbox{fill=white}{(0,0)}{.3}{0}{0}{$f$}
}
\right).
$$
Moreover, for a 1-morphism $X$ in $\fX$, pre/post-composing with $X^\vee$ is the unitary adjoint to pre/post-composing with $X$.

\begin{equation}
\begin{tikzcd}[column sep=5em]
\fX(a\to c)
\arrow[r,shift left=1,"{}_bX^\vee_a\otimes-"]
&
\fX(b\to c)
\arrow[l,shift left=1,"{}_aX_b\otimes-"]
\end{tikzcd}
\end{equation}
\end{remark}

\begin{ex}
\label{ex:WeightOnModC}
Let $(\cC,\vee,\psi)$ be an $\rmH^*$-multifusion category, and recall from Lemma \ref{lem:UAFonModC} that unitary adjunction is a UAF on $\Mod^\dag(\cC)$.
We claim that defining
$$
\Psi_\cM^{\Mod^\dag(\cC)} (\eta) \coloneqq \sum_{m \in \Irr(\cM)} d_m \Tr^\cM_{m}(\eta_m)
$$
for $\cC$-module natural transformations $\eta: \id_\cM\Rightarrow\id_\cM$ defines a spherical weight on $\Mod^\dag(\cC)$ endowing it with the structure of a pre-3-Hilbert space.

Indeed, for a $\cC$-module functor $F\colon \cM \to \cN$ and each $m\in\Irr(\cM)$ and $n\in\Irr(\cN)$, we fix an ONB 

$$
\left\{
\,
\tikzmath{
\begin{scope}
\clip[rounded corners=5pt] (.7,-.7) rectangle (-.7,.7);
\fill[fill=gray!50] (-.2,-.7) rectangle (.2,-.3);
\fill[fill=gray!20] (0,.7) -- (0,0) -- (-.2,-.3) -- (-.2,-.7) -- (-.7,-.7) -- (-.7,.7);
\end{scope}
\draw[thick] (-.2,-.7) node[below]{$\scriptstyle F$} -- (-.2,-.3);
\draw[thick] (.2,-.7) node[below]{$\scriptstyle m$} -- (.2,-.3);
\draw[thick] (0,.3) -- (0,.7) node[above]{$\scriptstyle n$};
\roundNbox{fill=white}{(0,0)}{.3}{.1}{.1}{$\phi$}
}
\,
 \right\}
 \subset
 \cN(F(m) \to n),
\qquad\quad
\tikzmath{
\fill[rounded corners=5pt, fill=gray!50] (0,0) rectangle (.6,.6);
}
=
\cM
\qquad
\tikzmath{
\fill[rounded corners=5pt, fill=gray!20] (0,0) rectangle (.6,.6);
}
=
\cN
\qquad
\tikzmath{
\draw[rounded corners=5pt,dotted] (0,0) rectangle (.6,.6);
}
=
\cC
$$
so that we have the relations
\begin{equation}
\tikzmath{
\begin{scope}
\clip[rounded corners=5pt] (.7,-1.7) rectangle (-.7,.7);
\fill[fill=gray!50] (-.2,-.7) rectangle (.2,-.3);
\fill[fill=gray!20] (0,.7) -- (0,0) -- (-.2,-.3) -- (-.2,-.7) -- (0,-1.3) -- (0,-1.7) -- (-.7,-1.7) -- (-.7,.7);
\end{scope}
\draw[thick] (-.2,-.7) -- node[left]{$\scriptstyle F$} (-.2,-.3);
\draw[thick] (.2,-.7) -- node[right]{$\scriptstyle m$} (.2,-.3);
\draw[thick] (0,.3) -- (0,.7) node[above]{$\scriptstyle n$};
\draw[thick] (0,-1.3) -- (0,-1.7) node[below]{$\scriptstyle n$};
\roundNbox{fill=white}{(0,0)}{.3}{.1}{.1}{$\phi'$}
\roundNbox{fill=white}{(0,-1)}{.3}{.1}{.1}{$\phi^\dag$}
}
= 
\delta_{\phi=\phi'}d_n^{-1} \,
\tikzmath{
\begin{scope}
\clip[rounded corners=5pt] (.7,.7) rectangle (-.7,-1.7);
\fill[fill=gray!20] (0,.7) rectangle (-.7,-1.7);
\end{scope}
\draw[thick] (0,-1.7) node[below]{$\scriptstyle n$} -- (0,.7) node[above]{$\scriptstyle n$};
}
\qquad\text{and}\qquad
\tikzmath{
\begin{scope}
\clip[rounded corners=5pt] (.7,.7) rectangle (-.7,-1.7);
\fill[fill=gray!20] (-.2,.7) rectangle (-.7,-1.7);
\fill[fill=gray!50] (-.2,.7) rectangle (.2,-1.7);
\end{scope}
\draw[thick] (-.2,-1.7) node[below]{$\scriptstyle F$} -- (-.2,.7) node[above]{$\scriptstyle F$};
\draw[thick] (.2,-1.7) node[below]{$\scriptstyle m$} -- (.2,.7) node[above]{$\scriptstyle m$};
}
=
\sum_{n\in \Irr(\cN)}
\sum_\phi
d_n \;
\tikzmath{
\begin{scope}
\clip[rounded corners=5pt] (.7,.7) rectangle (-.7,-1.7);
\fill[fill=gray!20] (-.2,-1.7) -- (-.2,-1.2) arc (180:90:.2cm) -- (0,0) arc (-90:-180:.2cm) -- (-.2,.7) -- (-.7,.7) -- (-.7,-1.7);
\end{scope}
\fill[fill=gray!50] (-.2,-1.7) rectangle (.2,-1);
\fill[fill=gray!50] (-.2,.7) rectangle (.2,0);
\draw[thick] (.2,-1.7)  node[below]{$\scriptstyle m$} -- (.2,-1);
\draw[thick] (.2,.7)  node[above]{$\scriptstyle m$} -- (.2,0);
\draw[thick] (-.2,-1.7)  node[below]{$\scriptstyle F$} -- (-.2,-1);
\draw[thick] (-.2,.7)  node[above]{$\scriptstyle F$} -- (-.2,0);
\draw[thick] (0,0) --node[right]{$\scriptstyle n$} (0,-1);
\roundNbox{fill=white}{(0,0)}{.3}{.1}{.1}{$\phi^\dag$}
\roundNbox{fill=white}{(0,-1)}{.3}{.1}{.1}{$\phi$}
}
\;.
\label{eq:Fm_splicing}
\end{equation}
Since $F\dashv^\dag F^*$, we also have the relations
\begin{equation}
\tikzmath{
\begin{scope}
\clip[rounded corners=5pt] (.5,-1.2) rectangle (-1.2,1.2);
\fill[gray!50] (-1.2,-1.2) -- (.2,-1.2) -- (.2,-.5) -- (0,-.5) -- (0,.5) -- (.2,.5) -- (.2,1.2) -- (-1.2,1.2);
\fill[gray!20] (-.2,.8) arc (0:180:.2) -- (-.6,-.8) arc (-180:0:.2) -- (0,-.5) -- (0,.5);
\end{scope}
\draw[thick] (-.2,.8) arc (0:180:.2) -- node[left] {$\scriptstyle F^*$} (-.6,-.8) arc (-180:0:.2);
\draw[thick] (.2,.5) -- (.2,1.2) node[above] {$\scriptstyle m$};
\draw[thick] (.2,-.5) -- (.2,-1.2) node[below] {$\scriptstyle m$};
\draw[thick] (0,-.5) -- node[right] {$\scriptstyle n$} (0,.5);
\roundNbox{fill=white}{(0,-.5)}{.3}{.1}{.1}{$\phi'$};
\roundNbox{fill=white}{(0,.5)}{.3}{.1}{.1}{$\phi^\dag$};
}
=
\delta_{\phi=\phi'}d_m^{-1} \,
\tikzmath{
\begin{scope}
\clip[rounded corners=5pt] (.7,.7) rectangle (-.7,-1.7);
\fill[fill=gray!50] (0,.7) rectangle (-.7,-1.7);
\end{scope}
\draw[thick] (0,-1.7) node[below]{$\scriptstyle m$} -- (0,.7) node[above]{$\scriptstyle m$};
}
\qquad\text{and}\qquad
\tikzmath{
\begin{scope}
\clip[rounded corners=5pt] (.7,.7) rectangle (-.7,-1.7);
\fill[fill=gray!50] (-.2,.7) rectangle (-.7,-1.7);
\fill[fill=gray!20] (-.2,.7) rectangle (.2,-1.7);
\end{scope}
\draw[thick] (-.2,-1.7) node[below]{$\scriptstyle F^*$} -- (-.2,.7) node[above]{$\scriptstyle F^*$};
\draw[thick] (.2,-1.7) node[below]{$\scriptstyle n$} -- (.2,.7) node[above]{$\scriptstyle n$};
}
=
\sum_{m \in \Irr(\cN)} \sum_i d_m \;
\tikzmath{
\begin{scope}
\clip[rounded corners=5pt] (.5,-1.2) rectangle (-1,1.2);   
\fill[gray!50] (-1,-1.2) -- (0,-1.2) -- (0,-.6) -- (.2,-.3) -- (.2,.3) -- (0,.6) -- (0,1.2) -- (-1,1.2);
\fill[gray!20] (-.2,.3) arc (0:-180:.2) -- (-.6,1.2) -- (0,1.2) -- (0,.6);
\fill[gray!20] (-.2,-.3) arc (0:180:.2) -- (-.6,-1.2) -- (0,-1.2) -- (0,-.6);
\end{scope}
\draw[thick] (-.2,.3) arc (0:-180:.2) -- (-.6,1.2) node[above] {$\scriptstyle F^*$};
\draw[thick] (-.2,-.3) arc (0:180:.2) -- (-.6,-1.2) node[below] {$\scriptstyle F^*$};
\draw[thick] (.2,-.3) -- node[right] {$\scriptstyle m$} (.2,.3);
\draw[thick] (0,.6) -- (0,1.2) node[above] {$\scriptstyle n$};
\draw[thick] (0,-.6) -- (0,-1.2) node[below] {$\scriptstyle n$};
\roundNbox{fill=white}{(0,.6)}{.3}{.1}{.1}{$\phi$};
\roundNbox{fill=white}{(0,-.6)}{.3}{.1}{.1}{$\phi^\dag$};
}\,,
\label{eq:Fstarn_splicing}
\end{equation}
so that the one-click rotation of $\{\phi\}$ is an ONB for $\cM(m\to F^*(n))$.
We then calculate that for $\eta \in \Mod^\dag(\cC)(F \Rightarrow F)$,
\begin{align*}
\Psi^{\Mod^\dag(\cC)}_\cN\left(\,
\tikzmath{
\fill[rounded corners=5pt, fill=gray!20] (-.5,-.8) rectangle (1,.8);
\filldraw[thick,fill=gray!50] (0,.3) arc (180:0:.3cm) -- (.6,-.3) arc (0:-180:.3cm);
\roundNbox{fill=white}{(0,0)}{.3}{0}{0}{$\eta$}
\node at (-.15,.5) {$\scriptstyle F$};
}
\,\right) 
&=
\sum_{n \in \Irr(\cN)} d_n \Tr^\cN_n \left( 
\,
\tikzmath{
\begin{scope}
\clip[rounded corners=5pt] (.9,-.8) rectangle (-.9,.8); \fill[fill=gray!20] (.7,-1.2) rectangle (-.9,1.2);
\fill[fill=gray!50] (0.3,.3) arc (0:180:.3cm) -- (-.3,-.3) arc (-180:0:.3cm);
\end{scope}
\draw[thick] (.7,-.8) node[below]{$\scriptstyle n$} -- (.7,.8) node[above]{$\scriptstyle n$};
\draw[thick] (-.3,.3) arc (180:0:.3cm) -- (.3,-.3) arc (0:-180:.3cm);
\roundNbox{fill=white}{(-.3,0)}{.3}{0}{0}{$\eta$};
}\,\right) 
\displaybreak[1]
\\&
\underset{\eqref{eq:Fstarn_splicing}}{=} 
\sum_{\substack{
m \in \Irr(\cM)
\\
n \in \Irr(\cN)
}} 
\sum_{\phi\in \ONB(m\to F^*(n))}
d_n d_m \Tr^\cN_n \left( 
\tikzmath{
\draw[thick] (0,-.8) node[below]{$\scriptstyle n$} -- (0,.8) node[above]{$\scriptstyle n$};
\begin{scope}
\clip[rounded corners=5pt] (-.8,-1.7) rectangle (.5,1.7); 
\fill[gray!20] (-.8,-1.7) rectangle (0,1.7);
\fill[gray!50] (-.3,-1) rectangle (.3,1);
\end{scope}
\draw[thick] (-.3,-1) -- (-.3,1);
\draw[thick] (.3,-1) -- node[right] {$\scriptstyle m$} (.3,1);
\draw[thick] (0,1) -- (0,1.7) node[above] {$\scriptstyle n$};
\draw[thick] (0,-1) -- (0,-1.7) node[below] {$\scriptstyle n$};
\roundNbox{fill=white}{(-.3,0)}{.3}{0}{0}{$\eta$};
\roundNbox{fill=white}{(0,1)}{.3}{.2}{.2}{$\phi$}
\roundNbox{fill=white}{(0,-1)}{.3}{.2}{.2}{$\phi^\dag$}
}\right) 
\displaybreak[1]\\&= 
\sum_{\substack{
m \in \Irr(\cM)
\\
n \in \Irr(\cN)
}} 
\sum_{\phi\in \ONB(m\to F^*(n))}
d_n d_m \Tr^\cN_{F(m)} \left( 
\tikzmath{
\begin{scope}
\clip[rounded corners=5pt] (-.9,-.7) rectangle (.5,2.7); 
\fill[gray!20]  (-.9,-.7) rectangle (0,2.7); 
\fill[gray!50]  (-.3,-.7) rectangle (.3,1); 
\fill[gray!50]  (-.3,2) rectangle (.3,2.7); 
\end{scope}
\draw[thick] (-.3,-.7) -- (-.3,1);
\draw[thick] (-.3,2) -- (-.3,2.7);
\draw[thick] (.3,-.7) node[below] {$\scriptstyle m$} -- (.3,1);
\draw[thick] (.3,2) -- (.3,2.7) node[above] {$\scriptstyle m$};
\draw[thick] (0,1) -- node[right] {$\scriptstyle n$} (0,2);
\roundNbox{fill=white}{(-.3,0)}{.3}{0}{0}{$\eta$};
\roundNbox{fill=white}{(0,1)}{.3}{.2}{.2}{$\phi$}
\roundNbox{fill=white}{(0,2)}{.3}{.2}{.2}{$\phi^\dag$}
}
\right) 
\displaybreak[1]\\&
\underset{\eqref{eq:Fm_splicing}}{=} \sum_{m \in \Irr(\cM)} d_m \Tr^\cN_{F(m)} \left( 
\,
\tikzmath{
\begin{scope}
\clip (.5,-1.2) rectangle (-.9,1.2); 
\fill[rounded corners=5pt, fill=gray!50] (.9,-.8) rectangle (-.8,.8);
\clip[rounded corners=5pt] (.9,-.8) rectangle (-.8,.8);
\filldraw[fill=gray!20] (-.9,-.9) rectangle (0,.9);
\end{scope}
\draw[thick] (0,-.8) node[below]{$\scriptstyle F$} -- (0,.8) node[above]{$\scriptstyle F$};
\draw[thick] (.5,-.8) node[below]{$\scriptstyle m$} -- (.5,.8) node[above]{$\scriptstyle m$};
\roundNbox{fill=white}{(0,0)}{.3}{0}{0}{$\eta$};
}
\;\;
\right)
\displaybreak[1]\\&= 
\sum_{m \in \Irr(\cM)} d_m
\Tr_m^{\cM}\left(\,
\tikzmath{
\begin{scope}
\clip[rounded corners=5pt] (.7,-.8) rectangle (-1,.8); 
\fill[gray!50] (.5,-.8) rectangle (-1,.8);
\filldraw[thick,fill=gray!20] (0,.3) arc (0:180:.3cm) -- (-.6,-.3) arc (-180:0:.3cm);
\end{scope}
\draw[thick] (.5,-.8) node[below]{$\scriptstyle m$} -- (.5,.8) node[above]{$\scriptstyle m$};
\roundNbox{fill=white}{(0,0)}{.3}{0}{0}{$\eta$};
}
\;\;\right)
\displaybreak[1]\\&= 
\Psi_\cM^{\Mod^\dag(\cC)}\left(\,
\tikzmath{
\fill[rounded corners=5pt, fill=gray!50] (.5,-.8) rectangle (-.8,.8);
\filldraw[thick,fill=gray!20] (0,.3) arc (0:180:.3cm) -- (-.6,-.3) arc (-180:0:.3cm);
\roundNbox{fill=white}{(0,0)}{.3}{0}{0}{$\eta$}
}
\,\right).
\end{align*}
Note, however, that under this choice of spherical weight, the induced 2-Hilbert space structure on $\Fun^\dag_\cC({}_\cC\cM\to {}_\cC\cN)$ does not agree with the 2-Hilbert space structure from \eqref{eq:RenormalizedTraceOnFun}.
We thus \emph{renormalize} $\Psi^{\Mod^\dag(\cC)}$ on each component so that these 2-Hilbert space structures agree. 
In explicit detail, suppose 
$\cC=\boxplus \cC_j$ is a decomposition of $\cC$ into indecomposable summands.
Then $\Mod^\dag(\cC)=\boxplus \Mod^\dag(\cC_j)$ is a decomposition into components.
We rescale $\Psi^{\Mod^\dag(\cC)}$ on the $j$-th component by multiplying by the scalar 
$$
\frac{
\dim(\End_{\cC_j}(1_{\cC_j}))^2
}{
\operatorname{FPdim}(\cC_j)
\cdot
\psi_\cC(\id_{1_{\cC_j}})
}
$$
which appeared in \eqref{eq:RenormalizedTraceOnFun} above.
\end{ex}

\begin{sub-ex}
\label{ex:WeightOn2Hilb}
Choosing $\cC=\Hilb$ in the above example with its canonical UDF and trace, we see that $2\Hilb$ is a pre-3-Hilbert space with
$$
\Psi_\cA(\eta:\id_\cA\Rightarrow\id_\cA)
:=
\sum_{a\in\Irr(\cA)}d_a \Tr^\cA_a(\eta_a).
$$
Observe that no renormalization is necessary for the $\rmH^*$-multifusion category $\Hilb$ equipped with its canonical spherical state normalized so that $\psi(\id_\bbC)=1$.
\end{sub-ex}

\begin{prop}\label{prop:uniquenessofUAF}
Suppose $\fX$ is a finite unitary 2-category with a weight $\Psi$, and let $(-)^\vee$ and $(-)^*$ be
two UAFs on $\fX$ giving it the structure of a pre-3-Hilbert space. 
Then the canonical iconic natural isomorphism $\zeta:(-)^\vee\Rightarrow (-)^*$ 
given by
$$
\zeta_a = 1_a
\quad
\forall\,a\in\fX
\qquad\text{and}\qquad
\zeta_X:=
\tikzmath{
\begin{scope}
\clip[rounded corners=5pt] (-1.2,-.7) rectangle (2.2,1.7);
\fill[\brColor] (-1.2,-.7) rectangle (2.2,1.7);
\fill[\arColor] (-.5,-.7) -- (-.5,.7) -- (.5,.7) -- (.5,.3) -- (1.5,.3) -- (1.5,1.7) -- (2.2,1.7) -- (2.2,-.7);
\end{scope}
\draw[thick] (-.5,-.7) --node[left]{$\scriptstyle X^\vee$} (-.5,.7);
\draw[thick] (.5,.3) --node[right]{$\scriptstyle X$} (.5,.7);
\draw[thick] (1.5,.3) --node[right]{$\scriptstyle X^*$} (1.5,1.7);
\roundNbox{fill=white}{(0,1)}{.3}{.5}{.5}{$\ev_X^\vee$}
\roundNbox{fill=white}{(1,0)}{.3}{.5}{.5}{$\coev_X^*$}
}
$$
is unitary.
Thus a fixed weight $\Psi$ on $\fX$ admits at most one UAF under which it becomes a pre-3-Hilbert space.
\end{prop}

\begin{proof}
Without loss of generality, we may assume that the underlying 1-morphisms $X^\vee$ and $X^*$ are identical.
Indeed, we know they are isomorphic and hence unitarily isomorphic, and showing $\zeta_X$ is unitary is equivalent to proving its composite with another unitary is again unitary.
Even though we now identify $X^*=X^\vee$, we use the notation
\[
\coev_X^\vee
=
\tikzmath{
\fill[\arColor, rounded corners = 5pt]  (-.2,-.5) rectangle (.8,0);
\filldraw[fill = \brColor, thick]  (0,0) node[above]{$\scriptstyle X$} arc (-180:0:.3) node[above]{$\scriptstyle X^\vee$};
}
\qquad\qquad
\ev_X^\vee
=
\tikzmath{
\fill[\brColor, rounded corners = 5pt]  (-.2,-.5) rectangle (.8,0);
\filldraw[fill = \arColor, thick] (0,-.5) node[below]{$\scriptstyle X^\vee$} arc (180:0:.3) node[below]{$\scriptstyle X$};
}
\qquad\qquad
\coev_X^*
=
\tikzmath{
\fill[\arColor, rounded corners = 5pt]  (-.2,-.5) rectangle (.8,0);
\filldraw[fill = \brColor, thick]  (0,0) node[above]{$\scriptstyle X$} arc (-180:0:.3) node[above]{$\scriptstyle X^*$};
}
\qquad\qquad
\ev_X^*
=
\tikzmath{
\fill[\brColor, rounded corners = 5pt]  (-.2,-.5) rectangle (.8,0);
\filldraw[fill = \arColor, thick] (0,-.5) node[below]{$\scriptstyle X^*$} arc (180:0:.3) node[below]{$\scriptstyle X$};
}
\]
to distinguish their evaluations and coevaluations.

For all $f \in \fX(X^\vee \to X^\vee)$, we compute
\begin{align*}
\Psi_b\left(\,
\tikzmath{
\fill[rounded corners = 5pt, fill=\brColor] (-1.2,-1.5) rectangle (.6,1.5);
\fill[\arColor] (0,1.1) arc (0:180:.3cm) -- (-.6,-1.1) arc (-180:0:.3cm);
\draw[thick] (0,1.1) arc (0:180:.3cm) node[left,yshift=.2cm]{$\scriptstyle X^*$} -- (-.6,-1.1) node[left,yshift=-.2cm]{$\scriptstyle X^*$} arc (-180:0:.3cm) -- cycle;
\node at (.2,0) {$\scriptstyle X$};
\roundNbox{fill=white}{(-.6,-.8)}{.3}{0}{0}{$f$};
\roundNbox{fill=white}{(-.6,0)}{.3}{0}{0}{$\zeta_X$};
\roundNbox{fill=white}{(-.6,.8)}{.3}{0}{0}{$\zeta^\dag_X$};
}
\,\right)
&=
\Psi_b\left(\,
\tikzmath{
\fill[rounded corners = 5pt, fill=\brColor] (-1.2,-1.5) rectangle (.6,1.5);
\fill[\arColor] (0,1.1) arc (0:180:.3cm) -- (-.6,-1.1) arc (-180:0:.3cm);
\draw[thick] (0,1.1) arc (0:180:.3cm) node[left,yshift=.2cm]{$\scriptstyle X^*$} -- (-.6,-1.1) node[left,yshift=-.2cm]{$\scriptstyle X^*$} arc (-180:0:.3cm) -- cycle;
\node at (.2,0) {$\scriptstyle X$};
\roundNbox{fill=white}{(-.6,0)}{.3}{0}{0}{$f$};
\roundNbox{fill=white}{(-.6,.8)}{.3}{0}{0}{$\zeta_X$};
\roundNbox{fill=white}{(-.6,-.8)}{.3}{0}{0}{$\zeta^\dag_X$};
}
\,\right)
=
\Psi_b\left(\,
\tikzmath{
\fill[rounded corners = 5pt, fill=\brColor] (1.2,-.9) rectangle (-.6,.9);
\fill[\arColor]  (0,.3) arc (180:0:.3cm) -- (.6,-.3) arc (0:-180:.3cm);
\draw[thick] (0,.3) node[left,yshift=.2cm]{$\scriptstyle X^\vee$} arc (180:0:.3cm) --node[right]{$\scriptstyle X$} (.6,-.3) arc (0:-180:.3cm) node[left,yshift=-.2cm]{$\scriptstyle X^\vee$};
\roundNbox{fill=white}{(0,0)}{.3}{0}{0}{$f$};
}
\right)
=
\Psi_a\left(\,
\tikzmath{
\fill[rounded corners = 5pt, fill=\arColor] (-1.2,-.9) rectangle (.6,.9);
\fill[\brColor] (0,.3) arc (0:180:.3cm) -- (-.6,-.3) arc (-180:0:.3cm);
\draw[thick] (0,.3) node[right,yshift=.2cm]{$\scriptstyle X^\vee$} arc (0:180:.3cm) --node[left]{$\scriptstyle X$} (-.6,-.3) arc (-180:0:.3cm) node[right,yshift=-.2cm]{$\scriptstyle X^\vee$};
\roundNbox{fill=white}{(0,0)}{.3}{0}{0}{$f$};
}
\,\right)\\
&=
\Psi_a\left(\,
\tikzmath{
\fill[rounded corners = 5pt, fill=\arColor] (-1.2,-1.5) rectangle (.6,1.5);
\fill[\brColor] (0,1.1) arc (0:180:.3cm) -- (-.6,-1.1) arc (-180:0:.3cm);
\draw[thick] (0,1.1)  node[right,yshift=.2cm]{$\scriptstyle X^*$} arc (0:180:.3cm) -- (-.6,-1.1) arc (-180:0:.3cm) node[right,yshift=-.2cm]{$\scriptstyle X^*$} -- cycle;
\node at (.2,0) {$\scriptstyle X$};
\roundNbox{fill=white}{(0,-.8)}{.3}{.2}{.2}{$\scriptstyle \zeta_X^{-1}$};
\roundNbox{fill=white}{(0,0)}{.3}{0}{0}{$f$};
\roundNbox{fill=white}{(0,.8)}{.3}{0.2}{0.2}{$\scriptstyle (\zeta_X^\dag)^{-1}$};
\node at (-.8,0) {$\scriptstyle X$};
}
\,\right)
=
\Psi_b\left(\,
\tikzmath{
\fill[rounded corners = 5pt, fill=\brColor] (-1.2,-1.5) rectangle (.6,1.5);
\fill[\arColor] (0,1.1) arc (0:180:.3cm) -- (-.6,-1.1) arc (-180:0:.3cm);
\draw[thick] (0,1.1) arc (0:180:.3cm) node[left,yshift=.2cm]{$\scriptstyle X^*$} -- (-.6,-1.1) node[left,yshift=-.2cm]{$\scriptstyle X^*$} arc (-180:0:.3cm) -- cycle;
\node at (.2,0) {$\scriptstyle X$};
\roundNbox{fill=white}{(-.6,0)}{.3}{0}{0}{$f$};
\roundNbox{fill=white}{(-.6,.8)}{.3}{0.2}{0.2}{$\scriptstyle (\zeta_X^\dag)^{-1}$};
\roundNbox{fill=white}{(-.6,-.8)}{.3}{0.2}{0.2}{$\scriptstyle \zeta_X^{-1}$};
}
\,\right)
=
\Psi_b\left(\,
\tikzmath{
\fill[rounded corners = 5pt, fill=\brColor] (-1.2,-1.5) rectangle (.6,1.5);
\fill[\arColor] (0,1.1) arc (0:180:.3cm) -- (-.6,-1.1) arc (-180:0:.3cm);
\draw[thick] (0,1.1) arc (0:180:.3cm) node[left,yshift=.2cm]{$\scriptstyle X^*$} -- (-.6,-1.1) node[left,yshift=-.2cm]{$\scriptstyle X^*$} arc (-180:0:.3cm) -- cycle;
\node at (.2,0) {$\scriptstyle X$};
\roundNbox{fill=white}{(-.6,-.8)}{.3}{0}{0}{$f$};
\roundNbox{fill=white}{(-.6,0)}{.3}{0.2}{0.2}{$\scriptstyle (\zeta_X^\dag)^{-1}$};
\roundNbox{fill=white}{(-.6,.8)}{.3}{0.2}{0.2}{$\scriptstyle \zeta_X^{-1}$};
}
\,\right).
\end{align*}
Note that in the first and last equations, we used traciality of pivotal trace associated to $(X^*,\ev^*_X,\coev_X^*)$. 
By non-degeneracy of $\Psi$, we see that $\zeta^\dag_X \zeta_X = \zeta^{-1}_X (\zeta^\dag_X)^{-1}$. 
But since $\zeta^\dag_X \zeta_X$ is positive, the only way it can equal its own inverse is if it equals $\id_{X^*}$. 
We conclude that $\zeta_X$ is unitary.
\end{proof}
\subsection{Isometries, coisometries, and Hilbert direct sums}
\label{sec:HilbDirectSum}

In a unitary 2-category, it makes sense to talk about orthogonal direct sums.
Given $a,b\in\fX$, their \emph{orthogonal direct sum} is a direct sum $a\boxplus b$ such that the isomorphisms $I_c\otimes P_c \cong 1_c$ are unitary for $c=a,b$,
and the isomorphism
$(P_a \otimes I_a) \oplus (P_b \otimes I_b) \cong 1_{a\boxplus b}$
is both unitary and witnesses $1_{a\boxplus b}$ as an orthogonal direct sum.

As in Remark \ref{rem:UpgradeDirectSums}, we may sill arrange that $I_c$ is both left and right adjoint to $P_c$ for $c=a,b$ while keeping the desired direct sum orthogonal.
However, when working in a pre-3-Hilbert space, these adjunctions need not be compatible with the ambient UAF (see Remark \ref{rem:Evil2Hilb} below).  
Thus the notion of orthogonal direct sum is not sufficient for the purposes of this article.

\begin{defn}
\label{defn:IsometryInPre3Hilb}
Suppose $(\fX,\vee,\Psi)$ is a pre-3-Hilbert space and $a,b\in\fX$.
A 1-morphism ${}_aX_b$ is called an \emph{isometry} if $\coev_X: 1_a\to {}_aX\otimes_b X^\vee_a$ 
is unitary
and
$\ev_X: {}_b X^\vee\otimes_a X_b\to 1_b$ is an isometry.
Denoting $a,b,{}_aX_b,{}_bX^\vee_a$ graphically by
$$
\tikzmath{\fill[fill=\arColor, rounded corners=5] (-.3,-.3) rectangle (.3,.3);}=a
\qquad\qquad
\tikzmath{\fill[fill=\brColor, rounded corners=5] (-.3,-.3) rectangle (.3,.3);}=b
\qquad\qquad
\tikzmath{
\begin{scope}
\clip[rounded corners=5pt] (-.3,-.3) rectangle (.3,.3);
\fill[\arColor] (0,-.3) rectangle (-.3,.3);
\fill[\brColor] (0,-.3) rectangle (.3,.3);
\end{scope}
\draw[thick] (0,-.3) -- (0,.3);
}= {}_aX_b
\qquad\qquad
\tikzmath{
\begin{scope}
\clip[rounded corners=5pt] (-.3,-.3) rectangle (.3,.3);
\fill[\brColor] (0,-.3) rectangle (-.3,.3);
\fill[\arColor] (0,-.3) rectangle (.3,.3);
\end{scope}
\draw[thick] (0,-.3) -- (0,.3);
}= {}_bX^\vee_a
$$
the relations for an isometry are
$$
\tikzmath{
\fill[rounded corners = 5pt, fill=\arColor] (-.9,-.6) rectangle (.9,.6);
\filldraw[fill=\brColor, thick] (0,0) circle (.3cm);
\node at (-.6,0) {$\scriptstyle X$};
\node at (.6,0) {$\scriptstyle X^\vee$};
}
=
\tikzmath{
\fill[rounded corners = 5pt, fill=\arColor] (-.6,-.6) rectangle (.6,.6);
}
\qquad\qquad
\tikzmath{
\fill[rounded corners = 5pt, fill=\arColor] (-.9,-.6) rectangle (.9,.6);
\filldraw[fill=\brColor, thick] (-.3,.6) arc(-180:0:.3cm);
\filldraw[fill=\brColor, thick] (-.3,-.6) arc(180:0:.3cm);
}
=
\tikzmath{
\fill[rounded corners = 5pt, fill=\arColor] (-.9,-.6) rectangle (.9,.6);
\fill[fill=\brColor] (-.3,-.6) rectangle (.3,.6);
\draw[thick] (-.3,-.6) -- (-.3,.6);
\draw[thick] (.3,-.6) -- (.3,.6);
\node at (.6,0) {$\scriptstyle X^\vee$};
\node at (-.6,0) {$\scriptstyle X$};
}
\qquad\qquad
\tikzmath{
\fill[rounded corners = 5pt, fill=\brColor] (-.9,-.6) rectangle (.9,.6);
\filldraw[fill=\arColor, thick] (-.3,.6) arc(-180:0:.3cm);
\filldraw[fill=\arColor, thick] (-.3,-.6) arc(180:0:.3cm);
}
=
\tikzmath{
\fill[rounded corners = 5pt, fill=\brColor] (-.9,-.6) rectangle (.9,.6);
\fill[fill=\arColor] (-.3,-.6) rectangle (.3,.6);
\draw[thick] (-.3,-.6) -- (-.3,.6);
\draw[thick] (.3,-.6) -- (.3,.6);
\node at (.6,0) {$\scriptstyle X$};
\node at (-.6,0) {$\scriptstyle X^\vee$};
}\,.
$$
Observe that if ${}_aX_b$ is an isometry and $f:\id_a \to \id_a$, then $\Psi_a$ is completely determined in terms of $\Psi_b$:
\begin{equation}
\label{eq:CoisometryDeterminesPsiOnSmaller}
\Psi_a(f)
=
\Psi_a\left(\,
\tikzmath{
\fill[rounded corners = 5pt, fill=\arColor] (-.7,-.7) rectangle (.7,.7);
\roundNbox{fill=white}{(0,0)}{.3}{0}{0}{$f$}
}
\,\right)
=
\Psi_a\left(\,
\tikzmath{
\fill[rounded corners = 5pt, fill=\arColor] (-1.9,-.7) rectangle (.9,.7);
\filldraw[fill=\brColor, thick] (0,0) circle (.3cm);
\roundNbox{fill=white}{(-1.3,0)}{.3}{0}{0}{$f$}
\node at (-.6,0) {$\scriptstyle X$};
\node at (.6,0) {$\scriptstyle X^\vee$};
}
\,\right)
=
\Psi_b\left(\,
\tikzmath{
\fill[rounded corners = 5pt, fill=\brColor] (-1.2,-.8) rectangle (1.2,.8);
\filldraw[fill=\arColor, thick] (0,0) circle (.6cm);
\roundNbox{fill=white}{(0,0)}{.3}{0}{0}{$f$}
\node at (-.9,0) {$\scriptstyle X^\vee$};
\node at (.9,0) {$\scriptstyle X$};
}
\,\right).
\end{equation}
A 1-morphism ${}_aX_b$ is called a \emph{co-isometry} if ${}_bX^\vee_a$ is an isometry, i.e.,
$\ev_X: {}_bX^\vee\otimes_a X_b \to 1_b$ is unitary
and
$\coev_X: 1_a\to {}_aX\otimes_b X^\vee_a$
is a coisometry.
Graphically, the relations for a co-isometry are
$$
\tikzmath{
\fill[rounded corners = 5pt, fill=\arColor] (-.9,-.6) rectangle (.9,.6);
\filldraw[fill=\brColor, thick] (-.3,.6) arc(-180:0:.3cm);
\filldraw[fill=\brColor, thick] (-.3,-.6) arc(180:0:.3cm);
}
=
\tikzmath{
\fill[rounded corners = 5pt, fill=\arColor] (-.9,-.6) rectangle (.9,.6);
\fill[fill=\brColor] (-.3,-.6) rectangle (.3,.6);
\draw[thick] (-.3,-.6) -- (-.3,.6);
\draw[thick] (.3,-.6) -- (.3,.6);
\node at (.6,0) {$\scriptstyle X^\vee$};
\node at (-.6,0) {$\scriptstyle X$};
}
\qquad\qquad
\tikzmath{
\fill[rounded corners = 5pt, fill=\brColor] (-.9,-.6) rectangle (.9,.6);
\filldraw[fill=\arColor, thick] (0,0) circle (.3cm);
\node at (-.6,0) {$\scriptstyle X^\vee$};
\node at (.6,0) {$\scriptstyle X$};
}
=
\tikzmath{
\fill[rounded corners = 5pt, fill=\brColor] (-.6,-.6) rectangle (.6,.6);
}
\qquad\qquad
\tikzmath{
\fill[rounded corners = 5pt, fill=\brColor] (-.9,-.6) rectangle (.9,.6);
\filldraw[fill=\arColor, thick] (-.3,.6) arc(-180:0:.3cm);
\filldraw[fill=\arColor, thick] (-.3,-.6) arc(180:0:.3cm);
}
=
\tikzmath{
\fill[rounded corners = 5pt, fill=\brColor] (-.9,-.6) rectangle (.9,.6);
\fill[fill=\arColor] (-.3,-.6) rectangle (.3,.6);
\draw[thick] (-.3,-.6) -- (-.3,.6);
\draw[thick] (.3,-.6) -- (.3,.6);
\node at (.6,0) {$\scriptstyle X$};
\node at (-.6,0) {$\scriptstyle X^\vee$};
}\,.
$$

An \emph{isometric equivalence} is a 1-morphism ${}_aX_b$ that is both an isometry and a co-isometry, i.e., an adjoint equivalence with respect to our UAF $\vee$.
\end{defn}

\begin{rem}
A coisometry is stronger than the notion of a dagger condensation in the sense of \cite{MR4369356}; see also \cite{1905.09566}.
\end{rem}

\begin{rem}
For any 1-morphism ${}_aX_b$ in a pre-3-Hilbert space,
\begin{equation}
\label{eq:EquivalentRecabling}
\tikzmath{
\fill[rounded corners = 5pt, fill=\arColor] (-.9,-.6) rectangle (.9,.6);
\filldraw[fill=\brColor, thick] (-.3,.6) arc(-180:0:.3cm);
\filldraw[fill=\brColor, thick] (-.3,-.6) arc(180:0:.3cm);
}
=
\tikzmath{
\fill[rounded corners = 5pt, fill=\arColor] (-.9,-.6) rectangle (.9,.6);
\fill[fill=\brColor] (-.3,-.6) rectangle (.3,.6);
\draw[thick] (-.3,-.6) -- (-.3,.6);
\draw[thick] (.3,-.6) -- (.3,.6);
\node at (.6,0) {$\scriptstyle X^\vee$};
\node at (-.6,0) {$\scriptstyle X$};
}
\qquad
\Longleftrightarrow
\qquad
\tikzmath{
\fill[rounded corners = 5pt, fill=\brColor] (-.9,-.6) rectangle (.9,.6);
\filldraw[fill=\arColor, thick] (-.3,.6) arc(-180:0:.3cm);
\filldraw[fill=\arColor, thick] (-.3,-.6) arc(180:0:.3cm);
}
=
\tikzmath{
\fill[rounded corners = 5pt, fill=\brColor] (-.9,-.6) rectangle (.9,.6);
\fill[fill=\arColor] (-.3,-.6) rectangle (.3,.6);
\draw[thick] (-.3,-.6) -- (-.3,.6);
\draw[thick] (.3,-.6) -- (.3,.6);
\node at (.6,0) {$\scriptstyle X$};
\node at (-.6,0) {$\scriptstyle X^\vee$};
}\,.
\end{equation}
Hence, $_{a}X_{b}$ is an isometry (resp. co-isometry) if and only if $\coev_X$ (resp. $\ev_X$) is unitary. 

We show the forward direction, as the two directions are symmetric.
First, observe that the relation on the left implies 
$$
\tikzmath{
\fill[rounded corners = 5pt, fill=\brColor] (-.9,-.6) rectangle (.9,.6);
\filldraw[fill=\arColor, thick] (-.3,.6) arc(-180:0:.3cm);
}
=
\tikzmath{
\fill[rounded corners = 5pt, fill=\brColor] (-1.8,-1.2) rectangle (.9,1.2);
\filldraw[fill=\arColor, thick] (-.3,1.2) -- (-.3,.6) arc(0:-180:.3cm) arc (0:180:.3cm) -- (-1.5,-.6) arc(-180:0:.3cm) arc(180:0:.3cm) arc (-180:0:.3cm) -- (.3,1.2);
}
=
\tikzmath{
\fill[rounded corners = 5pt, fill=\brColor] (-1.8,-1.2) rectangle (.9,1.2);
\filldraw[fill=\arColor, thick] (-.3,1.2) -- (-.3,-.6) arc (-180:0:.3cm) -- (.3,1.2);
\filldraw[fill=\arColor, thick] (-.9,.6) arc (0:180:.3cm) -- (-1.5,-.6) arc(-180:0:.3cm) -- (-.9,.6);
}
=
\tikzmath{
\fill[rounded corners = 5pt, fill=\brColor] (-.9,-.9) rectangle (.9,.6);
\filldraw[fill=\arColor, thick] (-.3,.6) arc(-180:0:.3cm);
\filldraw[fill=\arColor, thick] (0,-.3) circle (.3cm);
}\,.
$$
Since both diagrams on the right hand side of \eqref{eq:EquivalentRecabling} are self-adjoint and
$$
\left(\,
\tikzmath{
\fill[rounded corners = 5pt, fill=\brColor] (-.6,-.6) rectangle (.6,.6);
\fill[fill=\arColor] (-.3,-.6) rectangle (.3,.6);
\draw[thick] (-.3,-.6) -- (-.3,.6);
\draw[thick] (.3,-.6) -- (.3,.6);
}
-
\tikzmath{
\fill[rounded corners = 5pt, fill=\brColor] (-.6,-.6) rectangle (.6,.6);
\filldraw[fill=\arColor, thick] (-.3,.6) arc(-180:0:.3cm);
\filldraw[fill=\arColor, thick] (-.3,-.6) arc(180:0:.3cm);
}
\,\right)^2
=
\tikzmath{
\fill[rounded corners = 5pt, fill=\brColor] (-.6,-.6) rectangle (.6,.6);
\fill[fill=\arColor] (-.3,-.6) rectangle (.3,.6);
\draw[thick] (-.3,-.6) -- (-.3,.6);
\draw[thick] (.3,-.6) -- (.3,.6);
}
-
2\cdot
\tikzmath{
\fill[rounded corners = 5pt, fill=\brColor] (-.6,-.6) rectangle (.6,.6);
\filldraw[fill=\arColor, thick] (-.3,.6) arc(-180:0:.3cm);
\filldraw[fill=\arColor, thick] (-.3,-.6) arc(180:0:.3cm);
}
+
\tikzmath{
\fill[rounded corners = 5pt, fill=\brColor] (-.9,-.9) rectangle (.9,.9);
\filldraw[fill=\arColor, thick] (-.3,.9) arc(-180:0:.3cm);
\filldraw[fill=\arColor, thick] (0,0) circle (.3cm);
\filldraw[fill=\arColor, thick] (-.3,-.9) arc(180:0:.3cm);
}
=
\tikzmath{
\fill[rounded corners = 5pt, fill=\brColor] (-.6,-.6) rectangle (.6,.6);
\fill[fill=\arColor] (-.3,-.6) rectangle (.3,.6);
\draw[thick] (-.3,-.6) -- (-.3,.6);
\draw[thick] (.3,-.6) -- (.3,.6);
}
-
\tikzmath{
\fill[rounded corners = 5pt, fill=\brColor] (-.6,-.6) rectangle (.6,.6);
\filldraw[fill=\arColor, thick] (-.3,.6) arc(-180:0:.3cm);
\filldraw[fill=\arColor, thick] (-.3,-.6) arc(180:0:.3cm);
}\,,
$$
the result follows by taking trace and applying $\Psi$ to the above morphism:
$$
\Psi_b\left(\,
\tikzmath{
\fill[rounded corners = 5pt, fill=\brColor] (-.9,-.9) rectangle (.9,.9);
\filldraw[fill=\arColor, thick] (0,0) circle (.6cm);
\filldraw[fill=\brColor, thick] (0,0) circle (.3cm);
}
-
\tikzmath{
\fill[rounded corners = 5pt, fill=\brColor] (-.6,-.6) rectangle (.6,.6);
\filldraw[fill=\arColor, thick] (0,0) circle (.3cm);
}
\,\right)
=
\Psi_a\left(\,
\tikzmath{
\fill[rounded corners = 5pt, fill=\arColor] (-.6,-1.05) rectangle (.6,1.05);
\filldraw[fill=\brColor, thick] (0,.45) circle (.3cm);
\filldraw[fill=\brColor, thick] (0,-.45) circle (.3cm);
}
-
\tikzmath{
\fill[rounded corners = 5pt, fill=\arColor] (-.6,-.6) rectangle (.6,.6);
\filldraw[fill=\brColor, thick] (0,0) circle (.3cm);
}
\,\right)
=
\Psi_a\left(\,
\tikzmath{
\fill[rounded corners = 5pt, fill=\arColor] (-.6,-1.05) rectangle (.6,1.05);
\filldraw[fill=\brColor, thick] (.3,.45) arc (0:180:.3cm) -- (-.3,-.45) arc(-180:0:.3cm) -- (.3,.45);
}
-
\tikzmath{
\fill[rounded corners = 5pt, fill=\arColor] (-.6,-.6) rectangle (.6,.6);
\filldraw[fill=\brColor, thick] (0,0) circle (.3cm);
}
\,\right)
=0.
$$
\end{rem}

\begin{ex}
\label{ex:IsometriesIn2Hilb}
We show that the isometries in $2\Hilb$ are the fully faithful $\Tr$-preserving $\dag$-functors.
Thus the two notions of isometry for unitary functors between two 2-Hilbert spaces from Definitions \ref{defn:pre2Hilb} and \ref{defn:IsometryInPre3Hilb} agree.

Suppose $F: (\cA,\Tr^\cA)\to (\cB,\Tr^\cB)$ is a unitary functor.
When $F$ is fully faithful, then $F$ maps distinct simples in $\cA$ to distinct simples in $\cB$.
Moreover, $F^*:\cB\to \cA$ is the projection map onto the image of $\cA$, i.e., if $b\in\Irr(\cB)$, then $F^*(b)= a\in\Irr(\cA)$ if $F(a)= b$, and otherwise $F^*(b)=0$.
Since $F\dashv^\dag F^*$, we have a unitary isomorphism
$$
\cA(a\to a)=\cA(F^*F(a)\to a)\cong \cB(F(a)\to F(a)),
$$
and since $(\coev_F)_a=\mate(\id_{F(a)})$, we see that $(\coev_F)_a=\lambda_a \id_a$
for some $\lambda_a\in\bbC^\times$ which is unimodular if and only if $d_{F(a)}=d_a$.
We conclude that when $F$ is fully faithful, $\coev_F$ is unitary if and only if $F$ is trace-preserving.
It follows by a similar calculation or by \eqref{eq:EquivalentRecabling} that $\ev_F$ is an isometry.

Conversely, if $F$ is an isometry, then $\coev_F: \id_\cA\Rightarrow F^*\circ F$ is unitary and thus invertible, so $F$ is fully faithful.
The above argument then applies to show that $F$ is trace preserving.
\end{ex}

\begin{defn}
\label{Defn:IsometricEquiv}
Given two pre-3-Hilbert spaces $(\fX, \vee_\fX,\Psi^\fX)$ and $(\fY, \vee_\fY,\Psi^\fY)$, 
we write $\Fun^{\dag}(\fX \to \fY)$ for the unitary 2-category of unitary 2-functors $\fX\to \fY$.

We say a unitary 2-functor $F: \fX\to\fY$ is \emph{UAF-preserving} if for every 1-morphism ${}_aX_b\in \fX(a\to b)$, the canonical isomorphism

$$
\tikzmath{
\begin{scope}
\clip[rounded corners=5pt] (-1.6,-.7) rectangle (2.6,2.4);
\fill[\brColor] (-1.6,-.7) rectangle (2.6,2.4);
\fill[\arColor] (-.5,-.7) -- (-.5,.7) -- (.5,.7) -- (.5,.3) -- (1.5,.3) -- (1.5,2.4) -- (2.6,2.4) -- (2.6,-.7);
\end{scope}
\draw[thick, double] (0,1) -- (0,1.9);
\draw[thick] (-.5,-.7) --node[left]{$\scriptstyle F(X^\vee)$} (-.5,.7);
\draw[thick] (.5,.3) --node[right]{$\scriptstyle F(X)$} (.5,.7);
\draw[thick] (1.5,.3) --node[right]{$\scriptstyle F(X)^{\vee}$} (1.5,2.4);
\roundNbox{fill=white}{(0,1)}{.3}{.5}{.5}{$F^2_{X^\vee,X}$};
\roundNbox{fill=white}{(0,1.9)}{.3}{.5}{.5}{$F(\ev_X)$};
\roundNbox{fill=white}{(1,0)}{.3}{.55}{.55}{$\coev_{F(X)}$};
}
$$
is unitary. 
We write 
$\Fun^{\dag,\vee}(\fX \to \fY)$
for the full unitary 2-subcategory of $\Fun^{\dag}(\fX \to \fY)$ of UAF-preserving unitary 2-functors.

We say a $\dag$ 2-functor $F: \fX\to\fY$ is  an \emph{isometry} 
if 
$$
\Psi^\fY_{F(a)}(F(f)) = \Psi^\fX_a(f) \qquad\qquad\forall\, f \in \End_\fX(1_a),\,\,\forall\,a \in \fX.
$$
We write $\Isom(\fX\to \fY)$ for the full unitary 2-subcategory of $\Fun^\dag(\fX\to \fY)$ whose objects are isometries.
By Proposition \ref{prop:uniquenessofUAF}, $\Isom(\fX\to \fY)$ is a full unitary 2-subcategory of  $\Fun^{\dag,\vee}(\fX \to \fY)$.

We say two pre-3-Hilbert spaces are \emph{isometrically equivalent} if there is an isometric $\dag$-equivalence $(\fX, \vee_\fX,\Psi^\fX) \to (\fY, \vee_\fY,\Psi^\fY)$ which is \emph{isometrically essentially surjective}, i.e., for every $b\in \fY$, there is an $a\in\fX$ and an isometric equivalence ${}_{F(a)}Y_b$.
Note that the notion of isometric equivalence is strictly stronger than $\dag$-equivalence of the underlying unitary 2-categories.
\end{defn}

\begin{rem}
We conjecture there is a canonical pre-3-Hilbert structure on $\Fun^\dag(\fX \to \fY)$ such that the following are equivalent:
\begin{enumerate}
\item $F \in \Fun^\dag(\fX \to \fY)$ is an isometric equivalence between pre-3-Hilbert spaces;
\item There exist isometric equivalences $\eta \colon F^* F \Rightarrow \id_\fX$ in $\Fun^\dag(\fX \to \fX)$ and $\epsilon \colon \id_\fY \Rightarrow F F^*$ in $\Fun^\dag(\fY \to \fY)$, where $F^*$ is the unitary adjoint of $F$.
\end{enumerate}
\end{rem}

\begin{ex}
Given $\rmH^*$-multifusion categories $(\cC,\vee,\psi)$ and $(\cD,\vee,\psi)$, and a $\cC$-$\cD$ bimodule ${}_\cC \cN _\cD$ equipped with a $\cC$-$\cD$ bimodule trace $\Tr^\cN$, there is a $\dag$-2-functor
$$
- \boxtimes_{\cC} \cN_\cD 
\colon 
\Mod^\dag(\cC) \to \Mod^\dag(\cD)
$$
given by
\begin{align*}
m \boxtimes n &\mapsto F(m) \boxtimes n\\
\tikzmath{
\draw[thick] (0,-.8) -- node[right]{$\scriptstyle m$} (0,-1.2);
\draw[thick] (0,-.2) -- (0,1.2) node[left]{$\scriptstyle m'$};
\draw[thick] (1,.8) -- (1,1.2) node[right]{$\scriptstyle n'$};
\draw[thick] (1,.2) --node[right]{$\scriptstyle n$} (1,-1.2);
\draw[thick, blue] (0,-.5) --node[above]{$\scriptstyle c$} (1,.5);
\roundNbox{fill=white}{(0,-.5)}{0.3}{.1}{.1}{$f$};
\roundNbox{fill=white}{(1,.5)}{0.3}{.1}{.1}{$g$};
}    
&\mapsto 
\tikzmath{
\draw[thick] (.1,-.8) -- node[right]{$\scriptstyle m$} (.1,-1.2);
\draw[thick] (-.6,-1.2) --node[left]{$\scriptstyle F$} (-.6,1.2);
\draw[thick] (0,-.2) -- (0,1.2) node[right]{$\scriptstyle m'$};
\draw[thick] (1,.8) -- (1,1.2) node[right]{$\scriptstyle n'$};
\draw[thick] (1,.2) --node[right]{$\scriptstyle n$} (1,-1.2);
\draw[thick, blue] (0,-.5) --node[above]{$\scriptstyle c$} (1,.5);
\roundNbox{fill=white}{(0,-.5)}{0.3}{.1}{.1}{$f$};
\roundNbox{fill=white}{(1,.5)}{0.3}{.1}{.1}{$g$};
}
\end{align*}
We later show in Lemma \ref{lem:boxtimesisUAFpreserving} that this map is UAF-preserving.
\end{ex}

\begin{defn}
A \emph{Hilbert direct sum} of objects $a_1,a_2$ in a pre-3-Hilbert space consists of an object $a_1\boxplus a_2$ and isometries 
$I_j\colon a_j \to a_1\boxplus a_2$
for $j=1,2$
such that 
$\ev_{1}\ev_{1}^\dag+\ev_2\ev_{2}^\dag = \id_{1_{a_1\boxplus a_2}}$ where $\ev_j=\ev_{I_j}$.
This means that $I_1,I_2$ with their evaluations and coevaluations satisfy isometry relations, and we have the additional graphical relation
$$
\sum_{j=1}^2
\tikzmath{
\fill[rounded corners = 5pt, fill=green!20] (-.9,-.6) rectangle (.9,.6);
\filldraw[fill=blue!20, thick] (0,0) circle (.3cm);
\node at (-.6,0) {$\scriptstyle I_j^\vee$};
\node at (.6,0) {$\scriptstyle I_j$};
\node at (0,0) {$\scriptstyle j$};
}
=
\tikzmath{
\fill[rounded corners = 5pt, fill=green!20] (-.6,-.6) rectangle (.6,.6);
}
\qquad\qquad
\tikzmath{
\fill[fill=blue!20, rounded corners=5] (-.3,-.3) rectangle (.3,.3);
\node at (0,0) {$\scriptstyle j$};
}=a_j
\qquad\qquad
\tikzmath{\fill[fill=green!20, rounded corners=5] (-.3,-.3) rectangle (.3,.3);}=a_1\boxplus a_2.
$$
Observe that $\Psi_{a_1\boxplus a_2}$ is completely determined by $\Psi_{a_1}$ and $\Psi_{a_2}$.
Indeed,
if $f:1_{a_1\boxplus a_2}\to 1_{a_1\boxplus a_2}$,
then
\begin{equation}
\label{eq:PsiOfHilbertDirectSum}
\Psi_{a_1\boxplus a_2}\left(\,
\tikzmath{
\fill[rounded corners = 5pt, fill=green!20] (-.7,-.7) rectangle (.7,.7);
\roundNbox{fill=white}{(0,0)}{.3}{0}{0}{$f$}
}
\,\right)
=
\sum_{j=1}^2
\Psi_{a_1\boxplus a_2}\left(\,
\tikzmath{
\fill[rounded corners = 5pt, fill=green!20] (-1.9,-.7) rectangle (.9,.7);
\filldraw[fill=blue!20, thick] (0,0) circle (.3cm);
\roundNbox{fill=white}{(-1.3,0)}{.3}{0}{0}{$f$}
\node at (-.6,0) {$\scriptstyle I_{j}^\vee$};
\node at (.6,0) {$\scriptstyle I_{j}$};
\node at (0,0) {$\scriptstyle j$};
}
\,\right)
=
\sum_{j=1}^2
\Psi_{a_j}\left(\,
\tikzmath{
\fill[rounded corners = 5pt, fill=blue!20] (-1.2,-.8) rectangle (1.2,.8);
\filldraw[fill=green!20, thick] (0,0) circle (.6cm);
\roundNbox{fill=white}{(0,0)}{.3}{0}{0}{$f$}
\node at (-.9,0) {$\scriptstyle I_j$};
\node at (.9,0) {$\scriptstyle I_j^\vee$};
}
\,\right).
\end{equation}

A pre-3-Hilbert space is called \emph{Hilbert direct sum complete} ($\boxplus$-complete) if every two objects $a_1,a_2\in\fX$ admit a Hilbert direct sum.
\end{defn}

\begin{rem}
Given two Hilbert direct sums $a_1\boxplus a_2$ and $a_1\boxplus' a_2$, which we represent graphically by
$$
\tikzmath{
\fill[fill=blue!20, rounded corners=5] (-.3,-.3) rectangle (.3,.3);
\node at (0,0) {$\scriptstyle j$};
}=a_j
\qquad\qquad
\tikzmath{\fill[fill=green!20, rounded corners=5] (-.3,-.3) rectangle (.3,.3);}=a_1\boxplus a_2
\qquad\qquad
\tikzmath{\fill[fill=green!30!black!20, rounded corners=5] (-.3,-.3) rectangle (.3,.3);}=a_1\boxplus' a_2,
$$
there is a canonical isometric equivalence 
$a_1\boxplus a_2 \to a_1\boxplus' a_2$ given by the orthogonal direct sum
$$
\tikzmath{
\begin{scope}
\clip[rounded corners=5pt] (-.9,-.6) rectangle (.9,.6);
\fill[fill=green!20] (-.9,-.6) rectangle (-.3,.6);    
\fill[fill=blue!20] (-.3,-.6) rectangle (.3,.6);
\fill[fill=green!30!black!20] (.9,-.6) rectangle (.3,.6);    
\end{scope}
\draw[thick] (-.3,-.6) -- (-.3,.6);
\draw[thick] (.3,-.6) -- (.3,.6);
\node at (.6,0) {$\scriptstyle I_1'$};
\node at (-.6,0) {$\scriptstyle I_1^\vee$};
\node at (0,0) {$\scriptstyle 1$};
}
\,
\oplus
\,
\tikzmath{
\begin{scope}
\clip[rounded corners=5pt] (-.9,-.6) rectangle (.9,.6);
\fill[fill=green!20] (-.9,-.6) rectangle (-.3,.6);    
\fill[fill=blue!20] (-.3,-.6) rectangle (.3,.6);
\fill[fill=green!30!black!20] (.9,-.6) rectangle (.3,.6);    
\end{scope}
\draw[thick] (-.3,-.6) -- (-.3,.6);
\draw[thick] (.3,-.6) -- (.3,.6);
\node at (.6,0) {$\scriptstyle I_2'$};
\node at (-.6,0) {$\scriptstyle I_2^\vee$};
\node at (0,0) {$\scriptstyle 2$};
}\,.
$$
These isometric equivalences are compatible under composition, showing that the space of Hilbert direct sums of $a_1,a_2$ is isometrically contractible.
\end{rem}

\begin{rem}
\label{rem:Evil2Hilb}
The notion of Hilbert direct sum is strictly stronger than orthogonal direct sum.
Consider the full 2-subcategory of $2\Hilb$ whose objects are
$(\Hilb^{\boxplus n},n\Tr)$ for $n\geq 0$
with its restricted UAF given by unitary adjunction and spherical weight given by Sub-Example \ref{ex:WeightOn2Hilb}.
This 2-subcategory admits orthogonal direct sums, but not Hilbert direct sums.
In particular, the adjoint functors $I_j:\Hilb \to\Hilb\boxplus \Hilb$ and $P_j:\Hilb\boxplus \Hilb \to \Hilb$ are not unitary adjoints with respect to the scaled traces.
\end{rem}

\begin{ex}
\label{ex:ModCAdmitsHilbertDirectSums}
Given an $\rmH^*$-multifusion category $(\cC,\vee,\psi)$, the pre-3-Hilbert space $\Mod^\dag(\cC)$ from Example \ref{ex:WeightOnModC} admits Hilbert direct sums.
Moreover, every $(\cM,\Tr^\cM)\in \Mod^\dag(\cC)$ is a Hilbert direct sum of simple modules.
In particular, these results hold for $2\Hilb$.
\end{ex}

\begin{ex}
Given a pre-3-Hilbert space $(\fX,\vee,\Psi)$ and $a_1,\dots, a_n\in\fX$,
the linking algebra
$$
\cL(a_1,\dots, a_n)
=
\bigboxplus_{i,j=1}^n\fX(a_j\to a_i)
$$ 
is the Hilbert direct sum of the 2-Hilbert spaces $\fX(a_j\to a_i)$.
\end{ex}

\begin{rem}
If the pre-3-Hilbert space $(\fX,\vee,\Psi)$ is Hilbert direct sum complete, then the linking algebra $\cL(a_1,\dots, a_n)$ (which is organically an $\rmH^*$-multifusion category by Example \ref{ex:LinkingH*Alg}) is isometrically equivalent to $\End_\fX(\bigboxplus a_i)$.
\end{rem}

\begin{lem}
\label{lem:PreservesHilbertDirectSums}
Suppose $F:\fX\to \fY$ is a unitary functor between pre-3-Hilbert spaces.
If $F$ preserves the UAFs, then $F$ maps (co)isometries to (co)isometries.
In particular, if $F$ preserves the UAFs, then $F$ preserves Hilbert direct sums.
\end{lem}
\begin{proof}
Since $F$ preserves identity 2-morphisms, daggers, and the UAFs, it preserves the (co)isometry relations from Definition \ref{defn:IsometryInPre3Hilb}.
In particular, $F$ preserves the relation $$\ev_1\ev_1^\dag +\ev_2\ev_2^\dag = \id,$$ and thus $F$ preserves Hilbert direct sums.
\end{proof}

\begin{construction}[Hilbert direct sum completion]
Given a pre-3-Hilbert space $(\fX,\vee,\Psi)$, its \emph{Hilbert direct sum completion} is defined similarly to the matrix category of $\fX$ defined in \cite{MR4482713}.
It has
\begin{itemize}
\item
objects finite lists $(a_1,\dots, a_r)$ of objects of $\fX$,
\item
1-morphisms $(a_1,\dots, a_t)\to (b_1,\dots, b_s)$ are formal matrices $X=(X_{ij} :a_i \to b_j)$ of 1-morphisms in $\fX$ 
\item
2-morphisms $X=(X_{ij})\Rightarrow Y=(Y_{ij})$ are formal matrices $f=(f_{ij}: X_{ij}\Rightarrow Y_{ij})$ of 2-morphisms in $\fX$.
\end{itemize}
Composition of 2-morphisms happens component-wise, i.e., $(f\circ g)_{ij}=f_{ij}\circ g_{ij}$.
Composition of 1-morphisms is a matrix multiplication formula taking an orthogonal direct sum over $\otimes_\fX$:
$$
(X\otimes Y)_{ik} := \bigoplus_j X_{ij} \otimes_{b_j} Y_{jk} 
\qquad\qquad\text{if}\qquad\qquad
\begin{aligned}
X : (a_1,\dots, a_r)&\to (b_1,\dots, b_s)
\\
Y : (b_1,\dots, b_s)&\to (c_1,\dots, c_t).
\end{aligned}
$$
The unit of $(a_1,\dots, a_r)$ is $\operatorname{diag}(1_{a_1},\dots, 1_{a_r})$.

The UAF of $\Hilb_\boxplus(\fX)$ on $X=(X_{ij})$ is given $(X^\vee)_{ij}:=(X^\vee_{ji})$ and 
$$
(\ev_X: X^\vee\otimes X\Rightarrow 1_{(b_1,\dots, b_s)})_{ij}:= 
\delta_{i=j}\cdot
\bigoplus_k\ev_{X_{ik}}: X_{ik}^\vee\otimes_{a_k} X_{ki}\Rightarrow 1_{b_i}
$$
and similarly for $\coev_X$.
The spherical weight $\Psi^{\HilbPlus(\fX)}$ is given on 
$$
(f_{ij})\in\End_{\HilbPlus(\fX)}(1_{(a_1,\dots, a_r)})=\bigoplus \End_\fX(1_{a_j})
$$ 
by $\bigoplus_j \Psi^\fX_{a_j}(f_{jj})$.

To see that this weight is indeed spherical, for $X=(X_{ij}):(a_1,\dots, a_r) \to (b_1,\dots, b_s)$ and $f=(f_{ij})\in \End_{\HilbPlus(\fX)}(X)$,
$$
\Psi^{\HilbPlus(\fX)}_{(b_1,\dots, b_s)}(f)
=
\sum_{i,j}
\Psi^\fX_{b_j}
\left(\,
\tikzmath{
\fill[rounded corners = 5pt, fill=\brColor] (-1.2,-.9) rectangle (.6,.9);
\fill[\arColor] (0,.3) arc (0:180:.3cm) -- (-.6,-.3) arc (-180:0:.3cm);
\draw[thick] (0,.3) node[right,yshift=.2cm]{$\scriptstyle X_{ij}$} arc (0:180:.3cm) --node[left]{$\scriptstyle X_{ij}^\vee$} (-.6,-.3) arc (-180:0:.3cm) node[right,yshift=-.2cm]{$\scriptstyle X_{ij}$};
\roundNbox{fill=white}{(0,0)}{.3}{0}{0}{$f$}
}
\,\right)
=
\sum_{i,j}
\Psi^\fX_{a_i}
\left(\,
\tikzmath{
\fill[rounded corners = 5pt, fill=\arColor] (1.2,-.9) rectangle (-.6,.9);
\fill[\brColor]  (0,.3) arc (180:0:.3cm) -- (.6,-.3) arc (0:-180:.3cm);
\draw[thick] (0,.3) node[left,yshift=.2cm]{$\scriptstyle X_{ij}$} arc (180:0:.3cm) --node[right]{$\scriptstyle X_{ij}^\vee$} (.6,-.3) arc (0:-180:.3cm) node[left,yshift=-.2cm]{$\scriptstyle X_{ij}$};
\roundNbox{fill=white}{(0,0)}{.3}{0}{0}{$f$}
}
\right)
=
\Psi^{\HilbPlus(\fX)}_{(a_1,\dots, a_r)}(f).
$$

Observe that $\Hilb_\boxplus(\fX)$ comes equipped with an organic isometric inclusion $\iota_X:\fX\hookrightarrow \Hilb_\boxplus(\fX)$ given by $a\mapsto (a)$, ${}_aX_b\mapsto ({}_aX_b)$, and $(f: X\Rightarrow Y)\mapsto (f)$.
\end{construction}

\begin{rem}
The object $(x_1,\dots, x_n)\in \Hilb_\boxplus(\fX)$ is manifestly the Hilbert direct sum of $(x_1),\dots, (x_n)\in\iota(\fX)$.
\end{rem}

\begin{lem}
\label{lem:HilbCompletionIsHilbComplete}
$\Hilb_\boxplus(\fX)$ is Hilbert direct sum complete.
\end{lem}
\begin{proof}
One verifies that the Hilbert direct sum of $(a_1,\dots, a_m)$ and $(b_1,\dots, b_n)$ is the concatenation
$(a_1,\dots, a_m,b_1,\dots, b_n)$.
\end{proof}

\begin{rem}
In fact, $\HilbPlus$ is a $\dag$-3-functor on the $\rm C^*$-3-category of pre-3-Hilbert spaces.
We omit the proof which is similar to the proof that Q-system completion is a 3-functor on the $\rm C^*$ 3-category of $\rmC^*$ 2-categories from \cite{MR4369356}.
\end{rem}

The Hilbert direct sum completion satisfies a universal property in terms of the canonical isometric inclusion of pre-3-Hilbert spaces $\iota_\fX:\fX \hookrightarrow \HilbPlus(\fX)$.

\begin{prop}
\label{prop:HilbertDirectSumUniversal}
For any Hilbert direct sum complete pre-3-Hilbert space $\fY$,
precomposition with $\iota_\fX$ is an equivalence of unitary 2-categories
$$
\Fun^\dag(\HilbPlus(\fX)\to \fY) 
\xrightarrow{\iota_\fX^*}
\Fun^\dag(\fX\to \fY)
$$
which maps 
\begin{itemize}
\item 
$\Fun^{\dag,\vee}(\HilbPlus(\fX)\to \fY)$ onto $\Fun^{\dag,\vee}(\fX\to \fY)$, and
\item 
$\Isom(\HilbPlus(\fX)\to \fY)$ onto $\Isom(\fX\to \fY)$.
\end{itemize}
\end{prop}

To prove this proposition, we use the `overlay' graphical calculus for $\dag$-2-functors from \cite[\S2.4]{MR4369356}.
We represent 2-functors $A,B: \fZ\to \fY$ by \emph{textured regions}, transformations $\rho,\sigma: A\Rightarrow B$ by \emph{textured strings}, and modifications $m: \rho \Rrightarrow \sigma$ by coupons.

\begin{equation}
\label{eq:TexturedConventions}
\tikzmath{
\begin{scope}
\clip[rounded corners = 5pt] (-.7,-.7) rectangle (.7,.7);
\filldraw[primedregion=white] (0,.7) rectangle (-.7,-.7);
\filldraw[plusregion=white] (0,-.7) rectangle (.7,.7);
\end{scope}
\draw[\rhoColor,thick] (0,-.7) -- (0,-.3);
\draw[\sigmaColor,thick] (0,.3) -- (0,.7);
\draw[thin, dotted, rounded corners = 5pt] (-.7,-.7) rectangle (.7,.7);
\roundNbox{fill=white}{(0,0)}{.3}{0}{0}{$m$}
}
=
m:{}_A\rho_B\Rrightarrow{}_A\sigma_B
\qquad
\text{where}
\qquad
\begin{aligned}
\tikzmath{
\begin{scope}
\clip[rounded corners = 5pt] (0,0) rectangle (.6,.6);
\fill[primedregion=white] (0,0) rectangle (.3,.6);
\fill[plusregion=white] (.3,0) rectangle (.6,.6);
\end{scope}
\draw[\rhoColor,thick] (.3,0) -- (.3,.6);
\draw[thin, dotted, rounded corners = 5pt] (0,0) rectangle (.6,.6);
}
&=
\rho: A\Rightarrow B
\\
\tikzmath{
\begin{scope}
\clip[rounded corners = 5pt] (0,0) rectangle (.6,.6);
\fill[primedregion=white] (0,0) rectangle (.3,.6);
\fill[plusregion=white] (.3,0) rectangle (.6,.6);
\end{scope}
\draw[\sigmaColor,thick] (.3,0) -- (.3,.6);
\draw[thin, dotted, rounded corners = 5pt] (0,0) rectangle (.6,.6);
}
&=
\sigma: A\Rightarrow B
\end{aligned}
\qquad
\text{and}
\qquad
\begin{aligned}
\tikzmath{
\filldraw[primedregion=white, rounded corners = 5pt] (0,0) rectangle (.6,.6);
\draw[thin, dotted, rounded corners = 5pt] (0,0) rectangle (.6,.6);
}
&=
A
\\
\tikzmath{
\filldraw[plusregion=white, rounded corners = 5pt] (0,0) rectangle (.6,.6);
\draw[thin, dotted, rounded corners = 5pt] (0,0) rectangle (.6,.6);
}
&=
B
\end{aligned}
\end{equation}
A 2-morphism in $\fY$ in the image of $F$ or associated to a 2-functor $\fZ\to \fY$, 2-transformation, or 2-modification is depicted by overlaying these textured diagrams with diagrams internal to $\fZ$.
For example, 
$$
\left(
\tikzmath{
\filldraw[primedregion=white, rounded corners = 5pt] (0,0) rectangle (.6,.6);
\draw[thin, dotted, rounded corners = 5pt] (0,0) rectangle (.6,.6);
\node at (.3,-.2) {$\scriptstyle A$};
\node at (.3,.8) {$\scriptstyle \phantom A$};}
\right)
\left(
\tikzmath{
\begin{scope}
\clip[rounded corners = 5pt] (-.7,-.7) rectangle (.7,.7);
\filldraw[\arColor] (0,.7) rectangle (-.7,-.7);
\filldraw[\brColor] (0,-.7) rectangle (.7,.7);
\end{scope}
\draw[red,thick] (0,-.7) node[below]{$\scriptstyle X$}-- (0,-.3);
\draw[orange,thick] (0,.3) -- (0,.7) node[above]{$\scriptstyle Y$};
\draw[thin, dotted, rounded corners = 5pt] (-.7,-.7) rectangle (.7,.7);
\roundNbox{fill=white}{(0,0)}{.3}{0}{0}{$f$}
\node at (-.5,0) {$\scriptstyle a$};
\node at (.5,0) {$\scriptstyle b$};
}
\right)
=:
\tikzmath{
\begin{scope}
\clip[rounded corners = 5pt] (-.7,-.7) rectangle (.7,.7);
\filldraw[primedregion=\arColor] (0,.7) rectangle (-.7,-.7);
\filldraw[primedregion=\brColor] (0,-.7) rectangle (.7,.7);
\end{scope}
\draw[red,thick] (0,-.7) node[below]{$\scriptstyle A(X)$}-- (0,-.3);
\draw[orange,thick] (0,.3) -- (0,.7) node[above]{$\scriptstyle A(Y)$};
\draw[thin, dotted, rounded corners = 5pt] (-.7,-.7) rectangle (.7,.7);
\roundNbox{fill=white}{(0,0)}{.3}{.2}{.2}{$A(f)$}
}
=
A(f): {}_{A(a)}A(X)_{A(b)}\Rightarrow {}_{A(a)}A(Y)_{A(b)}
\qquad\qquad
\begin{aligned}
\tikzmath{
\filldraw[primedregion=\arColor, rounded corners = 5pt] (0,0) rectangle (.6,.6);
}
&=
A(a)
\\
\tikzmath{
\filldraw[primedregion=\brColor, rounded corners = 5pt] (0,0) rectangle (.6,.6);
}
&=
A(b).
\end{aligned}
$$
and
$$
\left(
\tikzmath{
\begin{scope}
\clip[rounded corners = 5pt] (-.7,-.7) rectangle (.7,.7);
\filldraw[primedregion=white] (0,.7) rectangle (-.7,-.7);
\filldraw[plusregion=white] (0,-.7) rectangle (.7,.7);
\end{scope}
\draw[\rhoColor,thick] (0,-.7) node[below]{$\scriptstyle \rho$}-- (0,-.3);
\draw[\sigmaColor,thick] (0,.3) -- (0,.7) node[above]{$\scriptstyle \sigma$};
\draw[thin, dotted, rounded corners = 5pt] (-.7,-.7) rectangle (.7,.7);
\roundNbox{fill=white}{(0,0)}{.3}{0}{0}{$m$}
}
\right)
\left(
\tikzmath{
\filldraw[\arColor, rounded corners = 5pt] (0,0) rectangle (.6,.6);
\node at (.3,-.2) {$\scriptstyle a$};
\node at (.3,.8) {$\scriptstyle \phantom a$};
}
\right)
=:
\tikzmath{
\begin{scope}
\clip[rounded corners = 5pt] (-.7,-.7) rectangle (.7,.7);
\filldraw[primedregion=\arColor] (0,.7) rectangle (-.7,-.7);
\filldraw[plusregion=\arColor] (0,-.7) rectangle (.7,.7);
\end{scope}
\draw[\rhoColor,thick] (0,-.7) node[below]{$\scriptstyle \rho_a$}-- (0,-.3);
\draw[\sigmaColor,thick] (0,.3) -- (0,.7) node[above]{$\scriptstyle \sigma_a$};
\roundNbox{fill=white}{(0,0)}{.3}{0}{0}{$m_a$}
}
\qquad\qquad\qquad
\begin{aligned}
\tikzmath{
\filldraw[primedregion=\arColor, rounded corners = 5pt] (0,0) rectangle (.6,.6);
}
&=
A(a)
\\
\tikzmath{
\filldraw[plusregion=\arColor, rounded corners = 5pt] (0,0) rectangle (.6,.6);
}
&=
B(a).
\end{aligned}
$$
One should think of this as identifying $\fZ\cong \Fun^\dag(*\to \fZ)$, where $*$ is the category with one object $*$, one 1-morphism $1_*$, and 2-morphisms $\End(1_*)=\bbC$,
and identifying $A: \fZ\to \fY$ with $A\circ -: \Fun^\dag(*\to \fZ)\to \Fun^\dag(*\to \fY)$.
We refer the reader to \cite[\S2.4]{MR4369356} for further details.

We now extend the results on dominance and truncation from \cite[\S4.6]{MR4369356} to the setting of pre-3-Hilbert spaces.
\begin{defn}
\label{defn:BoxplusDominance}
Suppose $\fX,\fZ$ are pre-3-Hilbert spaces where $\fZ$ admits Hilbert direct sums.
A $\dag$ 2-functor $G:\fX\to \fZ$ is called
\begin{itemize}
\item 
$\boxplus$-0-dominant 
if for all $z\in\fZ$, there are
$x_1,\dots, x_n\in\fX$ and a coisometry 
$\bigboxplus G(x_i) \to z$.
\item 
$\boxplus$-dominant if it is both $\boxplus$-0-dominant and locally dominant, i.e., every hom $\dag$-functor $G_{x\to x'}$ is dominant (further unpacked, every 1-morphism ${}_{G(x)}Z_{G(x')}$ is unitarily isomorphic to an orthogonal direct summand of some ${}_{G(x)}G(X)_{G(x')}$).
\end{itemize}
\end{defn}

\begin{rem}
In the algebraic linear 2-category setting, when $\fX,\fZ$ are locally Cauchy complete linear 2-categories such that $\fZ$ admits direct sums of objects, $\boxplus$-0-dominance means that for every $z\in \fZ$, there are $x_1,\dots, x_n\in\fX$ and a condensation $\boxplus G(x_i)\condense z$.
When $\fX,\fZ$ are unitary 2-categories such that $\fZ$ admits orthogonal direct sums, we modify the algebraic definition above by using orthogonal direct sums and $\dag$-condensations.
\end{rem}

\begin{lem}
\label{lem:DominanceAndFaithfulness}
Suppose $\fX,\fZ$ are pre-3-Hilbert spaces where $\fZ$ admits Hilbert direct sums.
If a $\dag$-2-functor $G:\fX\to \fZ$ is $\boxplus$-0-dominant, then 
$$
-\circ G : \Fun^\dag(\fZ\to \fY) \to \Fun^\dag(\fX\to \fY)
$$
is faithful on 2-morphisms.
If $G$ is $\boxplus$-dominant, then $-\circ G$ is fully faithful on 2-morphisms.
\end{lem}
\begin{proof}
To prove the first result, we modify the proof of \cite[Prop.~4.10]{MR4369356} using $\boxplus$-0-dominance instead of $0$-dominance.
The second result is then a straightforward adaptation of \cite[Prop.~4.11]{MR4369356} mutatis mutandis which we leave to the reader.

Suppose $A,B\in \Fun^\dag(\fZ\to \fY)$, $\rho,\sigma: A\Rightarrow B$, and $m: \rho\Rrightarrow \sigma$, all represented graphically as in \eqref{eq:TexturedConventions}.
We prove that if $m\circ G = 0$, then $m=0$.
As $G$ is $\boxplus$-0-dominant,
for each $z\in \fZ$, there are $x_1,\dots, x_n\in \fX$  and a coisometry $Z: \boxplus G(x_j) \to z$, which we represent graphically by
$$
\begin{aligned}
\tikzmath{
\fill[fill=blue!20, rounded corners=5] (-.3,-.3) rectangle (.3,.3);
\node at (0,0) {$\scriptstyle j$};
}&=G(x_j)
\\
\tikzmath{\fill[fill=green!20, rounded corners=5] (-.3,-.3) rectangle (.3,.3);}&=\bigboxplus G(x_j)
\end{aligned}
\qquad\qquad
\begin{aligned}
\tikzmath{
\begin{scope}
\clip[rounded corners = 5pt] (0,0) rectangle (.6,.6);
\fill[fill=blue!20] (0,0) rectangle (.3,.6);
\fill[fill=green!20] (.3,0) rectangle (.6,.6);
\end{scope}
\draw[blue,thick] (.3,0) -- (.3,.6);
\node at (.15,.3) {$\scriptstyle j$};
}
&=
{}_{G(x_j)}(I_j)_{\boxplus G(x_j)}
\\
\tikzmath{
\begin{scope}
\clip[rounded corners = 5pt] (0,0) rectangle (.6,.6);
\fill[fill=green!20] (0,0) rectangle (.3,.6);
\fill[fill=\AColor] (.3,0) rectangle (.6,.6);
\end{scope}
\draw[DarkGreen,thick] (.3,0) -- (.3,.6);
}
&=
{}_{\boxplus G(x_j)}Z_z
\end{aligned}
\qquad\qquad
\tikzmath{
\fill[fill=\AColor, rounded corners=5] (-.3,-.3) rectangle (.3,.3);
}=z
.
$$
We now compute that $m_{\boxplus G(x_j)}=0$ as it is determined by the $m_{G(x_j)}=0$ by
\begin{align*}
\tikzmath{
\begin{scope}
\clip[rounded corners=5pt] (-1,-.7) rectangle (1,.7);
\fill[primedregion=green!20] (-1,-.7) rectangle (0,.7);
\fill[plusregion=green!20] (1,-.7) rectangle (0,.7);
\end{scope}
\roundNbox{fill=white}{(0,0)}{.3}{.5}{.5}{$m_{\boxplus G(x_j)}$}
\draw[\rhoColor,thick] (0,-.7) node[below]{$\scriptstyle \rho_{\boxplus G(x_j)}$}-- (0,-.3);
\draw[\sigmaColor,thick] (0,.3) -- (0,.7) node[above]{$\scriptstyle \sigma_{\boxplus G(x_j)}$};
}
&=
\sum_j
\tikzmath{
\begin{scope}
\clip[rounded corners=5pt] (-.6,-.7) rectangle (2.3,.7);
\fill[primedregion=green!20] (-.6,-.7) rectangle (1.3,.7);
\fill[plusregion=green!20] (2.3,-.7) rectangle (1.3,.7);
\end{scope}
\draw[preaction={primedregion=blue!20},blue,thick] (0,0) circle (.3cm);
\roundNbox{fill=white}{(1.3,0)}{.3}{.5}{.5}{$m_{\boxplus G(x_j)}$}
\draw[\rhoColor,thick] (1.3,-.7) node[below]{$\scriptstyle \rho_{\boxplus G(x_j)}$}-- (1.3,-.3);
\draw[\sigmaColor,thick] (1.3,.3) -- (1.3,.7) node[above]{$\scriptstyle \sigma_{\boxplus G(x_j)}$};
\node at (0,0) {$\scriptstyle j$};
}
=
\sum_j
\tikzmath{
\begin{scope}
\clip[rounded corners=5pt] (-1,-1.1) rectangle (1,1.1);
\fill[primedregion=green!20] (-1,-1.1) rectangle (0,1.1);
\fill[plusregion=green!20] (0,-1.1) rectangle (1,1.1);
\end{scope}
\draw[preaction={primedregion=blue!20},blue,thick] (0,-.2) arc(90:270:.3cm);
\draw[preaction={plusregion=blue!20},blue,thick] (0,-.2) arc(90:-90:.3cm);
\roundNbox{fill=white}{(0,.5)}{.3}{.5}{.5}{$m_{\boxplus G(x_j)}$}
\draw[\rhoColor,thick] (0,-1.1) node[below]{$\scriptstyle \rho_{\boxplus G(x_j)}$}-- (0,.2);
\draw[\sigmaColor,thick] (0,.8) -- (0,1.1) node[above]{$\scriptstyle \sigma_{\boxplus G(x_j)}$};
\node at (-.15,-.5) {$\scriptstyle j$};
}
=
\sum_j
\tikzmath{
\begin{scope}
\clip[rounded corners=5pt] (-1.5,-1.5) rectangle (1.5,1.5);
\fill[primedregion=green!20] (-1.5,-1.5) rectangle (0,1.5);
\fill[plusregion=green!20] (0,-1.5) rectangle (1.5,1.5);
\end{scope}
\draw[preaction={primedregion=blue!20},blue,thick] (0,1) arc(90:270:1cm);
\draw[preaction={plusregion=blue!20},blue,thick] (0,1) arc(90:-90:1cm);
\roundNbox{fill=white}{(0,0)}{.3}{.4}{.4}{$m_{G(x_j)}$}
\draw[\rhoColor,thick] (0,-1.5) node[below]{$\scriptstyle \rho_{\boxplus G(x_j)}$}-- (0,-.3);
\draw[\sigmaColor,thick] (0,.3) -- (0,1.5) node[above]{$\scriptstyle \sigma_{\boxplus G(x_j)}$};
\node at (-.3,-.6) {$\scriptstyle j$};
}
=
0.
\end{align*}
By a similar calculation as above replacing
the green $\boxplus G(x_j)$ region with the gray $z$ region, 
the sum over the blue $G(x_j)$ bubbles with a single green $\boxplus G(x_j)$ bubble,
and the isometries $I_j$ with the coisometry $Z$,
one sees that
\[
\tikzmath{
\begin{scope}
\clip[rounded corners=5pt] (-.7,-.7) rectangle (.7,.7);
\fill[primedregion=\AColor] (-.7,-.7) rectangle (0,.7);
\fill[plusregion=\AColor] (.7,-.7) rectangle (0,.7);
\end{scope}
\roundNbox{fill=white}{(0,0)}{.3}{0}{0}{$m_{z}$}
\draw[\rhoColor,thick] (0,-.7) node[below]{$\scriptstyle \rho_{z}$}-- (0,-.3);
\draw[\sigmaColor,thick] (0,.3) -- (0,.7) node[above]{$\scriptstyle \sigma_{z}$};
}
=
\tikzmath{
\begin{scope}
\clip[rounded corners=5pt] (-.2,-.7) rectangle (1.9,.7);
\fill[primedregion=\AColor] (-.2,-.7) rectangle (1.3,.7);
\fill[plusregion=\AColor] (1.9,-.7) rectangle (1.3,.7);
\end{scope}
\draw[thick, DarkGreen, preaction={primedregion=green!20}] (.4,0) circle (.3cm);
\roundNbox{fill=white}{(1.3,0)}{.3}{0}{0}{$m_{z}$}
\draw[\rhoColor,thick] (1.3,-.7) node[below]{$\scriptstyle \rho_{z}$}-- (1.3,-.3);
\draw[\sigmaColor,thick] (1.3,.3) -- (1.3,.7) node[above]{$\scriptstyle \sigma_{z}$};
}
=
\tikzmath{
\begin{scope}
\clip[rounded corners=5pt] (-.6,-1.1) rectangle (.6,1.1);
\fill[primedregion=\AColor] (-.6,-1.1) rectangle (0,1.1);
\fill[plusregion=\AColor] (0,-1.1) rectangle (.6,1.1);
\end{scope}
\draw[thick, DarkGreen, preaction={primedregion=green!20}] (0,-.2) arc(90:270:.3cm);
\draw[thick, DarkGreen, preaction={plusregion=green!20}] (0,-.2) arc(90:-90:.3cm);
\roundNbox{fill=white}{(0,.5)}{.3}{0}{0}{$m_{z}$}
\draw[\rhoColor,thick] (0,-1.1) node[below]{$\scriptstyle \rho_{z}$}-- (0,.2);
\draw[\sigmaColor,thick] (0,.8) -- (0,1.1) node[above]{$\scriptstyle \sigma_{z}$};
}
=
\tikzmath{
\begin{scope}
\clip[rounded corners=5pt] (-1.5,-1.5) rectangle (1.5,1.5);
\fill[primedregion=\AColor] (-1.5,-1.5) rectangle (0,1.5);
\fill[plusregion=\AColor] (0,-1.5) rectangle (1.5,1.5);
\end{scope}
\draw[thick, DarkGreen, preaction={primedregion=green!20}] (0,1) arc(90:270:1cm);
\draw[thick, DarkGreen, preaction={plusregion=green!20}] (0,1) arc(90:-90:1cm);
\roundNbox{fill=white}{(0,0)}{.3}{.5}{.5}{$m_{\boxplus G(x_j)}$}
\draw[\rhoColor,thick] (0,-1.5) node[below]{$\scriptstyle \rho_{z}$}-- (0,-.3);
\draw[\sigmaColor,thick] (0,.3) -- (0,1.5) node[above]{$\scriptstyle \sigma_{z}$};
}
=
0.
\qedhere\]
\end{proof}

\begin{proof}[Proof of Proposition \ref{prop:HilbertDirectSumUniversal}]
The proof is similar to the proof of the universal property of Q-system completion from \cite[\S4]{MR4369356}.

\item[\underline{Step 1:}] 
Since every object of $\Hilb_\boxplus(\fX)$ is a Hilbert direct sum of objects in the image of $\iota_\fX$, $\iota_\fX$ is $\boxplus$-0-dominant.
By Lemma \ref{lem:DominanceAndFaithfulness}, $-\circ \iota$ is faithful on 2-morphisms.

\item[\underline{Step 2:}] 
Observe that $\iota_\fX$ is a local $\dag$-equivalence.
Again by Lemma \ref{lem:DominanceAndFaithfulness},  $-\circ \iota_\fX$ is fully faithful on 2-morphisms.

\item[\underline{Step 3:}] 
We now show
$-\circ \iota_\fX$ is a $\dag$-equivalence on hom-categories.
It remains to prove essential surjectivity on hom-categories.
Suppose we have a 2-transformation $\rho:A\circ \iota_\fX \Rightarrow B\circ \iota_\fX$ for two functors $A,B: \Hilb_\boxplus(\fX)\to \fY$.
We construct a 2-transformation $\sigma: A\Rightarrow B$ such that $\sigma\circ \iota_\fX \cong \rho$.
For each $x_1,\dots, x_n\in\fX$,
$(x_1,\dots, x_n)\in \Hilb_\boxplus(\fX)$ is the Hilbert direct sum $\bigboxplus \iota_\fX(x_j)$.
Hence we have a canonical isometric equivalence $A(x_1,\dots, x_n)\cong \bigboxplus A(x_j)$
(where we abbreviate $A(x_j):=(A\circ \iota_\fX)(x_j)$) whose canonical isometries are denoted by $I_j^{A(x)}: A(x_j)\to A(x_1,\dots, x_n)$, and similarly for $B$.
We define $\sigma_{(x_1,\dots, x_n)}:=\bigoplus {I^{A(x)}_j}^\vee\otimes \rho_{x_j}\otimes I^{B(x)}_j : \bigboxplus A(x_j)\to \bigboxplus B(x_j)$, which we represent graphically by
$$
\sigma_{(x_1,\dots, x_n)}
=
\bigoplus
\tikzmath{
\begin{scope}
\clip[rounded corners = 5pt] (-1.7,-.5) rectangle (1.5,.5);
\fill[primedregion=blue!40] (-.5,-.5) rectangle (-1.7,.5);
\fill[primedregion=blue!20] (-.5,-.5) rectangle (0,.5);
\fill[plusregion=blue!20] (0,-.5) rectangle (.5,.5);
\fill[plusregion=blue!40] (.5,-.5) rectangle (1.5,.5);
\end{scope}
\draw[thick, blue] (-.5,-.5) --node[left]{$\scriptstyle {I_j^{A(x)}}^\vee$} (-.5,.5);
\draw[thick, black, snake] (0,-.5) --node[left]{$\scriptstyle j$} node[right]{$\scriptstyle j$} (0,.5);
\draw[thick, blue] (.5,-.5) --node[right]{$\scriptstyle I_j^{B(x)}$} (.5,.5);
}
\qquad\qquad
\begin{aligned}
\tikzmath{
\filldraw[primedregion=white, rounded corners = 5pt] (0,0) rectangle (.6,.6);
\draw[thin, dotted, rounded corners = 5pt] (0,0) rectangle (.6,.6);
}
&=
A
\\
\tikzmath{
\filldraw[plusregion=white, rounded corners = 5pt] (0,0) rectangle (.6,.6);
\draw[thin, dotted, rounded corners = 5pt] (0,0) rectangle (.6,.6);
}
&=
B
\\
\tikzmath{
\begin{scope}
\clip[rounded corners = 5pt] (0,0) rectangle (.6,.6);
\fill[primedregion=white] (0,0) rectangle (.3,.6);
\fill[plusregion=white] (.3,0) rectangle (.6,.6);
\end{scope}
\draw[\sigmaColor,thick] (.3,0) -- (.3,.6);
\draw[thin, dotted, rounded corners = 5pt] (0,0) rectangle (.6,.6);
}
&=
\rho: A\circ\iota\Rightarrow B\circ\iota
\end{aligned}
\qquad\qquad
\begin{aligned}
\tikzmath{
\fill[blue!40, rounded corners = 5pt] (0,0) rectangle (.6,.6);
}
&=
(x_1,\dots, x_n)
\\
\tikzmath{
\fill[blue!20, rounded corners = 5pt] (0,0) rectangle (.6,.6);
\node at (.3,.3) {$\scriptstyle j$};
}
&=
x_j
\end{aligned}
$$
For a 1-morphism $X=(X_{ij}): (x_1,\dots, x_n) \to (y_1,\dots, y_m)$,
keeping the graphical notation above, we define 
$$
\sigma_X:= \sum_{i,j}
\tikzmath{
\begin{scope}
\clip[rounded corners=5pt] (-2.2,-1) rectangle (2,1);
\fill[primedregion=blue!40] (-1,-1) rectangle (-2.2,1);
\fill[primedregion=blue!20] (-1,-1) rectangle (-.6,1);
\fill[plusregion=blue!20] (-.6,1) -- (-.3,1) arc(-180:0:.3cm) -- (.6,1) -- (.6,.3) --  (-.6,.3);
\fill[plusregion=red!40] (1,-1) rectangle (2,1);
\fill[plusregion=red!20] (1,-1) rectangle (.6,1);
\fill[primedregion=red!20] (-.6,-1) -- (-.3,-1) arc(180:0:.3cm) -- (.6,-1) -- (.6,-.3) --  (-.6,-.3);
\draw[thick, red, preaction={primedregion=red!40}] (-.3,-1) arc(180:0:.3cm);
\draw[thick, blue, preaction={plusregion=blue!40}] (-.3,1) arc(-180:0:.3cm);
\end{scope}
\draw[thick, blue] (-1,-1) --node[left]{$\scriptstyle {I_i^{A(x)}}^\vee$} (-1,1);
\draw[thick] (-.6,-.3) -- (-.6,-1) node[below]{$\scriptstyle A(X_{ij})$};
\draw[thick, black, snake] (-.6,.3) -- (-.6,1) node[above]{$\scriptstyle \rho_{x_i}$};
\draw[thick, black, snake] (.6,-.3) -- (.6,-1) node[below]{$\scriptstyle \rho_{x_j}$};
\draw[thick] (.6,.3) -- (.6,1) node[above]{$\scriptstyle B(X_{ij})$};
\draw[thick, red] (1,-1) --node[right]{$\scriptstyle I_j^{B(y)}$} (1,1);
\roundNbox{fill=white}{(0,0)}{.3}{.5}{.5}{$\rho_{X_{ij}}$}
}
\qquad\qquad\qquad\qquad
\begin{aligned}
\tikzmath{
\fill[red!40, rounded corners = 5pt] (0,0) rectangle (.6,.6);
}
&=
(y_1,\dots, y_m)
\\
\tikzmath{
\fill[red!20, rounded corners = 5pt] (0,0) rectangle (.6,.6);
\node at (.3,.3) {$\scriptstyle j$};
}
&=
y_j
\end{aligned}
$$
again abbreviating $A(X_{ij}):=(A\circ\iota)(X_{ij})$, etc.
One readily checks $\sigma: A\Rightarrow B$ is a 2-transformation such that $\sigma\circ\iota_\fX \cong\rho$.

\item[\underline{Step 4:}] 
$-\circ \iota_\fX$ is essentially surjective.
Given $A\in \Fun^\dag(\fX\to \fY)$, using that $\Hilb_\boxplus$ is a 3-functor, one easily constructs a 
lift $A'\in \Fun^\dag(\Hilb_\boxplus(\fX)\to \fY)$ 
with $A'\circ \iota_\fX \cong A$
as in \cite[Proof of Thm.~1.2]{MR4369356} by
$$
\begin{tikzcd}
\Hilb_\boxplus(\fX)
\arrow[rr, "\Hilb_\boxplus(A)"]
&&
\Hilb_\boxplus(\fY)
\arrow[dr, "\iota_\fY^{-1}"]
\\
\fX
\arrow[r,swap, "A"]
\arrow[u, "\iota_\fX"]
&
\fY
\arrow[rr,swap,"\id_\fY"]
\arrow[ur, "\iota_\fY"]
\arrow[ul,Rightarrow,shorten <= 1em, shorten >= 1em,"\cong"']
&
\arrow[u,Rightarrow,shorten <= .5em, 
shorten >= .5em,
"\cong"]
&
\fY.
\end{tikzcd}
$$
One can also write down a candidate $A'$ directly.

\item[\underline{Step 5:}] 
$-\circ \iota_\fX$ restricted to $\Fun^{\dag,\vee}(\Hilb_\boxplus(\fX) \to \fY)$ is essentially surjective onto $\Fun^{\dag,\vee}(\fX \to \fY).$

Notice that $\iota_\fX$ is UAF-preserving. Since the composite of UAF-preserving functors is again UAF-preserving, the restriction of $\iota^*_\fX$ is well-defined. 

Now, if $A \in \Fun^{\dag,\vee}(\fX \to \fY)$, it is clear that $\Hilb_\boxplus(A)$ is UAF-preserving. Since $\iota_{\fY}^{-1}$ is also UAF-preserving, we thus have that the lift $\iota_{\fY}^{-1} \circ \Hilb_{\boxplus}(\cA)$ of $A$ is UAF-preserving.

\item[\underline{Step 6:}] 
$-\circ \iota_\fX$ restricted to $\Isom(\HilbPlus(\fX)\to \fY)$ is essentially surjective onto $\Isom(\fX\to \fY)$.

Notice that $\iota_\fX$ is in fact isomtetric. Since the composite of isomtetric functors is again isometric, the restriction of $\iota^*_\fX$ is well-defined. 

Now if $A \in \Isom(\fX \to \fY)$, it is clear that $\Hilb_\boxplus(A)$ is isometric.
Since $\iota_\fY^{-1}$ can be taken to be the isometric `evaluation' functor $\Hilb_\boxplus(\fY)\to \fY$ given by $(y_1,\dots, y_n)\mapsto \boxplus y_j$ as $\fY$ is Hilbert direct sum complete, we conclude that the lift $\iota_{\fY}^{-1} \circ \Hilb_{\boxplus}(\cA)$ of $A$ is isometric.
\end{proof}

\begin{rem}
We expect that the universal property in Proposition \ref{prop:HilbertDirectSumUniversal} may be stated in terms of an isometric equivalence of pre-3-Hilbert spaces in certain cases.
In order to do so, one would need an organic pre-3-Hilbert space structure on $\Fun^\dag(\fX\to \fY)$.
We refer the reader to \S\ref{sec:OpenQuestions} below for further discussion.
\end{rem}

\subsection{\texorpdfstring{$\rmH^*$}{H*}-monads and \texorpdfstring{$\rmH^*$}{H*}-monad completeness}
\label{sec:H*Monad}

An \emph{$\rmH^*$-monad} in a pre-3-Hilbert space $\fX$ is an object $a\in\fX$ together with an $\rmH^*$-algebra $({}_aA_a,\mu,\iota)\in \End_\fX(a)$. Notice when $\fX$ is the delooping $\rmB \cC$ for an H*-multifusion category $\cC$, the H*-monads in $\fX$ correspond to the H*-algebras in $\cC$.

\begin{ex}\label{ex:summandsof1}
Given an object $a$ in a pre-3-Hilbert space $\fX$, any summand of $1_a$ has a canonical structure of a standard Q-system.
\end{ex}

\begin{ex}
\label{ex:CondensationH*algebra}
Similar to Example \ref{ex:separabledualyieldsH*algebra}, for any 1-morphism ${}_aX_b$,
${}_aX\otimes_b X^\vee_a$ has a canonical structure of an $\rmH^*$-monad
with multiplication and unit given by
$$
\mu:=\id_X\otimes \ev_X\otimes \id_{X^\vee}
\qquad\text{and}\qquad 
\iota:= \coev_X.
$$
When ${}_aX_b$ is a coisometry, ${}_aX\otimes_b X^\vee_a$ is a standard Q-system as $\ev_X\ev^\dag_X=\id_{1_b}$.
\end{ex}

\begin{defn}
A \emph{splitting}
of an $\rmH^*$-monad $A\in\fX(a\to a)$ in a pre-3-Hilbert space $\fX$ is 
an object $b\in\fX$ 
and a 1-morphism ${}_aX_b$ such that $\ev_X\ev_X^\dag\in\End(1_b)$ is invertible
together with a unitary algebra isomorphism $u: {}_aA_a\to {}_aX\otimes_b X^\vee_a$ .

We say a pre-3-Hilbert space is \emph{idempotent complete} if every $\rmH^*$-algebra splits.
\end{defn}

\begin{rem}
If $({}_aX_b,u)$ splits the $\rmH^*$-monad ${}_aA_a$,
then $X$ is canonically a unital left $A$-module and $X^\vee$ is canonically a unital right $A$-module where
\begin{equation}
\label{eq:XisA-Abimod}
\tikzmath{
\begin{scope}
\clip[rounded corners=5pt] (-.7,-.5) rectangle (.5,.5);
\fill[fill=\arColor] (-.7,-.5) rectangle (0,.5);
\fill[fill=\brColor] (0,-.5) rectangle (.5,.5);
\end{scope}
\draw[thick, \AsColor] (-.5,-.5) -- (0,0);
\draw[thick] (0,-.5) --node[right]{$\scriptstyle X$} (0,.5);
}
:=
\tikzmath{
\begin{scope}
\clip[rounded corners=5pt] (-.7,-.7) rectangle (1.1,.7);
\fill[fill=\arColor] (-.7,-.7) rectangle (.6,.7);
\fill[fill=\brColor] (1.1,.7) -- (-.2,.7) -- (-.2,.3) -- (.2,.3) arc(180:0:.2cm) -- (.6,-.7) -- (1.1,-.7);
\end{scope}
\draw[thick, \AsColor] (0,-.7) -- (0,-.3);
\draw[thick] (-.2,.3) --node[left]{$\scriptstyle X$} (-.2,.7);
\draw[thick] (.2,.3) arc(180:0:.2cm) --node[right]{$\scriptstyle X$} (.6,-.7);
\roundNbox{fill=white}{(0,0)}{.3}{.1}{.1}{$u$}
}
\qquad\text{and}\qquad
\tikzmath{
\begin{scope}
\clip[rounded corners=5pt] (-.7,-.5) rectangle (.7,.5);
\fill[\brColor] (-.7,-.5) rectangle (0,.5);
\fill[\arColor] (0,-.5) rectangle (.7,.5);
\end{scope}
\draw[thick, \AsColor] (.5,-.5) -- (0,0);
\draw[thick] (0,-.5) --node[left]{$\scriptstyle X^\vee$} (0,.5);
}
:=
\tikzmath{
\begin{scope}
\clip[rounded corners=5pt] (-1.3,-.7) rectangle (.8,.7);
\fill[\brColor] (-1.3,-.7) rectangle (.8,.7);
\fill[\arColor] (-.2,.3) arc(0:180:.2cm) -- (-.6,-.7) -- (.8,-.7) -- (.8,.7) -- (.2,.7) -- (.2,.3);
\end{scope}
\draw[thick, \AsColor] (0,-.7) -- (0,-.3);
\draw[thick] (.2,.3) --node[right]{$\scriptstyle X^\vee$} (.2,.7);
\draw[thick] (-.2,.3) arc(0:180:.2cm) --node[left]{$\scriptstyle X^\vee$} (-.6,-.7);
\roundNbox{fill=white}{(0,0)}{.3}{.1}{.1}{$u$}
}
\,.
\end{equation}
Moreover, under these actions, $u: {}_aA_a\to {}_aX\otimes_b X^\vee_a$ is an $A$-$A$ bimodule map.
\end{rem}

\begin{lem}
\label{lem:DecomposeIntoSimplesInA3Hilb}
In an idempotent complete pre-3-Hilbert space, every object is isometrically equivalent to a Hilbert direct sum of simple objects.
\end{lem}
\begin{proof}
Let $a$ be an object in an idempotent complete pre-3-Hilbert space $\fX$.
We can write $1_a = \bigoplus 1_i$ as an orthogonal direct sum of simple objects in the $\rmH^*$-multifusion category $\End(a)$.
Each $1_i$ has the structure of a standard Q-system with multiplication $\mu_i$ and unit $\iota_i$  given by the canonical isomorphism and isometry
$$
\mu_i:=
1_i\otimes 1_i \xrightarrow{\cong} 1_i
\qquad\qquad
\iota_i:=
1_\fX \xrightarrow{\coev_{1_i}} 1_i\otimes 1_i \xrightarrow{\cong} 1_i.
$$
Since $\fX$ is idempotent complete, we have objects $a_i$ and 1-morphisms ${}_{a}(X_i)_{a_i}$ such that $\ev_{X_i}\ev_{X_i}^\dag$ is invertible, together with a unitary unital algebra maps $u_i: 1_i\to {}_aX_i\otimes_{a_i}(X_i)_a^\vee$.
Unitality of $u_i$ implies that
$$
\tikzmath{
\begin{scope}
\clip[rounded corners=5pt] (-.9,-1.1) rectangle (.9,.7);
\fill[green!20] (-.9,-1.1) rectangle (.9,.7);
\fill[blue!20] (-.3,.3) rectangle (.3,.7);
\end{scope}
\draw[dotted] (0,-1.1) --node[right]{$\scriptstyle 1_\fX$} (0,-.7) --node[right]{$\scriptstyle 1_i$} (0,-.3);
\filldraw (0,-.7) node[left]{$\scriptstyle \iota_i$} circle (.05cm);
\draw[thick] (.3,.3) --node[right]{$\scriptstyle X_i^\vee$} (.3,.7);
\draw[thick] (-.3,.3) --node[left]{$\scriptstyle X_i$} (-.3,.7);
\roundNbox{fill=white}{(0,0)}{.3}{.2}{.3}{$u_i$}
\node at (0,.5) {$\scriptstyle i$};
}
=
\tikzmath{
\begin{scope}
\clip[rounded corners=5pt] (-.9,-.3) rectangle (.9,.7);
\fill[green!20] (-.9,-.3) rectangle (.9,.7);
\fill[blue!20] (-.3,.7) -- (-.3,.3) arc(-180:0:.3cm) -- (.3,.7);
\end{scope}
\draw[thick] (-.3,.7) --node[left]{$\scriptstyle X_i$} (-.3,.3) arc(-180:0:.3cm) --node[right]{$\scriptstyle X_i^\vee$} (.3,.7);
\node at (0,.5) {$\scriptstyle i$};
}
\qquad\qquad
\tikzmath{
\fill[fill=blue!20, rounded corners=5] (-.3,-.3) rectangle (.3,.3);
\node at (0,0) {$\scriptstyle i$};
} = a_i
\qquad\qquad
\tikzmath{\fill[fill=green!20, rounded corners=5] (-.3,-.3) rectangle (.3,.3);} = a,
$$
so $\coev_{X_i}$ is manifestly a coisometry satisfying $\coev_{X_i}^\dag\coev_{X_i}=p_i$, the orthogonal projection onto $1_i$.
It immediately follows that $\sum \coev_{X_i}^\dag\coev_{X_i} = \sum p_i =\id_{1_a}$.
By \eqref{eq:EquivalentRecabling}, we also get that $\ev_{X_i}$ is also an isometry.
Since ${}_aX_i\otimes_{a_i}(X_i)_a^\vee$ is a standard Q-system, we see that $\ev_{X_i}\ev_{X_i}^\dag=\id_{a_i}$.
Thus each $X_i$ is a coisometry and $\sum \coev^\dag_{X_i}\coev_{X_i}=\id_{1_a}$, so the $X_i$ witness $a$ as the Hilbert direct sum $\bigboxplus a_i$.
\end{proof}

\begin{warn}
When $\fX$ is a finite idempotent complete 3-Hilbert space, there are infinitely many isometric equivalence classes of simple objects.
For example, $2\Hilb$ has infinitely many simples $(\Hilb, r\Tr)$ for $r>0$ up to isometric equivalence.
\end{warn}

\begin{construction}[$\rmH^*$-monad completion]
Similar to the $\rmH^*$-algebra completion of an $\rmH^*$-multifusion category from Definition \ref{defn:HStarAlgC}, we may take the $\rmH^*$-\emph{monad completion} $\HstarAlg(\fX)$ of a pre-3-Hilbert space $\fX$.
Again, it is closely related to the Q-system completion $\QSys(\fX)$ from \cite{MR4419534}, but we must now correct for bubbles.
As before, 
\begin{itemize}
\item 
the objects are
$\rmH^*$-monads
$(a,{}_aA_a)$ in $\fX$,
\item 
1-morphisms ${}_AM_B:A\to B$ are ${}_aA_a-{}_bB_b$ bimodules (these can be interpreted as $A$-$B$ bimodules in the $\rmH^*$-linking algebra $\cL(a,b)$), and
\item 
2-morphisms are intertwiners.
\end{itemize}
Composition of 2-morphisms is composition of intertwiners.
Composition of 1-morphisms is the relative product achieved by splitting the appropriate separability idempotent.
As before, $\HstarAlg(\fX)$ carries a canonical UAF given by choosing $\delta=0$ in \eqref{eq:1ParameterUAFOnH*Alg}.

We now endow $\HstarAlg(\fX)$ with the structure of a pre-3-Hilbert space.
The spherical trace $\Psi^{\HstarAlg(\fX)}$ is defined by
$$
\Psi^{\HstarAlg(\fX)}_A(f: {}_AA_A\Rightarrow {}_AA_A)
:=
\Psi^\fX_a\left(
\tikzmath{
\begin{scope}
\clip[rounded corners=5pt] (-.5,-1) rectangle (.5,.8);
\fill[fill=\arColor] (-.5,-1) rectangle (.5,.8);
\end{scope}
\draw[thick,\AsColor] (0,-.6) -- (0,-.3);
\draw[thick,\AsColor] (0,.3) -- (0,.6);
\roundNbox{fill=white}{(0,0)}{0.3}{0}{0}{$f$};
\filldraw[\AsColor] (0,.6) circle (.05);
\draw[\AsColor, thick] (0,-.7) circle (.1);
\node[\AsColor] at (.25,-.5) {$\scriptstyle -1$};
}
\right).
$$

Indeed, for an $A$-$B$ bimodule map $f: {}_AX_B\Rightarrow {}_AX_B$, where $A$ is an H*-monad on $a \in \fX$ and $B$ is an H*-monad on $b \in \fX$, we graphically denote
\begin{align*}
& 
&&& 
& 
&&& 
\tikzmath{\fill[fill=\ArColor, rounded corners=5] (-.3,-.3) rectangle (.3,.3);}
&= 
A 
&&& 
\tikzmath{\fill[fill=\BrColor, rounded corners=5] (-.3,-.3) rectangle (.3,.3);}
&= 
B  
&&& 
\tikzmath{
\begin{scope}
\clip[rounded corners=5] (-.3,-.3) rectangle (.3,.3);
\fill[fill=\ArColor] (-.3,-.3) rectangle (0,.3);
\fill[fill=\BrColor] (0,-.3) rectangle (.3,.3);
\end{scope}
\draw[thick] (0,-.3) -- (0,.3);} 
&=
X
&&& 
\text{in } \mathsf{H}^*\mathsf{Alg}(\fX),
\\  
\tikzmath{\fill[fill=\arColor, rounded corners=5] (-.3,-.3) rectangle (.3,.3);} 
&= 
a  
&&& 
\tikzmath{\fill[fill=\brColor, rounded corners=5] (-.3,-.3) rectangle (.3,.3);}
&=
b
&&& 
\tikzmath{\fill[fill=\arColor, rounded corners=5] (-.3,-.3) rectangle (.3,.3);
\draw[thick,red] (0,-.3) -- (0,.3);} 
&= 
A
&&& 
\tikzmath{\fill[fill=\brColor, rounded corners=5] (-.3,-.3) rectangle (.3,.3);
\draw[thick,blue] (0,-.3) -- (0,.3);} 
&= 
B 
&&& 
\tikzmath{
\begin{scope}
\clip[rounded corners=5] (-.3,-.3) rectangle (.3,.3);
\fill[fill=\arColor] (-.3,-.3) rectangle (0,.3);
\fill[fill=\brColor] (0,-.3) rectangle (.3,.3);
\end{scope}
\draw[thick] (0,-.3) -- (0,.3);} 
&=
X
&&&
\text{in } \fX;  
\end{align*}
and compute
$$
\Psi^{\HstarAlg(\fX)}_B\left(
\tikzmath{
\fill[rounded corners = 5pt, fill=\BrColor] (-1.2,-.9) rectangle (.6,.9);
\fill[\ArColor] (0,.3) arc (0:180:.3cm) -- (-.6,-.3) arc (-180:0:.3cm);
\draw[thick] (0,.3) node[right,yshift=.2cm]{$\scriptstyle X$} arc (0:180:.3cm) --node[left]{$\scriptstyle X^\vee$} (-.6,-.3) arc (-180:0:.3cm) node[right,yshift=-.2cm]{$\scriptstyle X$};
\roundNbox{fill=white}{(0,0)}{.3}{0}{0}{$f$}
}
\,\right)
=
\Psi_b^\fX\left(
\tikzmath{
\begin{scope}
\clip[rounded corners=5pt] (-.9,-1.8) rectangle (.7,1.8);
\fill[\brColor] (-.9,-2) rectangle (.7,2);
\fill[\arColor] (0,.3) -- (0,.9) arc(0:180:.3cm) -- (-.6,-.9) arc (-180:0:.3cm) -- (0,-.3);
\end{scope}
\draw[thick, black] (0,.3) node[right,yshift=.2cm]{$\scriptstyle X$}-- (0,.9) arc(0:180:.3cm) -- (-.6,-.9) arc (-180:0:.3cm) -- (0,-.3) node[right,yshift=-.2cm]{$\scriptstyle X$};
\draw[thick, \BsColor] (.3,-1.4)  to[out=90,in=-45] node[above]{$\scriptstyle B$}(0,-.9);
\draw[thick, \BsColor] (.3,1.4)  to[out=270,in=45] node[right]{$\scriptstyle B$}(0,.6);
\filldraw[thick,\BsColor] (.3,1.4) circle (.05cm);
\draw[thick,\BsColor] (.3,-1.5) node[above,xshift=.2cm]{$\scriptstyle -1$} circle (.1cm);
\draw[thick, \AsColor] (0,.9) -- (-.3,.6);
\filldraw[thick,\AsColor,fill=\arColor] (-.3,.6) node[above, xshift = -1,yshift = 1]{$\scriptstyle -\frac{1}{2}$} circle (.1cm);
\draw[thick, \AsColor] (0,-.6) -- (-.3,-.9);
\filldraw[thick,\AsColor,fill=\arColor] (-.3,-.9) node[above, xshift = -1,yshift = 1]{$\scriptstyle -\frac{1}{2}$} circle (.1cm);
\roundNbox{fill=white}{(0,0)}{.3}{0}{0}{$f$}
}
\right)
=
\Psi_a^\fX\left(
\tikzmath{
\begin{scope}
\clip[rounded corners=5pt] (-.7,-1.8) rectangle (.9,1.8);
\fill[\arColor] (-.7,-2) rectangle (.9,2);
\fill[\brColor] (0,.3) -- (0,.9) arc(180:0:.3cm) -- (.6,-.9) arc (0:-180:.3cm) -- (0,-.3);
\end{scope}
\draw[thick, black] (0,.3) node[left,yshift=.2cm]{$\scriptstyle X$}-- (0,.9) arc(180:0:.3cm) -- (.6,-.9) arc (0:-180:.3cm) -- (0,-.3) node[left,yshift=-.2cm]{$\scriptstyle X$};
\draw[thick, \AsColor] (-.3,-1.4)  to[out=90,in=-135] node[above]{$\scriptstyle A$}(0,-.9);
\draw[thick, \AsColor] (-.3,1.4)  to[out=270,in=135] node[left]{$\scriptstyle A$}(0,.6);
\filldraw[thick,\AsColor] (-.3,1.4) circle (.05cm);
\draw[thick,\AsColor] (-.3,-1.5) node[above,xshift=.25cm]{$\scriptstyle -1$} circle (.1cm);
\draw[thick, \BsColor] (0,.9) -- (.3,.6);
\filldraw[thick,\BsColor,fill=\brColor] (.3,.6)node[above]{$\scriptstyle -\frac{1}{2}$} circle (.1cm);
\draw[thick, \BsColor] (0,-.6) -- (.3,-.9);
\filldraw[thick,\BsColor,fill=\brColor] (.3,-.9)node[above]{$\scriptstyle -\frac{1}{2}$} circle (.1cm);
\roundNbox{fill=white}{(0,0)}{.3}{0}{0}{$f$}
}
\right)
=
\Psi^{\HstarAlg(\fX)}_A\left(\,
\tikzmath{
\fill[rounded corners = 5pt, \ArColor] (1.2,-.9) rectangle (-.6,.9);
\fill[\BrColor]  (0,.3) arc (180:0:.3cm) -- (.6,-.3) arc (0:-180:.3cm);
\draw[thick] (0,.3) node[left,yshift=.2cm]{$\scriptstyle X$} arc (180:0:.3cm) --node[right]{$\scriptstyle X^\vee$} (.6,-.3) arc (0:-180:.3cm) node[left,yshift=-.2cm]{$\scriptstyle X$};
\roundNbox{fill=white}{(0,0)}{.3}{0}{0}{$f$}
}
\right).
$$
Hence $\HstarAlg(\fX)$ is a pre-3-Hilbert space.

Moreover, there is a canonical isometric embedding $\iota_\fX: \fX\hookrightarrow \HstarAlg(\fX)$ by $a\mapsto 1_a$.
\end{construction}

\begin{rem}
\label{rem:PromoteSplittingToIsometricEquivalence}
If $({}_aX_b,u)$ splits the $\rmH^*$-monad ${}_aA_a$ in $\fX$,
then considering $X$ as an $A$-$1_b$ bimodule in $\HstarAlg(\fX)$,
$\ev_X^{\HstarAlg(\fX)}: {}_{1_b}X^\vee\otimes_{A} X_{1_b} \to 1_b$
is a unitary isomorphism between standard Q-systems.
To see this, we first calculate that

$$
(\ev_X^{\HstarAlg(\fX)})^\dag\ev_X^{\HstarAlg(\fX)}
=
\tikzmath{
\begin{scope}
\clip[rounded corners=5pt] (-.9,-1.2) rectangle (.9,1.2);
\fill[\brColor] (-.9,-1.2) rectangle (.9,1.2);
\fill[\arColor] (-.3,1.2) -- (-.3,.6) arc(-180:0:.3cm) -- (.3,1.2);
\fill[\arColor] (-.3,-1.2) -- (-.3,-.6) arc(180:0:.3cm) -- (.3,-1.2);
\end{scope}
\draw[thick] (-.3,1.2) --node[left]{$\scriptstyle X^\vee$} (-.3,.6) arc(-180:0:.3cm) --node[right]{$\scriptstyle X$} (.3,1.2);
\draw[thick] (-.3,-1.2) --node[left]{$\scriptstyle X^\vee$} (-.3,-.6) arc(180:0:.3cm) --node[right]{$\scriptstyle X$} (.3,-1.2);
\draw[thick, \AsColor] (.3,.9) -- (0,.6);
\filldraw[thick, \AsColor, fill=\arColor] (0,.6) node[above,yshift=.1cm]{$\scriptstyle -\frac{1}{2}$} circle (.1cm);
\draw[thick, \AsColor] (.3,-.9) -- (0,-.6);
\filldraw[thick, \AsColor, fill=\arColor] (0,-.6) node[below, yshift=-.1cm]{$\scriptstyle -\frac{1}{2}$} circle (.1cm);
}
=
\tikzmath{
\begin{scope}
\clip[rounded corners=5pt] (-.9,-1.2) rectangle (1.7,1.2);
\fill[\brColor] (-.9,-1.2) rectangle (1.7,1.2);
\fill[\arColor] (-.3,1.2) -- (-.3,-.6) arc(-180:0:.3cm) -- (.3,1.2);
\fill[\arColor] (.6,-1.2) -- (.6,.6) arc(180:0:.3cm) -- (1.2,-1.2);
\end{scope}
\draw[thick] (-.3,1.2) --node[left]{$\scriptstyle X^\vee$} (-.3,-.6) arc(-180:0:.3cm) -- (.3,1.2);
\draw[thick] (.6,-1.2) -- (.6,.6) arc(180:0:.3cm) --node[right]{$\scriptstyle X$} (1.2,-1.2);
\draw[thick, \AsColor] (.3,.8) -- (0,.5);
\filldraw[thick, \AsColor, fill=\arColor] (0,.5) node[above,yshift=.1cm]{$\scriptstyle -\frac{1}{2}$} circle (.1cm);
\draw[thick, \AsColor] (1.2,-.6) -- (.9,-.9);
\filldraw[thick, \AsColor, fill=\arColor] (.9,-.9) node[above,yshift=.1cm]{$\scriptstyle -\frac{1}{2}$} circle (.1cm);
}
=
\tikzmath{
\begin{scope}
\clip[rounded corners=5pt] (-.9,-1.2) rectangle (1.7,1.2);
\fill[\brColor] (-.9,-1.2) rectangle (1.7,1.2);
\fill[\arColor] (-.3,1.2) -- (-.3,-.8) arc(-180:0:.3cm) -- (.6,-.8) -- (.6,-1.2) -- (1.2,-1.2) -- (1.2,.8) arc (0:180:.3cm) -- (.3,.8) -- (.3,1.2);
\end{scope}
\draw[thick] (-.3,1.2) --node[left]{$\scriptstyle X^\vee$} (-.3,-.8) arc(-180:0:.3cm);
\draw[thick] (.3,.8) -- (.3,1.2);
\draw[thick] (.6,-1.2) -- (.6,-.8);
\draw[thick] (.6,.8) arc(180:0:.3cm) --node[right]{$\scriptstyle X$} (1.2,-1.2);
\draw[thick, \AsColor] (.45,.2) -- (.45,-.2);
\draw[thick, \AsColor] (.3,1) -- (0,.7);
\filldraw[thick, \AsColor, fill=\arColor] (0,.7) node[above,xshift=-.05cm]{$\scriptstyle -\frac{1}{2}$} circle (.1cm);
\draw[thick, \AsColor] (1.2,-.6) -- (.9,-.9);
\filldraw[thick, \AsColor, fill=\arColor] (.9,-.9) node[above,yshift=.15cm,xshift=.1cm]{$\scriptstyle -\frac{1}{2}$} circle (.1cm);
\roundNbox{fill=white}{(.45,.5)}{.3}{0}{0}{$u$}
\roundNbox{fill=white}{(.45,-.5)}{.3}{0}{0}{$u^\dag$}
}
\underset{(\ref{eq:XisA-Abimod})}{=}
\tikzmath{
\begin{scope}
\clip[rounded corners=5pt] (-.9,-.5) rectangle (.9,.5);
\fill[\brColor] (-.9,-1.2) rectangle (.9,1.2);
\fill[\arColor] (-.3,-1.2) rectangle (.3,1.2);
\end{scope}
\draw[thick] (-.3,-.5) --node[left]{$\scriptstyle X^\vee$} (-.3,.5);
\draw[thick,\AsColor] (-.3,0) -- (.3,0);
\draw[thick] (.3,-.5) --node[right]{$\scriptstyle X$} (.3,.5);
}\,,
$$
so $\ev_X^{\HstarAlg(\fX)}$ is a partial isometry.
Since the blue bubble is invertible and positive in $\End({}_AA_A)$, both it and its inverse are bounded below, i.e., there is a $C>0$ such that
$$
\tikzmath{
\begin{scope}
\clip[rounded corners=5pt] (-.4,-.5) rectangle (.4,.5);
\fill[fill=\arColor] (-.5,-1) rectangle (.5,.8);
\end{scope}
\draw[thick, \AsColor] (0,-.5) -- (0,.5);
\filldraw[thick, \AsColor, fill=\arColor] (0,0) node[above,xshift=.25cm]{$\scriptstyle -1$} circle (.1cm);
}
\geq
C\cdot 
\tikzmath{
\begin{scope}
\clip[rounded corners=5pt] (-.3,-.5) rectangle (.3,.5);
\fill[fill=\arColor] (-.5,-1) rectangle (.5,.8);
\end{scope}
\draw[thick, \AsColor] (0,-.5) -- (0,.5);
}
\qquad\qquad
\Longrightarrow
\qquad\qquad
\tikzmath{
\begin{scope}
\clip[rounded corners=5pt] (-.4,-.5) rectangle (.4,.5);
\fill[fill=\arColor] (-.5,-1) rectangle (.5,.8);
\end{scope}
\draw[thick, \AsColor] (0,-.5) -- (0,.5);
\filldraw[thick, \AsColor, fill=\arColor] (0,0) node[above,xshift=.25cm]{$\scriptstyle -2$} circle (.1cm);
}
\geq
C\cdot 
\tikzmath{
\begin{scope}
\clip[rounded corners=5pt] (-.4,-.5) rectangle (.4,.5);
\fill[fill=\arColor] (-.5,-1) rectangle (.5,.8);
\end{scope}
\draw[thick, \AsColor] (0,-.5) -- (0,.5);
\filldraw[thick, \AsColor, fill=\arColor] (0,0) node[above,xshift=.25cm]{$\scriptstyle -1$} circle (.1cm);
}\,.
$$
We now see that the projection
$$
\ev_X^{\HstarAlg(\fX)}(\ev_X^{\HstarAlg(\fX)})^\dag
=
\tikzmath{
\begin{scope}
\clip[rounded corners=5pt] (-.9,-1.1) rectangle (.9,1.1);
\fill[fill=\brColor]  (-.9,-1.1) rectangle (.9,1.1);
\fill[fill=\arColor] (-.3,-.6) arc(-180:0:.3cm) -- (.3,.6) arc(0:180:.3cm) -- (-.3,-.6);
\end{scope}
\draw[thick] (-.3,-.6) arc(-180:0:.3cm) --node[right]{$\scriptstyle X$} (.3,.6) arc(0:180:.3cm) --node[left]{$\scriptstyle X^\vee$} (-.3,-.6);
\draw[thick, \AsColor] (.3,.6) -- (0,.3);
\filldraw[thick, \AsColor, fill=\arColor] (0,.3) node[above,xshift=-.05cm,yshift=0.05cm]{$\scriptstyle -\frac{1}{2}$} circle (.1cm);
\draw[thick, \AsColor] (.3,-.3) -- (0,-.6);
\filldraw[thick, \AsColor, fill=\arColor] (0,-.6) node[above,xshift=-.05cm,yshift=0.05cm]{$\scriptstyle -\frac{1}{2}$} circle (.1cm);
}
=
\tikzmath{
\begin{scope}
\clip[rounded corners=5pt] (-.9,-1) rectangle (.9,1);
\fill[fill=\brColor]  (-.9,-1.1) rectangle (.9,1.1);
\fill[fill=\arColor] (-.3,-.4) arc(-180:0:.3cm) -- (.3,.4) arc(0:180:.3cm) -- (-.3,-.4);
\end{scope}
\draw[thick] (-.3,-.4) arc(-180:0:.3cm) --node[right]{$\scriptstyle X$} (.3,.4) arc(0:180:.3cm) --node[left]{$\scriptstyle X^\vee$} (-.3,-.4);
\draw[thick, \AsColor] (.3,.4) to[out=-135,in=90] (0,0) to[out=-90,in=135] (.3,-.4);
\filldraw[thick, \AsColor, fill=\arColor] (0,0) node[above,yshift=.2cm]{$\scriptstyle -2$} circle (.1cm);
}
\geq
C\cdot
\tikzmath{
\begin{scope}
\clip[rounded corners=5pt] (-.9,-1) rectangle (.9,1);
\fill[fill=\brColor]  (-.9,-1.1) rectangle (.9,1.1);
\fill[fill=\arColor] (-.3,-.4) arc(-180:0:.3cm) -- (.3,.4) arc(0:180:.3cm) -- (-.3,-.4);
\end{scope}
\draw[thick] (-.3,-.4) arc(-180:0:.3cm) --node[right]{$\scriptstyle X$} (.3,.4) arc(0:180:.3cm) --node[left]{$\scriptstyle X^\vee$} (-.3,-.4);
\draw[thick, \AsColor] (.3,.4) to[out=-135,in=90] (0,0) to[out=-90,in=135] (.3,-.4);
\filldraw[thick, \AsColor, fill=\arColor] (0,0) node[above,yshift=.2cm]{$\scriptstyle -1$} circle (.1cm);
}
=
C\cdot
\tikzmath{
\begin{scope}
\clip[rounded corners=5pt] (-.9,-.6) rectangle (.9,.6);
\fill[fill=\brColor]  (-.9,-1.1) rectangle (.9,1.1);
\fill[fill=\arColor] (0,0) circle (.3cm);
\end{scope}
\draw[thick] (0,0) circle (.3cm);
\node at (-.6,0) {$\scriptstyle X^\vee$};
\node at (.5,0) {$\scriptstyle X$};
}
$$
which is invertible by assumption.
Thus $\ev_X^{\HstarAlg(\fX)}(\ev_X^{\HstarAlg(\fX)})^\dag=\id_{1_b}$.
Finally, compatibility with multiplication follows from
$$
\tikzmath{
\begin{scope}
\clip[rounded corners=5pt] (-.6,-1.2) rectangle (1.5,.3);
\fill[\brColor] (-.6,-1.2) rectangle (1.5,.3);
\fill[\arColor] (-.3,-1.2) -- (-.3,-.3) arc(180:0:.3cm) -- (.3,-1.2);
\fill[\arColor] (.6,-1.2) -- (.6,-.3) arc(180:0:.3cm) -- (1.2,-1.2);
\end{scope}
\draw[thick] (-.3,-1.2) -- (-.3,-.3) arc(180:0:.3cm) -- (.3,-1.2);
\draw[thick] (.6,-1.2) -- (.6,-.3) arc(180:0:.3cm) -- (1.2,-1.2);
\draw[thick, \AsColor] (.3,-.7) -- (0,-1);
\filldraw[thick, \AsColor, fill=\arColor] (0,-1) node[above,xshift=-.05cm,yshift=0.05cm]{$\scriptstyle -\frac{1}{2}$} circle (.1cm);
\draw[thick, \AsColor] (1.2,-.7) -- (.9,-1);
\filldraw[thick, \AsColor, fill=\arColor] (.9,-1) node[above,xshift=-.05cm,yshift=0.05cm]{$\scriptstyle -\frac{1}{2}$} circle (.1cm);
}
=
\tikzmath{
\begin{scope}
\clip[rounded corners=5pt] (-.6,-1.4) rectangle (1.5,1.4);
\fill[\brColor] (-.6,-1.4) rectangle (1.5,1.4);
\fill[\arColor] (-.3,-1.4) -- (-.3,.8) arc(180:0:.3cm) -- (.6,.8) arc(180:0:.3cm) -- (1.2,-1.4) -- (.6,-1.4) -- (.6,-.8) -- (.3,-.8) -- (.3,-1.4);
\end{scope}
\draw[thick] (-.3,-1.4) -- (-.3,.8) arc(180:0:.3cm);
\draw[thick] (.3,-.8) -- (.3,-1.4);
\draw[thick] (.6,-1.4) -- (.6,-.8);
\draw[thick] (.6,.8) arc(180:0:.3cm) -- (1.2,-1.4);
\draw[thick, \AsColor] (.45,-.2) -- (.45,.2);
\draw[thick, \AsColor] (.3,-1) -- (0,-1.2);
\filldraw[thick, \AsColor, fill=\arColor] (0,-1.2) node[above,xshift=-.05cm,yshift=0.05cm]{$\scriptstyle -\frac{1}{2}$} circle (.1cm);
\draw[thick, \AsColor] (1.2,-1) -- (.9,-1.2);
\filldraw[thick, \AsColor, fill=\arColor] (.9,-1.2) node[above,yshift=0.05cm]{$\scriptstyle -\frac{1}{2}$} circle (.1cm);
\roundNbox{fill=white}{(.45,.5)}{.3}{0}{0}{$u$}
\roundNbox{fill=white}{(.45,-.5)}{.3}{0}{0}{$u^\dag$}
}
\underset{(\ref{eq:XisA-Abimod})}{=}
\tikzmath{
\begin{scope}
\clip[rounded corners=5pt] (-1.2,-.9) rectangle (1.2,1.6);
\fill[\brColor] (-1.2,-.9) rectangle (1.2,1.6);
\fill[\arColor] (-.9,-.9) -- (-.9,-.3) to[out=90,in=270] (-.3,1) arc(180:0:.3cm) to[out=270,in=90] (.9,-.3) -- (.9,-.9);
\fill[\brColor] (-.3,-.9) -- (-.3,-.3) arc(180:0:.3cm) -- (.3,-.9);
\end{scope}
\draw[thick] (-.3,-.9) -- (-.3,-.3) arc(180:0:.3cm) -- (.3,-.9);
\draw[thick] (-.9,-.9) -- (-.9,-.3) to[out=90,in=270] (-.3,1) arc(180:0:.3cm) to[out=270,in=90] (.9,-.3) -- (.9,-.9);
\draw[thick, \AsColor] (-.9,-.3) -- (-.3,-.3);
\draw[thick, \AsColor] (.9,-.6) -- (.3,-.6);
\draw[thick, \AsColor] (.3,-.3) to[out=45,in=-135] (.6,.37);
\filldraw[thick, \AsColor, fill=\arColor] (.45,.05) circle (.1cm);
\node[\AsColor] at (.64,-.15) {$\scriptstyle -\frac{1}{2}$};
\draw[thick, \AsColor] (.3,1) -- (0,.6);
\filldraw[thick, \AsColor, fill=\arColor] (0,.6) node[above,xshift=-.05cm,yshift=0.05cm]{$\scriptstyle -\frac{1}{2}$} circle (.1cm);
}\,.
$$
Hence the $\rmH^*$-algebra ${}_aA_a$ is adjoint equivalent to the standard Q-system $1_b$ in $\HstarAlg(\fX)$.
\end{rem}

\begin{lem}
\label{lem:HStarAlgCompleteIffIsometricEquivalence}
A pre-3-Hilbert space $\fX$ is idempotent complete if and only if the canonical isometric inclusion $\iota_\fX:\fX\hookrightarrow \HstarAlg(\fX)$ is isometrically essentially surjective and thus an isometric equivalence of pre-3-Hilbert spaces.
\end{lem}
\begin{proof}
If $\fX$ is idempotent complete, then every $\rmH^*$-monad ${}_aA_a$ splits.
By Remark \ref{rem:PromoteSplittingToIsometricEquivalence}, any splitting $({}_aX_b,u)$ can be promoted to an isometric adjoint equivalence $({}_AX_{1_b},u)$, and thus every $\rmH^*$-monad is in the isometric essential image of the canonical isometric inclusion $\iota_\fX:\fX\hookrightarrow\HstarAlg(\fX)$.

Conversely, if $\iota_\fX$ is isometrically essentially surjective, then given any $\rmH^*$-monad ${}_aA_a$ in $\fX$, there is an isometric adjoint equivalence ${}_AX_{1_b}$ between ${}_aA_a$ and a standard Q-system $1_b$ in the image of $\iota_\fX$.
A standard calculation shows that $X$ splits ${}_aA_a$.
\end{proof}

\begin{lem}
\label{lem:HstarAlgIsIdempotentComplete}
The $\rmH^*$-algebra completion $\HstarAlg(\fX)$ of a pre-3-Hilbert space $\fX$ is idempotent complete.
\end{lem}
\begin{proof}
The proof is entirely similar to \cite[Cor.~3.37]{MR4419534}; one shows that if ${}_AB_A$ is an $\rmH^*$-monad in $\HstarAlg(\fX)$ where ${}_aA_a$ is an $\rmH^*$-monad in $\fX$, then ${}_aB_a$ already has an $\rmH^*$-monad structure in $\fX$, and ${}_AB_B$ splits ${}_AB_A$ in $\HstarAlg(\fX)$.
\end{proof}

We now state the universal property of $\HstarAlg(\fX)$ which is similar to the universal property of Q-system completion from \cite{MR4369356} (see also \cite{MR4372801}).
The proof is entirely similar to the above proof and the proof of Proposition \ref{prop:HilbertDirectSumUniversal}, so we will not include it here.

\begin{prop}
\label{prop:UniversalPropertyOfHstarAlgCompletion}
For any idempotent complete pre-3-Hilbert space $\fY$,
precomposition with $\iota_\fX$ is an equivalence of unitary 2-categories
$$
\Fun^\dag(\HstarAlg(\fX)\to \fY) 
\xrightarrow{\iota^*}
\Fun^\dag(\fX\to \fY)
$$
which maps 
\begin{itemize}
\item 
$\Fun^{\dag,\vee}(\HstarAlg(\fX)\to \fY)$ onto $\Fun^{\dag,\vee}(\fX\to \fY)$, and
\item 
$\Isom(\HstarAlg(\fX)\to \fY)$ onto $\Isom(\fX\to \fY)$.
\end{itemize}
\end{prop}

Finally, the next result shows that Hilbert direct sum completeness is preserved by taking the $\rmH^*$-monad completion.

\begin{prop}
\label{prop:HstarCompletionOfHilbCompleteIsHilbComplete}
If $\fX$ is a Hilbert direct sum complete pre-3-Hilbert space, then $\HstarAlg(\fX)$ is also Hilbert direct sum complete.
\end{prop}
\begin{proof}
Suppose ${}_aA_a$ and ${}_bB_b$ are $\rmH^*$-monads in $\fX$.
Then there are canonical isometries $I_a: a\to a\boxplus b$ and $I_b: b\to a\boxplus b$, and we can form the $\rmH^*$-monad
$$
A\boxplus B
:=
I_a^\vee \otimes_a A\otimes_a I_a \oplus I_b^\vee \otimes_b B\otimes_b I_b
$$ 
with component-wise multiplication and unit.
A straightforward calculation shows that
$A\otimes_a I_a : a\to a\boxplus b$ is an $A$-$A\boxplus B$ bimodule which is an isometry in $\HstarAlg(\fX)$,
and similarly for $B\otimes_b I_b : b\to a\boxplus b$ as a $B$-$A\boxplus B$ bimodule,
and that $A\otimes_a I_a , B\otimes_b I_b$ witness 
$A\boxplus B$ as the Hilbert direct sum of $A$ and $B$ in $\HstarAlg(\fX)$.
\end{proof}

\subsection{3-Hilbert spaces}

We are finally ready to introduce the notion of a 3-Hilbert space.

\begin{defn}
A pre-3-Hilbert space $\fX$ is called a \emph{3-Hilbert space} if it admits Hilbert direct sums and all $\rmH^*$-algebras split.
The notion of isometric equivalence of 3-Hilbert spaces is the same as that for pre-3-Hilbert spaces.
\end{defn}

\begin{rem}
The underlying 2-category of a 3-Hilbert space $\fX$ is finite semisimple.
To see this, it suffices to show that every unital condensation monad ${}_aA_a$ in $\fX$ the sense of \cite{1812.11933} splits.
By \cite{MR4724964}, every finite semisimple module category over a unitary multifusion category is uniquely unitarizable. 
Hence, $A$ is equivalent to a standard Q-system in the unitary multifusion category $\End(a)$, and we know every standard Q-system splits.
In the spirit of this paper, we provide a manifestly unitary approach to these results in Proposition \ref{prop:ModulesUnitarizable} and Corollary \ref{cor:AlgebraModulesUnitarizable}.
\end{rem}

\begin{defn}\label{def:unitaryCauchyCompletion3Hil}
If $(\fX,\vee,\Psi)$ is a pre-3-Hilbert space, then $\HstarAlg(\Hilb_\boxplus(\fX))$ is a 3-Hilbert space by Lemmas \ref{lem:HilbCompletionIsHilbComplete} and \ref{lem:HstarAlgIsIdempotentComplete} and Proposition \ref{prop:HstarCompletionOfHilbCompleteIsHilbComplete}. We denote this unitary Cauchy completion 3-Hilbert space by $\cent^\dag(\fX)$.
\end{defn}

\begin{ex}
If $(\cC,\vee,\psi)$ is an $\rmH^*$-multifusion category,
then the 2-category
$\HstarAlg(\cC)$ of $\rmH^*$-algebras in $\cC$ is a 3-Hilbert space as it already admits Hilbert direct sums.
Indeed, if $A,B\in\cC$ are $\rmH^*$-algebras, then 
as in the proof of Proposition \ref{prop:HstarCompletionOfHilbCompleteIsHilbComplete}, 
${}_A(A\oplus B)_{A\oplus B}$ and ${}_B(A\oplus B)_{A\oplus B}$ witness $A\oplus B$ as a Hilbert direct sum of $A,B$ in $\HstarAlg(\cC)$.
\end{ex}

\begin{ex}
If $(\cC,\vee,\psi)$ is an $\rmH^*$-multifusion category,
then $\Mod^\dag(\cC)$ is also a 3-Hilbert space.
We already saw that $\Mod^\dag(\cC)$ admits Hilbert direct sums in Example \ref{ex:ModCAdmitsHilbertDirectSums} above.
To see that every $\rmH^*$-monad splits, observe that 
by Theorem \ref{thm:EveryModuleComesFromH*Alg} and Example \ref{ex:IsometriesIn2Hilb}, the functor $\HstarAlg(\cC)\to \Mod^\dag(\cC)$ from \S\ref{sec:H*Alg==ModC} is automatically isometrically essentially surjective, as every $(\cM,\Tr^\cM)$ is isometrically adjoint equivalent to a $\cC$-module of the form $(\cC_A,\Tr^{\cC_A})$.
Hence one can pull back an $\rmH^*$-monad from 
$\Mod^\dag(\cC)$ to $\HstarAlg(\cC)$ and split it in $\HstarAlg(\cC)$.
Since the functor $\HstarAlg(\cC)\to \Mod^\dag(\cC)$  preserves the UAF, pushing forward the splitting achieves our goal.
\end{ex}

\begin{rem}
\label{rem:WeightOfComponentAtObject}
Suppose $(\fX,\vee,\Psi)$ is a 3-Hilbert space.
Then $\Psi$ is completely determined on each component of $\fX$ by its value at a single (simple) object.
Indeed, if $a,b\in\fX$ are both simple and ${}_aX_b$ is dualizable, then the formula
$$
\Psi_b\left(\,
\tikzmath{
\fill[rounded corners = 5pt, fill=\brColor] (-.9,-.7) rectangle (.9,.7);
\filldraw[fill=\arColor,thick] (0,0) circle (.3cm);
\node at (-.6,0) {$\scriptstyle X^\vee$};
\node at (.6,0) {$\scriptstyle X$};
}
\,\right)
=
\Psi_a\left(\,
\tikzmath{
\fill[rounded corners = 5pt, fill=\arColor] (-.9,-.7) rectangle (.9,.7);
\filldraw[fill=\brColor,thick] (0,0) circle (.3cm);
\node at (-.6,0) {$\scriptstyle X$};
\node at (.6,0) {$\scriptstyle X^\vee$};
}
\,\right)
\qquad\qquad
\qquad\qquad
\begin{aligned}
\tikzmath{\fill[fill=\arColor, rounded corners=5] (-.3,-.3) rectangle (.3,.3);}&=a
\\
\tikzmath{\fill[fill=\brColor, rounded corners=5] (-.3,-.3) rectangle (.3,.3);}&=b
\end{aligned}
$$
completely determines $\Psi_b$ in terms of $\Psi_a$ and ${}_aX_b$.
Hence $\Psi$ is completely determined by combining Lemma \ref{lem:DecomposeIntoSimplesInA3Hilb} to write an arbitrary $c\in\fX$ in the component of $a$ as a Hilbert direct sum of simples together with \eqref{eq:PsiOfHilbertDirectSum}.

Now if $F: \fX\to \fY$ is a functor of connected 3-Hilbert spaces which preserves the UAF, it is isometric if and only if it preserves $\Psi$ at a single simple $a\in\fX$.
To see this, first note that if $b\in\fX$ is another simple, then
$$
\Psi_{F(b)}^\fY\left(\,
\tikzmath{
\fill[rounded corners = 5pt, fill=\brColor] (-1.1,-.5) rectangle (1.1,.5);
\filldraw[fill=\arColor,thick] (0,0) circle (.2cm);
\node at (-.68,0) {$\scriptstyle F(X)^\vee$};
\node at (.65,0) {$\scriptstyle F(X)$};
}
\,\right)
=
\Psi_{F(a)}^\fY\left(\,
\tikzmath{
\fill[rounded corners = 5pt, fill=\arColor] (-1.1,-.5) rectangle (1.1,.5);
\filldraw[fill=\brColor,thick] (0,0) circle (.2cm);
\node at (-.65,0) {$\scriptstyle F(X)$};
\node at (.68,0) {$\scriptstyle F(X)^\vee$};
}
\,\right)
=
\Psi_a^\fX\left(\,
\tikzmath{
\fill[rounded corners = 5pt, fill=\arColor] (-.7,-.5) rectangle (.7,.5);
\filldraw[fill=\brColor,thick] (0,0) circle (.2cm);
\node at (-.45,0) {$\scriptstyle X$};
\node at (.48,0) {$\scriptstyle X^\vee$};
}
\,\right)
=
\Psi_b^\fX\left(\,
\tikzmath{
\fill[rounded corners = 5pt, fill=\brColor] (-.7,-.5) rectangle (.7,.5);
\filldraw[fill=\arColor,thick] (0,0) circle (.2cm);
\node at (-.48,0) {$\scriptstyle X^\vee$};
\node at (.45,0) {$\scriptstyle X$};
}
\,\right),
$$
and thus $F$ preserves $\Psi$ at $b$.
The result now follows as every object in $\fX$ is a Hilbert direct sum of simples by Lemma \ref{lem:DecomposeIntoSimplesInA3Hilb},
and $F$ preserves Hilbert direct sums by Lemma \ref{lem:PreservesHilbertDirectSums}.
\end{rem}

\begin{thm}
\label{thm:HstarAlgC==ModC}
Suppose $(\cC,\vee,\psi)$ is an $\rmH^*$-multifusion category.
The unitary equivalence $\HstarAlg(\cC)\to \Mod^\dag(\cC)$
from \S\ref{sec:H*Alg==ModC}
is an isometric equivalence of 3-Hilbert spaces.
\end{thm}
\begin{proof}
The unitary 2-equivalence $\HstarAlg(\cC)\to \Mod^\dag(\cC)$ preserves the UAF and is isometrically essentially surjective by Theorem \ref{thm:EveryModuleComesFromH*Alg} and Example \ref{ex:IsometriesIn2Hilb}.
It remains to prove it is isometric.
On each component, we know that the unitary equivalence
$\HstarAlg(\cC)\cong\Mod^\dag(\cC)$ must 
map
$\Psi^{\HstarAlg(\cC)}$
to
$\Psi^{\Mod^\dag(\cC)}$ 
up to a uniform scalar by Remark \ref{rem:WeightOfComponentAtObject}.
We may thus reduce to the case that $\cC$ is indecomposable.
Again by Remark \ref{rem:WeightOfComponentAtObject}, it suffices to check the values of the $\Psi$s at a single simple object in $\HstarAlg(\cC)$; we choose the standard Q-system $1_1$ where $1_\cC=\bigoplus_{j=1}^k 1_j$ is an orthogonal direct sum decomposition into simples.

The corresponding module category in $\Mod^\dag(\cC)$ is $\cC_1:=\cC\otimes 1_1$, the `first column' of $\cC$ where $\cC$ acts on the left.
The trace on $\cC_1$ from \eqref{eq:TraceOnC_A} from Proposition \ref{prop:InternalEndRecognition}
is exactly
$\Tr^{\cC_1}=(\psi_\cC\circ \tr^\vee_\cC)|_{\cC_1}$.
Since $\End_\cC(1_1)=\bbC$, we compute
\begin{align*}
\Psi^{\Mod^\dag(\cC)}_{(\cC_1,\Tr^{\cC_1})}
(\id_{\id_{\cC_1}})
&
\underset{\text{(Ex.~\ref{ex:WeightOnModC})}}{=}
\frac{d_{1_1}}{\dim(\cC_1)}
\sum_{c\in \Irr(\cC_1)} d_{c}^2
=
d_{1_1}
=
\Psi^{\HstarAlg(\cC)}_{1_1}
(\id_{1_1}) 
\end{align*}
as claimed.
\end{proof}

\begin{cor}
\label{cor:All3HilbsAreModC}
Every 3-Hilbert space is of the form $\HstarAlg(\cC)\cong\Mod^\dag(\cC)$ for some $\rmH^*$-multifusion category $\cC$.
\end{cor}
\begin{proof}
It suffices to consider the case that $\fX$ is a connected 3-Hilbert space.
Let $a\in\fX$ be simple, and consider the $\rmH^*$-multifusion category $\cC=\Omega_a=\End_\fX(a)$.
By Proposition \ref{prop:UniversalPropertyOfHstarAlgCompletion}, precomposition with the canonical isometric inclusion $\iota_\cC: \cC\hookrightarrow \HstarAlg(\cC)$ gives a canonical equivalence
$$
\Isom(\rmB\cC \to \fX)
\cong
\Isom(\HstarAlg(\cC)\to \fX);
$$
we thus have a canonical fully faithful isometric unitary functor 
$\HstarAlg(\cC)\hookrightarrow \fX$; we identify $\HstarAlg(\cC)$ with its image in $\fX$.
To see that this functor is isometrically essentially surjective, suppose $b\in\fX$ and let ${}_aX_b$ be a dualizable 1-morphism.
The $\rmH^*$-monad $A:={}_aX\otimes_bX^\vee_a \in \Omega_a$ splits in $\fX$ and also defines an object of $\HstarAlg(\cC)\subset \fX$.
By Remark \ref{rem:PromoteSplittingToIsometricEquivalence} and Lemma \ref{lem:HStarAlgCompleteIffIsometricEquivalence}, ${}_AX_b$ is an isometric adjoint equivalence in $\fX$, as desired.
\end{proof}

The proof of Theorem \ref{thmalpha:All3Hilbs} follows immediately from Theorem \ref{thm:HstarAlgC==ModC} and Corollary \ref{cor:All3HilbsAreModC}.
\qed

\section{The 3-category of 3-Hilbert spaces}

In this section, we promote the map $(\cC,\vee,\psi)\mapsto (\Mod^\dag(\cC),\Psi)$
to a unitary equivalence of $\rmC^*$ 3-categories
$\sH^*\mFC\to 3\Hilb$.
We refer the reader to \cite{2404.05193} for the recently introduced notion of $\rmC^*$ 3-category.
In future work, we plan to promote this unitary equivalence to an isometric equivalence of 4-Hilbert spaces.

\begin{defn}
The $\rmC^*$ 3-category $\sH^*\mFC$ is given as follows.
\begin{itemize}
\item 
The objects are the $\rmH^*$-multifusion categories $(\cC,\vee,\psi)$.
\item 
The 1-morphisms $\cC\to \cD$ are simply finite semisimple unitary $\cC$-$\cD$ bimodule categories equipped with $\cC$-$\cD$ bimodule traces.
The 1-composition operation is the relative Deligne product ${}_\cC\cM\boxtimes_\cD \cN_\cE$ equipped with its unitary trace from Lemma \ref{lem:RelativeDeligneTrace}
or \eqref{eq:RenormalizedTraceOnFun} depending on the model chosen for the relative Deligne product.
\item 
The 2-morphisms ${}_\cC\cM_\cD\Rightarrow {}_\cC\cN_\cD$ are unitary $\cC$-$\cD$ bimodule functors with no compatibility requirement with the tracial data.
The 2-composition is composition of functors, which is strictly associative.
\item 
Given $\cC$-$\cD$ bimodule functors $F,G:{}_\cC\cM_\cD\to {}_\cC\cN_\cD$,
the 3-morphisms $F\Rrightarrow G$ are bimodule natural transformations, and 3-composition is the usual composition of natural transformations.
\item 
The $\dag$-operation on 3-morphisms is the adjoint of natural transformations, which is again a natural transformations.
\end{itemize}
All coherence data for this 3-category is unitary as discussed in Remark \ref{rem:unitary-assoc-constraint}.
\end{defn}

\begin{rem}
Observe that the hom $\rmC^*$ 2-categories in $\sH^*\mFC$ can be equipped with 3-Hilbert space structures by Example \ref{ex:WeightOnModC} and the folding trick:
$$
\mathsf{Bim}^{\dag}(\cC,\cD) 
\cong 
\Mod^\dag(\cC\boxtimes \cD^{\rm mp}).
$$
\end{rem}

\begin{rem}
\label{rem:ForgetTraces}
As the 2-morphisms in $\sH^*\mFC$ are not required to be compatible with the bimodule traces, equivalence for 1-morphisms in $\sH^*\mFC$ is unitary equivalence of the underlying bimodule categories forgetting the tracial data.
\end{rem}

\begin{defn}
The $\rmC^*$ 3-category $3\Hilb$ is given as follows.
\begin{itemize}
\item 
The objects are the 3-Hilbert spaces $(\fX,\vee,\Psi)$.
\item 
The 1-morphisms $\fX\to \fY$ are the UAF-preserving $\dag$-2-functors.
The 1-composition operation is composition of 2-functors.
\item 
Given $F,G:\fX\to \fY$, the 2-morphisms $F\Rightarrow G$ are $\dag$-2-transformations,
and 2-composition is composition of 2-transformations.
\item 
Given 2-transformations $\rho,\sigma: F\Rightarrow G$, a 3-morphism $m: \rho \Rrightarrow \sigma$ is a 2-modification, and 3-composition is the usual composition of 2-modifications.
\item 
The $\dag$-operation on 3-morphisms is the adjoint of modifications, which is again a modification due to the unitarity of the source and target $\dag$-2-natural transformations constraint data.
\end{itemize}
All coherence data for this 3-category is unitary.
\end{defn}

\begin{rem}
For 3-Hilbert spaces $\fX$ and $\fY$, we expect $\Fun^\dag(\fX \to \fY)$ and $\Fun^{\dag,\vee}(\fX \to \fY)$ to be equivalent as unitary 2-categories, but not isometrically so as 3-Hilbert spaces for some putative 3-Hilbert space structure.
Non-pivotal unitary tensor functors between $\rmH^*$-multifusion categories $\cC\to \cD$ provide one obstruction.
By completing $\cC,\cD$, these yield 2-functors $\Mod(\cC)\to \Mod(\cD)$ which does not preserve UAFs.
However, for every such 2-functor, one may equip $\cD$ with a different $\cC$-$\cD$ bimodule trace such that it will be UAF-preserving.
These have the same underlying 2-functor, and they are unitarily equivalent, but they will not be isometrically so.
\end{rem}

\subsection{Unitary 3-equivalence between \texorpdfstring{$3\Hilb$}{3Hilb} and \texorpdfstring{$\sH^*\mFC$}{H*mFC}}
\label{sec:3Hilb==H*mFC}

We prove Theorem \ref{thm:3Hilb=H*mFC} which states that $\Mod^\dag$ gives a unitary equivalence between the $\rmC^*$ 3-categories $\sH^*\mFC$ and $3\Hilb$.
We expect this equivalence to be as 4-Hilbert spaces. 
Defining a spherical structure and isometries within 4-Hilbert spaces requires a detailed treatment of unitary adjoints for 2-functors, which we leave to future work.

\begin{lem}\label{lem:boxtimesisUAFpreserving}
Given $\rmH^*$-multifusion categories $(\cC,\vee,\psi)$ and $(\cD,\vee,\psi)$, and a $\cC$-$\cD$ bimodule ${}_\cC \cN _\cD$ equipped with a $\cC$-$\cD$ bimodule trace $\Tr^\cN$, the $\dag$-2-functor
$$
- \boxtimes_{\cC} \cN_\cD 
\colon 
\Mod^\dag(\cC) \to \Mod^\dag(\cD)
$$
is UAF-preserving.
\end{lem}

\begin{proof}
Consider $F \in \Fun_\cC(\cM \to \cM')$ with unitary adjoint $F^*$.
We wish to show the map
\begin{align*}
(\cM' \boxtimes_\cC \cN_\cD)( Fm \boxtimes n, m' \boxtimes n') 
&\rightarrow(\cM \boxtimes_\cC \cN_\cD)(m \boxtimes n, F^*m' \boxtimes n')\\
\tikzmath{
\draw[thick] (.1,-.8) -- node[right]{$\scriptstyle m$} (.1,-1.2);
\draw[thick] (-.1,-.8) --node[left]{$\scriptstyle F$} (-.1,-1.2);
\draw[thick] (0,-.2) -- (0,1.2) node[left]{$\scriptstyle m'$};
\draw[thick] (1,.8) -- (1,1.2) node[right]{$\scriptstyle n'$};
\draw[thick] (1,.2) --node[right]{$\scriptstyle n$} (1,-1.2);
\draw[thick, blue] (0,-.5) --node[above]{$\scriptstyle c$} (1,.5);
\roundNbox{fill=white}{(0,-.5)}{0.3}{.1}{.1}{$f$};
\roundNbox{fill=white}{(1,.5)}{0.3}{.1}{.1}{$g$};
}
&\mapsto 
\tikzmath{
\draw[thick] (0.1,-.8) --node[right]{$\scriptstyle m$} (0.1,-1.2);
\draw[thick] (-0.1,-.8) arc (0:-180:.2) -- (-.5,1.2) node[left,yshift=-.125]{$\scriptstyle F^*$};
\draw[thick] (.1,-.2) -- (.1,1.2) node[right]{$\scriptstyle m'$};
\draw[thick] (1,.8) -- (1,1.2) node[right]{$\scriptstyle n'$};
\draw[thick] (1,.2) --node[right]{$\scriptstyle n$} (1,-1.2);
\draw[thick, blue] (0,-.5) --node[above]{$\scriptstyle c$} (1,.5);
\roundNbox{fill=white}{(0,-.5)}{0.3}{.1}{.1}{$f$};
\roundNbox{fill=white}{(1,.5)}{0.3}{.1}{.1}{$g$};
}
\end{align*}
is unitary for $m, \in \cM$ , $m' \in \cM'$, $n, n' \in \cN$.
Since this map admits an inverse, it suffices to show it is isometric.
Consider a morphism
$$
\sum_{c\in\Irr(\cC)} \sum_{r=1}^{k_c} f^c_r\otimes g^c_r
\in
\bigoplus_{c\in\Irr(\cC)} \cM'(Fm\to m'\lhd c)\otimes \cN(c\rhd n\to n').
$$
Observe
\begin{align*}
\sum_{\substack{i,j \\ b,c\in\Irr(\cC_{ij}) \\ r,l}}
\Tr^{\cM\boxtimes_\cC\cN}_{Fm_i\boxtimes n_i}\left(
\tikzmath{
\begin{scope}
\clip (0,-.2) rectangle (1,4.2);
\fill[\arColor] (0,-.2) rectangle (1,4.2);
\fill[\brColor] (0,0.5) -- (1,1.5) -- (1,2.5) -- (0,3.5);
\end{scope}
\draw[thick] (0,-0.2) node[right,yshift=.1cm]{$\scriptstyle m_i$} -- node[left]{$\scriptstyle m'_j$} (0,4.2) node[right,yshift=-.1cm]{$\scriptstyle m_i$};
\draw [thick] (-.2,-0.2) -- (-.2,0.2) node[left,yshift=-.25cm]{$\scriptstyle F$};
\draw [thick] (-.2,3.8) -- (-.2,4.2) node[left,yshift=-.125cm]{$\scriptstyle F$};
\draw[thick] (1,-0.2) node[right,yshift=.1cm]{$\scriptstyle n_i$} -- node[right]{$\scriptstyle n'_j$} (1,4.2) node[right,yshift=-.1cm]{$\scriptstyle n_i$};
\draw[thick, DarkGreen] (0,0.5) --node[above]{$\scriptstyle b$} (1,1.5);
\draw[thick, blue] (0,3.5) --node[above]{$\scriptstyle c$} (1,2.5);
\roundNbox{fill=white}{(0,0.5)}{0.3}{.1}{.1}{$f^b_r$};
\roundNbox{fill=white}{(0,3.5)}{0.3}{.1}{.1}{\scriptsize{$(f^c_l)^\dag$}};
\roundNbox{fill=white}{(1,1.5)}{0.3}{.1}{.1}{$g^b_r$};
\roundNbox{fill=white}{(1,2.5)}{0.3}{.1}{.1}{\scriptsize{$(g^c_l)^\dag$}};
}
\right) 
&=
\sum_{\substack{i,j \\ c\in\Irr(\cC_{ij}) \\ r,l}}
d_c^{-1}
\Tr_{Fm_i}^{\cM}\left(
\tikzmath{
\begin{scope}
\clip[rounded corners = 5pt] (-.2,-.7) rectangle (.7,1.7);
\fill[\arColor] (0,-.7) rectangle (.7,1.7);
\fill[\brColor] (-.15,0) rectangle (.15,1);
\end{scope}
\draw[thick] (0,1.3) --node[right]{$\scriptstyle m_i$} (0,1.7);
\draw[thick] (0,-.7) --node[right]{$\scriptstyle m_i$} (0,-.3);
\draw [thick] (-.2,-0.7) -- (-.2,-.3) node[left,yshift=-.15cm]{$\scriptstyle F$};
\draw [thick] (-.2,1.3) -- (-.2,1.7) node[left,yshift=-.125cm]{$\scriptstyle F$};
\draw[thick] (-.15,.3) --node[left]{$\scriptstyle m'_j$} (-.15,.7);
\draw[thick, blue] (.15,.3) --node[right]{$\scriptstyle c$} (.15,.7);
\roundNbox{fill=white}{(0,0)}{0.3}{0.1}{0.1}{$f^c_r$};
\roundNbox{fill=white}{(0,1)}{0.3}{0.1}{0.1}{\scriptsize{$(f^c_l)^\dag$}};
}
\right)
\Tr^\cN_{n_i}\left(
\tikzmath{
\begin{scope}
\clip[rounded corners = 5pt] (-1.15,-.7) rectangle (.2,1.7);
\fill[\arColor] (-1.15,-.7) -- (.15,-.7) -- (.15,0) -- (0,0) -- (0,1) -- (.15,1) -- (.15,1.7) -- (-1.15,1.7);
\fill[\brColor] (-.15,1.3) arc (0:180:.3cm) -- (-.75,-.3) arc (-180:0:.3cm) -- (0,0) -- (0,1);
\end{scope}
\draw[thick, blue] (-.15,1.3) arc (0:180:.3cm) --node[right,xshift=-.1cm]{$\scriptstyle c^\vee$} (-.75,-.3) arc (-180:0:.3cm);
\draw[thick] (0,.3) --node[right]{$\scriptstyle n'_j$} (0,.7);
\draw[thick] (.15,1.3) --node[right]{$\scriptstyle n_i$} (.15,1.7);
\draw[thick] (.15,-.3) --node[right]{$\scriptstyle n_i$} (.15,-.7);
\roundNbox{fill=white}{(0,0)}{0.3}{0.1}{0.1}{$g^c_r$};
\roundNbox{fill=white}{(0,1)}{0.3}{0.1}{0.1}{\scriptsize{$(g^c_l)^\dag$}};
}
\right)
\displaybreak[1]\\
\sum_{\substack{i,j \\ b,c\in\Irr(\cC_{ij}) \\ r,l}}
\Tr^{\cM\boxtimes_\cC\cN}_{m_i\boxtimes n_i}\left(
\tikzmath{
\begin{scope}
\clip (0,-.2) rectangle (1,4.2);
\fill[\arColor] (0,-.2) rectangle (1,4.2);
\fill[\brColor] (0,0.5) -- (1,1.5) -- (1,2.5) -- (0,3.5);
\end{scope}
\draw[thick] (0,-0.2) node[right,yshift=.1cm]{$\scriptstyle m_i$} -- node[right]{$\scriptstyle m'_j$} (0,4.2) node[right,yshift=-.1cm]{$\scriptstyle m_i$};
\draw [thick] (-.2,3.8) arc (0:180:.2) -- node[left]{$\scriptstyle F^*$} (-.6,0.2) arc (180:360:.2); 
\draw[thick] (1,-0.2) node[right,yshift=.1cm]{$\scriptstyle n_i$} -- node[right]{$\scriptstyle n'_j$} (1,4.2) node[right,yshift=-.1cm]{$\scriptstyle n_i$};
\draw[thick, DarkGreen] (0,0.5) --node[above]{$\scriptstyle b$} (1,1.5);
\draw[thick, blue] (0,3.5) --node[above]{$\scriptstyle c$} (1,2.5);
\roundNbox{fill=white}{(0,0.5)}{0.3}{.1}{.1}{$f^b_r$};
\roundNbox{fill=white}{(0,3.5)}{0.3}{.1}{.1}{\scriptsize{$(f^c_l)^\dag$}};
\roundNbox{fill=white}{(1,1.5)}{0.3}{.1}{.1}{$g^b_r$};
\roundNbox{fill=white}{(1,2.5)}{0.3}{.1}{.1}{\scriptsize{$(g^c_l)^\dag$}};
}
\right) 
&=
\sum_{\substack{i,j \\ c\in\Irr(\cC_{ij}) \\ r,l}}
d_c^{-1}
\Tr_{m_i}^{\cM}\left(
\tikzmath{
\begin{scope}
\clip[rounded corners = 5pt] (-.2,-.7) rectangle (.7,1.7);
\fill[\arColor] (0,-.7) rectangle (.7,1.7);
\fill[\brColor] (-.15,0) rectangle (.15,1);
\end{scope}
\draw[thick] (0,1.3) --node[right]{$\scriptstyle m_i$} (0,1.7);
\draw[thick] (0,-.7) --node[right]{$\scriptstyle m_i$} (0,-.3);
\draw [thick] (-.2,1.3) arc (0:180:.3) -- node [left]{$\scriptstyle F^*$} (-.8,-.3) arc (180:360:.3);
\draw[thick] (-.15,.3) --node[left]{$\scriptstyle m'_j$} (-.15,.7);
\draw[thick, blue] (.15,.3) --node[right]{$\scriptstyle c$} (.15,.7);
\roundNbox{fill=white}{(0,0)}{0.3}{0.1}{0.1}{$f^c_r$};
\roundNbox{fill=white}{(0,1)}{0.3}{0.1}{0.1}{\scriptsize{$(f^c_l)^\dag$}};
}
\right)
\Tr^\cN_{n_i}\left(
\tikzmath{
\begin{scope}
\clip[rounded corners = 5pt] (-1.15,-.7) rectangle (.2,1.7);
\fill[\arColor] (-1.15,-.7) -- (.15,-.7) -- (.15,0) -- (0,0) -- (0,1) -- (.15,1) -- (.15,1.7) -- (-1.15,1.7);
\fill[\brColor] (-.15,1.3) arc (0:180:.3cm) -- (-.75,-.3) arc (-180:0:.3cm) -- (0,0) -- (0,1);
\end{scope}
\draw[thick, blue] (-.15,1.3) arc (0:180:.3cm) --node[right,xshift=-.1cm]{$\scriptstyle c^\vee$} (-.75,-.3) arc (-180:0:.3cm);
\draw[thick] (0,.3) --node[right]{$\scriptstyle n'_j$} (0,.7);
\draw[thick] (.15,1.3) --node[right]{$\scriptstyle n_i$} (.15,1.7);
\draw[thick] (.15,-.3) --node[right]{$\scriptstyle n_i$} (.15,-.7);
\roundNbox{fill=white}{(0,0)}{0.3}{0.1}{0.1}{$g^c_r$};
\roundNbox{fill=white}{(0,1)}{0.3}{0.1}{0.1}{\scriptsize{$(g^c_l)^\dag$}};
}
\right)
\end{align*}
Since $F \dashv_\dag F^*$, 
we have
$
\Tr_{Fm_i}^{\cM}\left(
\tikzmath{
\begin{scope}
\clip[rounded corners = 5pt] (-.2,-.7) rectangle (.7,1.7);
\fill[\arColor] (0,-.7) rectangle (.7,1.7);
\fill[\brColor] (-.15,0) rectangle (.15,1);
\end{scope}
\draw[thick] (0,1.3) --node[right]{$\scriptstyle m_i$} (0,1.7);
\draw[thick] (0,-.7) --node[right]{$\scriptstyle m_i$} (0,-.3);
\draw [thick] (-.2,-0.7) -- (-.2,-.3) node[left,yshift=-.15cm]{$\scriptstyle F$};
\draw [thick] (-.2,1.3) -- (-.2,1.7) node[left,yshift=-.125cm]{$\scriptstyle F$};
\draw[thick] (-.15,.3) --node[left]{$\scriptstyle m'_j$} (-.15,.7);
\draw[thick, blue] (.15,.3) --node[right]{$\scriptstyle c$} (.15,.7);
\roundNbox{fill=white}{(0,0)}{0.3}{0.1}{0.1}{$f^c_r$};
\roundNbox{fill=white}{(0,1)}{0.3}{0.1}{0.1}{\scriptsize{$(f^c_l)^\dag$}};
}
\right)
=
\Tr_{m_i}^{\cM}\left(
\tikzmath{
\begin{scope}
\clip[rounded corners = 5pt] (-.2,-.7) rectangle (.7,1.7);
\fill[\arColor] (0,-.7) rectangle (.7,1.7);
\fill[\brColor] (-.15,0) rectangle (.15,1);
\end{scope}
\draw[thick] (0,1.3) --node[right]{$\scriptstyle m_i$} (0,1.7);
\draw[thick] (0,-.7) --node[right]{$\scriptstyle m_i$} (0,-.3);
\draw [thick] (-.2,1.3) arc (0:180:.3) -- node [left]{$\scriptstyle F^*$} (-.8,-.3) arc (180:360:.3);
\draw[thick] (-.15,.3) --node[left]{$\scriptstyle m'_j$} (-.15,.7);
\draw[thick, blue] (.15,.3) --node[right]{$\scriptstyle c$} (.15,.7);
\roundNbox{fill=white}{(0,0)}{0.3}{0.1}{0.1}{$f^c_r$};
\roundNbox{fill=white}{(0,1)}{0.3}{0.1}{0.1}{\scriptsize{$(f^c_l)^\dag$}};
}
\right)$. 
\end{proof}

We now describe the 3-functor $\Mod^\dag \colon \sH^*\mFC\to 3\Hilb$.
We recall that the non-unitary version $\mFC\to 3\Vect$ was shown to be an equivalence in \cite[Thm.~2.2.2]{MR4372801}.

\begin{itemize}
\item 
An $\rmH^*$-multifusion category $\cC\in\rmH^*\mFC$ maps to $\Mod^\dag(\cC)\in 3\Hilb$.
\item 
A unitary finite semisimple $\cC$-$\cD$ bimodule category ${}_\cC\cM_\cD$ with a bimodule trace maps to the UAF-preserving $\dag$-2-functor $-\boxtimes_\cC\cM_\cD : \Mod^\dag(\cC)\to \Mod^\dag(\cD)$.

\item 
A unitary $\cC$-$\cD$ bimodule functor $F: {}_\cC\cM_\cD\to {}_\cC\cN_\cD$ maps to the $\dag$-2-transformation $-\boxtimes F: -\boxtimes_\cC\cM_\cD\Rightarrow -\boxtimes_\cC\cN_\cD$.
\item 
A $\cC$-$\cD$ bimodule natural tranformation $\rho: F\Rightarrow G$ maps to the 2-modification $-\boxtimes \rho : -\boxtimes F\Rrightarrow -\boxtimes G$.
\end{itemize}
In fact, $\Mod^\dag \colon \sH^*\mFC \to 3\Hilb$ can be described as the Yoneda map
\begin{align*}
\sH^*\mFC( \Hilb \to -) \colon \sH^*\mFC  &\to 3\Hilb
\\
\cC &\mapsto \sH^*\mFC( \Hilb \to \cC) = \Mod^\dag(\cC),
\end{align*}
which is clearly a (weak) $\dag$-3-functor.

\begin{proof}[Proof of Theorem \ref{thm:3Hilb=H*mFC}]
By Corollary \ref{cor:All3HilbsAreModC}, the functor $\Mod^\dag\colon \sH^*\mFC\to 3\Hilb$ is unitarily essentially surjective (even isometrically so!).
It remains to prove it is fully faithful, i.e.
$$
\Mod^\dag:
\sH^*\mFC(\cC\to \cD)
\to
\Fun^{\dag,\vee}(\Mod^\dag(\cC)\to \Mod^\dag(\cD))
$$
is a $\dag$-equivalence on every hom 2-category.

We will use the fact that, viewing $\cC$ as a right $\cC$-module $\cC_\cC \in \Mod^\dag(\cC)$, Proposition \ref{prop:RenormalizedTraceOnFun} yields an isometric equivalence $\rmB\cC \to \rmB \End(\cC_{\cC})$ given on hom-spaces by $c \mapsto c \otimes -$.
Moreover, the data of a unitary $\cC$-$\cD$ bimodule $_{\cC}\cM_{\cD}$ equipped with a $\cC$-$\cD$ bimodule trace $\Tr^\cM$ is precisely the data of a UAF-preserving $\dag$-2-functor 
\begin{align*}
\rmB \cC &\to \Mod^\dag(\cD)\\
\star &\mapsto \cM_{\cD}\\
(c \colon \star \to \star) &\mapsto (c \triangleright - \colon \cM_{\cD} \to \cM_{\cD})
\end{align*}
In particular, recall that a right $\cD$-module trace on $\cM$ is a $\cC$-$\cD$ bimodule trace exactly when $c \triangleright - \dashv_\dag c^\vee \triangleright -$ for all $c \in \cC$.
This observation yields an equivalence of $\rmC^*$-2-categories 
$$
\Fun^{\dag,\vee}(\rmB\cC \to \Mod^\dag(\cD)) \cong \sH^*\mFC(\cC \to \cD).
$$
(We conjecture this to be an isometric equivalence of 3-Hilbert spaces. 
We leave this to future work.)

Moreover, by Proposition \ref{prop:UniversalPropertyOfHstarAlgCompletion} we know
$$
\Fun^{\dag,\vee}(\Mod^\dag(\cC) \to \Mod^\dag(\cD)) \cong \Fun^{\dag,\vee}(\rmB\End(\cC_{\cC}) \to \Mod^\dag(\cD))
$$
given by restriction to the subcategory $\rmB\End(\cC_{\cC})$ of $\Mod^\dag(\cC)$. 
Putting this all together, this defines the pseudo-inverse
$$
\Fun^{\dag,\vee}(\Mod^\dag(\cC) \to \Mod^\dag(\cD)) \to \sH^*\mFC(\cC \to \cD).
$$
We will only unpack this equivalence at the level of functors and bimodules, as the desired facts for higher morphisms follow similarly. 
Indeed, given some $F \colon \Mod^\dag(\cC) \to \Mod^\dag(\cD)$, we determine a $\cC$-$\cD$ bimodule $_{\cC}F(\cC)_{\cD}$ by 
$$
\begin{tikzcd}[row sep=0]
\rmB\cC 
\arrow[r]
&
\rmB\End(\cC_\cC) 
\arrow[r,hook]
&
\Mod^\dag(\cC) 
\arrow[r,"F"]
&
\Mod^\dag(\cD)
\\ 
\star
\arrow[r,mapsto]
&
\star
\arrow[r,mapsto]
&
\cC_\cC 
\arrow[r,mapsto]
&
F(\cC)_\cD
\end{tikzcd}
$$
By Corollary \ref{cor:RightCActionIsometric}, $F(\cC)_\cD \cong \cC \boxtimes_{\cC} F(\cC)_\cD$ isometrically, hence $F \cong - \boxtimes_\cC F(\cC)_\cD$ when restricted to $\rmB\End(\cC_\cC)$. 
The universal property of $\Mod^\dag(\cC) = \sH^*\Alg(\cC)$ then implies $F \cong - \boxtimes_\cC F(\cC)_\cD$ isometrically. 
\end{proof}

\subsection{The forgetful 3-functor \texorpdfstring{$3\Hilb\to 3\Vect$}{3Hilb->3Vect} is fully faithful}

In this section, we prove Corollary \ref{coralpha:ForgetFF} which states that the forgetful functor $3\Hilb \to 3\Vect$ is fully faithful.

The unitary equivalence $\cM\cong \cM_1\boxtimes_{\cC_{11}}\cC$ of right $\cC$-module categories used in the proof of Corollary \ref{cor:ExistsModuleTrace} can also be used to prove that every module category over a unitary multifusion category is unitarizable.
The main ingredient is the fact that every connected separable algebra in a unitary fusion category is equivalent to a Q-system by \cite{MR4616673}.

\begin{prop}[{\cite{MR4724964}}]
\label{prop:ModulesUnitarizable}
Every finite semisimple module category $\cM$ over a unitary multifusion category $\cC$ is uniquely unitarizable.
\end{prop}
\begin{proof}
The proof proceeds in steps on properties of the module category $\cM$ and the unitary multifusion category $\cC$.
\item[\underline{Step 1:}]
If $\cC$ is fusion and $\cM$ is indecomposable,
then $\cM$ is unitarizable by \cite{MR4616673}.
Indeed, $\cM$ is equivalent to ${}_A\cC$ where $A=\underline{\End}_\cC(m)$ for any simple object $m\in\cM$.
Since $m$ is simple, $A$ is connected, i.e., $\cC(1_\cC\to A)\cong \cM(m\to m)=\bbC$.
By \cite{MR4616673}, $A$ is equivalent to a Q-system, so there is a Q-system $B\in\cC$ and an equivalence $\cM\cong {}_B\cC$, which is unitary.
    
\item[\underline{Step 2:}]
If $\cC$ is fusion and $\cM$ is arbitrary, then $\cM$ is a direct sum of indecomposable module categories.
Each of these indecomposable summands is unitarizable, so $\cM$ is too.

\item[\underline{Step 3:}]
If $\cC$ is indecomposable unitary multifusion and $\cM$ is indecomposable, then
$\cC$ is unitarily Morita equivalent to any of its unitary fusion corners of the form $\cC_{ii}=1_i\otimes \cC\otimes 1_i$ for a simple $1_i\subseteq 1_\cC$.
Consider $\cM_1:= \cM\otimes 1_1 \in \Mod(\cC_{11})$.
Since $\cC_{11}$ is unitary fusion, every module category is unitarizable by Step 2, so we may assume $\cM_1$ is unitary.
Since $\cM$ is indecomposable, $\cM$ is equivalent as a right $\cC$-module category to the unitary relative Deligne product $\cM_1\boxtimes_{\cC_{11}}\cC$ \cite{MR3975865}, \cite[Prop.~3.12]{MR4598730}, which is again unitary.
Hence $\cM$ is unitarizable.

\item[\underline{Step 4:}]
If $\cC$ is multifusion and $\cM$ is arbitrary, then $\cM$ is a direct sum of indecomposable module categories.
Each of these indecomposable summands is unitarizable, so $\cM$ is too.

\item[\underline{Uniqueness:}]
This follows from \cite[Rem.~4]{MR4538281} which also applies to Theorem 8 therein.
\end{proof}

By the folding trick, we get the following immediate corollary.

\begin{cor}
\label{cor:BimodulesUniquelyUnitarizable}
For unitary multifusion categories $\cC,\cD$, every finite semisimple $\cC$-$\cD$ bimodule category is uniquely unitarizable.
\end{cor}

We also have the following immediate corollary about unitarizability of finite semisimple 2-categories.

\begin{cor}
If a finite semisimple 2-category $\fX$ is unitarizable (is equivalent to the underlying 2-category of a $\rmC^*$ 2-category),
then it is uniquely unitarizable.
That is, if $\fY,\fZ$ are two $\rmC^*$ 2-categories whose underlying 2-categories  are equivalent to $\fX$, then $\fY$ is unitarily equivalent to $\fZ$.
\end{cor}
\begin{proof}
By \cite[Thm.~1.4.9]{1812.11933}, 
there is a multifusion category $\cC$ such that $\fX$ is equivalent to $\Mod(\cC)$.
Observe that $\Mod(\cC)$ is unitarizable if and only if $\cC$ is unitarizable.
Indeed, if $\Mod(\cC)$ is unitarizable, then so is $\End(\cC_\cC)\cong \cC$.
Conversely, if $\cC$ is unitarizable, then $\Mod^\dag(\cC)$ is equivalent to $\Mod(\cC)$ by Proposition \ref{prop:ModulesUnitarizable}.
By \cite[Rem.~4]{MR4538281}, a unitarizable multifusion category is uniquely unitarizable.
The result now follows from the equivalence $\sH^*\mFC\cong 3\Hilb$.
\end{proof}

Proposition \ref{prop:ModulesUnitarizable} affords a second quick proof of the corollary from \cite{MR4724964} that every separable algebra in a unitary multifusion category is equivalent to a Q-system.
Our proof relies on Sub-Example \ref{ex:H*AlgOnEnd(m)} which used unitary adjunction to prove that the internal end $[m,m]_\cC$ has a canonical $\rmH^*$-algebra structure.
We remark that unitary adjunction is implicitly used in the proof of \cite{MR4724964}, where certain adjunction isomorphims are first shown to be unitary; our approach is manifestly unitary.

\begin{lem}
\label{lem:QSystemIffUnitarizable}
A separable algebra $A$ in a unitary multifusion category $\cC$ is equivalent to a standard Q-system if and only if its category of right modules $\cC_A$ is unitarizable.
\end{lem}
\begin{proof}
For a Q-system $A$, $\cC_A$ is well-known to be unitarizable, e.g., see \cite[\S3.2]{MR4419534}.
On the other hand, if $A$ is an algebra in $\cC$ such that $\cC_A$ is unitarizable, there exists a faithful $\cC$-module trace on $\cC_A$ by Corollary \ref{cor:ExistsModuleTrace}.
By Sub-Example \ref{ex:H*AlgOnEnd(m)}, the internal end 
$[A,A]_\cC$ has a canonical $\rmH^*$-algebra structure, and it is well-known to be isomorphic as an algebra to $A$ (perhaps non-unitarily).
Finally, $[A,A]_\cC$ is equivalent to a Q-system by Fact \ref{rem:HStarAlgsQSystems}.
\end{proof}

\begin{cor}
\label{cor:AlgebraModulesUnitarizable}
Every separable algebra is equivalent to a standard Q-system.
\end{cor}
\begin{proof}
Immediate from Proposition \ref{prop:ModulesUnitarizable} and Lemma \ref{lem:QSystemIffUnitarizable}.
\end{proof}

\begin{lem}
\label{lem:ForgetBimoduleFunctorsFullyFaithful}
Given any two finite semisimple unitary $\cC$-$\cD$ bimodule categories $\cM,\cN$, the forgetful functor
$\Fun_{\cC-\cD}^{\dag}(\cM\to \cN) \to \Fun_{\cC-\cD}(\cM\to \cN)$ is an equivalence.
\end{lem}

\begin{proof}
By the folding trick, it suffices to prove the forgetful functor
$\Fun_{\cC}^{\dag}(\cM_\cC\to \cN_\cC) \to \Fun_{\cC}(\cM_\cC\to \cN_\cC)$
is an equivalence for right $\cC$-modules.
If we only cared about $\cC$-module equivalences, this would follow directly from \cite[Prop.~2.11]{MR4538281}, but the proof there does not obviously generalize to general $\cC$-module functors.
Instead, we note that the category of (unitary) $\cC$-module functors is a model for the (unitary) relative Deligne product \cite[\S3]{MR2677836}, and we have a commuting square
$$
\begin{tikzcd}
\cM\boxtimes^\dag_\cC \cN^{\rm op}
\arrow[d, "\simeq"']
\arrow[r,"\Forget"']
&
\cM\boxtimes_\cC \cN^{\rm op}
\arrow[d, "\simeq"]
\\
\Fun_{\cC}^\dag(\cM_\cC\to \cN_\cC)
\arrow[r,"\Forget"]
&
\Fun_{\cC}(\cM_\cC\to \cN_\cC)
\end{tikzcd}
$$
where the top left corner denotes the unitary relative Deligne product.
One can show that the top forgetful functor is an equivalence using the ladder category model of the relative Deligne product \cite{MR3975865} (see also \cite[\S3]{2307.13822}), and thus the bottom forgetful functor is also an equivalence.
\end{proof}

\begin{cor}
\label{cor:ForgetH*mFC->mFC is FullyFaithful}
The forgetful functor $\sH^*\mFC\to \mFC$ is fully faithful.
\end{cor}
\begin{proof}
Suppose $\cC,\cD$ are $\rmH^*$-multifusion categories.
The 2-functor $\sH^*\mFC(\cC\to \cD) \to \mFC(\cC\to \cD)$ is fully faithful by Lemma \ref{lem:ForgetBimoduleFunctorsFullyFaithful}
and essentially surjective by Corollary \ref{cor:BimodulesUniquelyUnitarizable}, which states that 
every finite semisimple $\cC$-$\cD$ bimodule category is uniquely unitarizable.
\end{proof}

\begin{rem}
The reader may be worried here that the existence of distinct $\cC$-$\cD$ bimodule traces on a $\cC$-$\cD$ bimodule category may be a counter-example to the forgetful functor $\sH^*\mFC\to \mFC$ being fully faithful. 
However, equivalence of 1-morphisms is determined by existence of a 2-isomorphism, and the 2-morphisms in $\sH^*\mFC$ are not required to be compatible with bimodule traces.
\end{rem}

\begin{proof}[Proof of Corollary \ref{coralpha:ForgetFF}]
By commutativity of the diagram
$$
\begin{tikzcd}
\sH^*\mFC
\arrow[r,"\Forget"']
\arrow[d, "\simeq"']
&
\mFC
\arrow[d, "\simeq"]
\\
3\Hilb
\arrow[r,"\Forget"]
&
3\Vect
\end{tikzcd}
$$
it suffices to prove that the top arrow is fully faithful.
This is exactly the content of Corollary \ref{cor:ForgetH*mFC->mFC is FullyFaithful}.
\end{proof}

\subsection{Is \texorpdfstring{$3\Hilb$}{3Hilb} enriched in 3-Hilbert spaces?}
\label{sec:OpenQuestions}

We end this section with an question on 3-Hilbert spaces that we will consider in the future.
In light of Example \ref{ex:FunCatsArePre2Hilbs} which gives an organic 2-Hilbert space structure on dagger functors between 2-Hilbert spaces, we ask:

\begin{quest}
Suppose $\fX$ is a pre-3-Hilbert space and $\fY$ is a 3-Hilbert space.
Is there an organic 3-Hilbert space structure on $\Fun^{\dag,\vee}(\fX\to \fY)$?
\end{quest}

Here are several desiderata for this putative 3-Hilbert space structure on $\Fun^{\dag,\vee}(\fX\to \fY)$.

\begin{itemize}
\item 
By the universal properties in Propositions \ref{prop:HilbertDirectSumUniversal} and \ref{prop:UniversalPropertyOfHstarAlgCompletion}, we have a unitary equivalence
$$
\Fun^\dag(\fX\to \fY)
\cong
\Fun^\dag(\HstarAlg(\Hilb_\boxplus(\fX))\to \fY),
$$
which preserves UAF-preserving functors and isometries,
so we may consider the case that $\fX$ is a 3-Hilbert space.
Given $\rmH^*$-multifusion categories $\cC,\cD$ and isometric equivalences
$\fX\cong \Mod^\dag(\cC)$ and $\fY\cong \Mod^\dag(\cD)$,
the unitary 2-equivalence 
$$
\Fun^{\dag,\vee}(\fX\to \fY)
\cong
\Fun^{\dag,\vee}(\Mod^\dag(\cC)\to \Mod^\dag(\cD))
\underset{\text{(\S\ref{sec:3Hilb==H*mFC})}}{\cong}
\mathsf{Bim}^\dag(\cC,\cD)
\cong
\Mod^\dag(\cC\boxtimes \cD^{\rm mp})
$$
should be an isometric equivalence where the latter category is equipped with the renormalized 3-Hilbert space structure from Example \ref{ex:WeightOnModC}.
\item 
Given a pre-3-Hilbert space $\fX$,
the Yoneda embedding $\fX\hookrightarrow \Fun^{\dag,\vee}(\fX^{\op}\to 2\Hilb)$ 
should be isometric.
Moreover, the Yoneda embedding should be an isometric equivalence if and only if $\fX$ is a 3-Hilbert space.
\item 
For a pre-3-Hilbert space $\fX$ and a 3-Hilbert space $\fY$, 
precomposition with $\fX\hookrightarrow \Hilb_\boxplus(\fX)$ should be an isometric equivalence
$$
\Fun^{\dag,\vee}(\fX\to \fY)
\cong
\Fun^{\dag,\vee}(\Hilb_\boxplus(\fX)\to \fY),
$$
and
precomposition with $\fX\hookrightarrow \HstarAlg(\fX)$ should be an isometric equivalence
$$
\Fun^{\dag,\vee}(\fX\to \fY)
\cong
\Fun^{\dag,\vee}(\HstarAlg(\fX)\to \fY).
$$
\end{itemize}

\section{\texorpdfstring{$\rmH^*$}{H*}-multifusion categories are the \texorpdfstring{$\rmH^*$}{H*}-algebras in \texorpdfstring{$\rmB2\Hilb$}{B2Hilb}}

In this section, we prove Theorem \ref{thmalpha:H*AlgebrasIn2Hilb} which asserts a unitary equivalence of $\rmC^*$ 3-categories $\sH^*\mFC\cong 3\Hilb$.

\subsection{\texorpdfstring{$E_1$}{E1}-algebras in \texorpdfstring{$\rmB2\Vect$}{B2Vect}}
\label{sec:E1algIn2Vect}

In this section, we assume $2\Vect$ is a $\Gray$-monoid by \cite{MR3076451}, and we work with the graphical calculus for $\Gray$-monoids from \cite{1409.2148}.

\begin{defn}
An \emph{algebra} $(A,\mu,\alpha)$ in a $\Gray$-monoid $\fC$\footnote{One can use the results of \cite{1903.05777} to apply graphical calculus to any weak monoidal 2-category.} consists of:
\begin{itemize}
\item
An object $A \in \fC$,
\item
A \emph{monoidal product} 1-morphism $\mu : A\xz A \to A$ denoted graphically by a trivalent vertex
$$
\mu =
\tikzmath{
\draw (-.3,-.3) arc (180:0:.3cm);
\draw (0,0) -- (0,.3);
\filldraw (0,0) circle (.05cm);
}
$$
\item
An \emph{associator} 2-isomorphism $\alpha: \mu \xo (\mu \xz \id_M) \Rightarrow \mu \xo (\id_M \xz \mu)$ such that the following diagram commutes:
\begin{equation}
\label{eq:2AlgebraAssociator}
\begin{tikzpicture}[baseline= (a).base]\node[scale=1] (a) at (0,0){\begin{tikzcd}
\tikzmath{
\draw (-.3,-.3) arc (180:0:.3cm);
\draw (0,0) arc (180:0:.3cm) -- (.6,-.3);
\draw (.3,.3) arc (180:0:.3cm) -- (.9,-.3);
\draw (.6,.6) -- (.6,.9);
\filldraw (0,0) circle (.05cm);
\filldraw (.3,.3) circle (.05cm);
\filldraw (.6,.6) circle (.05cm);
\draw[dashed, red, rounded corners = 5pt] (-.4,-.2) rectangle (.7,.45);
\draw[dashed, blue, rounded corners = 5pt] (-.2,.15) rectangle (1,.8);
}
\arrow[r, red, Rightarrow, "\alpha"]
\arrow[d, blue, Rightarrow, "\alpha"]
&
\tikzmath{
\draw (0,-.3) arc (180:0:.3cm);
\draw (-.3,-.3) -- (-.3,0) arc (180:0:.3cm);
\draw (0,.3) arc (180:0:.45cm) -- (.9,-.3);
\draw (.45,.75) -- (.45,1.05);
\filldraw (.3,0) circle (.05cm);
\filldraw (0,.3) circle (.05cm);
\filldraw (.45,.75) circle (.05cm);
\draw[dashed, rounded corners = 5pt] (-.4,.15) rectangle (1,.95);
}
\arrow[r,Rightarrow,"\alpha"]
&
\tikzmath[xscale=-1]{
\draw (0,-.3) arc (180:0:.3cm);
\draw (-.3,-.3) -- (-.3,0) arc (180:0:.3cm);
\draw (0,.3) arc (180:0:.45cm) -- (.9,-.3);
\draw (.45,.75) -- (.45,1.05);
\filldraw (.3,0) circle (.05cm);
\filldraw (0,.3) circle (.05cm);
\filldraw (.45,.75) circle (.05cm);
\draw[dashed, rounded corners = 5pt] (-.4,-.2) rectangle (.7,.45);
}
\arrow[d,Rightarrow,"\alpha"]
\\
\tikzmath{
\draw (-.3,-.3) arc (180:0:.3cm);
\draw (0,0) -- (0,.3) arc (180:0:.45cm);
\draw (.6,-.3) -- (.6,0) arc (180:0:.3cm) -- (1.2,-.3);
\draw (.45,.75) -- (.45,1.05);
\filldraw (0,0) circle (.05cm);
\filldraw (.9,.3) circle (.05cm);
\filldraw (.45,.75) circle (.05cm);
\draw[dashed, rounded corners = 5pt] (-.4,-.2) rectangle (1.3,.5);
}
\arrow[r,Rightarrow,"\phi^{-1}"]
&
\tikzmath[xscale=-1]{
\draw (-.3,-.3) arc (180:0:.3cm);
\draw (0,0) -- (0,.3) arc (180:0:.45cm);
\draw (.6,-.3) -- (.6,0) arc (180:0:.3cm) -- (1.2,-.3);
\draw (.45,.75) -- (.45,1.05);
\filldraw (0,0) circle (.05cm);
\filldraw (.9,.3) circle (.05cm);
\filldraw (.45,.75) circle (.05cm);
\draw[dashed, rounded corners = 5pt] (-.1,.15) rectangle (1.3,.95);
}
\arrow[r,Rightarrow,"\alpha"]
&
\tikzmath[xscale=-1]{
\draw (-.3,-.3) arc (180:0:.3cm);
\draw (0,0) arc (180:0:.3cm) -- (.6,-.3);
\draw (.3,.3) arc (180:0:.3cm) -- (.9,-.3);
\draw (.6,.6) -- (.6,.9);
\filldraw (0,0) circle (.05cm);
\filldraw (.3,.3) circle (.05cm);
\filldraw (.6,.6) circle (.05cm);
}
\end{tikzcd}};\end{tikzpicture}
\end{equation}
\end{itemize}
We call an algebra \emph{unital} if if there exist
\begin{itemize}
\item
A \emph{unit} 1-morphism $\iota: \mathbf{1}_\fC \to A$ denoted graphically by a univalent vertex
$$
\iota =
\tikzmath{
\draw (0,0) -- (0,.3);
\filldraw (0,0) circle (.05cm);
}
$$
\item
and \emph{unitor} 2-morphisms
$\lambda: \mu \xo (\iota \xz \id_A) \Rightarrow \id_A$ and $\rho: \mu\xo (\id_A \xz \iota)\Rightarrow \id_A$
such that the following diagram commutes:
\begin{equation}
\label{eq:2AlgebraUnit}
\begin{tikzpicture}[baseline= (a).base]\node[scale=1] (a) at (0,0){\begin{tikzcd}
\tikzmath{
\draw (-.3,-.6) -- (-.3,-.3) arc (180:0:.3cm);
\draw (0,0) arc (180:0:.3cm) -- (.6,-.6);
\draw (.3,.3) -- (.3,.6);
\filldraw (0,0) circle (.05cm);
\filldraw (.3,-.3) circle (.05cm);
\filldraw (.3,.3) circle (.05cm);
\draw[dashed, red, rounded corners = 5pt] (-.4,-.15) rectangle (.7,.45);
\draw[dashed, blue, rounded corners = 5pt] (-.4,-.5) rectangle (.45,.15);
}
\arrow[rr, red, Rightarrow,"\alpha"]
\arrow[dr, blue,Rightarrow, "\rho"']
&&
\tikzmath[xscale=-1]{
\draw (-.3,-.6) -- (-.3,-.3) arc (180:0:.3cm);
\draw (0,0) arc (180:0:.3cm) -- (.6,-.6);
\draw (.3,.3) -- (.3,.6);
\filldraw (0,0) circle (.05cm);
\filldraw (.3,-.3) circle (.05cm);
\filldraw (.3,.3) circle (.05cm);
\draw[dashed, rounded corners = 5pt] (-.4,-.5) rectangle (.45,.15);
}
\arrow[dl,Rightarrow,"\lambda"]
\\
&
\tikzmath{
\draw (-.3,-.3) arc (180:0:.3cm);
\draw (0,0) -- (0,.3);
\filldraw (0,0) circle (.05cm);
}
\end{tikzcd}};\end{tikzpicture}
\end{equation}
\end{itemize}
Observe that unitality of an algebra $A$ is a property and not additional structure; the space of choices of units and unitors is either empty or contractible.

We call an algebra \emph{rigid} if it is unital and $\mu$ admits a right adjoint $\mu^R: A \to A\xz A$ as an $A$-$A$ bimodule map.
We call the 
unit of this adjunction $\eta_\mu: \id_A \Rightarrow \mu^R\xo \mu$ and 
we call the counit of this adjuntion $\varepsilon_\mu: \mu\xo \mu^R \Rightarrow \id_A$.
The $A$-$A$ bimodularity means we have \emph{Frobeniator} natural isomorphisms 
\begin{equation}
\label{eq:Frobeniators}
\tikzmath{
\draw (-.4,.5) -- (-.4,0) arc (-180:0:.2cm) arc (180:0:.2cm) -- (.4,-.5);
\draw (-.2,-.2) -- (-.2,-.5);
\draw (.2,.2) -- (.2,.5);
\filldraw (-.2,-.2) circle (.05cm);
\filldraw (.2,.2) circle (.05cm);
}
\overset{\theta}{\Longrightarrow}
\tikzmath{
\draw (-.3,-.5) arc (180:0:.3cm);
\draw (-.3,.5) arc (-180:0:.3cm);
\draw (0,-.2) -- (0,.2);
\filldraw (0,-.2) circle (.05cm);
\filldraw (0,.2) circle (.05cm);
}
\qquad\text{and}\qquad
\tikzmath[xscale=-1]{
\draw (-.4,.5) -- (-.4,0) arc (-180:0:.2cm) arc (180:0:.2cm) -- (.4,-.5);
\draw (-.2,-.2) -- (-.2,-.5);
\draw (.2,.2) -- (.2,.5);
\filldraw (-.2,-.2) circle (.05cm);
\filldraw (.2,.2) circle (.05cm);
}
\overset{\kappa}{\Longrightarrow}
\tikzmath{
\draw (-.3,-.5) arc (180:0:.3cm);
\draw (-.3,.5) arc (-180:0:.3cm);
\draw (0,-.2) -- (0,.2);
\filldraw (0,-.2) circle (.05cm);
\filldraw (0,.2) circle (.05cm);
}\,,
\end{equation}
which endow $\mu^R$ with the structure of an $A$-$A$ bimodule functor such that $\eta_\mu,\varepsilon_\mu$ witness an adjunction of $A$-$A$ bimodule functors.
Observe that rigidity of an algebra is also a property and not additional structure; the space of choices of $\mu^R$ is contractible.

A rigid algebra is called \emph{separable} if for any choice of $\mu^R$, the counit $\varepsilon_\mu$ admits a splitting as an $A$-$A$ bimodule natural transformation, i.e., there exists $\delta : \id_A \Rightarrow \mu \xo \mu^R$ such that $\varepsilon \circ \delta=\id_{\id_A}$.
Observe that separability is \emph{independent} of the choice of $\mu^R$. 
\end{defn}

\begin{rem}
\label{rem:FrobeniusatorOverdetermined}
It may appear that choices of Frobeniators $\theta,\kappa$ for $\mu^R$ as in \eqref{eq:Frobeniators}  are additional structure.
It was pointed out to us by David Reutter that the Frobeniator $\kappa$ is determined by $\alpha$ and the adjunction data by the following commuting diagram:
$$
\begin{tikzpicture}[baseline= (a).base]\node[scale=1] (a) at (0,0){\begin{tikzcd}
\tikzmath{
\draw (.4,1.1) -- (.4,0) arc (0:-180:.2cm) arc (0:180:.2cm) -- (-.4,-.6);
\draw (.2,-.2) -- (.2,-.6);
\draw (-.2,.2) -- (-.2,1.1);
\filldraw (.2,-.2) circle (.05cm);
\filldraw (-.2,.2) circle (.05cm);
\draw[dashed, blue, rounded corners=5pt] (.6,.45) rectangle (-.4,.9);
\draw[dashed, red, rounded corners=5pt] (.6,-.4) rectangle (-.6,.35);
}
\arrow[rr, red, Rightarrow, "\kappa"]
\arrow[dd, blue, Rightarrow, "\eta"]
&&
\tikzmath{
\draw (-.3,.2) arc (180:0:.3cm);
\draw (-.3,2) -- (-.3,1.1) arc (-180:0:.3cm) -- (.3,2);
\draw (0,.5) -- (0,.8);
\filldraw (0,.5) circle (.05cm);
\filldraw (0,.8) circle (.05cm);
\draw[dashed, blue, rounded corners=5pt] (-.6,1.2) rectangle (.6,1.8);
}
\arrow[dl, blue, bend right=30, Rightarrow, "\eta"]
\\
&
\tikzmath{
\draw (-.3,-.9) arc (180:0:.3cm);
\draw (-.3,.9) arc (-180:0:.3cm);
\draw (0,0) circle (.3cm);
\draw (0,.3) -- (0,.6);
\draw (0,-.3) -- (0,-.6);
\filldraw (0,-.3) circle (.05cm);
\filldraw (0,-.6) circle (.05cm);
\filldraw (0,.3) circle (.05cm);
\filldraw (0,.6) circle (.05cm);
\draw[dashed, rounded corners=5pt] (-.6,-.45) rectangle (.6,.45);
}
\arrow[ur, bend right=30, Rightarrow, "\varepsilon"]
\arrow[dr, phantom, "\text{\tiny $\kappa$ an $\cA-\cA$ bimod nat iso}\qquad\qquad"]
\arrow[ur, phantom, "\text{\tiny zig-zag}"]
\\
\tikzmath{
\draw (.4,0) arc (0:-180:.2cm) arc (0:180:.2cm) -- (-.4,-.6);
\draw (.2,-.2) -- (.2,-.6);
\draw (-.2,.2) arc (180:0:.3cm) -- (.4,0);
\draw (.4,1.1) arc (0:-180:.3cm);
\draw (.1,.5) -- (.1,.8);
\filldraw (.2,-.2) circle (.05cm);
\filldraw (-.2,.2) circle (.05cm);
\filldraw (.1,.5) circle (.05cm);
\filldraw (.1,.8) circle (.05cm);
\draw[dashed, blue, rounded corners=5pt] (.6,0) rectangle (-.6,.65);
\draw[dashed, red, rounded corners=5pt] (.6,-.4) rectangle (-.6,.35);
}
\arrow[ur, red, Rightarrow, "\kappa"]
\arrow[rr, blue, Rightarrow, "\alpha"]
&&
\tikzmath{
\draw (.3,0) circle (.2cm);
\draw (.3,-.2) -- (.3,-.6);
\draw (.3,.2) arc (0:180:.3cm) -- (-.3,-.6);
\draw (-.3,1.1) arc (-180:0:.3cm);
\draw (0,.5) -- (0,.8);
\filldraw (.3,-.2) circle (.05cm);
\filldraw (.3,.2) circle (.05cm);
\filldraw (0,.5) circle (.05cm);
\filldraw (0,.8) circle (.05cm);
\draw[dashed, rounded corners=5pt] (.6,-.4) rectangle (0,.4);
}
\arrow[uu, Rightarrow, "\varepsilon"]
\end{tikzcd}};\end{tikzpicture}
$$
That the lower right triangle commutes is best seen by inverting the bottom associator arrow.
A similar statement holds for $\theta$.
\end{rem}

\begin{ex}
\label{ex:RigidAlgebrasIn2Vect}
The rigid algebras in $\rmB2\Vect$ are exactly the multifusion categories \cite[\S II]{MR4444089}.
We provide a complete proof below summarized in Corollary \ref{cor:Separable2AlgebrasIn2Vec}.
Construction \ref{const:MultiFusCatFromRigid2Alg} constructs a monoidal product on $|\cA|:=\Hom(\Vect\to \cA)$ for any algebra $\cA\in 2\Vect$, and it is straightforward to prove that the canonical equivalence $|\cA|\to \cA$ given by evaluation at $\bbC$ is monoidal.
When $\cA$ is rigid, we construct explicit evaluation and coevaluation morphisms on $|\cA|$ and we prove they satisfy the zig-zag identities in Proposition \ref{prop:Separable2AlgebrasInNiceGrayMonoid}.

Our main reason for providing the explicit detail below is that we will check later in Section \ref{sec:H*algInB2Hilb} that the duals constructed above give a UDF on $|\cA|$ in the unitary setting.
\end{ex}

Our construction bears many similarities with \cite{MR4535015}.

\begin{construction}
\label{const:MultiFusCatFromRigid2Alg}
Suppose $\fC$ is a $\Gray$-monoid and $A\in \fC$ is a unital algebra.
Define $|A|:= \fC(1_\fC \to A)$.
In the graphical calculus, we represent $A$ by a black strand, and we represent objects $a,b,c\in |A|$ by shaded disks with an $A$-strand emanating from the top:
$$
\tikzmath{
    \draw (0,0)  -- (0,.3);
    \filldraw[thick, fill=\aColor] (0,0) circle (.1cm);
}
:=a
\qquad
\qquad
\tikzmath{
    \draw (0,0)  -- (0,.3);
    \filldraw[thick, fill=\bColor] (0,0) circle (.1cm);
}
:=b
\qquad
\qquad
\tikzmath{
    \draw (0,0)  -- (0,.3);
    \filldraw[fill=\cColor, thick] (0,0) circle (.1cm);
}
:=c.
$$
We endow $|A|$ with a monoidal product by
$$ 
\tikzmath{
\draw (-.4,0) -- (-.4,.4);
\draw (.4,0) -- (.4,.4);
\filldraw[thick, fill=\aColor] (-.4,0) circle (.1cm);
\filldraw[thick, fill=\bColor] (.4,0) circle (.1cm);
\node at (0,.2) {$\times$};
}
\longmapsto
\tikzmath{
    \draw (-.3,-.1) -- (-.3,0) arc (180:0:.3cm) -- (.3,-.3);
    \draw (0,.3) -- (0,.6);
    \filldraw (0,.3) circle (.05cm);
    \filldraw[thick, fill=\aColor] (-.3,-.1) circle (.1cm);
    \filldraw[thick, fill=\bColor] (.3,-.3) circle (.1cm);
}\,,
$$
and we define an associator by
\begin{equation}
\label{eq:HomCatAssociator}
\tikzmath{
\draw (-.3,-.5) -- (-.3,-.3) arc (180:0:.3cm)  to (.3,-.7);
\draw (0,0) arc (180:0:.45) -- ++(0,-.8);
\draw (.45,.45) -- (.45,.7);
\filldraw (0,0) circle (.05cm);
\filldraw (.45,.45) circle (.05cm);
\filldraw[thick, fill=\aColor] (-.3,-.5) circle (.1cm);
\filldraw[thick, fill=\bColor] (0.3,-.7) circle (.1cm);
\filldraw[fill=\cColor, thick] (0.9,-.9) circle (.1cm);
\draw[rounded corners= 5pt, dashed, thick] (-.6, -.2) rectangle (1.2, .6);
}
\overset{\alpha}{\Longrightarrow}
\tikzmath{
\draw (.3,-.7) -- (.3,-.3) arc (180:0:.3cm) to (.9,-.9);
\draw (.6,0) arc (0:180:.45) -- ++(0,-.3);
\draw (.15,.45) -- (.15,.7);
\filldraw (.6,0) circle (.05cm);
\filldraw (.15,.45) circle (.05cm);
\filldraw[thick, fill=\aColor] (-.3,-.3) circle (.1cm);
\filldraw[thick, fill=\bColor] (0.3,-.7) circle (.1cm);
\filldraw[fill=\cColor, thick] (0.9,-.9) circle (.1cm);
\draw[rounded corners= 5pt, dashed, thick] (-.7, -.5) rectangle (1.2, .2);
}
\overset{\phi^{-1}}{\Longrightarrow}
\tikzmath{
\draw (.3,-.7) arc (180:0:.3cm) to (.9,-.9);
\draw (.6,-.4) -- (.6,0) arc (0:180:.45) -- ++ (0,0);
\draw (.15,.45) -- (.15,.7);
\filldraw (.6,-.4) circle (.05cm);
\filldraw (.15,.45) circle (.05cm);
\filldraw[thick, fill=\aColor] (-.3,0) circle (.1cm);
\filldraw[thick, fill=\bColor] (0.3,-.7) circle (.1cm);
\filldraw[fill=\cColor, thick] (0.9,-.9) circle (.1cm);
}\,.
\end{equation}
Observe that \eqref{eq:2AlgebraAssociator} implies that the associators \eqref{eq:HomCatAssociator} satisfy the pentagon axiom.

We define the unit object $1_{|A|} := \iota\in \fC(1_\fC\to A)$. 
The unitors are given by
$$
\tikzmath{
    \draw (-.3,-.1) -- (-.3,0) arc (180:0:.3cm)  -- (.3,-.5);
    \draw (0,.3) -- (0,.7);
    \draw[rounded corners=5pt, thick, dashed] (-.9,-.3) rectangle (.9,.5);
    \filldraw (0,.3) circle (.05cm);
    \filldraw (-.3,-.1) circle (.05cm);
    \filldraw[thick, fill=\aColor] (.3,-.5) circle (.1cm);
}
\,\,\,
\overset{\lambda}{\Longrightarrow}
\tikzmath{
\draw (-.4,0) -- (-.4,1);
\filldraw[thick, fill=\aColor] (-.4,0) circle (.1cm);
}
\qquad\qquad\text{and}\qquad\qquad
\tikzmath[xscale=-1]{
    \draw (-.3,-.5) -- (-.3,0) arc (180:0:.3cm) -- (.3,-.3);
    \draw (0,.3) -- (0,.7);
    \draw[rounded corners=5pt, thick, dashed] (-.9,-.8) rectangle (.9,0); 
    \filldraw (0,.3) circle (.05cm);
    \filldraw (-.3,-.5) circle (.05cm);
    \filldraw[thick, fill=\aColor] (.3,-.3) circle (.1cm);
}
\overset{\phi}{\Longrightarrow}\,\,\,
\tikzmath[xscale=-1]{
    \draw (-.3,-.1) -- (-.3,0) arc (180:0:.3cm)  -- (.3,-.5);
    \draw (0,.3) -- (0,.7);
    \draw[rounded corners=5pt, thick, dashed] (-.9,-.3) rectangle (.9,.5);
    \filldraw (0,.3) circle (.05cm);
    \filldraw (-.3,-.1) circle (.05cm);
    \filldraw[thick, fill=\aColor] (.3,-.5) circle (.1cm);
}
\,\,\,
\overset{\rho}{\Longrightarrow}
\tikzmath{
\draw (0,0) -- (0,1);
\filldraw[thick, fill=\aColor] (0,0) circle (.1cm);
}\,.
$$
Observe that \eqref{eq:2AlgebraUnit} implies that the unitors satisfy the triangle axiom.
Hence $|A|$ is a semisimple monoidal category.

Now suppose that 
$\fC$ admits right adjoints for all 1-morphisms
and
$A\in \fC$ is rigid.
Every $a\in |A|=\fC(1_\fC\to A)$ has a right adjoint $a^R\in \fC(A\to 1_\fC)$, which we denote by a shaded disk with a string emanating from the bottom:
$$
\tikzmath{
\draw (0,0) -- (0,-.5);
\filldraw[thick, fill=\aColor] (0,0) circle (.1cm);
}
:=
a^R
\qquad\qquad
\tikzmath{
\draw (0,0) -- (0,-.5);
\filldraw[thick, fill=\bColor] (0,0) circle (.1cm);
}
:=
b^R
\qquad\qquad
\tikzmath{
\draw (0,0) -- (0,-.5);
\filldraw[thick, fill=\cColor] (0,0) circle (.1cm);
}
:=
c^R.
$$
Using the graphical calculus for $\Gray$-monoids, we define left and right duals $a^\vee$ and ${}^\vee a$ of $a\in|A|$ by
\[
a^\vee := 
\tikzmath{
\draw (0,-.1) arc (0:-180:.2cm) -- (-.4,.4);
\filldraw[thick, fill=\aColor] (0,0) circle (.1cm);
\filldraw (-.2,-.3) circle (.05cm);
\draw (-.2,-.3) -- (-.2,-.5);
\filldraw (-.2,-.5) circle (.05cm);
}
\qquad\text{and}\qquad
{}^\vee a:= 
\tikzmath{
\draw (0,-.1) arc (-180:0:.2cm) -- (.4,.4);
\filldraw[thick, fill=\aColor] (0,0) circle (.1cm);
\filldraw (.2,-.3) circle (.05cm);
\draw (.2,-.3) -- (.2,-.5);
\filldraw (.2,-.5) circle (.05cm);
}
\,.
\]
The adjunction $a\dashv a^R$ allows the following evaluation and coevaluation morphisms to exhibit 
$a^\vee$ and ${}^\vee a$ as left and right duals of $a$ respectively.
\begin{align*}
\ev_a^L
&:
a^\vee\otimes a
=
\,\,\,
\tikzmath{
\draw (.6,-.9) -- (.6,-.2) arc (0:180:.4cm) -- (-.2,-.3) arc (-180:0:.2cm);
\draw (0,-.5) -- (0,-.7);
\filldraw[thick, fill=\aColor] (.2,-.2) circle (.1cm);
\filldraw[thick, fill=\aColor] (.6,-.9) circle (.1cm);
\filldraw (0,-.5) circle (.05cm);
\filldraw (.2,.2) circle (.05cm);
\draw (0,-.5) -- (0,-.7);
\draw (.2,.2) -- (.2,.5);
\filldraw (0,-.7) circle (.05cm);
\draw[dashed, rounded corners=5pt] (.8,-1.1) rectangle (-.3,0);
}
\,\,\,
\overset{\phi}{\Longrightarrow}
\,\,\,
\tikzmath{
\draw (.2,.3) arc (0:180:.2cm) -- (-.2,-.3) arc (-180:0:.2cm);
\draw (0,.5) -- (0,.7);
\draw (0,-.5) -- (0,-.7);
\filldraw[thick, fill=\aColor] (.2,-.2) circle (.1cm);
\filldraw[thick, fill=\aColor] (.2,.2) circle (.1cm);
\filldraw (0,-.5) circle (.05cm);
\filldraw (0,.5) circle (.05cm);
\filldraw (0,-.7) circle (.05cm);
\draw[dashed, rounded corners=5pt] (.4,-.4) rectangle (0,.4);
}
\,\,\,
\overset{\varepsilon_a}{\Longrightarrow}
\,\,\,
\tikzmath{
\draw (-.2,.3) arc (180:0:.2cm) -- (.2,-.3) arc (0:-180:.2cm) -- (-.2,.3);
\draw (0,.5) -- (0,.7);
\draw (0,-.5) -- (0,-.7);
\filldraw (0,-.5) circle (.05cm);
\filldraw (0,.5) circle (.05cm);
\filldraw (0,-.7) circle (.05cm);
\draw[dashed, rounded corners=5pt] (-.3,-.6) rectangle (.3,.6);
}
\,\,\,
\overset{\varepsilon_\mu}{\Longrightarrow}
\,\,\,
\tikzmath{
\draw (0,-.7) -- (0,.7);
\filldraw (0,-.7) circle (.05cm);
}
\,\,\,
=
1_{|A|}
\\
\coev_a^L
&:
1_{|A|}=
\tikzmath{
\draw (0,-.8) -- (0,.8);
\filldraw (0,-.8) circle (.05cm);
\draw[dashed, rounded corners=5pt] (.2,-.6) rectangle (.5,.6);
}
\,\,\,
\overset{\eta_a}{\Longrightarrow}
\,\,\,
\tikzmath{
\draw (0,-.8) -- (0,.8);
\filldraw (0,-.8) circle (.05cm);
\draw (.4,-.4) -- (.4,.4);
\filldraw[thick, fill=\aColor] (.4,-.5) circle (.1cm);
\filldraw[thick, fill=\aColor] (.4,.5) circle (.1cm);
\draw[dashed, rounded corners=5pt] (-.2,-.3) rectangle (.6,.3);
}
\,\,\,
\overset{\eta_\mu}{\Longrightarrow}
\,\,\,
\tikzmath{
\draw (.4,-.4) arc (0:180:.2cm) -- (0,-.8);
\draw (.4,.4) arc (0:-180:.2cm) -- (0,.8);
\draw (.2,.2) -- (.2,-.2);
\filldraw (.2,-.2) circle (.05cm);
\filldraw (.2,.2) circle (.05cm);
\filldraw (0,-.8) circle (.05cm);
\filldraw[thick, fill=\aColor] (.4,-.5) circle (.1cm);
\filldraw[thick, fill=\aColor] (.4,.5) circle (.1cm);
}
\,\,\,
\overset{\text{\eqref{eq:AbbreviatedCoheretor}}}{\Longrightarrow}
\,\,\,
\tikzmath{
\draw (0,-.1) arc (0:-180:.2cm) -- (-.4,.5) arc (0:180:.2cm);
\filldraw[thick, fill=\aColor] (0,0) circle (.1cm);
\filldraw[thick, fill=\aColor] (-.8,.4) circle (.1cm);
\draw (-.6,.7) -- (-.6,1);
\filldraw (-.6,.7) circle (.05cm);
\filldraw (-.2,-.3) circle (.05cm);
\draw (-.2,-.3) -- (-.2,-.5);
\filldraw (-.2,-.5) circle (.05cm);
}
=
a\otimes a^\vee
\end{align*}
Here, the exclamation point \eqref{eq:AbbreviatedCoheretor} in $\coev_a^L$ denotes the composite of coheretors
\begin{equation}
\label{eq:AbbreviatedCoheretor}
\tikzmath{
\draw (.4,-.4) arc (0:180:.2cm) -- (0,-.8);
\draw (.4,.4) arc (0:-180:.2cm) -- (0,.8);
\draw (.2,.2) -- (.2,-.2);
\filldraw (.2,-.2) circle (.05cm);
\filldraw (.2,.2) circle (.05cm);
\filldraw (0,-.8) circle (.05cm);
\filldraw[thick, fill=\aColor] (.4,-.5) circle (.1cm);
\filldraw[thick, fill=\aColor] (.4,.5) circle (.1cm);
\draw[dashed, rounded corners=5pt] (.6,-1) rectangle (-.2,-.3);
}
\,\,\,
\overset{\phi}{\Longrightarrow}
\,\,\,
\tikzmath{
\draw (.4,-.7) -- (.4,-.4) arc (0:180:.2cm) -- (0,-.5);
\draw (.4,.4) arc (0:-180:.2cm) -- (0,.8);
\draw (.2,.2) -- (.2,-.2);
\filldraw (.2,-.2) circle (.05cm);
\filldraw (.2,.2) circle (.05cm);
\filldraw (0,-.5) circle (.05cm);
\filldraw[thick, fill=\aColor] (.4,-.8) circle (.1cm);
\filldraw[thick, fill=\aColor] (.4,.5) circle (.1cm);
\draw[dashed, rounded corners=5pt] (.6,0) rectangle (-.2,-.65);
}
\,\,\,
\overset{\lambda}{\Longrightarrow}
\,\,\,
\tikzmath{
\draw (.4,.4) arc (0:-180:.2cm) -- (0,.8);
\draw (.2,.2) -- (.2,-.6);
\filldraw (.2,.2) circle (.05cm);
\filldraw[thick, fill=\aColor] (.2,-.6) circle (.1cm);
\filldraw[thick, fill=\aColor] (.4,.5) circle (.1cm);
\draw[dashed, rounded corners=5pt] (0,0) rectangle (.4,-.4);
}
\,\,\,
\overset{\rho^{-1}}{\Longrightarrow}
\,\,\,
\tikzmath{
\draw (.4,-.5) -- (.4,-.4) arc (0:180:.2cm) -- (0,-.8);
\draw (.4,.4) arc (0:-180:.2cm) -- (0,.8);
\draw (.2,.2) -- (.2,-.2);
\filldraw (.2,-.2) circle (.05cm);
\filldraw (.2,.2) circle (.05cm);
\filldraw (.4,-.5) circle (.05cm);
\filldraw[thick, fill=\aColor] (0,-.8) circle (.1cm);
\filldraw[thick, fill=\aColor] (.4,.5) circle (.1cm);
\draw[dashed, rounded corners=5pt] (-.2,-.4) rectangle (.6,.3);
}
\,\,\,
\overset{\kappa^{-1}}{\Longrightarrow}
\,\,\,
\tikzmath{
\draw (.8,.5) -- (.8,0) arc (0:-180:.2cm) arc (0:180:.2cm) -- (0,-.8);
\draw (.2,.2) -- (.2,.8);
\draw (.6,-.2) -- (.6,-.5);
\filldraw (.6,-.2) circle (.05cm);
\filldraw (.2,.2) circle (.05cm);
\filldraw (.6,-.5) circle (.05cm);
\filldraw[thick, fill=\aColor] (0,-.8) circle (.1cm);
\filldraw[thick, fill=\aColor] (.8,.5) circle (.1cm);
}
\,\,\,
\overset{\phi}{\Longrightarrow}
\,\,\,
\tikzmath{
\draw (0,-.1) arc (0:-180:.2cm) -- (-.4,.5) arc (0:180:.2cm);
\filldraw[thick, fill=\aColor] (0,0) circle (.1cm);
\filldraw[thick, fill=\aColor] (-.8,.4) circle (.1cm);
\draw (-.6,.7) -- (-.6,1);
\filldraw (-.6,.7) circle (.05cm);
\filldraw (-.2,-.3) circle (.05cm);
\draw (-.2,-.3) -- (-.2,-.5);
\filldraw (-.2,-.5) circle (.05cm);
}
\tag{!}
\end{equation}
Similarly, we can define evaluation and coevaluation for the right dual ${}^\vee a$, where a similar composite of coheretors is needed for $\coev_a^R$
\begin{align*}
\ev_a^R
&:
a\otimes {}^\vee a
=
\,\,\,
\tikzmath{
\draw (-.2,.3) arc (180:0:.2cm) -- (.2,-.3) arc (0:-180:.2cm);
\draw (0,.5) -- (0,.7);
\draw (0,-.5) -- (0,-.7);
\filldraw[thick, fill=\aColor] (-.2,-.2) circle (.1cm);
\filldraw[thick, fill=\aColor] (-.2,.2) circle (.1cm);
\filldraw (0,-.5) circle (.05cm);
\filldraw (0,.5) circle (.05cm);
\filldraw (0,-.7) circle (.05cm);
\draw[dashed, rounded corners=5pt] (-.4,-.4) rectangle (0,.4);
}
\,\,\,
\overset{\varepsilon_a}{\Longrightarrow}
\,\,\,
\tikzmath{
\draw (-.2,.3) arc (180:0:.2cm) -- (.2,-.3) arc (0:-180:.2cm) -- (-.2,.3);
\draw (0,.5) -- (0,.7);
\draw (0,-.5) -- (0,-.7);
\filldraw (0,-.5) circle (.05cm);
\filldraw (0,.5) circle (.05cm);
\filldraw (0,-.7) circle (.05cm);
\draw[dashed, rounded corners=5pt] (-.3,-.6) rectangle (.3,.6);
}
\,\,\,
\overset{\varepsilon_\mu}{\Longrightarrow}
\,\,\,
\tikzmath{
\draw (0,-.7) -- (0,.7);
\filldraw (0,-.7) circle (.05cm);
}
\,\,\,
=
1_{|A|}
\\
\coev_a^R
&:
1_{|A|}=
\tikzmath{
\draw (0,-.8) -- (0,.8);
\filldraw (0,-.8) circle (.05cm);
\draw[dashed, rounded corners=5pt] (-.2,-.6) rectangle (-.5,.6);
}
\,\,\,
\overset{\eta_a}{\Longrightarrow}
\,\,\,
\tikzmath{
\draw (0,-.8) -- (0,.8);
\filldraw (0,-.8) circle (.05cm);
\draw (-.4,-.4) -- (-.4,.4);
\filldraw[thick, fill=\aColor] (-.4,-.5) circle (.1cm);
\filldraw[thick, fill=\aColor] (-.4,.5) circle (.1cm);
\draw[dashed, rounded corners=5pt] (-.6,-.3) rectangle (.2,.3);
}
\,\,\,
\overset{\eta_\mu}{\Longrightarrow}
\,\,\,
\tikzmath{
\draw (-.4,-.4) arc (180:0:.2cm) -- (0,-.8);
\draw (-.4,.4) arc (-180:0:.2cm) -- (0,.8);
\draw (-.2,.2) -- (-.2,-.2);
\filldraw (-.2,-.2) circle (.05cm);
\filldraw (-.2,.2) circle (.05cm);
\filldraw (0,-.8) circle (.05cm);
\filldraw[thick, fill=\aColor] (-.4,-.5) circle (.1cm);
\filldraw[thick, fill=\aColor] (-.4,.5) circle (.1cm);
}
\,\,\,
\overset{\text{(!!)}}{\Longrightarrow}
\,\,\,
\tikzmath{
\draw (0,-.1) arc (-180:0:.2cm) -- (.4,.1) arc (180:0:.2cm) -- (.8,-.8);
\filldraw[thick, fill=\aColor] (0,0) circle (.1cm);
\filldraw[thick, fill=\aColor] (.8,-.8) circle (.1cm);
\draw (.6,.3) -- (.6,.6);
\filldraw (.6,.3) circle (.05cm);
\filldraw (.2,-.3) circle (.05cm);
\draw (.2,-.3) -- (.2,-.5);
\filldraw (.2,-.5) circle (.05cm);
}
=
{}^\vee a\otimes a
\end{align*}
for another similarly defined composite of coheretors (!!).
In Proposition \ref{prop:Separable2AlgebrasInNiceGrayMonoid} below, we prove that the zig-zag axiom is satisfied.
Thus $|A|$ is a rigid finitely semisimple monoidal category, i.e., a multifusion category.
\end{construction}

\begin{prop}
\label{prop:Separable2AlgebrasInNiceGrayMonoid}
Suppose $A$ is a rigid algebra in a $\Gray$-monoid $\fC$ in which all 1-morphisms admit right adjoints.
The monoidal category $|A|=\fC(1_\fC\to A)$ from Construction \ref{const:MultiFusCatFromRigid2Alg} above is rigid.
\end{prop}
\begin{proof}
All that remains is to verify the zig-zag axioms.
We explicitly prove the relation $(\id_{a}\otimes \ev^L_a)\circ (\coev^L_{a} \otimes \id_a)= \id_a$; the other 3 relations are left to the reader.

The zig-zag relation follows from commutativity of the diagram in Figure \ref{fig:RealizationDualsZigZag} below; the composite map along the outside of the diagram is the zig-zag formula, and each 2-cell commutes.
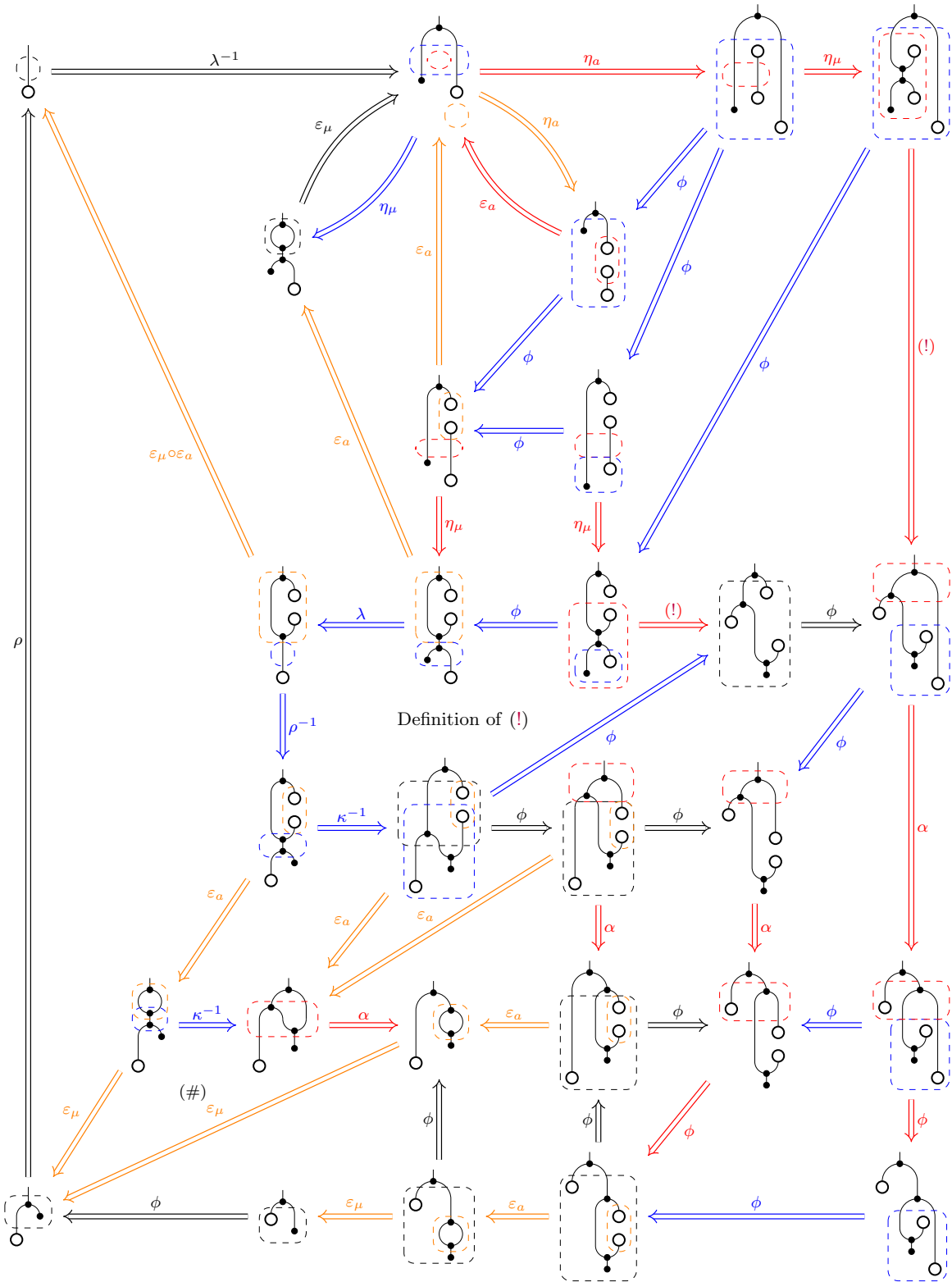
\begin{figure}
\begin{tikzpicture}[baseline= (a).base]\node[scale=.95] (a) at (0,0){\begin{tikzcd}
\tikzmath{
\draw (0,0) -- (0,.8);
\filldraw[thick, fill=\aColor] (0,0) circle (.1cm);
\draw[dashed, rounded corners=5pt] (-.2,.2) rectangle (.2,.6);
}
\arrow[rrr, Rightarrow, "\lambda^{-1}"]
&&&
\tikzmath{
\draw (-.5,.4) -- (-.5,1) arc (180:0:.3cm) -- (.1,.3);
\draw (-.2,1.3) -- (-.2,1.5);
\filldraw (-.2,1.3) circle (.05cm);
\filldraw (-.5,.4) circle (.05cm);
\filldraw[thick, fill=\aColor] (.1,.2) circle (.1cm);
\draw[orange, dashed, rounded corners=5pt] (-.1,0) rectangle (.3,-.4);
\draw[red, dashed, rounded corners=5pt] (-.4,.6) rectangle (0,.9);
\draw[blue, dashed, rounded corners=5pt] (-.7,.5) rectangle (.3,1);
}
\arrow[dl, blue, bend left=20, Rightarrow,"\eta_\mu"]
\arrow[rr, red, Rightarrow,"\eta_a"]
\arrow[dr, orange, bend left=20, Rightarrow,"\eta_a"]
&&
\tikzmath{
\draw (.6,-1.4) -- (.6,0) arc (0:180:.4cm) -- (-.2,-1.2);
\draw (.2,.4) -- (.2,.6);
\draw (.2,-.2) -- (.2,-1);
\filldraw[thick, fill=\aColor] (.2,-.2) circle (.1cm);
\filldraw[thick, fill=\aColor] (.2,-1) circle (.1cm);
\filldraw (.2,.4) circle (.05cm);
\filldraw (-.2,-1.2) circle (.05cm);
\filldraw[thick, fill=\aColor] (.6,-1.5) circle (.1cm);
\draw[dashed, red, rounded corners=5pt] (-.4,-.4) rectangle (.4,-.8);
\draw[dashed, blue, rounded corners=5pt] (-.5,-1.7) rectangle (.8,0);
}
\arrow[r, red, Rightarrow,"\eta_\mu"]
\arrow[ddl, blue, Rightarrow,"\phi"]
\arrow[dl, blue, Rightarrow,"\phi"]
&
\tikzmath{
\draw (.6,-1.4) -- (.6,0) arc (0:180:.4cm) -- (-.2,-.3) arc (-180:0:.2cm);
\draw (-.2,-1.2) -- (-.2,-.9) arc (180:0:.2cm);
\draw (.2,.4) -- (.2,.6);
\draw (0,-.5) -- (0,-.7);
\filldraw[thick, fill=\aColor] (.2,-.2) circle (.1cm);
\filldraw[thick, fill=\aColor] (.2,-1) circle (.1cm);
\filldraw (0,-.5) circle (.05cm);
\filldraw (.2,.4) circle (.05cm);
\filldraw (0,-.7) circle (.05cm);
\filldraw (-.2,-1.2) circle (.05cm);
\filldraw[thick, fill=\aColor] (.6,-1.5) circle (.1cm);
\draw[dashed, red, rounded corners=5pt] (-.4,-1.35) rectangle (.4,.1);
\draw[dashed, blue, rounded corners=5pt] (-.5,-1.7) rectangle (.8,.2);
}
\arrow[dddll, blue, Rightarrow,"\phi"]
\arrow[ddd, red, Rightarrow, "\text{\eqref{eq:AbbreviatedCoheretor}}"]
\\
&&
\tikzmath{
\draw (0,-.3) circle (.2cm);
\draw (0,-.1) -- (0,.1);
\draw (0,-.5) -- (0,-.7);
\filldraw (0,-.5) circle (.05cm);
\filldraw (0,-.1) circle (.05cm);
\filldraw (0,-.7) circle (.05cm);
\filldraw (-.2,-.9) circle (.05cm);
\draw (-.2,-.9) arc (180:0:.2cm) -- (.2,-1.1);
\filldraw[thick, fill=\aColor] (.2,-1.2) circle (.1cm);
\draw[dashed, rounded corners=5pt] (-.3,-.6) rectangle (.3,0);
}
\arrow[ur, bend left=20, Rightarrow, "\varepsilon_\mu"]
&&
\tikzmath{
\draw (-.2,-.2) -- (-.2,-.1) arc (180:0:.2cm) -- (.2,-.5);
\draw (0,.1) -- (0,.3);
\draw (.2,-1) -- (.2,-1.2);
\filldraw (-.2,-.2) circle (.05cm);
\filldraw (0,.1) circle (.05cm);
\filldraw[thick, fill=\aColor] (.2,-.5) circle (.1cm);
\filldraw[thick, fill=\aColor] (.2,-.9) circle (.1cm);
\filldraw[thick, fill=\aColor] (.2,-1.3) circle (.1cm);
\draw[red, dashed, rounded corners=5pt] (0,-1.1) rectangle (.4,-.3);
\draw[blue, dashed, rounded corners=5pt] (-.4,-1.5) rectangle (.5,0);
}
\arrow[ul,red, bend left=20, Rightarrow,"\varepsilon_a"]
\arrow[dl,blue, Rightarrow,"\phi"]
\\
&&
&
\tikzmath{
\draw (-.4,-.3) -- (-.4,.8) arc (180:0:.2cm);
\draw (-.2,1) -- (-.2,1.2);
\draw (0,-.5) -- (0,.2);
\filldraw (-.2,1) circle (.05cm);
\filldraw (-.4,-.3) circle (.05cm);
\filldraw[thick, fill=\aColor] (0,.7) circle (.1cm);
\filldraw[thick, fill=\aColor] (0,.3) circle (.1cm);
\filldraw[thick, fill=\aColor] (0,-.6) circle (.1cm);
\draw[orange, dashed, rounded corners=5pt] (-.2,.1) rectangle (.2,.9);
\draw[red, dashed, rounded corners=5pt] (-.6,-.2) rectangle (.2,.1);
}
\arrow[uu,orange, Rightarrow,"\varepsilon_a"]
\arrow[d, red, Rightarrow, "\eta_\mu"]
&
\tikzmath{
\draw (-.2,-1.8) -- (-.2,-.2) arc (180:0:.2cm);
\draw (0,0) -- (0,.2);
\draw (.2,-.8) -- (.2,-1.4);
\filldraw (-.2,-1.8) circle (.05cm);
\filldraw (0,0) circle (.05cm);
\filldraw[thick, fill=\aColor] (.2,-.3) circle (.1cm);
\filldraw[thick, fill=\aColor] (.2,-.7) circle (.1cm);
\filldraw[thick, fill=\aColor] (.2,-1.5) circle (.1cm);
\draw[red, dashed, rounded corners=5pt] (-.4,-.9) rectangle (.4,-1.3);
\draw[blue, dashed, rounded corners=5pt] (-.4,-1.9) rectangle (.4,-1.3);
}
\arrow[l, blue, Rightarrow, "\phi"]
\arrow[d, red, Rightarrow, "\eta_\mu", swap]
\\
&&
\tikzmath{
\draw (.2,.3) arc (0:180:.2cm) -- (-.2,-.3) arc (-180:0:.2cm);
\draw (0,.5) -- (0,.7);
\draw (0,-.5) -- (0,-1.1);
\filldraw[thick, fill=\aColor] (.2,-.2) circle (.1cm);
\filldraw[thick, fill=\aColor] (.2,.2) circle (.1cm);
\filldraw (0,-.5) circle (.05cm);
\filldraw (0,.5) circle (.05cm);
\filldraw[thick, fill=\aColor] (0,-1.2) circle (.1cm);
\draw[orange, dashed, rounded corners=5pt] (-.4,-.6) rectangle (.4,.6);
\draw[blue, dashed, rounded corners=5pt] (-.2,-.6) rectangle (.2,-1);
}
\arrow[uuull, orange, Rightarrow, near start, "\varepsilon_\mu\circ\varepsilon_a"]
\arrow[d, blue, Rightarrow, "\rho^{-1}"]
&
\tikzmath{
\draw (.2,.3) arc (0:180:.2cm) -- (-.2,-.3) arc (-180:0:.2cm);
\draw (0,.5) -- (0,.7);
\draw (0,-.5) -- (0,-.7);
\filldraw[thick, fill=\aColor] (.2,-.2) circle (.1cm);
\filldraw[thick, fill=\aColor] (.2,.2) circle (.1cm);
\filldraw (0,-.5) circle (.05cm);
\filldraw (0,.5) circle (.05cm);
\filldraw (0,-.7) circle (.05cm);
\filldraw (-.2,-.9) circle (.05cm);
\draw (-.2,-.9) arc (180:0:.2cm) -- (.2,-1.1);
\filldraw[thick, fill=\aColor] (.2,-1.2) circle (.1cm);
\draw[orange, dashed, rounded corners=5pt] (-.4,-.6) rectangle (.4,.6);
\draw[blue, dashed, rounded corners=5pt] (-.4,-.6) rectangle (.4,-1);
}
\arrow[l, blue, Rightarrow, "\lambda"']
\arrow[uul, orange, Rightarrow, "\varepsilon_a"]
&
\tikzmath{
\draw (.2,.3) arc (0:180:.2cm) -- (-.2,-.3) arc (-180:0:.2cm);
\draw (0,.5) -- (0,.7);
\draw (0,-.5) -- (0,-.7);
\filldraw[thick, fill=\aColor] (.2,-.2) circle (.1cm);
\filldraw[thick, fill=\aColor] (.2,.2) circle (.1cm);
\filldraw (0,-.5) circle (.05cm);
\filldraw (0,.5) circle (.05cm);
\filldraw (0,-.7) circle (.05cm);
\filldraw (-.2,-1.2) circle (.05cm);
\draw (-.2,-1.2) -- (-.2,-.9) arc (180:0:.2cm);
\filldraw[thick, fill=\aColor] (.2,-1) circle (.1cm);
\draw[dashed, blue, rounded corners=5pt] (-.4,-1.35) rectangle (.4,-.8);
\draw[dashed, red, rounded corners=5pt] (-.5,-1.45) rectangle (.5,0);
}
\arrow[l, blue, Rightarrow, "\phi"']
\arrow[r, red, Rightarrow, "\text{\eqref{eq:AbbreviatedCoheretor}}"]
&
\tikzmath{
\draw (0,-.1) arc (0:-180:.2cm) -- (-.4,.5) arc (0:180:.2cm);
\filldraw[thick, fill=\aColor] (0,0) circle (.1cm);
\filldraw[thick, fill=\aColor] (-.8,.4) circle (.1cm);
\filldraw[thick, fill=\aColor] (-.2,.9) circle (.1cm);
\draw (-.6,.7) -- (-.6,1) arc (180:0:.2cm);
\draw (-.4,1.2) -- (-.4,1.4);
\filldraw (-.4,1.2) circle (.05cm);
\filldraw (-.6,.7) circle (.05cm);
\filldraw (-.2,-.3) circle (.05cm);
\draw (-.2,-.3) -- (-.2,-.5);
\filldraw (-.2,-.5) circle (.05cm);
\draw[dashed, rounded corners=5pt] (-1,-.7) rectangle (.2,1.1);
}
\arrow[r, Rightarrow, "\phi"]
&
\tikzmath{
\draw (0,-.1) arc (0:-180:.2cm) -- (-.4,.5) arc (0:180:.2cm);
\filldraw[thick, fill=\aColor] (0,0) circle (.1cm);
\filldraw[thick, fill=\aColor] (-.8,.4) circle (.1cm);
\filldraw[thick, fill=\aColor] (.2,-.8) circle (.1cm);
\draw (-.6,.7) arc (180:0:.4cm) -- (.2,-.7);
\draw (-.2,1.1) -- (-.2,1.4);
\filldraw (-.2,1.1) circle (.05cm);
\filldraw (-.6,.7) circle (.05cm);
\filldraw (-.2,-.3) circle (.05cm);
\draw (-.2,-.3) -- (-.2,-.5);
\filldraw (-.2,-.5) circle (.05cm);
\draw[dashed, red, rounded corners=5pt] (-.9,.6) rectangle (.4,1.25);
\draw[dashed, blue, rounded corners=5pt] (-.6,.2) rectangle (.4,-1);
}
\arrow[dl, blue, Rightarrow, "\phi"]
\arrow[dd, red, Rightarrow, "\alpha"]
\\
&&
\tikzmath{
\draw (.2,.3) arc (0:180:.2cm) -- (-.2,-.3) arc (-180:0:.2cm);
\draw (0,.5) -- (0,.7);
\draw (0,-.5) -- (0,-.7);
\filldraw[thick, fill=\aColor] (.2,-.2) circle (.1cm);
\filldraw[thick, fill=\aColor] (.2,.2) circle (.1cm);
\filldraw (0,-.5) circle (.05cm);
\filldraw (0,.5) circle (.05cm);
\filldraw (0,-.7) circle (.05cm);
\filldraw (.2,-.9) circle (.05cm);
\draw (-.2,-1.1) -- (-.2,-.9) arc (180:0:.2cm);
\filldraw[thick, fill=\aColor] (-.2,-1.2) circle (.1cm);
\draw[orange, dashed, rounded corners=5pt] (0,-.4) rectangle (.4,.4);
\draw[blue, dashed, rounded corners=5pt] (-.4,-.4) rectangle (.4,-.8);
}
\arrow[dl, orange, Rightarrow, near start, "\varepsilon_a"']
\arrow[r, blue, Rightarrow, "{\kappa^{-1}}"]
&
\tikzmath{
\draw (.2,.7) arc (0:180:.3cm) -- (-.4,-.1);
\draw (-.6,-.3) arc (180:0:.2cm) -- (-.2,-.3) arc (-180:0:.2cm) -- (.2,.1);
\draw (-.1,1) -- (-.1,1.2);
\draw (0,-.5) -- (0,-.7);
\draw (-.6,-.9) -- (-.6,-.3);
\filldraw[thick, fill=\aColor] (-.6,-1) circle (.1cm);
\filldraw[thick, fill=\aColor] (.2,.2) circle (.1cm);
\filldraw[thick, fill=\aColor] (.2,.6) circle (.1cm);
\filldraw (-.4,-.1) circle (.05cm);
\filldraw (0,-.5) circle (.05cm);
\filldraw (-.1,1) circle (.05cm);
\filldraw (0,-.7) circle (.05cm);
\draw[dashed, rounded corners=5pt] (-.9,-.3) rectangle (.5,.8);
\draw[dashed, blue, rounded corners=5pt] (-.8,-1.2) rectangle (.4,.4);
\draw[dashed, orange, rounded corners=5pt] (0,0) rectangle (.4,.8);
}
\arrow[u, phantom, "\text{\qquad\scriptsize{Definition of \eqref{eq:AbbreviatedCoheretor}}}"]
\arrow[r, Rightarrow, "\phi"]
\arrow[urr, blue, Rightarrow, swap, "\phi"]
\arrow[dl, orange, Rightarrow, "\varepsilon_a"']
&
\tikzmath{
\draw (.2,.3) -- (.2,.5) arc (0:180:.3cm);
\draw (-.6,.3) arc (180:0:.2cm) -- (-.2,-.3) arc (-180:0:.2cm);
\draw (-.1,.8) -- (-.1,1.1);
\draw (0,-.5) -- (0,-.7);
\draw (-.6,-.9) -- (-.6,.3);
\filldraw[thick, fill=\aColor] (-.6,-1) circle (.1cm);
\filldraw[thick, fill=\aColor] (.2,-.2) circle (.1cm);
\filldraw[thick, fill=\aColor] (.2,.2) circle (.1cm);
\filldraw (-.4,.5) circle (.05cm);
\filldraw (0,-.5) circle (.05cm);
\filldraw (-.1,.8) circle (.05cm);
\filldraw (0,-.7) circle (.05cm);
\draw[dashed, red, rounded corners=5pt] (-.7,.4) rectangle (.4,1.05);
\draw[dashed, rounded corners=5pt] (-.8,.4) rectangle (.4,-1.2);
\draw[dashed, orange, rounded corners=5pt] (0,-.4) rectangle (.4,.4);
}
\arrow[d, red, Rightarrow, "\alpha"]
\arrow[r, Rightarrow, "\phi"]
\arrow[dll, orange, Rightarrow, "\varepsilon_a"']
&
\tikzmath{
\draw (0,-.5) arc (0:-180:.2cm) -- (-.4,.5) arc (0:180:.2cm);
\filldraw[thick, fill=\aColor] (0,-.4) circle (.1cm);
\filldraw[thick, fill=\aColor] (-.8,.4) circle (.1cm);
\filldraw[thick, fill=\aColor] (0,0) circle (.1cm);
\draw (-.6,.7) arc (180:0:.3cm) -- (0,.1);
\draw (-.3,1) -- (-.3,1.3);
\filldraw (-.3,1) circle (.05cm);
\filldraw (-.6,.7) circle (.05cm);
\filldraw (-.2,-.7) circle (.05cm);
\draw (-.2,-.7) -- (-.2,-.9);
\filldraw (-.2,-.9) circle (.05cm);
\draw[dashed, red, rounded corners=5pt] (-.9,.6) rectangle (.2,1.15);
}
\arrow[d, red, Rightarrow, "\alpha"]
\\
&
\tikzmath{
\draw (-.2,1) circle (.2cm);
\draw (0,.4) arc (0:180:.2cm) -- (-.4,0);
\draw (-.2,.6) -- (-.2,.8);
\draw (-.2,1.2) -- (-.2,1.4);
\filldraw (-.2,.6) circle (.05cm);
\filldraw (0,.4) circle (.05cm);
\filldraw (-.2,.8) circle (.05cm);
\filldraw (-.2,1.2) circle (.05cm);
\filldraw[thick, fill=\aColor] (-.4,-.1) circle (.1cm);
\draw[dashed, blue, rounded corners=5pt] (-.5,.5) rectangle (.1,.9);
\draw[dashed, orange, rounded corners=5pt] (-.5,.7) rectangle (.1,1.3);
}
\arrow[dl, orange, Rightarrow, "\varepsilon_\mu"']
\arrow[r, blue, Rightarrow, "\kappa^{-1}"]
&
\tikzmath{
\draw (.2,.3) -- (.2,.5) arc (0:180:.3cm);
\draw (-.2,.3) arc (-180:0:.2cm);
\draw (-.2,.3) arc (0:180:.2cm) -- (-.6,-.4);
\draw (-.1,.8) -- (-.1,1);
\draw (0,-.2) -- (0,.1);
\filldraw (0,.1) circle (.05cm);
\filldraw (-.4,.5) circle (.05cm);
\filldraw (0,-.2) circle (.05cm);
\filldraw (-.1,.8) circle (.05cm);
\filldraw[thick, fill=\aColor] (-.6,-.5) circle (.1cm);
\draw[dashed, red, rounded corners=5pt] (-.8,0) rectangle (.4,.6);
}
\arrow[r, red, Rightarrow, "\alpha"]
\arrow[dll,phantom,near start,"\scriptstyle(\#)"]
&
\tikzmath{
\draw (0,.3) circle (.2cm);
\draw (0,.5) arc (0:180:.3cm) -- (-.6,-.3);
\draw (-.3,.8) -- (-.3,1);
\draw (0,-.1) -- (0,.1);
\filldraw (0,-.1) circle (.05cm);
\filldraw (0,.5) circle (.05cm);
\filldraw (0,.1) circle (.05cm);
\filldraw (-.3,.8) circle (.05cm);
\filldraw[thick, fill=\aColor] (-.6,-.4) circle (.1cm);
\draw[dashed, orange, rounded corners=5pt] (-.3,0) rectangle (.3,.6);
}
\arrow[dlll, orange, Rightarrow, "\varepsilon_\mu"']
&
\tikzmath{
\draw (0,-.5) arc (0:-180:.2cm) -- (-.4,.1) arc (180:0:.2cm);
\filldraw[thick, fill=\aColor] (0,-.4) circle (.1cm);
\filldraw[thick, fill=\aColor] (-.8,-1.2) circle (.1cm);
\filldraw[thick, fill=\aColor] (0,0) circle (.1cm);
\draw (-.2,.3) arc (0:180:.3cm) -- (-.8,-1.1);
\draw (-.5,.6) -- (-.5,.8);
\filldraw (-.2,.3) circle (.05cm);
\filldraw (-.5,.6) circle (.05cm);
\filldraw (-.2,-.7) circle (.05cm);
\draw (-.2,-.7) -- (-.2,-.9);
\filldraw (-.2,-.9) circle (.05cm);
\draw[orange, dashed, rounded corners=5pt] (-.2,-.6) rectangle (.2,.2);
\draw[dashed, rounded corners=5pt] (-1,-1.4) rectangle (.3,.2);
}
\arrow[r, Rightarrow, "\phi"]
\arrow[l, orange, Rightarrow, "\varepsilon_a"']
&
\tikzmath{
\draw (0,-.5) arc (0:-180:.2cm) -- (-.4,.5) arc (180:0:.2cm)-- (0,.1);
\filldraw[thick, fill=\aColor] (0,-.4) circle (.1cm);
\filldraw[thick, fill=\aColor] (-.8,.4) circle (.1cm);
\filldraw[thick, fill=\aColor] (0,0) circle (.1cm);
\draw (-.8,.5) -- (-.8,.7) arc (180:0:.3cm);
\draw (-.5,1) -- (-.5,1.2);
\filldraw (-.5,1) circle (.05cm);
\filldraw (-.2,.7) circle (.05cm);
\filldraw (-.2,-.7) circle (.05cm);
\draw (-.2,-.7) -- (-.2,-.9);
\filldraw (-.2,-.9) circle (.05cm);
\draw[dashed, red, rounded corners=5pt] (-1,.2) rectangle (.2,.85);
}
\arrow[dl, red, Rightarrow, "\phi"]
&
\tikzmath{
\draw (0,-.1) arc (0:-180:.2cm) -- (-.4,.4) arc (180:0:.3cm) -- (.2,-.7);
\filldraw[thick, fill=\aColor] (0,0) circle (.1cm);
\filldraw[thick, fill=\aColor] (-.7,.4) circle (.1cm);
\filldraw[thick, fill=\aColor] (.2,-.8) circle (.1cm);
\draw (-.7,.5) -- (-.7,.7) arc (180:0:.3cm);
\draw (-.4,1) -- (-.4,1.2);
\filldraw (-.4,1) circle (.05cm);
\filldraw (-.1,.7) circle (.05cm);
\filldraw (-.2,-.3) circle (.05cm);
\draw (-.2,-.3) -- (-.2,-.5);
\filldraw (-.2,-.5) circle (.05cm);
\draw[dashed, blue, rounded corners=5pt] (-.6,.2) rectangle (.4,-1);
\draw[dashed, red, rounded corners=5pt] (-.9,.2) rectangle (.4,.85);
}
\arrow[l, blue, Rightarrow, "\phi"']
\arrow[d, red, Rightarrow, "\phi"]
\\
\tikzmath{
\draw (0,.4) arc (0:180:.2cm) -- (-.4,.1);
\draw (-.2,.6) -- (-.2,.9);
\filldraw (-.2,.6) circle (.05cm);
\filldraw (0,.4) circle (.05cm);
\filldraw[thick, fill=\aColor] (-.4,0) circle (.1cm);
\draw[dashed, rounded corners=5pt] (-.6,.75) rectangle (.2,.2);
}
\arrow[uuuuuu, Rightarrow, "\rho"]
&&
\tikzmath{
\draw (0,.1) -- (0,.4) arc (0:180:.2cm);
\draw (-.2,.6) -- (-.2,.8);
\filldraw (-.2,.6) circle (.05cm);
\filldraw (0,.1) circle (.05cm);
\filldraw[thick, fill=\aColor] (-.4,.3) circle (.1cm);
\draw[dashed, rounded corners=5pt] (-.6,-.1) rectangle (.2,.5);
}
\arrow[ll, Rightarrow, "\phi"']
&
\tikzmath{
\draw (0,-.1) circle (.2cm);
\draw (0,.1) -- (0,.5) arc (0:180:.3cm);
\draw (-.3,.8) -- (-.3,1);
\draw (0,-.3) -- (0,-.5);
\filldraw (0,-.5) circle (.05cm);
\filldraw (0,.1) circle (.05cm);
\filldraw (0,-.3) circle (.05cm);
\filldraw (-.3,.8) circle (.05cm);
\filldraw[thick, fill=\aColor] (-.6,.4) circle (.1cm);
\draw[dashed, rounded corners=5pt] (-.8,-.6) rectangle (.4,.7);
\draw[dashed, orange, rounded corners=5pt] (-.3,.2) rectangle (.3,-.4);
}
\arrow[l, orange, Rightarrow, "\varepsilon_\mu"']
\arrow[u, Rightarrow, "\phi"]
&
\tikzmath{
\draw (0,-.5) arc (0:-180:.2cm) -- (-.4,.1) arc (180:0:.2cm);
\filldraw[thick, fill=\aColor] (0,-.4) circle (.1cm);
\filldraw[thick, fill=\aColor] (-.8,.5) circle (.1cm);
\filldraw[thick, fill=\aColor] (0,0) circle (.1cm);
\draw (-.8,.6) arc (180:0:.3cm) -- (-.2,.3);
\draw (-.5,.9) -- (-.5,1.1);
\filldraw (-.2,.3) circle (.05cm);
\filldraw (-.5,.9) circle (.05cm);
\filldraw (-.2,-.7) circle (.05cm);
\draw (-.2,-.7) -- (-.2,-.9);
\filldraw (-.2,-.9) circle (.05cm);
\draw[orange, dashed, rounded corners=5pt] (-.2,-.6) rectangle (.2,.2);
\draw[dashed, rounded corners=5pt] (-1,-1.1) rectangle (.3,.7);
}
\arrow[l, orange, Rightarrow, "\varepsilon_a"']\arrow[u, Rightarrow, "\phi"]
&&
\tikzmath{
\draw (0,-.1) arc (0:-180:.2cm) -- (-.4,.1) arc (180:0:.3cm) -- (.2,-.7);
\filldraw[thick, fill=\aColor] (0,0) circle (.1cm);
\filldraw[thick, fill=\aColor] (-.7,.6) circle (.1cm);
\filldraw[thick, fill=\aColor] (.2,-.8) circle (.1cm);
\draw (-.7,.7) arc (180:0:.3cm) -- (-.1,.4);
\draw (-.4,1) -- (-.4,1.2);
\filldraw (-.4,1) circle (.05cm);
\filldraw (-.1,.4) circle (.05cm);
\filldraw (-.2,-.3) circle (.05cm);
\draw (-.2,-.3) -- (-.2,-.5);
\filldraw (-.2,-.5) circle (.05cm);
\draw[dashed, blue, rounded corners=5pt] (-.6,.2) rectangle (.4,-1);
}
\arrow[ll, blue, Rightarrow, "\phi"']
\end{tikzcd}};\end{tikzpicture}
\caption{\label{fig:RealizationDualsZigZag} Pasting diagram showing that $\vee$ on $|A|$ satisfies the zig-zag relations.}
\end{figure}
The square marked $(\#)$ commutes as $\varepsilon_\mu$ is an $A$-$A$ bimodule map.
This means that going along the outside of the diagram (the zig-zag relation) is equal to the following composite of 2-morphisms
$$
\tikzmath{
\draw (0,0) -- (0,.8);
\filldraw[thick, fill=\aColor] (0,0) circle (.1cm);
\draw[dashed, rounded corners=5pt] (-.2,.2) rectangle (.2,.6);
}
\overset{\lambda^{-1}}{\Longrightarrow}
\tikzmath{
\draw (-.2,0) arc (180:0:.2cm) -- (.2,-.2);
\draw (0,.2) -- (0,.4);
\filldraw (-.2,0) circle (.05cm);
\filldraw (0,.2) circle (.05cm);
\filldraw[thick, fill=\aColor] (.2,-.3) circle (.1cm);
\draw[dashed, rounded corners=5pt] (0,-.5) rectangle (.4,-.9);
}
\overset{\eta_a}{\Longrightarrow}
\tikzmath{
\draw (-.2,-.2) -- (-.2,-.1) arc (180:0:.2cm) -- (.2,-.5);
\draw (0,.1) -- (0,.3);
\draw (.2,-1) -- (.2,-1.2);
\filldraw (-.2,-.2) circle (.05cm);
\filldraw (0,.1) circle (.05cm);
\filldraw[thick, fill=\aColor] (.2,-.5) circle (.1cm);
\filldraw[thick, fill=\aColor] (.2,-.9) circle (.1cm);
\filldraw[thick, fill=\aColor] (.2,-1.3) circle (.1cm);
\draw[dashed, rounded corners=5pt] (0,-1.1) rectangle (.4,-.3);
}
\overset{\varepsilon_a}{\Longrightarrow}
\tikzmath{
\draw (-.4,.4) -- (-.4,1) arc (180:0:.2cm) -- (0,.3);
\draw (-.2,1.2) -- (-.2,1.4);
\filldraw (-.2,1.2) circle (.05cm);
\filldraw (-.4,.4) circle (.05cm);
\filldraw[thick, fill=\aColor] (0,.2) circle (.1cm);
\draw[dashed, rounded corners=5pt] (-.6,.5) rectangle (.2,1);
}
\overset{\eta_\mu}{\Longrightarrow}
\tikzmath{
\draw (0,-.3) circle (.2cm);
\draw (0,-.1) -- (0,.1);
\draw (0,-.5) -- (0,-.7);
\filldraw (0,-.5) circle (.05cm);
\filldraw (0,-.1) circle (.05cm);
\filldraw (0,-.7) circle (.05cm);
\filldraw (-.2,-.9) circle (.05cm);
\draw (-.2,-.9) arc (180:0:.2cm) -- (.2,-1.1);
\filldraw[thick, fill=\aColor] (.2,-1.2) circle (.1cm);
\draw[dashed, rounded corners=5pt] (-.3,-.6) rectangle (.3,0);
}
\overset{\varepsilon_\mu}{\Longrightarrow}
\tikzmath{
\draw (-.2,0) arc (180:0:.2cm) -- (.2,-.3);
\draw (0,.2) -- (0,.6);
\filldraw (-.2,0) circle (.05cm);
\filldraw (0,.2) circle (.05cm);
\filldraw[thick, fill=\aColor] (.2,-.4) circle (.1cm);
\draw[dashed, rounded corners=5pt] (-.4,-.2) rectangle (.4,.4);
}
\overset{\lambda}{\Longrightarrow}
\tikzmath{
\draw (0,0) -- (0,.8);
\filldraw[thick, fill=\aColor] (0,0) circle (.1cm);
\draw[dashed, rounded corners=5pt] (-.2,.2) rectangle (.2,.6);
}
$$
which is the identity as claimed.
\end{proof}

We have the following immediate corollary.

\begin{cor}
Suppose $A$ is a rigid algebra in a $\Gray$-monoid $\fC$ in which all hom categories are finitely semisimple linear categories, and all 1-morphisms admit right adjoints.
The semisimple tensor category $|A|=\fC(1_\fC\to A)$ from Construction \ref{const:MultiFusCatFromRigid2Alg} above is a multifusion category.
When $A$ is \emph{connected} (the unit $\iota\in \fC(1_\fC \to A)$ is simple), $|A|$ is fusion.
\end{cor}

To connect Construction \ref{const:MultiFusCatFromRigid2Alg} with the classification of rigid algebras in $2\Vect$, we note that for an algebra $(\cA,\mu, \alpha, 1_\cA, \lambda,\rho) \in 2\Vect$, 
we have a monoidal equivalence
$\cA \cong |\cA|:=\Fun(\Vect \to \cA)$
intertwining the monoidal structure $\mu$ and the monoidal structure from Construction \ref{const:MultiFusCatFromRigid2Alg}.

\begin{cor}
\label{cor:Separable2AlgebrasIn2Vec}
Rigid algebras $(\cA,\mu, \alpha, 1_\cA, \lambda,\rho) \in 2\Vect$ are exactly multifusion categories.
\end{cor}
\begin{proof}
Every multifusion category gives a rigid algebra in $2\Vect$ as the 2-category of $\cA$-$\cA$ bimodules, bimodule functors, and bimodule natural transformations admits adjoints.
The converse direction is exactly Construction \ref{const:MultiFusCatFromRigid2Alg} above, noting that if $\cA\in 2\Vect$, then $\cA\cong |\cA|$.
\end{proof}

\subsection{\texorpdfstring{$\rmH^*$}{H*}-algebras in \texorpdfstring{$\rmB2\Hilb$}{B2Hilb}}
\label{sec:H*algInB2Hilb}
In this section, we assume $2\Hilb$ is a $\sC^*\Gray$-monoid by \cite{2404.05193}; we refer the reader to that article for the definition.

We now give a definition of an $\rmH^*$-algebra in $\rmB2\Hilb$ in the more general setting of a $\sC^*\Gray$-monoid 
whose underlying strict $\rmC^*$ 2-category is a pre-3-Hilbert space.
We expect the following definition can be generalized to any unitary multifusion 2-category, a notion which will be defined in future work.

\begin{defn}\label{defn:Hstaralg-in-pre3hilb}
Suppose $\fC$ is a $\sC^*\Gray$-monoid 
whose underlying strict $\rmC^*$ 2-category is a pre-3-Hilbert space.
An $\rmH^*$-\emph{algebra} in $\fC$ is a unital algebra
$(A,\mu,\alpha, \iota,\lambda,\rho)$ 
such that
\begin{enumerate}[label=($\rmH^*$\arabic*)]
\setcounter{enumi}{-1} 
    \item (Unitary) the 2-isomorphisms $\alpha,\lambda,\rho$ are all unitary.
    \item 
    \label{H:UnitarilyRigid}
    (Unitarily rigid) the unitary adjoint $\mu^*$ of $\mu$ is an $A$-$A$ bimodule functor (which is a property by Remark \ref{rem:FrobeniusatorOverdetermined}) and the unit $\eta_\mu: \id_A \Rightarrow \mu^*\xo \mu$ and counit $\varepsilon_\mu: \mu\xo \mu^* \Rightarrow \id_A$ are $A$-$A$ bimodule natural transformations.
    \item 
    \label{H:UnitarilySeparable}
    (Unitarily separable) $\varepsilon_\mu^\dag$ is an $A$-$A$ bimodule natural transformation and the bubble on the $A$-sheet $\varepsilon_\mu\circ \varepsilon_\mu^\dag : \id_A\Rightarrow\id_A$ is invertible.
\item\label{H:SphereStandard}
(Standard) for all $\alpha:\id_A \Rightarrow \id_A$, we have
$$
\Psi^{\fC}_{1_\fC}\left(
\tikzmath{
\draw[rounded corners=5pt,dashed] (0,0) rectangle +(.4,.4);
}
\xRightarrow{\epsilon_\mu^\dag\circ\eta_\iota}
\tikzmath{
\draw[thick] (0,-.4) -- (0,-.6);
\draw[thick] (0,.4) -- (0,.6);
\filldraw (0,.6) circle (.05);
\filldraw (0,.4) circle (.05);
\filldraw (0,-.4) circle (.05);
\filldraw (0,-.6) circle (.05);
\draw[thick] (-.2,-.2) arc (-180:0:.2cm) -- (.2,.2) arc(0:180:.2cm) -- (-.2,-.2);
\draw[rounded corners=5pt,dashed] (-.4,-.2) rectangle (0,.2);
}
\xRightarrow{\alpha}
\tikzmath{
\draw[thick] (0,-.4) -- (0,-.6);
\draw[thick] (0,.4) -- (0,.6);
\filldraw (0,.6) circle (.05);
\filldraw (0,.4) circle (.05);
\filldraw (0,-.4) circle (.05);
\filldraw (0,-.6) circle (.05);
\draw[thick] (-.2,-.2) arc (-180:0:.2cm) -- (.2,.2) arc(0:180:.2cm) -- (-.2,-.2);
}
\xRightarrow{\eta_\iota^\dag\circ\epsilon_\mu}
\tikzmath{
\draw[rounded corners=5pt,dashed] (0,0) rectangle +(.4,.4);
}
\right)
=
\Psi^{\fC}_{1_\fC}\left(
\tikzmath{
\draw[rounded corners=5pt,dashed] (0,0) rectangle +(.4,.4);
}
\xRightarrow{\epsilon_\mu^\dag\circ\eta_\iota}
\tikzmath{
\draw[thick] (0,-.4) -- (0,-.6);
\draw[thick] (0,.4) -- (0,.6);
\filldraw (0,.6) circle (.05);
\filldraw (0,.4) circle (.05);
\filldraw (0,-.4) circle (.05);
\filldraw (0,-.6) circle (.05);
\draw[thick] (-.2,-.2) arc (-180:0:.2cm) -- (.2,.2) arc(0:180:.2cm) -- (-.2,-.2);
\draw[rounded corners=5pt,dashed] (.4,-.2) rectangle (0,.2);
}
\xRightarrow{\alpha}
\tikzmath{
\draw[thick] (0,-.4) -- (0,-.6);
\draw[thick] (0,.4) -- (0,.6);
\filldraw (0,.6) circle (.05);
\filldraw (0,.4) circle (.05);
\filldraw (0,-.4) circle (.05);
\filldraw (0,-.6) circle (.05);
\draw[thick] (-.2,-.2) arc (-180:0:.2cm) -- (.2,.2) arc(0:180:.2cm) -- (-.2,-.2);
}
\xRightarrow{\eta_\iota^\dag\circ\epsilon_\mu}
\tikzmath{
\draw[rounded corners=5pt,dashed] (0,0) rectangle +(.4,.4);
}
\right).
$$
\end{enumerate}
\end{defn}

\begin{rem}
Recall that the standardness condition \ref{H:Standard} in Section \ref{sec:H*Algs} is identical to the sphericality of $\psi$ with evaluation 
constructed from the counit and multiplcation ($\ev=i^\dag\mu$)
and the coevaluation 
constructed from the unit and comultiplcation ($\coev=\mu^\dag i$).
If we similarly modify the notion of sphericality in \cite[\S2.3]{1812.11933}, we obtain the condition \ref{H:SphereStandard} above. 
David Reutter has made the intriguing suggestion that the non-unitary analog of \ref{H:SphereStandard} may be automatic in a fusion 2-category.
\end{rem}

\begin{prop}\label{prop:hstarmfc-to-hstaralg}
Every $(\cC,\vee,\psi_\cC)\in\sH^*\mFC$ is an $\rmH^*$-algebra in $\rmB2\Hilb$.
\end{prop}
\begin{proof}
To see that $\cC$ is unitarily rigid, we note that the category of $\cC$-$\cC$ unitary bimodules equipped with bimodule traces, unitary bimodule functors, and bimodule natural transformations admits adjoints, and in particular, unitary adjoints.

To see that $\cC$ is unitarily separable, we 
work in the category $\End^\dag_{\cC\text{-}\cC}(\cC\oplus \cC\boxtimes \cC)$,
which is equivalent to $\End_{\cC\text{-}\cC}(\cC\oplus \cC\boxtimes \cC)$ by Lemma \ref{lem:ForgetBimoduleFunctorsFullyFaithful}.
Since every multifusion category is separable, $\varepsilon_\mu$ admits a right inverse $\delta_\mu: \id_{\cC}\Rightarrow \otimes \circ \otimes^*$ in the latter category, and thus in the former.
Since $\End^\dag_{\cC\text{-}\cC}(\cC\oplus \cC\boxtimes \cC)\cong \Hilb^{\oplus n}$ for some finite $n$,
the argument of Remark \ref{rem:RightInvIffBubbleInv} applies, i.e., $\varepsilon_\mu$ admits a right inverse if and only if $\varepsilon_\mu\circ \varepsilon_\mu^\dag$ is invertible.
\end{proof}

\begin{construction}
\label{const:DaggerStructure}
Suppose $\fC$ is a $\sC^*\Gray$-monoid and $A\in\fC$ is a unital algebra.
As in Construction \ref{const:MultiFusCatFromRigid2Alg}, we define the monoidal category
$|A|:=\Hom_\fC(1_\fC\to A)$.
We endow $|A|$ with a dagger structure given by the dagger in $\Hom_\fC(1_\fC\to A)$, i.e.,
$$
\left(
f:
\tikzmath{
    \draw (0,0)  -- (0,.3);
    \filldraw[thick, fill=\aColor] (0,0) circle (.1cm);
}
\Rightarrow
\tikzmath{
    \draw (0,0)  -- (0,.3);
    \filldraw[thick, fill=\bColor] (0,0) circle (.1cm);
}
\right)^\dag
:=
\left(
f^\dag :
\tikzmath{
    \draw (0,0)  -- (0,.3);
    \filldraw[thick, fill=\bColor] (0,0) circle (.1cm);
}
\Rightarrow
\tikzmath{
    \draw (0,0)  -- (0,.3);
    \filldraw[thick, fill=\aColor] (0,0) circle (.1cm);
}
\right)\,.
$$
One verifies that this dagger structure is monoidal for $|A|$ using the unitarity of $\phi,\alpha,\lambda,\rho$, together with the calculation
$$
(f\otimes g)^\dag
=
\left(
\tikzmath{
    \draw (-.3,-.1) -- (-.3,0) arc (180:0:.3cm) -- (.3,-.3);
    \draw (0,.3) -- (0,.6);
    \filldraw (0,.3) circle (.05cm);
    \filldraw[thick, fill=\aColor] (-.3,-.1) circle (.1cm);
    \filldraw[thick, fill=\bColor] (.3,-.3) circle (.1cm);
    \draw[dashed, rounded corners=5pt] (-.5, -.5) rectangle (.5,.1);
}
\overset{f\otimes g}{\Longrightarrow}
\tikzmath{
    \draw (-.3,-.1) -- (-.3,0) arc (180:0:.3cm) -- (.3,-.3);
    \draw (0,.3) -- (0,.6);
    \filldraw (0,.3) circle (.05cm);
    \filldraw[fill=\cColor, thick] (-.3,-.1) circle (.1cm);
    \filldraw[fill=red, thick] (.3,-.3) circle (.1cm);
}
\right)^\dag
=
\left(
\tikzmath{
    \draw (-.3,-.1) -- (-.3,0) arc (180:0:.3cm) -- (.3,-.3);
    \draw (0,.3) -- (0,.6);
    \filldraw (0,.3) circle (.05cm);
    \filldraw[fill=\cColor, thick] (-.3,-.1) circle (.1cm);
    \filldraw[fill=red, thick] (.3,-.3) circle (.1cm);
    \draw[dashed, rounded corners=5pt] (-.5, -.5) rectangle (.5,.1);
}
\overset{f^\dag\otimes g^\dag}{\Longrightarrow}
\tikzmath{
    \draw (-.3,-.1) -- (-.3,0) arc (180:0:.3cm) -- (.3,-.3);
    \draw (0,.3) -- (0,.6);
    \filldraw (0,.3) circle (.05cm);
    \filldraw[thick, fill=\aColor] (-.3,-.1) circle (.1cm);
    \filldraw[thick, fill=\bColor] (.3,-.3) circle (.1cm);
}
\right).
$$
We thus see that $|A|$ is a finite semisimple $\rmC^*$ monoidal category, and when $\fC$ admits right adjoints for all 1-morphisms, $|A|$ is a unitary multifusion category.
\end{construction}

\begin{prop}\label{prop:underlying-hstar-udf}
Suppose $\fC$ is a $\sC^*\Gray$-monoid 
whose underlying strict $\rmC^*$ 2-category is equipped with a UAF $*$
and $A\in\fC$ is an $\rmH^*$-algebra.
Using the UAF on $\fC$, the dual functor $\vee$ from Construction \ref{const:MultiFusCatFromRigid2Alg} on $|A|$ is a UDF under the dagger structure from Construction \ref{const:DaggerStructure}.
\end{prop}
\begin{proof}
We must prove that the dual functor $\vee$ is a dagger functor, and the canonical tensorator 
$
\vee^2_{a,b}: a^{\vee} \otimes_{\rm mop} b^\vee = b^\vee\otimes a^\vee \to (a\otimes b)^\vee
$
given (in the graphical calculus of a monoidal category) by
\begin{equation}
\label{eq:CanonicalDualTensorator}
\tikzmath{
	\draw (.8,.6) -- (.8,0) arc (0:-180:.4cm) arc (0:180:.2cm) -- (-.4,-.6);
	\draw (.7,.6) -- (.7,0) arc (0:-180:.3cm) arc (0:180:.5cm) -- (-.9,-.6);
	\node at (1.4,.4) {$\scriptstyle (a\otimes b)^\vee$};
	\node at (-.6,-.4) {$\scriptstyle a^\vee$};
	\node at (-1.1,-.4) {$\scriptstyle b^\vee$};
}
=
\begin{aligned}
(\ev_b\otimes \id_{(a\otimes b)^\vee})
&\circ 
(\id_{b^\vee}\otimes \ev_a \otimes \id_b \otimes \id_{(a\otimes b)^\vee})
\\&\circ
(\id_{b^\vee\otimes a^\vee} \otimes \coev_{a\otimes b})
\end{aligned}
\end{equation}
is unitary.
Checking both of these amount to showing very large pasting diagrams commute.

Since the underlying strict $\rmC^*$ 2-category of $\fC$ is equipped with a unitary (right) adjoint 2-functor $*$ (which is the identity on objects), for all 2-morphisms $f:X\Rightarrow Y$, we have $f^{*\dag}=f^{\dag *}$.
In the graphical calculus for the underlying 2-category $\fC$ in the context of the graphical calculus for $\Gray$-monoids from \cite{1409.2148}, this means the following diagram commutes for all $f:a\Rightarrow b$ for  $a,b: 1\to A$ in $\fC$:
\begin{equation}
\label{eq:UnitaryAdjointsIn2Cat}
\begin{tikzpicture}[baseline= (a).base]\node[scale=1] (a) at (0,0){\begin{tikzcd}
\tikzmath{
\draw (0,0) -- (0,-.8);
\filldraw[thick, fill=\aColor] (0,0) circle (.1cm);
\draw[dashed, red, rounded corners=5pt] (-.2,.2) rectangle (.2,.6);
\draw[dashed, rounded corners=5pt] (-.2,-.2) rectangle (.2,-.6);
}
\arrow[r, Rightarrow, "\varepsilon_b^\dag"]
\arrow[d, red, Rightarrow, "\eta_b"]
&
\tikzmath{
\draw (.2,.2) -- (.2,.5);
\draw (.2,1) -- (.2,1.2);
\filldraw[thick, fill=\bColor] (.2,.5) circle (.1cm);
\filldraw[thick, fill=\bColor] (.2,.9) circle (.1cm);
\filldraw[thick, fill=\aColor] (.2,1.3) circle (.1cm);
\draw[dashed, rounded corners=5pt] (0,1.1) rectangle (.4,.7);
}
\arrow[r, Rightarrow, "f^\dag"]
&
\tikzmath{
\draw (.2,.2) -- (.2,.5);
\draw (.2,1) -- (.2,1.2);
\filldraw[thick, fill=\bColor] (.2,.5) circle (.1cm);
\filldraw[thick, fill=\aColor] (.2,.9) circle (.1cm);
\filldraw[thick, fill=\aColor] (.2,1.3) circle (.1cm);
\draw[dashed, rounded corners=5pt] (0,1.5) rectangle (.4,.7);
}
\arrow[d, Rightarrow, "\eta_a^\dag"]
\\
\tikzmath{
\draw (.2,.2) -- (.2,.5);
\draw (.2,1) -- (.2,1.2);
\filldraw[thick, fill=\aColor] (.2,.5) circle (.1cm);
\filldraw[thick, fill=\bColor] (.2,.9) circle (.1cm);
\filldraw[thick, fill=\bColor] (.2,1.3) circle (.1cm);
\draw[dashed, rounded corners=5pt] (0,1.1) rectangle (.4,.7);
}
\arrow[r, Rightarrow, "f^\dag"]
&
\tikzmath{
\draw (.2,.2) -- (.2,.5);
\draw (.2,1) -- (.2,1.2);
\filldraw[thick, fill=\aColor] (.2,.5) circle (.1cm);
\filldraw[thick, fill=\aColor] (.2,.9) circle (.1cm);
\filldraw[thick, fill=\bColor] (.2,1.3) circle (.1cm);
\draw[dashed, rounded corners=5pt] (0,1.1) rectangle (.4,.3);
}
\arrow[r, Rightarrow, "\varepsilon_a"]
&
\tikzmath{
\draw (0,0) -- (0,-.8);
\filldraw[thick, fill=\bColor] (0,0) circle (.1cm);
}
\end{tikzcd}};\end{tikzpicture}
\qquad\qquad\qquad
\forall
f:
\tikzmath{
    \draw (0,0)  -- (0,.3);
    \filldraw[thick, fill=\aColor] (0,0) circle (.1cm);
}
\Rightarrow
\tikzmath{
    \draw (0,0)  -- (0,.3);
    \filldraw[thick, fill=\bColor] (0,0) circle (.1cm);
}
\end{equation}
We also know that the canonical tensorators for $*$ are unitary. That is, suppose we have 1-morphisms ${}_xX_y$ and ${}_yY_z$, 
represented graphically:
\[
\tikzmath{
\draw[thick,black] (0,0) -- +(0,.6);
}
=
x
\,,\qquad
\tikzmath{
\draw[thick,red] (0,0) -- +(0,.6);
}
=
y\,,\qquad
\tikzmath{
\draw[thick,blue] (0,0) -- +(0,.6);
}
=
z\,,\qquad
\tikzmath{
\draw[thick,black] (0,0) -- +(0,.4) coordinate (a);
\draw[thick,red] (a) -- +(0,.4);
\filldraw[draw=black,thick, fill=\aColor] (a) circle (.1);
}
=
X\,,\qquad
\tikzmath{
\draw[thick,red] (0,0) -- +(0,.4) coordinate (a);
\draw[thick,blue] (a) -- +(0,.4);
\filldraw[draw=black,thick, fill=\bColor] (a) circle (.1);
}
=
Y\,.
\]
Then representing the unitary adjoint of
$
\tikzmath{
\draw[thick,black] (0,0) -- +(0,.4) coordinate (a1);
\draw[thick,red] (a1) -- +(0,.4) coordinate (b1);
\draw[thick,blue] (b1) -- +(0,.4);
\filldraw[draw=black,thick, fill=\aColor] (a1) circle (.1);
\filldraw[draw=black,thick, fill=\bColor] (b1) circle (.1);
}
$
with
$
\tikzmath[yscale=-1]{
\draw[thick,black] (0,0) -- +(0,.4) coordinate (a1);
\draw[thick,red] (a1) -- +(0,.4) coordinate (b1);
\draw[thick,blue] (b1) -- +(0,.4);
\draw (-.1,.2) rectangle (.1,1);
\filldraw[draw=black,thick, fill=\aColor] (a1) circle (.1);
\filldraw[draw=black,thick, fill=\bColor] (b1) circle (.1);
},
$
the following map is unitary:
\[
*^2:
\tikzmath{
\draw[thick,blue] (0,0) -- +(0,.4) coordinate (b);
\draw[thick,red] (b) -- +(0,.4) coordinate (a);
\draw[thick,black] (a) -- +(0,.8);
\draw[draw=none] (a) -- +(-.2,.2) coordinate (x);
\draw[rounded corners=5pt,dashed] (x) rectangle +(.4,.4);
\filldraw[draw=black,thick, fill=\aColor] (a) circle (.1);
\filldraw[draw=black,thick, fill=\bColor] (b) circle (.1);
}
\xRightarrow{\eta_{b\circ a}}
\tikzmath{
\draw[thick,blue] (0,0) -- +(0,.4) coordinate (b1);
\draw[thick,red] (b1) -- +(0,.4) coordinate (a1);
\draw[thick,black] (a1) -- +(0,.4) coordinate (a2);
\draw[thick,red] (a2) -- +(0,.4) coordinate (b2);
\draw[thick,blue] (b2) -- +(0,.4) coordinate (b3);
\draw[thick,red] (b3) -- +(0,.4) coordinate (a3);
\draw[thick,black] (a3) -- +(0,0.4);
\draw (-.1,1.8) rectangle (.1,2.6);
\draw[draw=none] (a1) -- +(-.2,-.2) coordinate (x);
\draw[rounded corners=5pt,dashed] (x) rectangle +(.4,.8);
\filldraw[draw=black,thick, fill=\aColor] (a1) circle (.1);
\filldraw[draw=black,thick, fill=\aColor] (a2) circle (.1);
\filldraw[draw=black,thick, fill=\aColor] (a3) circle (.1);
\filldraw[draw=black,thick, fill=\bColor] (b1) circle (.1);
\filldraw[draw=black,thick, fill=\bColor] (b2) circle (.1);
\filldraw[draw=black,thick, fill=\bColor] (b3) circle (.1);
}
\xRightarrow{\varepsilon_a}
\tikzmath{
\draw[thick,blue] (0,0) -- +(0,.4) coordinate (b1);
\draw[thick,red] (b1) -- +(0,.4) coordinate (b2);
\draw[thick,blue] (b2) -- +(0,.4) coordinate (b3);
\draw[thick,red] (b3) -- +(0,.4) coordinate (a3);
\draw[thick,black] (a3) -- +(0,0.4);
\draw (-.1,1) rectangle (.1,1.8);
\draw[draw=none] (b1) -- +(-.2,-.2) coordinate (x);
\draw[rounded corners=5pt,dashed] (x) rectangle +(.4,.8);
\filldraw[draw=black,thick, fill=\aColor] (a3) circle (.1);
\filldraw[draw=black,thick, fill=\bColor] (b1) circle (.1);
\filldraw[draw=black,thick, fill=\bColor] (b2) circle (.1);
\filldraw[draw=black,thick, fill=\bColor] (b3) circle (.1);
}
\xRightarrow{\varepsilon_b}
\tikzmath[yscale=-1]{
\draw[thick,black] (0,0) -- +(0,.4) coordinate (a1);
\draw[thick,red] (a1) -- +(0,.4) coordinate (b1);
\draw[thick,blue] (b1) -- +(0,.4);
\draw (-.1,.2) rectangle (.1,1);
\filldraw[draw=black,thick, fill=\aColor] (a1) circle (.1);
\filldraw[draw=black,thick, fill=\bColor] (b1) circle (.1);
}\,.
\]
(As above, a vertically reflected node represents the unitary adjoint of the node).

We first prove $f^{\vee\dag}=f^{\dag\vee}$ for all $f:a\to b$ in $|A|$.
This amounts to proving the pasting diagram in Figure \ref{fig:RealizationUDF-DaggerFunctor} below commutes.
Going to the right and down is $f^{\vee\dag}$, and going down and then right is $f^{\dag \vee}$.
The middle hexagon commutes by \eqref{eq:UnitaryAdjointsIn2Cat}.
The remaining cells of the pasting diagram also commute, but for the portions of the diagram below and above the middle hexagon, we suppress canonical isomorphisms of the form $\alpha,\kappa,\lambda,\rho,\phi$, and include at most one type of interesting morphism of the form $f^\dag,\varepsilon,\eta$ per cell of the pasting diagram.
We leave this straightforward exercise to the reader.
\end{proof}

\begin{figure}[!ht]
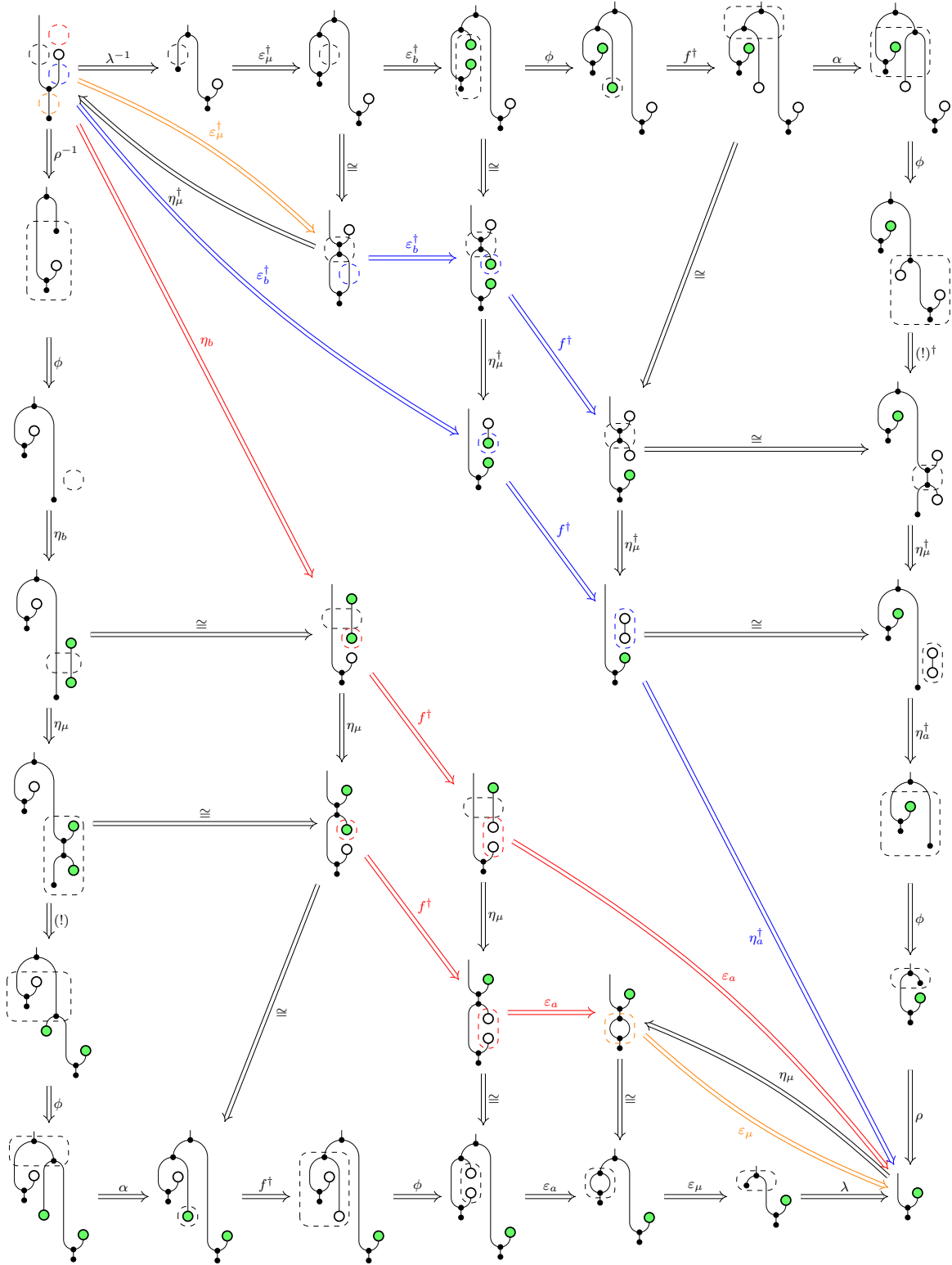
 

\caption{\label{fig:RealizationUDF-DaggerFunctor}Pasting diagram showing that $\vee$ on $A$ is a dagger functor.}
\end{figure}

We next prove that the canonical tensorators are unitary by considering the pasting diagram in Figure \ref{fig:UnitaryTensorator} below.
Again, we suppress all coherence isomorphisms of the form $\alpha,\kappa,\lambda,\rho,\phi$ and include at most one type of interesting morphism of the form $f^\dag,\varepsilon,\eta$ per cell of the pasting diagram.
Starting at the top left and going right, then down, then left again is the definition of the canonical tensorator.
As the diagram commutes, this composite is the same as going straight downward, which is manifestly a composite of unitaries.
\qedhere

\begin{figure}
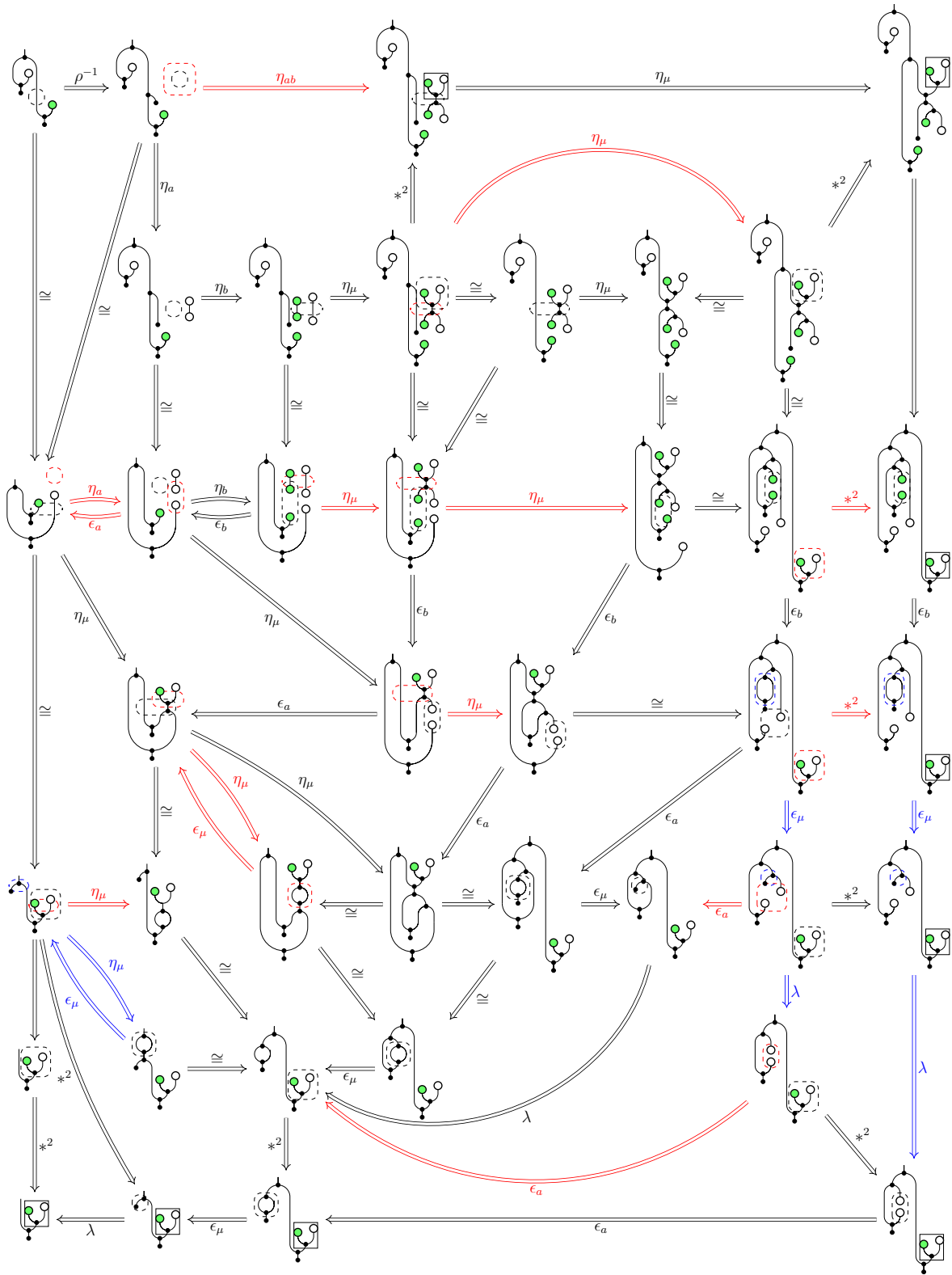

\[

\]
\caption{\label{fig:UnitaryTensorator} Pasting diagram showing the canonical tensorator for $\vee$ on $|A|$ is unitary.}
\end{figure}

\begin{prop}\label{prop:underlying-hstar-spherical-weight}
Suppose $\fC$ is a $\sC^*\Gray$ monoid 
whose underlying strict $\rmC^*$ 2-category is a pre-3-Hilbert space, and $A$ is a standard $\rmH^*$-algebra in $\fC$.
The weight
\[
\psi_{|A|}(f) := \Psi^{\fC}_{1_\fC}
\left(
\tikzmath{
\draw[rounded corners=5pt, dotted] (0,0) rectangle (.5,.5);
}
\xRightarrow{\eta_\iota}
\tikzmath{
\filldraw (0,0) circle (.05);
\draw (0,0) -- (0,.4);
\filldraw (0,.4) circle (.05);
\draw[rounded corners=5pt,dashed] (-.2,-.2) rectangle (.2,.2);
}
\xRightarrow{f}
\tikzmath{
\filldraw (0,0) circle (.05);
\draw (0,0) -- (0,.4);
\filldraw (0,.4) circle (.05);
}
\xRightarrow{\eta_\iota^\dag}
\tikzmath{
\draw[rounded corners=5pt, dotted] (0,0) rectangle (.5,.5);
}
\right)
\]
where $f$ is an endomorphism of $\iota$ defines a spherical weight on $|A|$.
\end{prop}
\begin{proof}
We must prove that
for any $a:1 \to A$ and any endomorphism $f:a \Rightarrow a$,
the left and right traces given by
\begin{align}
\label{eq:LeftTraceForSphericality}
(\psi_{|A|}\circ \tr^\vee_L)(f)
&=
\Psi^\fC_{1_\fC}\left(
\tikzmath{
\draw[rounded corners, dotted] (0,0) rectangle (.5,.5);
}
\xRightarrow{\eta_\iota}
\tikzmath{
\filldraw (0,0) circle (.05);
\draw (0,0) -- (0,.2);
\filldraw (0,.2) circle (.05);
}
\xRightarrow{\epsilon_\mu^\dag}
\tikzmath{
\filldraw (0,0) circle (.05);
\draw (0,0) -- (0,.2) arc (-90:450:.2) -- +(0,.2);
\filldraw (0,.8) circle (.05);
\filldraw (0,.2) circle (.05);
\filldraw (0,.6) circle (.05);
}
\xRightarrow{\epsilon_a^\dag}
\tikzmath{
\draw (0,0) coordinate (d1) -- +(0,.2) coordinate (d2) arc (-90:0:.2) -- +(0,.1) coordinate (a1) -- +(0,0) arc (0:-180:.2) -- +(0,.6) arc (180:90:.2) coordinate (d3) -- +(0,.2) coordinate (d4) -- +(0,0) arc (90:0:.2) -- +(0,-.1) coordinate (a2);
\filldraw (d1) circle (.05);
\filldraw (d2) circle (.05);
\filldraw (d3) circle (.05);
\filldraw (d4) circle (.05);
\filldraw[thick, fill=\aColor] (a1) circle (.1);
\filldraw[thick, fill=\aColor] (a2) circle (.1);
\draw[rounded corners, dashed] (0,.75) rectangle +(.4,.3);
}
\xRightarrow{f}
\tikzmath{
\draw (0,0) coordinate (d1) -- +(0,.2) coordinate (d2) arc (-90:0:.2) -- +(0,.1) coordinate (a1) -- +(0,0) arc (0:-180:.2) -- +(0,.6) arc (180:90:.2) coordinate (d3) -- +(0,.2) coordinate (d4) -- +(0,0) arc (90:0:.2) -- +(0,-.1) coordinate (a2);
\filldraw (d1) circle (.05);
\filldraw (d2) circle (.05);
\filldraw (d3) circle (.05);
\filldraw (d4) circle (.05);
\filldraw[thick, fill=\aColor] (a1) circle (.1);
\filldraw[thick, fill=\aColor] (a2) circle (.1);
}
\xRightarrow{\epsilon_a}
\tikzmath{
\filldraw (0,0) circle (.05);
\draw (0,0) -- (0,.2) arc (-90:450:.2) -- +(0,.2);
\filldraw (0,.8) circle (.05);
\filldraw (0,.2) circle (.05);
\filldraw (0,.6) circle (.05);
}
\xRightarrow{\epsilon_\mu}
\tikzmath{
\filldraw (0,0) circle (.05);
\draw (0,0) -- (0,.2);
\filldraw (0,.2) circle (.05);
}
\xRightarrow{\eta_\iota^\dag}
\tikzmath{
\draw[rounded corners, dotted] (0,0) rectangle (.5,.5);
}
\right)
\\
\label{eq:RightTraceForSphericality}
(\psi_{|A|}\circ \tr^\vee_{R})(f)
&=
\Psi^\fC_{1_\fC}\left(
\tikzmath{
\draw[rounded corners, dotted] (0,0) rectangle (.5,.5);
}
\xRightarrow{\eta_\iota}
\tikzmath{
\filldraw (0,0) circle (.05);
\draw (0,0) -- (0,.2);
\filldraw (0,.2) circle (.05);
}
\xRightarrow{\eta_a}
\tikzmath{
\draw (0,-.1) -- (0,.3);
\draw[thick, fill=white] (0,-.1) circle (.1);
\draw[thick, fill=white] (0,.3) circle (.1);
\filldraw (-.3,-.2) circle (.05);
\draw (-.3,-.2) -- (-.3,.4);
\filldraw (-.3,.4) circle (.05);
}
\xRightarrow{\eta_\mu}
\tikzmath{
\draw (0,0) coordinate (d1) -- +(0,.2) arc (180:90:.2) coordinate (d2) arc (90:0:.2) -- +(0,-.1) coordinate (a1) -- +(0,0) arc (0:90:.2) -- +(0,.2) coordinate (d3) arc (-90:0:.2) -- +(0,.1) coordinate (a2) -- +(0,0) arc (0:-180:.2) -- +(0,.2) coordinate (d4);
\filldraw (d1) circle (.05);
\filldraw (d2) circle (.05);
\filldraw (d3) circle (.05);
\filldraw (d4) circle (.05);
\filldraw[thick, fill=\aColor] (a1) circle (.1);
\filldraw[thick, fill=\aColor] (a2) circle (.1);
\draw[rounded corners, dashed] (.2,-.05) rectangle +(.4,.3);
}
\xRightarrow{f}
\tikzmath{
\draw (0,-.1) coordinate (d1) -- +(0,.2) arc (180:90:.2) coordinate (d2) arc (90:0:.2) -- +(0,-.1) coordinate (a1) -- +(0,0) arc (0:90:.2) -- +(0,.2) coordinate (d3) arc (-90:0:.2) -- +(0,.1) coordinate (a2) -- +(0,0) arc (0:-180:.2) -- +(0,.2) coordinate (d4);
\filldraw (d1) circle (.05);
\filldraw (d2) circle (.05);
\filldraw (d3) circle (.05);
\filldraw (d4) circle (.05);
\filldraw[thick, fill=\aColor] (a1) circle (.1);
\filldraw[thick, fill=\aColor] (a2) circle (.1);
}
\xRightarrow{\eta_\mu^\dag}
\tikzmath{
\draw (0,-.1) -- (0,.3);
\draw[thick, fill=white] (0,-.1) circle (.1);
\draw[thick, fill=white] (0,.3) circle (.1);
\filldraw (-.3,-.2) circle (.05);
\draw (-.3,-.2) -- (-.3,.4);
\filldraw (-.3,.4) circle (.05);
}
\xRightarrow{\eta_a^\dag}
\tikzmath{
\filldraw (0,0) circle (.05);
\draw (0,0) -- (0,.2);
\filldraw (0,.2) circle (.05);
}
\xRightarrow{\eta_\iota^\dag}
\tikzmath{
\draw[rounded corners, dotted] (0,0) rectangle (.5,.5);
}\right)
\end{align}
agree.
It is enough to show that the diagram
\begin{equation}\label{eq:sphericality-claim-diagram}
\begin{tikzcd}
&&
\tikzmath{
\draw (0,-.1) -- (0,.3);
\draw[thick, fill=white] (0,-.1) circle (.1);
\draw[thick, fill=white] (0,.3) circle (.1);
\filldraw (-.3,-.2) circle (.05);
\draw (-.3,-.2) -- (-.3,.4);
\filldraw (-.3,.4) circle (.05);
\draw[rounded corners, dashed] (-.5,0) rectangle +(.7,.2);
}
\arrow[r,Rightarrow,"\eta_\mu"]
&
\tikzmath{
\draw (0,0) coordinate (d1) -- +(0,.2) arc (180:90:.2) coordinate (d2) arc (90:0:.2) -- +(0,-.1) coordinate (a1) -- +(0,0) arc (0:90:.2) -- +(0,.2) coordinate (d3) arc (-90:0:.2) -- +(0,.1) coordinate (a2) -- +(0,0) arc (0:-180:.2) -- +(0,.4) coordinate (d4);
\filldraw (d1) circle (.05);
\filldraw (d2) circle (.05);
\filldraw (d3) circle (.05);
\filldraw (d4) circle (.05);
\filldraw[thick, fill=\aColor] (a1) circle (.1);
\filldraw[thick, fill=\aColor] (a2) circle (.1);
\draw[rounded corners, dashed, red] (.2,-.05) rectangle +(.4,.3);
}
\arrow[d,"(!)"',Rightarrow]
\arrow[r,red, Rightarrow,"f"]
&
\tikzmath{
\draw (0,0) coordinate (d1) -- +(0,.2) arc (180:90:.2) coordinate (d2) arc (90:0:.2) -- +(0,-.1) coordinate (a1) -- +(0,0) arc (0:90:.2) -- +(0,.2) coordinate (d3) arc (-90:0:.2) -- +(0,.1) coordinate (a2) -- +(0,0) arc (0:-180:.2) -- +(0,.4) coordinate (d4);
\filldraw (d1) circle (.05);
\filldraw (d2) circle (.05);
\filldraw (d3) circle (.05);
\filldraw (d4) circle (.05);
\filldraw[thick, fill=\aColor] (a1) circle (.1);
\filldraw[thick, fill=\aColor] (a2) circle (.1);
\draw[rounded corners, dashed,blue] (-.2,.25) rectangle +(.8,.5);
\draw[rounded corners, dashed] (-.3,-.2) rectangle +(1,1.3);
}
\arrow[r,Rightarrow,"\eta_\mu^\dag",blue]
\arrow[d,Rightarrow,"(!)"]
&
\tikzmath{
\draw (0,-.1) -- (0,.3);
\draw[thick, fill=white] (0,-.1) circle (.1);
\draw[thick, fill=white] (0,.3) circle (.1);
\filldraw (-.3,-.2) circle (.05);
\draw (-.3,-.2) -- (-.3,.4);
\filldraw (-.3,.4) circle (.05);
\draw[rounded corners, dashed] (-.15,-.3) rectangle +(.3,.8);
}
\arrow[dr,Rightarrow,"\eta_a^\dag"]
\\
\tikzmath{
\draw[rounded corners, dotted] (0,0) rectangle (.5,.5);
}
\arrow[r,Rightarrow,"\eta_\iota"]
&
\tikzmath{
\filldraw (0,0) circle (.05);
\draw (0,0) -- (0,.6);
\filldraw (0,.6) circle (.05);
\draw[rounded corners, dashed,orange] (-.15,.15) rectangle +(.3,.3);
\draw[rounded corners, dashed,blue] (.25,.15) rectangle +(.3,.3);
}
\arrow[ur,"\eta_a",Rightarrow,blue]
\arrow[dr,"\epsilon_\mu^\dag"',Rightarrow,orange]
&&
\tikzmath[xscale=-1]{
\draw (0,0) coordinate (d1) {} -- +(0,.2) coordinate (d2) {} arc (-90:-180:.2) -- +(0,.1) coordinate (a1) -- +(0,0) arc (-180:0:.2) -- +(0,.4) arc (180:90:.2) coordinate (d3) {} -- +(0,.2) coordinate (d4) {} -- +(0,0) arc (90:0:.2) -- +(0,-.1) coordinate (a2) {};
\draw[draw=none] (a1) -- +(-.2,-.15) coordinate (x1);
\draw[rounded corners, dashed, red] (x1) rectangle +(.4,.3);
\filldraw (d1) circle (.05);
\filldraw (d2) circle (.05);
\filldraw (d3) circle (.05);
\filldraw (d4) circle (.05);
\filldraw[thick,fill=\aColor] (a1) circle (.1);
\filldraw[thick,fill=\aColor] (a2) circle (.1);
}
\arrow[r,Rightarrow,"f",red]
&
\tikzmath[xscale=-1]{
\draw (0,0) coordinate (d1) {} -- +(0,.2) coordinate (d2) {} arc (-90:-180:.2) -- +(0,.1) coordinate (a1) -- +(0,0) arc (-180:0:.2) -- +(0,.4) arc (180:90:.2) coordinate (d3) {} -- +(0,.2) coordinate (d4) {} -- +(0,0) arc (90:0:.2) -- +(0,-.1) coordinate (a2) {};
\filldraw (d1) circle (.05);
\filldraw (d2) circle (.05);
\filldraw (d3) circle (.05);
\filldraw (d4) circle (.05);
\filldraw[thick,fill=\aColor] (a1) circle (.1);
\filldraw[thick,fill=\aColor] (a2) circle (.1);
}
&&
\tikzmath{
\filldraw (0,0) circle (.05);
\draw (0,0) -- (0,.6);
\filldraw (0,.6) circle (.05);
}
\arrow[r,Rightarrow,"\eta_\iota^\dag"]
&
\tikzmath{
\draw[rounded corners, dotted] (0,0) rectangle (.5,.5);
}
\\
&&
\tikzmath{
\filldraw (0,0) circle (.05);
\draw (0,0) -- (0,.2) arc (-90:450:.2) -- +(0,.2);
\filldraw (0,.8) circle (.05);
\filldraw (0,.2) circle (.05);
\filldraw (0,.6) circle (.05);
\draw[rounded corners, dashed] (-.3,.25) rectangle +(.3,.3);
}
\arrow[r,Rightarrow,"\epsilon_a^\dag"']
&
\tikzmath[xscale=-1]{
\draw (0,0) coordinate (d1) -- +(0,.2) coordinate (d2) arc (-90:0:.2) -- +(0,.1) coordinate (a1) -- +(0,0) arc (0:-180:.2) -- +(0,.6) arc (180:90:.2) coordinate (d3) -- +(0,.2) coordinate (d4) -- +(0,0) arc (90:0:.2) -- +(0,-.1) coordinate (a2);
\filldraw (d1) circle (.05);
\filldraw (d2) circle (.05);
\filldraw (d3) circle (.05);
\filldraw (d4) circle (.05);
\draw[rounded corners, dashed] (-.4,-.2) rectangle +(.8,.9);
\filldraw[thick, fill=\aColor] (a1) circle (.1);
\filldraw[thick, fill=\aColor] (a2) circle (.1);
\draw[rounded corners, dashed, red] (0,.75) rectangle +(.4,.3);
}
\arrow[u,Rightarrow,"p"]
\arrow[r,red,Rightarrow,"f"']
&
\tikzmath[xscale=-1]{
\draw (0,0) coordinate (d1) -- +(0,.2) coordinate (d2) arc (-90:0:.2) -- +(0,.1) coordinate (a1) -- +(0,0) arc (0:-180:.2) -- +(0,.6) arc (180:90:.2) coordinate (d3) -- +(0,.2) coordinate (d4) -- +(0,0) arc (90:0:.2) -- +(0,-.1) coordinate (a2);
\filldraw (d1) circle (.05);
\filldraw (d2) circle (.05);
\filldraw (d3) circle (.05);
\filldraw (d4) circle (.05);
\draw[rounded corners, dashed] (-.4,-.2) rectangle +(.8,.9);
\filldraw[thick, fill=\aColor] (a1) circle (.1);
\filldraw[thick, fill=\aColor] (a2) circle (.1);
\draw[rounded corners, dashed,orange] (-.1,.3) rectangle +(.6,.8);
}
\arrow[u,Rightarrow,"p"']
\arrow[r,Rightarrow,"\epsilon_a"',orange]
&
\tikzmath{
\filldraw (0,0) circle (.05);
\draw (0,0) -- (0,.2) arc (-90:450:.2) -- +(0,.2);
\filldraw (0,.8) circle (.05);
\filldraw (0,.2) circle (.05);
\filldraw (0,.6) circle (.05);
\draw[rounded corners, dashed] (-.3,.1) rectangle +(.6,.6);
}
\arrow[ur,Rightarrow,"\epsilon_\mu"']
\end{tikzcd}
\end{equation}
commutes, where $p$ denotes the natural pivotal map coming from the unitary dual functor on $\fC(1_{\fC}\to A)$ as defined above.
Indeed, the $\Psi^\fC_{1_\fC}$ applied to the top path in \eqref{eq:sphericality-claim-diagram} is \eqref{eq:RightTraceForSphericality},
and $\Psi^\fC_{1_\fC}$ to bottom path in \eqref{eq:sphericality-claim-diagram} is the horizontal reflection of \eqref{eq:LeftTraceForSphericality}, which is then equal to \eqref{eq:LeftTraceForSphericality} by the standardness condition \ref{H:SphereStandard}.
As the right hexagonal face in \eqref{eq:sphericality-claim-diagram} is the dagger of the one on the left (as $p$ and $(!)$ are both unitary isomorphisms), we only need to show the left hexagon commutes.
Unpacking (and suppressing interchangers, as before), the map $p$ is given by
the marked blue edges in the pasting diagram in Figure \ref{fig:RealizationSphericalityHelperDiagram} below.
Notice the perimeter of Figure \ref{fig:RealizationSphericalityHelperDiagram} is the left hexagon in \eqref{eq:sphericality-claim-diagram}, and thus \eqref{eq:sphericality-claim-diagram} commutes.
\qedhere
\begin{figure}[!ht]
\begin{tikzcd}
\tikzmath{
\filldraw (0,0) circle (.05);
\draw (0,0) -- (0,.6);
\filldraw (0,.6) circle (.05);
\draw[rounded corners, dashed,orange] (-.15,.2) rectangle +(.3,.3);
\draw[rounded corners, dashed] (.25,.2) rectangle +(.3,.3);
}
\arrow[r,Rightarrow,"\epsilon_\mu^\dag",orange]
\arrow[d,Rightarrow,"\eta_a"]
&
\tikzmath{
\filldraw (0,0) circle (.05);
\draw (0,0) -- (0,.2) arc (-90:450:.2) -- +(0,.2);
\filldraw (0,.8) circle (.05);
\filldraw (0,.2) circle (.05);
\filldraw (0,.6) circle (.05);
\draw[rounded corners, dashed,orange] (-.3,.25) rectangle +(.3,.3);
\draw[rounded corners, dashed] (.3,.25) rectangle +(.3,.3);
}
\arrow[rr,Rightarrow,"\epsilon_a^\dag",orange]
\arrow[d,Rightarrow,"\eta_a"]
&&
\tikzmath[xscale=-1]{
\draw (0,0) -- +(0,.2) arc (-90:0:.2) arc (0:-180:.2) -- +(0,.6) arc (180:90:.2) -- +(0,.2) -- +(0,0) arc (90:0:.2);
\filldraw (0,0) circle (.05);
\filldraw (0,.2) circle (.05);
\filldraw[thick, fill=\aColor] (.2,.5) circle (.1);
\filldraw[thick, fill=\aColor] (.2,.9) circle (.1);
\filldraw (0,1.4) circle (.05);
\filldraw (0,1.2) circle (.05);
\draw[rounded corners, dashed,blue] (-0.35,.55) rectangle +(.3,.3);
\draw[rounded corners, dashed] (-0.7,.55) rectangle +(.3,.3);
}
\arrow[d,Rightarrow,"\eta_a"]
\arrow[r,Rightarrow,"\rho^{-1}",blue]
&
\tikzmath[xscale=-1]{
\draw (0,0) -- +(0,.2) arc (-90:0:.2) arc (0:-180:.2) -- +(0,.6) arc (180:90:.2) -- +(0,.2) -- +(0,0) arc (90:0:.2);
\draw (-.2,.8) arc (90:180:.2);
\filldraw (0,0) circle (.05);
\filldraw (0,.2) circle (.05);
\filldraw[thick, fill=\aColor] (.2,.5) circle (.1);
\filldraw[thick, fill=\aColor] (.2,.9) circle (.1);
\filldraw (0,1.4) circle (.05);
\filldraw (0,1.2) circle (.05);
\filldraw (-.4,.6) circle (.05);
\filldraw (-.2,.8) circle (.05);
\draw[rounded corners, dashed,blue] (-0.8,.55) rectangle +(.3,.3);
}
\arrow[d,Rightarrow,"\eta_a",blue]
\\
\tikzmath[xscale=-1]{
\draw (0,0) -- +(0,.6);
\draw (-.3,.1) -- +(0,.4);
\filldraw (0,0) circle (.05);
\filldraw (0,.6) circle (.05);
\filldraw[thick, fill=\aColor] (-.3,.1) circle (.1);
\filldraw[thick, fill=\aColor] (-.3,.5) circle (.1);
\draw[rounded corners, dashed,orange] (-.15,0.15) rectangle +(.3,.3);
}
\arrow[r,Rightarrow,"\epsilon_\mu^\dag",orange]
\arrow[d,Rightarrow,"\eta_\mu"]
&
\tikzmath[xscale=-1]{
\draw (0,0) -- (0,.2) arc (-90:450:.2) -- +(0,.2);
\draw (-.5,.2) -- +(0,.3) node (end) {};
\filldraw[thick, fill=\aColor] (-.5,.2) circle (.1);
\filldraw[thick, fill=\aColor] (end) circle (.1);
\filldraw (0,0) circle (.05);
\filldraw (0,.8) circle (.05);
\filldraw (0,.2) circle (.05);
\filldraw (0,.6) circle (.05);
\draw[rounded corners, dashed,orange] (0,.25) rectangle +(.3,.3);
}
\arrow[rr,Rightarrow,"\epsilon_a^\dag",orange]
\arrow[d,Rightarrow,"\eta_\mu"]
&&
\tikzmath[xscale=-1]{
\draw (0,0) -- +(0,.2) arc (-90:0:.2) arc (0:-180:.2) -- +(0,.6) arc (180:90:.2) -- +(0,.2) -- +(0,0) arc (90:0:.2);
\draw (-.5,0.05) node (begin) {} -- +(0,1.1) node (end) {};
\filldraw[thick, fill=\aColor] (begin) circle (.1);
\filldraw[thick, fill=\aColor] (end) circle (.1);
\filldraw (0,0) circle (.05);
\filldraw (0,.2) circle (.05);
\filldraw[thick, fill=\aColor] (.2,.5) circle (.1);
\filldraw[thick, fill=\aColor] (.2,.9) circle (.1);
\filldraw (0,1.4) circle (.05);
\filldraw (0,1.2) circle (.05);
\draw[rounded corners, dashed,orange] (-0.35,.55) rectangle +(.3,.3);
}
\arrow[r,Rightarrow,"\rho^{-1}",orange]
\arrow[dl,Rightarrow,"\eta_\mu"]
&
\tikzmath[xscale=-1]{
\draw (0,0) -- +(0,.2) arc (-90:0:.2) arc (0:-180:.2) -- +(0,.6) arc (180:90:.2) -- +(0,.2) -- +(0,0) arc (90:0:.2);
\draw (-.2,.8) arc (90:180:.2);
\draw (-.7,.05) node (begin) {} -- +(0,1.1) node (end) {};
\filldraw[thick, fill=\aColor] (begin) circle (.1);
\filldraw[thick, fill=\aColor] (end) circle (.1);
\filldraw (0,0) circle (.05);
\filldraw (0,.2) circle (.05);
\filldraw[thick, fill=\aColor] (.2,.5) circle (.1);
\filldraw[thick, fill=\aColor] (.2,.9) circle (.1);
\filldraw (0,1.4) circle (.05);
\filldraw (0,1.2) circle (.05);
\filldraw (-.4,.6) circle (.05);
\filldraw (-.2,.8) circle (.05);
}
\arrow[d,Rightarrow,"\eta_\mu",blue]
\arrow[dl,Rightarrow,"\eta_\mu"]
\\
\tikzmath[xscale=-1]{
\draw (0,0) coordinate (a1) -- +(0,.1) arc (180:0:.2) -- +(0,-.2) coordinate (d1) {} -- +(0,0) arc (0:90:.2) coordinate (d3) {} -- +(0,.2) coordinate (d4) arc (-90:-180:.2) -- +(0,.1) coordinate (a2) -- +(0,0) arc (-180:0:.2) -- +(0,.6) coordinate (d2) {};
\filldraw (d1) circle (.05);
\filldraw (d2) circle (.05);
\filldraw (d3) circle (.05);
\filldraw (d4) circle (.05);
\filldraw[thick, fill=\aColor] (a1) circle (.1);
\filldraw[thick, fill=\aColor] (a2) circle (.1);
\draw[rounded corners, dashed,orange] (.25,.9) rectangle +(.3,.3);
\draw[rounded corners, dashed] (-.2,-.2) rectangle +(.8,1.1);
}
\arrow[r,Rightarrow,"\epsilon_\mu^\dag",orange]
\arrow[dddrr,Rightarrow,"\cong"]
&
\tikzmath[xscale=-1]{
\draw (0,0) coordinate (a1) -- +(0,.1) arc (180:0:.2) -- +(0,-.2) coordinate (d1) {} -- +(0,0) arc (0:90:.2) coordinate (d2) {} -- +(0,.2) coordinate (d3) {} arc (-90:-180:.2) -- +(0,.1) coordinate (a2) {} -- +(0,0) arc (-180:0:.2) -- +(0,.2) coordinate (d4) {} arc (-90:450:.2) coordinate (d5) {} -- +(0,.2) coordinate (d6) {};
\filldraw (d1) circle (.05);
\filldraw (d2) circle (.05);
\filldraw (d3) circle (.05);
\filldraw (d4) circle (.05);
\filldraw (d5) circle (.05);
\filldraw (d6) circle (.05);
\filldraw[thick, fill=\aColor] (a1) circle (.1);
\filldraw[thick, fill=\aColor] (a2) circle (.1);
\draw[rounded corners, dashed,orange] (.4,.95) rectangle +(.3,.3);
\draw[rounded corners, dashed] (-.2,-.2) rectangle +(.8,1.1);
}
\arrow[r,Rightarrow,"\epsilon_a^\dag",orange]
\arrow[dr,Rightarrow,"\cong"]
&
\tikzmath[xscale=-1]{
\draw (0,0) coordinate (a1) -- +(0,.1) arc (180:0:.2) -- +(0,-.2) coordinate (d1) {} -- +(0,0) arc (0:90:.2) coordinate (d2) {} -- +(0,.2) coordinate (d3) {} arc (-90:-180:.2) -- +(0,.1) coordinate (a2) {} -- +(0,0) arc (-180:0:.2) -- +(0,.2) coordinate (d4) {} arc (-90:0:.2) -- +(0,.1) coordinate (a3) {} -- +(0,0) arc (0:-180:.2) -- +(0,.6) arc (180:0:.2) -- +(0,-.1) coordinate (a4) {} -- +(0,0) arc (0:90:.2) coordinate (d5) {} -- +(0,.2) coordinate (d6) {};
\filldraw (d1) circle (.05);
\filldraw (d2) circle (.05);
\filldraw (d3) circle (.05);
\filldraw (d4) circle (.05);
\filldraw (d5) circle (.05);
\filldraw (d6) circle (.05);
\filldraw[thick, fill=\aColor] (a1) circle (.1);
\filldraw[thick, fill=\aColor] (a2) circle (.1);
\filldraw[thick, fill=\aColor] (a3) circle (.1);
\filldraw[thick, fill=\aColor] (a4) circle (.1);
\draw[rounded corners, dashed,orange] (0.05,1.25) rectangle +(.3,.3);
\draw[rounded corners, dashed] (-.2,-.2) rectangle +(.8,1.1);
}
\arrow[dr,Rightarrow,"\cong"]
\arrow[r,Rightarrow,"\rho^{-1}",orange]
&
\tikzmath[xscale=-1]{
\draw (0,0) coordinate (a1) -- +(0,.1) {} arc (180:0:.2) -- +(0,-.2) coordinate (d1) {} -- +(0,0) arc (0:90:.2) coordinate (d2) {} -- +(0,.2) coordinate (d3) {} arc (-90:-180:.2) -- +(0,.1) coordinate (a2) {} -- +(0,0) arc (-180:0:.2) -- +(0,.2) coordinate (d4) {} arc (-90:0:.2) -- +(0,.1) coordinate (a3) -- +(0,0) arc (0:-180:.2) -- +(0,.4) arc (0:180:.2) coordinate (d5) {} arc (180:90:.2) coordinate (d6) -- +(0,.2) arc (180:0:.2) -- +(0,-.1) coordinate (a4) -- +(0,0) arc (0:90:.2) coordinate (d7) {} -- +(0,.2) coordinate (d8) {};
\filldraw (d1) circle (.05);
\filldraw (d2) circle (.05);
\filldraw (d3) circle (.05);
\filldraw (d4) circle (.05);
\filldraw (d5) circle (.05);
\filldraw (d6) circle (.05);
\filldraw (d7) circle (.05);
\filldraw (d8) circle (.05);
\filldraw[thick, fill=\aColor] (a1) circle (.1);
\filldraw[thick, fill=\aColor] (a2) circle (.1);
\filldraw[thick, fill=\aColor] (a3) circle (.1);
\filldraw[thick, fill=\aColor] (a4) circle (.1);
\draw[rounded corners, dashed] (-.3,-.1) rectangle +(.8,1.8);
}
\arrow[r,Rightarrow,"\cong"]
&
\tikzmath[xscale=-1]{
\draw (0,0) coordinate (d1) {} -- +(0,.2) coordinate (d2) {} arc (-90:0:.2) -- +(0,.1) coordinate (a1) -- +(0,0) arc (0:-180:.2) -- +(0,1.4) arc (0:90:.2) coordinate (d6) -- +(0,.2) arc (180:0:.2) -- +(0,-.1) coordinate (a4) -- +(0,0) {} arc (0:90:.2) coordinate (d7) {} -- +(0,.2) coordinate (d8) {} -- +(0,0) arc (90:180:.2) -- +(0,-.2) arc (90:180:.2) -- +(0,-.2) arc (0:-180:.2) -- +(0,.1) coordinate (a2) {} -- +(0,0) arc (-180:-90:.2) coordinate (d3) {} -- +(0,-.2) coordinate (d4) {} arc (90:0:.2) -- +(0,-.2) coordinate (d5) {} -- +(0,0) arc (0:180:.2) -- +(0,-.1) coordinate (a3) {};
\filldraw (d1) circle (.05);
\filldraw (d2) circle (.05);
\filldraw (d3) circle (.05);
\filldraw (d4) circle (.05);
\filldraw (d5) circle (.05);
\filldraw (d6) circle (.05);
\filldraw (d7) circle (.05);
\filldraw (d8) circle (.05);
\filldraw[thick, fill=\aColor] (a1) circle (.1);
\filldraw[thick, fill=\aColor] (a2) circle (.1);
\filldraw[thick, fill=\aColor] (a3) circle (.1);
\filldraw[thick, fill=\aColor] (a4) circle (.1);
\draw[rounded corners, dashed,blue] (-1.2,.6) rectangle +(.8,1.25);
}
\arrow[d,Rightarrow,"(!)",blue]
\\
&&
\tikzmath[xscale=-1]{
\draw (0,0) coordinate (a1) {} -- +(0,.2) arc (0:180:.2) -- +(0,-.1) coordinate (d1) -- +(0,0) arc (180:90:.2) coordinate (d2) {} -- +(0,.8) arc (0:180:.2) -- +(0,-.2) arc (0:-90:.2) coordinate (d3) {} -- +(0,-.2) coordinate (d4) {} -- +(0,0) arc (-90:-180:.2) -- +(0,.1) coordinate (a2) -- +(0,0) arc (-180:0:.2) -- +(0,.2) arc (180:90:.2) coordinate (d5) {} -- +(0,.2) coordinate (d6) arc (-90:450:.2) coordinate (d7) {} -- +(0,.2) coordinate (d8) {};
\filldraw (d1) circle (.05);
\filldraw (d2) circle (.05);
\filldraw (d3) circle (.05);
\filldraw (d4) circle (.05);
\filldraw (d5) circle (.05);
\filldraw (d6) circle (.05);
\filldraw (d7) circle (.05);
\filldraw (d8) circle (.05);
\filldraw[thick, fill=\aColor] (a1) circle (.1);
\filldraw[thick, fill=\aColor] (a2) circle (.1);
\draw[rounded corners, dashed,red] (-.8,1.2) rectangle +(.8,.6);
\draw[rounded corners, dashed] (-1.2,.5) rectangle +(1.3,1.2);
\draw[rounded corners, dashed,orange] (-.4,1.65) rectangle +(.3,.3);
}
\arrow[r,Rightarrow,"\epsilon_a^\dag",orange]
\arrow[dr,Rightarrow,"\cong"]
\arrow[dd,Rightarrow,"\eta_\mu^\dag",bend left=10,red]
&
\tikzmath[xscale=-1]{
\draw (0,0) coordinate (a1) {} -- +(0,.2) arc (0:180:.2) -- +(0,-.1) coordinate (d1) -- +(0,0) arc (180:90:.2) coordinate (d2) {} -- +(0,.8) arc (0:180:.2) -- +(0,-.2) arc (0:-90:.2) coordinate (d3) {} -- +(0,-.2) coordinate (d4) {} -- +(0,0) arc (-90:-180:.2) -- +(0,.1) coordinate (a2) -- +(0,0) arc (-180:0:.2) -- +(0,.2) arc (180:90:.2) coordinate (d5) {} -- +(0,.2) coordinate (d6) arc (-90:0:.2) -- +(0,.1) coordinate (a3) -- +(0,0) arc (0:-180:.2) -- +(0,.6) arc (180:0:.2) -- +(0,-.1) coordinate (a4) -- +(0,0) arc (0:90:.2) coordinate (d7) {} -- +(0,.2) coordinate (d8) {};
\filldraw (d1) circle (.05);
\filldraw (d2) circle (.05);
\filldraw (d3) circle (.05);
\filldraw (d4) circle (.05);
\filldraw (d5) circle (.05);
\filldraw (d6) circle (.05);
\filldraw (d7) circle (.05);
\filldraw (d8) circle (.05);
\filldraw[thick, fill=\aColor] (a1) circle (.1);
\filldraw[thick, fill=\aColor] (a2) circle (.1);
\filldraw[thick, fill=\aColor] (a3) circle (.1);
\filldraw[thick, fill=\aColor] (a4) circle (.1);
\draw[rounded corners, dashed] (-1.2,.5) rectangle +(1.3,1.2);
}
\arrow[ddr,Rightarrow,"\cong"]
&
\tikzmath[xscale=-1]{
\draw (0,0) coordinate (d1) {} -- +(0,.2) coordinate (d8) {} arc (-90:0:.2) -- +(0,.1) coordinate (a1) -- +(0,0) arc (0:-180:.2) -- +(0,1.4) arc (0:90:.2) coordinate (d2) {} -- +(0,.2) arc (180:90:.2) coordinate (d3) {} -- +(0,.2) coordinate (d4) {} -- +(0,0) arc (90:0:.2) -- +(0,-.1) coordinate (a2) -- +(0,0) arc (0:180:.2) -- +(0,-.2) arc (90:180:.2) coordinate (d5) {} arc (90:0:.2) -- +(0,-.1) coordinate (a3) -- +(0,0) arc (0:180:.2) -- +(0,-.4) arc (0:-90:.2) coordinate (d6) {} -- +(0,-.2) coordinate (d7) -- +(0,0) arc (-90:-180:.2) -- +(0,.1) coordinate (a4) {};
\filldraw (d1) circle (.05);
\filldraw (d2) circle (.05);
\filldraw (d3) circle (.05);
\filldraw (d4) circle (.05);
\filldraw (d5) circle (.05);
\filldraw (d6) circle (.05);
\filldraw (d7) circle (.05);
\filldraw (d8) circle (.05);
\filldraw[thick, fill=\aColor] (a1) circle (.1);
\filldraw[thick, fill=\aColor] (a2) circle (.1);
\filldraw[thick, fill=\aColor] (a3) circle (.1);
\filldraw[thick, fill=\aColor] (a4) circle (.1);
\draw[rounded corners, dashed,blue] (-.95,1.3) rectangle +(.8,.8);
}
\arrow[d,Rightarrow,"\alpha",blue]
\\
&&&
\tikzmath[xscale=-1]{
\draw (0,0) coordinate (a1) {} -- +(0,.2) arc (0:180:.2) -- +(0,-.1) coordinate (d1) -- +(0,0) arc (180:90:.2) coordinate (d2) {} -- +(0,.2) coordinate (d3) {} arc (-90:-180:.2) -- +(0,.7) -- +(0,0) arc (-180:0:.2) -- +(0,.9) arc (0:90:.3) coordinate (d4) {} -- +(0,.2) coordinate (d5) {} -- +(0,0) arc (90:180:.3) coordinate (d6) {} arc (90:0:.2) arc (0:180:.2) -- +(0,-.2) arc (0:-90:.2) coordinate (d7) {} -- +(0,-.2) coordinate (d8) {} -- +(0,0) arc (-90:-180:.2) -- +(0,.1) coordinate (a2) {};
\filldraw (d1) circle (.05);
\filldraw (d2) circle (.05);
\filldraw (d3) circle (.05);
\filldraw (d4) circle (.05);
\filldraw (d5) circle (.05);
\filldraw (d6) circle (.05);
\filldraw (d7) circle (.05);
\filldraw (d8) circle (.05);
\filldraw[thick, fill=\aColor] (a1) circle (.1);
\filldraw[thick, fill=\aColor] (a2) circle (.1);
\draw[rounded corners, dashed,red] (-.6,.2) rectangle +(.8,.6);
\draw[rounded corners, dashed,orange] (-.15,1) rectangle +(.3,.3);
}
\arrow[dl,Rightarrow,"\eta_\mu^\dag",red]
\arrow[dr,Rightarrow,"\epsilon_a^\dag",orange]
&
\tikzmath[xscale=-1]{
\draw (0,0) coordinate (d1) {} -- +(0,.2) coordinate (d2) {} arc (-90:0:.2) -- +(0,.1) coordinate (a1) -- +(0,0) arc (0:-180:.2) -- +(0,.4) arc (0:180:.2) -- +(0,-.1) coordinate (a2) -- +(0,0) arc (180:90:.2) coordinate (d3) -- +(0,.8) arc (0:180:.2) -- +(0,-.2) arc (0:-90:.2) coordinate (d4) {} -- +(0,-.2) coordinate (d5) {} -- +(0,0) arc (-90:-180:.2) -- +(0,.1) coordinate (a3) -- +(0,0) arc (-180:0:.2) -- +(0,.2) arc (180:90:.2) coordinate (d6) {} -- +(0,.2) arc (180:90:.2) coordinate (d7) {} -- +(0,.2) coordinate (d8) {} -- +(0,0) arc (90:0:.2) -- +(0,-.1) coordinate (a4) {};
\filldraw (d1) circle (.05);
\filldraw (d2) circle (.05);
\filldraw (d3) circle (.05);
\filldraw (d4) circle (.05);
\filldraw (d5) circle (.05);
\filldraw (d6) circle (.05);
\filldraw (d7) circle (.05);
\filldraw (d8) circle (.05);
\filldraw[thick, fill=\aColor] (a1) circle (.1);
\filldraw[thick, fill=\aColor] (a2) circle (.1);
\filldraw[thick, fill=\aColor] (a3) circle (.1);
\filldraw[thick, fill=\aColor] (a4) circle (.1);
\draw[rounded corners, dashed,blue] (-.8,-.1) rectangle +(1.2,1.25);
}
\arrow[d,Rightarrow,"(!)^\dag",blue]
\\
\tikzmath[xscale=-1]{
\draw (0,0) coordinate (d1) {} -- +(0,.2) coordinate (d2) {} arc (-90:-180:.2) -- +(0,.1) coordinate (a1) -- +(0,0) arc (-180:0:.2) -- +(0,.4) arc (180:90:.2) coordinate (d3) {} -- +(0,.2) coordinate (d4) {} -- +(0,0) arc (90:0:.2) -- +(0,-.1) coordinate (a2) {};
\filldraw (d1) circle (.05);
\filldraw (d2) circle (.05);
\filldraw (d3) circle (.05);
\filldraw (d4) circle (.05);
\filldraw[thick, fill=\aColor] (a1) circle (.1);
\filldraw[thick, fill=\aColor] (a2) circle (.1);
\draw[rounded corners, dashed] (-.4,-0.1) rectangle +(1.2,1.2);
}
\arrow[uuu,Rightarrow,"(!)"]
&&
\tikzmath[xscale=-1]{
\draw (0,0) coordinate (d1) {} -- +(0,.2) coordinate (d2) {} arc (-90:-180:.2) -- +(0,.1) coordinate (a1) -- +(0,0) arc (-180:0:.2) -- +(0,.4) arc (180:0:.2) coordinate (d3) {} arc (0:90:.2) coordinate (d4) {} -- +(0,.2) arc (180:90:.2) coordinate (d5) {} -- +(0,.2) coordinate (d6) {} -- +(0,0) arc (90:0:.2) -- +(0,-.1) coordinate (a2) {};
\filldraw (d1) circle (.05);
\filldraw (d2) circle (.05);
\filldraw (d3) circle (.05);
\filldraw (d4) circle (.05);
\filldraw (d5) circle (.05);
\filldraw (d6) circle (.05);
\filldraw[thick, fill=\aColor] (a1) circle (.1);
\filldraw[thick, fill=\aColor] (a2) circle (.1);
\draw[rounded corners, dashed,orange] (.3,1.2) rectangle +(.6,.35);
\draw[rounded corners, dashed,blue] (.1,.7) rectangle +(.6,.35);
}
\arrow[ll,Rightarrow,"\lambda",blue]
\arrow[uu,Rightarrow,"\epsilon_\mu^\dag",bend left=10,orange]
\arrow[r,Rightarrow, "\epsilon_a^\dag", bend left=10,orange]
&
\tikzmath[xscale=-1]{
\draw (0,0) coordinate (d1) {} -- +(0,.2) coordinate (d2) {} arc (-90:-180:.2) -- +(0,.1) coordinate (a1) -- +(0,0) arc (-180:0:.2) -- +(0,.2) arc (180:0:.2) -- +(0,-1.8) coordinate (d3) {} -- +(0,0) arc (0:90:.2) coordinate (d4) {} arc (180:90:.2) coordinate (d5) {} -- +(0,.2) coordinate (d6) {} -- +(0,0) arc (90:0:.2) -- +(0,-1) coordinate (a2) {};
\draw (.8,-1) coordinate (a3) {} -- +(0,.4) coordinate (a4) {};
\filldraw (d1) circle (.05);
\filldraw (d2) circle (.05);
\filldraw (d3) circle (.05);
\filldraw (d4) circle (.05);
\filldraw (d5) circle (.05);
\filldraw (d6) circle (.05);
\filldraw[thick, fill=\aColor] (a1) circle (.1);
\filldraw[thick, fill=\aColor] (a2) circle (.1);
\filldraw[thick, fill=\aColor] (a3) circle (.1);
\filldraw[thick, fill=\aColor] (a4) circle (.1);
\draw[rounded corners, dashed,blue] (.65,-1.2) rectangle +(.3,.8);
}
\arrow[l,Rightarrow, "\eta_a^\dag",bend left=10,blue]
&
\tikzmath[xscale=-1]{
\draw (0,0) coordinate (a1) {} -- +(0,.1) arc (0:180:.2) -- +(0,-.2) coordinate (d1) -- +(0,0) arc (180:90:.2) coordinate (d2) {} -- +(0,.2) coordinate (d3) {} arc (-90:-180:.2)-- +(0,1) -- +(0,0) arc (-180:0:.2) -- +(0,.1) coordinate (a3) {};
\draw (0,2) coordinate (a4) -- +(0,.1) arc (0:90:.3) coordinate (d4) {} -- +(0,.2) coordinate (d5) {} -- +(0,0) arc (90:180:.3) -- +(0,-.2) coordinate (d6) {} arc (90:0:.2) arc (0:180:.2) -- +(0,-.2) arc (0:-90:.2) coordinate (d7) {} -- +(0,-.2) coordinate (d8) {} -- +(0,0) arc (-90:-180:.2) -- +(0,.1) coordinate (a2) {};
\filldraw (d1) circle (.05);
\filldraw (d2) circle (.05);
\filldraw (d3) circle (.05);
\filldraw (d4) circle (.05);
\filldraw (d5) circle (.05);
\filldraw (d6) circle (.05);
\filldraw (d7) circle (.05);
\filldraw (d8) circle (.05);
\filldraw[thick, fill=\aColor] (a1) circle (.1);
\filldraw[thick, fill=\aColor] (a2) circle (.1);
\filldraw[thick, fill=\aColor] (a3) circle (.1);
\filldraw[thick, fill=\aColor] (a4) circle (.1);
\draw[rounded corners, dashed,blue] (-.6,0.1) rectangle +(.8,.6);
}
\arrow[l,Rightarrow,"\eta_\mu^\dag",blue]
\end{tikzcd}
\caption{\label{fig:RealizationSphericalityHelperDiagram} Pasting diagram proving that the left hexagon in \eqref{eq:sphericality-claim-diagram} commutes. 
The map is blue is the unitary pivotal structure induced by $\vee$ on $|A|$.}
\end{figure}
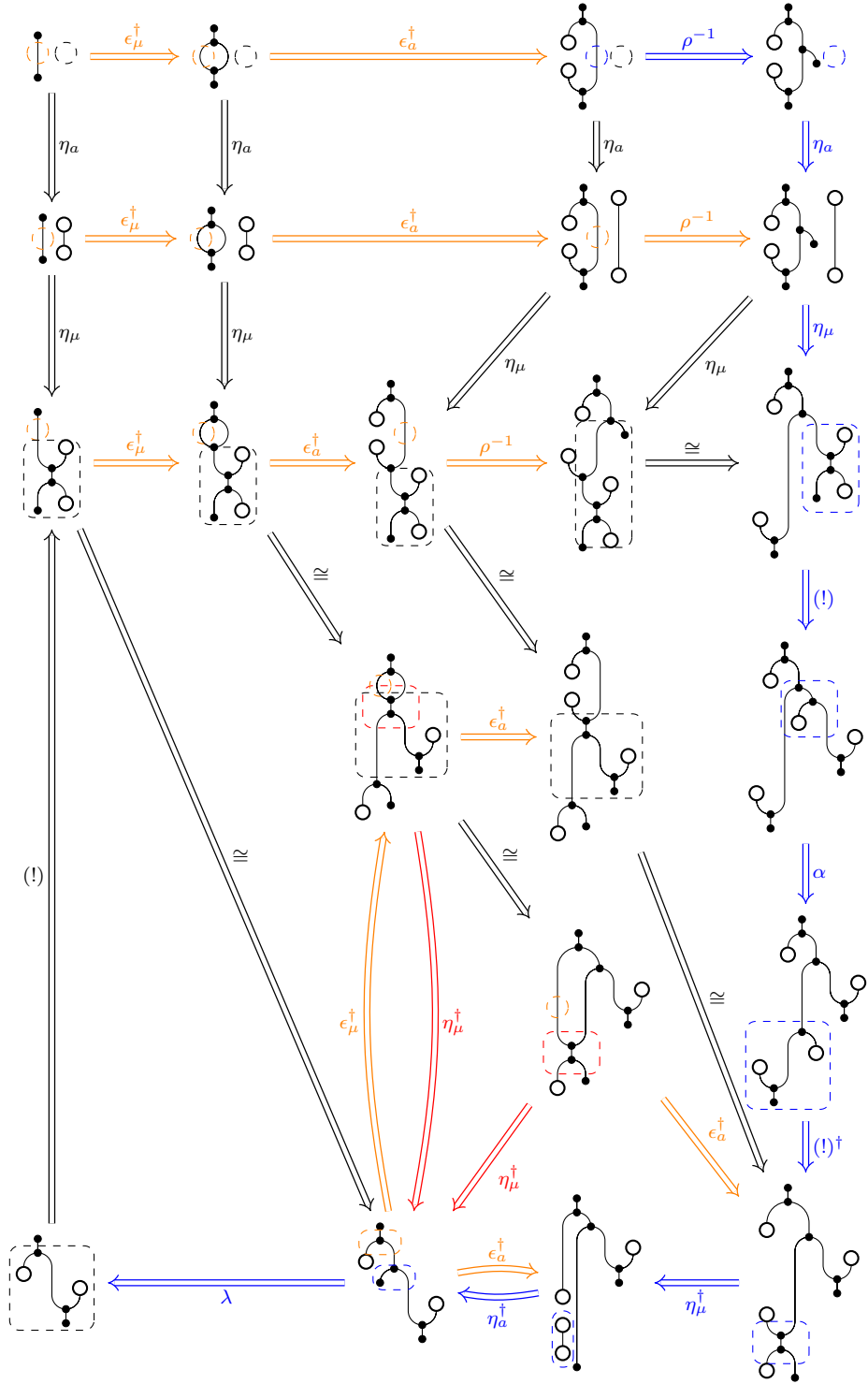
\end{proof}

\begin{cor}\label{cor:underlying-isometric}
When $\fC=2\Hilb$ and $(\cA,\Tr^\cA)\in 2\Hilb$ is an $\rmH^*$-algebra,
$(|\cA|, \vee,\psi_{|\cA|})$ is an
$\rmH^*$-multifusion category.
Moreover, the canonical unitary monoidal equivalence $\cA\cong |\cA|$ is isometric.
\end{cor}
\begin{proof}
Recall that the canonical map $\cA\to |\cA|$ is given by 
$a\mapsto (-\otimes a : H\mapsto H\otimes a)$
on objects and by
$(f:a\to a')\mapsto (-\otimes f : -\otimes a\Rightarrow -\otimes a')$.
It remains to show this equivalence is isometric from $(\cA,\Tr^\cA)$ to $(|\cA|, \psi_{|\cA|}\circ \tr_R^\vee)$.

We begin by noticing that since $\eta_\mu\circ\eta_a\circ\eta_\iota$ and $\eta_{\iota\otimes_{|\cA|} a}$ are related by a unitary (the canonical tensorator of the UAF), we have

\[
(\psi_{|A|}\circ \tr^\vee_{R})(f)
=
\Psi^{2\Hilb}_{\Hilb}\left(
\tikzmath{
\draw[rounded corners, dotted] (0,0) rectangle (.5,.5);
}
\xRightarrow{\eta_{\iota\otimes_{|\cA|} a}}
\tikzmath{
\draw (0,.1) coordinate (d1) -- +(0,.1) arc (180:90:.2) coordinate (d2) arc (90:0:.2) -- +(0,-.2) coordinate (a1) -- +(0,0) arc (0:90:.2) -- +(0,.2) coordinate (d3) arc (-90:0:.2) -- +(0,.2) coordinate (a2) -- +(0,0) arc (0:-180:.2) -- +(0,.1) coordinate (d4);
\filldraw (d1) circle (.05);
\filldraw (d2) circle (.05);
\filldraw (d3) circle (.05);
\filldraw (d4) circle (.05);
\filldraw[thick, fill=\aColor] (a1) circle (.1);
\filldraw[thick, fill=\aColor] (a2) circle (.1);
\draw[rounded corners, dashed] (.2,-.15) rectangle +(.4,.3);
\draw (.55,.5) rectangle +(-.65,.7);
}
\xRightarrow{f}
\tikzmath{
\draw (0,.1) coordinate (d1) -- +(0,.1) arc (180:90:.2) coordinate (d2) arc (90:0:.2) -- +(0,-.2) coordinate (a1) -- +(0,0) arc (0:90:.2) -- +(0,.2) coordinate (d3) arc (-90:0:.2) -- +(0,.2) coordinate (a2) -- +(0,0) arc (0:-180:.2) -- +(0,.1) coordinate (d4);
\filldraw (d1) circle (.05);
\filldraw (d2) circle (.05);
\filldraw (d3) circle (.05);
\filldraw (d4) circle (.05);
\filldraw[thick, fill=\aColor] (a1) circle (.1);
\filldraw[thick, fill=\aColor] (a2) circle (.1);
\draw (.55,.5) rectangle +(-.65,.7);
}
\xRightarrow{\eta_{\iota\otimes_{|\cA|} a}^\dag}
\tikzmath{
\draw[rounded corners, dotted] (0,0) rectangle (.5,.5);
}\right),
\]
where the box indicates the unitary adjoint, as in the proof of Proposition \ref{prop:underlying-hstar-udf}.
Now $\iota\otimes_{|\cA|} a$ and $a$ are unitarily isomorphic, so the above is equal to
\begin{equation}
\label{eq:FormulaForPsi|A|}
\Psi^{2\Hilb}_{\Hilb}\left(
\tikzmath{
\draw[rounded corners, dotted] (0,0) rectangle (.5,.5);
}
\xRightarrow{\eta_a}
\tikzmath{
\draw (0,0) coordinate (a1) -- +(0,.4) coordinate (a2);
\draw[rounded corners, dashed] (-.2,-.15) rectangle +(.4,.3);
\filldraw[thick, fill=\aColor] (a1) circle (.1);
\filldraw[thick, fill=\aColor] (a2) circle (.1);
}
\xRightarrow{f}
\tikzmath{
\draw (0,0) coordinate (a1) -- +(0,.4) coordinate (a2);
\filldraw[thick, fill=\aColor] (a1) circle (.1);
\filldraw[thick, fill=\aColor] (a2) circle (.1);
}
\xRightarrow{\eta_a^\dag}
\tikzmath{
\draw[rounded corners, dotted] (0,0) rectangle (.5,.5);
}
\right).
\end{equation}
Note that by Sub-Example \ref{sub-ex:UnitaryAdjointOf-otimesA}, $(\eta_a)_\bbC$ is the mate of $\id_a$ under the unitary adjunction
$$
\cA(\bbC\otimes a\to a) \cong \Hilb(\bbC\to \cA(a\to a)).
$$
We can now directly use the proof of \cite[Lem.~2.7]{MR4750417} and the definition of $\Psi^{2\Hilb}_\Hilb$ in Sub-Example \ref{ex:WeightOn2Hilb}:
\begin{align*}
\Tr^\cA_a(f)
&=
\langle \id_a|f\rangle_{\cA(a \to a)} 
\\&=
\langle \mate(\id_a)|\mate(f)\rangle_{\Hilb(\bbC\to \cA(a\to a))} 
\\
&=
\langle (\eta_a)_\bbC|(\id_{a^*}\otimes f)\circ (\eta_a)_\bbC\rangle_{\Hilb(\bbC\to \cA(a\to a))} 
\\
&= 
\Tr^{\Hilb}_{\bbC}( (\eta_a)_\bbC^\dag \circ (\id_{a^*}\otimes f) \circ (\eta_a)_\bbC)
\\
&=
\Psi^{2\Hilb}_\Hilb(\eta_a^\dag \circ (\id_{a^*}\otimes f)\circ \eta_a )
=
\eqref{eq:FormulaForPsi|A|}.
\qedhere
\end{align*}
\end{proof}

\begin{prop}
\label{prop:underlying-module-trace}
Suppose $\fC$ is a $\sC^*\Gray$ monoid 
whose underlying strict $\rmC^*$ 2-category is a pre-3-Hilbert space.
Suppose $A\in\fC$ is an $\rmH^*$-algebra and $M\in\fC$ is a unital $A$-module.
Similar to Construction \ref{const:MultiFusCatFromRigid2Alg}, $|M|:=\fC(1_\fC\to M)$ is a unital unitary $|A|$-module category.
The map
\[
\Tr^{|\cM|}_m(f) := \Psi^{\fC}_{1_\fC}\left(
\tikzmath{
\draw[rounded corners=5pt, dotted] (0,0) rectangle (.5,.5);
}
\xRightarrow{\eta_m}
\tikzmath{
\draw (0,0) -- (0,.4);
\filldraw[fill=white, thick] (0,0) circle (.1);
\filldraw[fill=white, thick] (0,.4) circle (.1);
\draw[rounded corners=5pt,dashed] (-.2,-.2) rectangle (.2,.2);
}
\xRightarrow{f}
\tikzmath{
\draw (0,0) -- (0,.4);
\filldraw[fill=white, thick] (0,0) circle (.1);
\filldraw[fill=white, thick] (0,.4) circle (.1);
}
\xRightarrow{\eta_m^\dag}
\tikzmath{
\draw[rounded corners=5pt, dotted] (0,0) rectangle (.5,.5);
}
\right)
\qquad\qquad
f:m\to m,
\quad
m\in |M|
\]
is an $|A|$-module trace on $|M|$
\end{prop}
\begin{proof}
We need to verify that for any endomorphism $f:a\lhd m \to a\lhd m$ (where $a \in |A|$), we have
\begin{equation}\label{eq:underlying-module-trace}
    \Tr^{|M|}_{a\rhd m}(f) = \Tr^{|M|}_{m}\left((\coev_a^\dag\rhd m )\circ (a^\vee\rhd f)\circ(\coev_a\rhd m)\right).
\end{equation}
The left side is given by evaluating $\Psi^{\fC}_{1_\fC}$ on
\[
\tikzmath{
\draw[rounded corners=5pt, dashed] (0,0) rectangle (.5,.5);
}
\xRightarrow{\eta_{a\lhd m}}
\tikzmath{
\draw[thick,\AsColor] (0,0) coordinate (a1) -- +(0,.1) arc (180:90:.2) coordinate (x1);
\draw[thick,\MsColor] (x1) -- +(0,-.5) coordinate (m1) -- +(0,.2) coordinate (x2) -- +(0,.6) coordinate (m2);
\draw[thick,\AsColor] (x2) arc (-90:-180:.2) coordinate (a2);
\filldraw[thick,draw=\AsColor,fill=white] (a1) circle (.1);
\filldraw[thick,draw=\AsColor,fill=white] (a2) circle (.1);
\filldraw[thick,draw=\MsColor,fill=white] (m1) circle (.1);
\filldraw[thick,draw=\MsColor,fill=white] (m2) circle (.1);
\filldraw[thick,\AsColor] (x1) circle (.05);
\filldraw[thick,\AsColor] (x2) circle (.05);
\draw[rounded corners=5pt,dashed] (-.2,-.4) rectangle (.4,.4);
}
\xRightarrow{f}
\tikzmath{
\draw[thick,\AsColor] (0,0) coordinate (a1) -- +(0,.1) arc (180:90:.2) coordinate (x1);
\draw[thick,\MsColor] (x1) -- +(0,-.5) coordinate (m1) -- +(0,.2) coordinate (x2) -- +(0,.6) coordinate (m2);
\draw[thick,\AsColor] (x2) arc (-90:-180:.2) coordinate (a2);
\filldraw[thick,draw=\AsColor,fill=white] (a1) circle (.1);
\filldraw[thick,draw=\AsColor,fill=white] (a2) circle (.1);
\filldraw[thick,draw=\MsColor,fill=white] (m1) circle (.1);
\filldraw[thick,draw=\MsColor,fill=white] (m2) circle (.1);
\filldraw[thick,\AsColor] (x1) circle (.05);
\filldraw[thick,\AsColor] (x2) circle (.05);
}
\xRightarrow{\eta_{a\lhd m}^\dag}
\tikzmath{
\draw[rounded corners=5pt, dashed] (0,0) rectangle (.5,.5);
}\,,
\qquad\qquad
\tikzmath{
\draw[thick,\MsColor] (0,0) -- +(0,.4);
} = M,\qquad
\tikzmath{
\draw[thick,\AsColor] (0,0) -- +(0,.4);
} = A,\qquad
\tikzmath{
\draw[thick,\MsColor] (0,0) -- +(0,.4);
\filldraw[thick,fill=white,draw=\MsColor] (0,0) circle (.1);
}=m,\qquad
\tikzmath{
\draw[thick,\AsColor] (0,0) -- +(0,.4);
\filldraw[thick,fill=white,draw=\AsColor] (0,0) circle (.1);
}=a.
\]
Slightly abusing the notation from Construction \ref{const:MultiFusCatFromRigid2Alg}, we can write the right side of Equation \eqref{eq:underlying-module-trace} as the evaluation of $\Psi^{\fC}_{1_\fC}$ on
\[
\tikzmath{
\draw[rounded corners=5pt, dashed] (0,0) rectangle (.5,.5);
}
\xRightarrow{\eta_m}
\tikzmath{
\draw[thick,\MsColor] (0,0) coordinate (m1) -- +(0,.4) coordinate (m2);
\filldraw[thick,fill=white,draw=\MsColor] (m1) circle (.1);
\filldraw[thick,fill=white,draw=\MsColor] (m2) circle (.1);
}
\xRightarrow{\coev^R_{a}}
\tikzmath{
\draw[thick,\AsColor] (0,0) coordinate (a1) -- +(0,.1) arc (180:90:.2) coordinate (x1);
\draw[thick,\MsColor] (x1) -- +(0,-.5) coordinate (m1) -- +(0,.8) coordinate (x2) -- +(0,1.1) coordinate (m2);
\draw[thick,\AsColor] (x2) arc (90:180:.2) arc (0:-90:.2) coordinate (x3) -- +(0,-.2) coordinate (x4) -- +(0,0) arc (-90:-180:.2) -- +(0,.1) coordinate (a2);
\filldraw[thick,fill=white,draw=\MsColor] (m1) circle (.1);
\filldraw[thick,fill=white,draw=\MsColor] (m2) circle (.1);
\filldraw[thick,fill=white,draw=\AsColor] (a1) circle (.1);
\filldraw[thick,fill=white,draw=\AsColor] (a2) circle (.1);
\filldraw[thick,\AsColor] (x1) circle (.05);
\filldraw[thick,\AsColor] (x2) circle (.05);
\filldraw[thick,\AsColor] (x3) circle (.05);
\filldraw[thick,\AsColor] (x4) circle (.05);
\draw[rounded corners=5pt,dashed] (-.2,-.4) rectangle (.4,.4);
}
\xRightarrow{f}
\tikzmath{
\draw[thick,\AsColor] (0,0) coordinate (a1) -- +(0,.1) arc (180:90:.2) coordinate (x1);
\draw[thick,\MsColor] (x1) -- +(0,-.5) coordinate (m1) -- +(0,.8) coordinate (x2) -- +(0,1.1) coordinate (m2);
\draw[thick,\AsColor] (x2) arc (90:180:.2) arc (0:-90:.2) coordinate (x3) -- +(0,-.2) coordinate (x4) -- +(0,0) arc (-90:-180:.2) -- +(0,.1) coordinate (a2);
\filldraw[thick,fill=white,draw=\MsColor] (m1) circle (.1);
\filldraw[thick,fill=white,draw=\MsColor] (m2) circle (.1);
\filldraw[thick,fill=white,draw=\AsColor] (a1) circle (.1);
\filldraw[thick,fill=white,draw=\AsColor] (a2) circle (.1);
\filldraw[thick,\AsColor] (x1) circle (.05);
\filldraw[thick,\AsColor] (x2) circle (.05);
\filldraw[thick,\AsColor] (x3) circle (.05);
\filldraw[thick,\AsColor] (x4) circle (.05);
}
\xRightarrow{(\coev_a^R)^\dag}
\tikzmath{
\draw[thick,\MsColor] (0,0) coordinate (m1) -- +(0,.4) coordinate (m2);
\filldraw[thick,fill=white,draw=\MsColor] (m1) circle (.1);
\filldraw[thick,fill=white,draw=\MsColor] (m2) circle (.1);
}
\xRightarrow{\eta_m^\dag}
\tikzmath{
\draw[rounded corners=5pt, dashed] (0,0) rectangle (.5,.5);
}.
\]
Suppressing interchangers, $\coev^R_a\rhd m$ is the composite
\[
\tikzmath{
\draw[thick,\MsColor] (0,0) coordinate (m1) -- +(0,.4) coordinate (m2);
\filldraw[thick,fill=white,draw=\MsColor] (m1) circle (.1);
}
\xRightarrow{\lambda^{-1}}
\tikzmath{
\draw[thick,\MsColor] (0,0) coordinate (m1) -- +(0,.4) coordinate (x1) -- +(0,.8);
\draw[thick,\AsColor] (x1) arc (90:180:.2) coordinate (x2);
\filldraw[thick,fill=\AsColor,draw=\AsColor] (x1) circle (.05);
\filldraw[thick,fill=\AsColor,draw=\AsColor] (x2) circle (.05);
\filldraw[thick,fill=white,draw=\MsColor] (m1) circle (.1);
}
\xRightarrow{\eta_a}
\tikzmath{
\draw[thick,\MsColor] (0,0) coordinate (m1) -- +(0,1) coordinate (x1) -- +(0,1.2);
\draw[thick,\AsColor] (x1) arc (90:180:.2) -- +(0,-.6) coordinate (x2);
\draw[draw=none] (x2) -- +(-.2,.2) coordinate (a1);
\draw[thick,\AsColor] (a1) -- +(0,.4) coordinate (a2);
\filldraw[thick,fill=\AsColor,draw=\AsColor] (x1) circle (.05);
\filldraw[thick,fill=\AsColor,draw=\AsColor] (x2) circle (.05);
\filldraw[thick,fill=white,draw=\MsColor] (m1) circle (.1);
\filldraw[thick,fill=white,draw=\AsColor] (a1) circle (.1);
\filldraw[thick,fill=white,draw=\AsColor] (a2) circle (.1);
}
\xRightarrow{\eta_\mu}
\tikzmath{
\draw[thick,\AsColor] (0,0) coordinate (a1) -- +(0,.1) arc (180:90:.2) coordinate (m1) arc (90:0:.2) coordinate (m2) arc (0:90:.2) -- +(0,.2) coordinate (m3) arc (-90:-180:.2) -- +(0,.1) coordinate (a2) -- +(0,0) arc (-180:0:.2) arc (180:90:.2) coordinate (m4);
\draw[thick,\MsColor] (m4) -- +(0,.3) coordinate (x1) -- +(0,-1.2) coordinate (x2);
\filldraw[thick,fill=white,draw=\AsColor] (a1) circle (.1);
\filldraw[thick,fill=white,draw=\AsColor] (a2) circle (.1);
\filldraw[thick,fill=\AsColor,draw=\AsColor] (m1) circle (.05);
\filldraw[thick,fill=\AsColor,draw=\AsColor] (m2) circle (.05);
\filldraw[thick,fill=\AsColor,draw=\AsColor] (m3) circle (.05);
\filldraw[thick,fill=\AsColor,draw=\AsColor] (m4) circle (.05);
\filldraw[thick,fill=white,draw=\MsColor] (x2) circle (.1);
}
\xRightarrow{\lambda^{-1}\circ\rho}
\tikzmath{
\draw[thick,\AsColor] (0,0) coordinate (aa1) arc (180:90:.2) coordinate (m1) arc (90:0:.2) -- +(0,-.1) coordinate (mm2) -- +(0,0) arc (0:90:.2) -- +(0,.2) coordinate (m3) arc (-90:-180:.2) -- +(0,.1) coordinate (a2) -- +(0,0) arc (-180:0:.2) arc (180:90:.2) coordinate (m4);
\draw[thick,\MsColor] (m4) -- +(0,.3) coordinate (x1) -- +(0,-1.2) coordinate (x2);
\filldraw[thick,fill=white,draw=\AsColor] (mm2) circle (.1);
\filldraw[thick,fill=white,draw=\AsColor] (a2) circle (.1);
\filldraw[thick,fill=\AsColor,draw=\AsColor] (m1) circle (.05);
\filldraw[thick,fill=\AsColor,draw=\AsColor] (aa1) circle (.05);
\filldraw[thick,fill=\AsColor,draw=\AsColor] (m3) circle (.05);
\filldraw[thick,fill=\AsColor,draw=\AsColor] (m4) circle (.05);
\filldraw[thick,fill=white,draw=\MsColor] (x2) circle (.1);
}
\xRightarrow{\kappa}
\tikzmath{
\draw[thick,\AsColor] (0,0) coordinate (a1) -- +(0,.7) arc (0:90:.2) coordinate (x1) arc (90:180:.2) -- +(0,-.2) arc (0:-90:.2) coordinate (x2) arc (-90:-180:.2) -- +(0,.1) coordinate (a2);
\draw[thick,\AsColor] (x2) -- +(0,-.2) coordinate (x3);
\draw[thick,\AsColor] (x1) arc (180:90:.4) coordinate (x4);
\draw[thick,\MsColor] (x4) -- +(0,.3) coordinate (m1) -- +(0,-1.5) coordinate (m2);
\filldraw[thick,fill=white,draw=\MsColor] (m2) circle (.1);
\filldraw[thick,fill=white,draw=\AsColor] (a1) circle (.1);
\filldraw[thick,fill=white,draw=\AsColor] (a2) circle (.1);
\filldraw[thick,\AsColor] (x1) circle (.05);
\filldraw[thick,\AsColor] (x2) circle (.05);
\filldraw[thick,\AsColor] (x3) circle (.05);
\filldraw[thick,\AsColor] (x4) circle (.05);
}
\xRightarrow{\alpha}
\tikzmath{
\draw[thick,\AsColor] (0,0) coordinate (a1) -- +(0,.1) arc (180:90:.2) coordinate (x1);
\draw[thick,\MsColor] (x1) -- +(0,-.5) coordinate (m1) -- +(0,.8) coordinate (x2) -- +(0,1.1) coordinate (m2);
\draw[thick,\AsColor] (x2) arc (90:180:.2) arc (0:-90:.2) coordinate (x3) -- +(0,-.2) coordinate (x4) -- +(0,0) arc (-90:-180:.2) -- +(0,.1) coordinate (a2);
\filldraw[thick,fill=white,draw=\MsColor] (m1) circle (.1);
\filldraw[thick,fill=white,draw=\AsColor] (a1) circle (.1);
\filldraw[thick,fill=white,draw=\AsColor] (a2) circle (.1);
\filldraw[thick,\AsColor] (x1) circle (.05);
\filldraw[thick,\AsColor] (x2) circle (.05);
\filldraw[thick,\AsColor] (x3) circle (.05);
\filldraw[thick,\AsColor] (x4) circle (.05);
}.
\]
We now show that the two morphisms on the left and right hand sides of Equation \eqref{eq:underlying-module-trace} are in fact the same by commutativity of the diagram

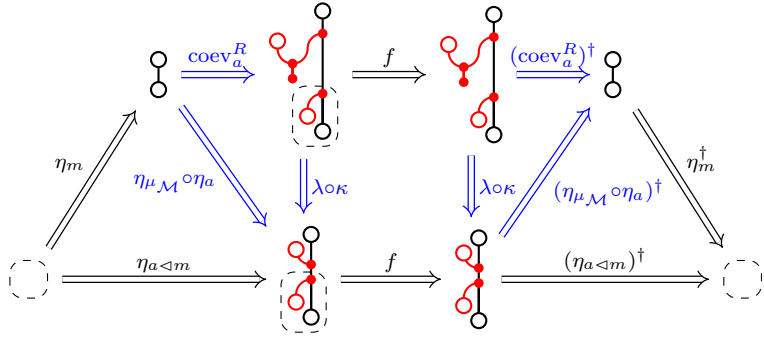
\[
\begin{tikzcd}
&
\tikzmath{
\draw[thick,\MsColor] (0,0) coordinate (m1) -- +(0,.4) coordinate (m2);
\filldraw[thick,fill=white,draw=\MsColor] (m1) circle (.1);
\filldraw[thick,fill=white,draw=\MsColor] (m2) circle (.1);
}
\arrow[r,Rightarrow,"\coev^R_a",blue]
\arrow[dr,Rightarrow,"\eta_{\mu_{\cM}}\circ\eta_a"',blue]
&
\tikzmath{
\draw[thick,\AsColor] (0,0) coordinate (a1) -- +(0,.1) arc (180:90:.2) coordinate (x1);
\draw[thick,\MsColor] (x1) -- +(0,-.5) coordinate (m1) -- +(0,.8) coordinate (x2) -- +(0,1.1) coordinate (m2);
\draw[thick,\AsColor] (x2) arc (90:180:.2) arc (0:-90:.2) coordinate (x3) -- +(0,-.2) coordinate (x4) -- +(0,0) arc (-90:-180:.2) -- +(0,.1) coordinate (a2);
\filldraw[thick,fill=white,draw=\MsColor] (m1) circle (.1);
\filldraw[thick,fill=white,draw=\MsColor] (m2) circle (.1);
\filldraw[thick,fill=white,draw=\AsColor] (a1) circle (.1);
\filldraw[thick,fill=white,draw=\AsColor] (a2) circle (.1);
\filldraw[thick,\AsColor] (x1) circle (.05);
\filldraw[thick,\AsColor] (x2) circle (.05);
\filldraw[thick,\AsColor] (x3) circle (.05);
\filldraw[thick,\AsColor] (x4) circle (.05);
\draw[rounded corners=5pt,dashed] (-.2,-.4) rectangle (.4,.4);
}
\arrow[r,Rightarrow,"f"]
\arrow[d,Rightarrow,"\lambda\circ\kappa",blue]
&
\tikzmath{
\draw[thick,\AsColor] (0,0) coordinate (a1) -- +(0,.1) arc (180:90:.2) coordinate (x1);
\draw[thick,\MsColor] (x1) -- +(0,-.5) coordinate (m1) -- +(0,.8) coordinate (x2) -- +(0,1.1) coordinate (m2);
\draw[thick,\AsColor] (x2) arc (90:180:.2) arc (0:-90:.2) coordinate (x3) -- +(0,-.2) coordinate (x4) -- +(0,0) arc (-90:-180:.2) -- +(0,.1) coordinate (a2);
\filldraw[thick,fill=white,draw=\MsColor] (m1) circle (.1);
\filldraw[thick,fill=white,draw=\MsColor] (m2) circle (.1);
\filldraw[thick,fill=white,draw=\AsColor] (a1) circle (.1);
\filldraw[thick,fill=white,draw=\AsColor] (a2) circle (.1);
\filldraw[thick,\AsColor] (x1) circle (.05);
\filldraw[thick,\AsColor] (x2) circle (.05);
\filldraw[thick,\AsColor] (x3) circle (.05);
\filldraw[thick,\AsColor] (x4) circle (.05);
}
\arrow[r,Rightarrow,"(\coev^R_a)^\dag",blue]
\arrow[d,Rightarrow,"\lambda\circ\kappa",blue]
&
\tikzmath{
\draw[thick,\MsColor] (0,0) coordinate (m1) -- +(0,.4) coordinate (m2);
\filldraw[thick,fill=white,draw=\MsColor] (m1) circle (.1);
\filldraw[thick,fill=white,draw=\MsColor] (m2) circle (.1);
}
\arrow[dr,Rightarrow,"\eta_m^\dag"]
\\
\tikzmath{
\draw[rounded corners=5pt, dashed] (0,0) rectangle (.5,.5);
}
\arrow[ur,Rightarrow,"\eta_m"]
\arrow[rr,Rightarrow,"\eta_{a\lhd m}"]
&&
\tikzmath{
\draw[thick,\AsColor] (0,0) coordinate (a1) -- +(0,.1) arc (180:90:.2) coordinate (x1);
\draw[thick,\MsColor] (x1) -- +(0,-.5) coordinate (m1) -- +(0,.2) coordinate (x2) -- +(0,.6) coordinate (m2);
\draw[thick,\AsColor] (x2) arc (-90:-180:.2) coordinate (a2);
\filldraw[thick,draw=\AsColor,fill=white] (a1) circle (.1);
\filldraw[thick,draw=\AsColor,fill=white] (a2) circle (.1);
\filldraw[thick,draw=\MsColor,fill=white] (m1) circle (.1);
\filldraw[thick,draw=\MsColor,fill=white] (m2) circle (.1);
\filldraw[thick,\AsColor] (x1) circle (.05);
\filldraw[thick,\AsColor] (x2) circle (.05);
\draw[rounded corners=5pt,dashed] (-.2,-.4) rectangle (.4,.4);
}
\arrow[r,Rightarrow,"f"]
&
\tikzmath{
\draw[thick,\AsColor] (0,0) coordinate (a1) -- +(0,.1) arc (180:90:.2) coordinate (x1);
\draw[thick,\MsColor] (x1) -- +(0,-.5) coordinate (m1) -- +(0,.2) coordinate (x2) -- +(0,.6) coordinate (m2);
\draw[thick,\AsColor] (x2) arc (-90:-180:.2) coordinate (a2);
\filldraw[thick,draw=\AsColor,fill=white] (a1) circle (.1);
\filldraw[thick,draw=\AsColor,fill=white] (a2) circle (.1);
\filldraw[thick,draw=\MsColor,fill=white] (m1) circle (.1);
\filldraw[thick,draw=\MsColor,fill=white] (m2) circle (.1);
\filldraw[thick,\AsColor] (x1) circle (.05);
\filldraw[thick,\AsColor] (x2) circle (.05);
}
\arrow[rr,Rightarrow,"(\eta_{a\lhd m})^\dag"]
\arrow[ur,Rightarrow,"(\eta_{\mu_{\cM}}\circ\eta_a)^\dag"',blue]
&&
\tikzmath{
\draw[rounded corners=5pt, dashed] (0,0) rectangle (.5,.5);
}
\end{tikzcd}.
\]
Notice the similarity to Proposition \ref{prop:underlying-hstar-spherical-weight}.
The blue triangles can be seen to commute by the diagram
\[
\begin{tikzcd}
\tikzmath{
\draw[thick,draw=\AsColor] (0,-0.1) coordinate (a1) -- +(0,.6) coordinate (a2);
\draw[thick,draw=\MsColor] (.5,-.3) coordinate (m1) -- +(0,.6) coordinate (x) -- +(0,1) coordinate (m2);
\draw[thick,draw=\AsColor] (x) arc (90:180:.2) -- +(0,-.1) coordinate (a3);
\filldraw[thick,fill=white,draw=\AsColor] (a1) circle (.1);
\filldraw[thick,fill=white,draw=\AsColor] (a2) circle (.1);
\filldraw[thick,fill=\AsColor,draw=\AsColor] (a3) circle (.05);
\filldraw[thick,fill=\AsColor,draw=\AsColor] (x) circle (.05);
\filldraw[thick,fill=white,draw=\MsColor] (m1) circle (.1);
\filldraw[thick,fill=white,draw=\MsColor] (m2) circle (.1);
\draw[rounded corners,dashed,orange] (-.2,0.1) rectangle +(.6,.25);
}
\arrow[r,Rightarrow,"\eta_{\mu_\cA}",orange]
\arrow[ddr,Rightarrow,"\eta_{\mu_\cM}"]
&
\tikzmath{
\draw[thick,\AsColor] (0,0) coordinate (a1) -- +(0,.1) arc (180:90:.2) coordinate (m1) arc (90:0:.2) coordinate (m2) arc (0:90:.2) -- +(0,.2) coordinate (m3) arc (-90:-180:.2) -- +(0,.1) coordinate (a2) -- +(0,0) arc (-180:0:.2) arc (180:90:.2) coordinate (m4);
\draw[thick,\MsColor] (m4) -- +(0,.3) coordinate (x1) -- +(0,-1.2) coordinate (x2);
\filldraw[thick,fill=white,draw=\AsColor] (a1) circle (.1);
\filldraw[thick,fill=white,draw=\AsColor] (a2) circle (.1);
\filldraw[thick,fill=\AsColor,draw=\AsColor] (m1) circle (.05);
\filldraw[thick,fill=\AsColor,draw=\AsColor] (m2) circle (.05);
\filldraw[thick,fill=\AsColor,draw=\AsColor] (m3) circle (.05);
\filldraw[thick,fill=\AsColor,draw=\AsColor] (m4) circle (.05);
\filldraw[thick,fill=white,draw=\MsColor] (x1) circle (.1);
\filldraw[thick,fill=white,draw=\MsColor] (x2) circle (.1);
\draw[rounded corners,dashed,blue] (-.2,0.1) rectangle +(.75,.3);
}
\arrow[drr,Rightarrow,"\eta_{\mu_\cM}"]
\arrow[rrr,Rightarrow,"\rho",blue]
&&&
\tikzmath{
\draw[thick,\AsColor] (0,0) coordinate (a1) -- +(0,.3) coordinate (m1) arc (-90:-180:.2) -- +(0,.3) coordinate (a2) -- +(0,0) arc (-180:0:.2) arc (180:90:.2) coordinate (m2);
\draw[thick,\MsColor] (m2) -- +(0,.3) coordinate (x1) -- +(0,-1) coordinate (x2);
\filldraw[thick,fill=white,draw=\AsColor] (a1) circle (.1);
\filldraw[thick,fill=white,draw=\AsColor] (a2) circle (.1);
\filldraw[thick,fill=\AsColor,draw=\AsColor] (m1) circle (.05);
\filldraw[thick,fill=\AsColor,draw=\AsColor] (m2) circle (.05);
\filldraw[thick,fill=white,draw=\MsColor] (x1) circle (.1);
\filldraw[thick,fill=white,draw=\MsColor] (x2) circle (.1);
\draw[rounded corners,dashed,blue] (-.15,0) rectangle +(.3,.3);
\draw[rounded corners,dashed,orange] (-.3,0.2) rectangle +(.85,.55);
}
\arrow[r,Rightarrow,"\lambda^\dag",blue]
\arrow[d,Rightarrow,"\eta_{\mu_\cM}",swap]
\arrow[ddd,Rightarrow,bend left=20,"\kappa",orange]
&
\tikzmath{
\draw[thick,\AsColor] (0,0) coordinate (aa1) arc (180:90:.2) coordinate (m1) arc (90:0:.2) -- +(0,-.1) coordinate (mm2) -- +(0,0) arc (0:90:.2) -- +(0,.2) coordinate (m3) arc (-90:-180:.2) -- +(0,.3) coordinate (a2) -- +(0,0) arc (-180:0:.2) arc (180:90:.2) coordinate (m4);
\draw[thick,\MsColor] (m4) -- +(0,.3) coordinate (x1) -- +(0,-1.2) coordinate (x2);
\filldraw[thick,fill=white,draw=\AsColor] (mm2) circle (.1);
\filldraw[thick,fill=white,draw=\AsColor] (a2) circle (.1);
\filldraw[thick,fill=\AsColor,draw=\AsColor] (m1) circle (.05);
\filldraw[thick,fill=\AsColor,draw=\AsColor] (aa1) circle (.05);
\filldraw[thick,fill=\AsColor,draw=\AsColor] (m3) circle (.05);
\filldraw[thick,fill=\AsColor,draw=\AsColor] (m4) circle (.05);
\filldraw[thick,fill=white,draw=\MsColor] (x1) circle (.1);
\filldraw[thick,fill=white,draw=\MsColor] (x2) circle (.1);
\draw[rounded corners,dashed,orange] (-.1,0.3) rectangle +(.85,.55);
\draw[rounded corners,dashed] (-.05,0) rectangle +(.55,.55);
}
\arrow[r,Rightarrow,"\kappa^{-1}"]
\arrow[dd,Rightarrow,"\kappa",orange]
&
\tikzmath{
\draw[thick,\AsColor] (0,0) coordinate (a1) -- +(0,.7) arc (0:90:.2) coordinate (x1) arc (90:180:.2) -- +(0,-.2) arc (0:-90:.2) coordinate (x2) arc (-90:-180:.2) -- +(0,.1) coordinate (a2);
\draw[thick,\AsColor] (x2) -- +(0,-.2) coordinate (x3);
\draw[thick,\AsColor] (x1) arc (180:90:.4) coordinate (x4);
\draw[thick,\MsColor] (x4) -- +(0,.3) coordinate (m1) -- +(0,-1.5) coordinate (m2);
\filldraw[thick,fill=white,draw=\MsColor] (m1) circle (.1);
\filldraw[thick,fill=white,draw=\MsColor] (m2) circle (.1);
\filldraw[thick,fill=white,draw=\AsColor] (a1) circle (.1);
\filldraw[thick,fill=white,draw=\AsColor] (a2) circle (.1);
\filldraw[thick,\AsColor] (x1) circle (.05);
\filldraw[thick,\AsColor] (x2) circle (.05);
\filldraw[thick,\AsColor] (x3) circle (.05);
\filldraw[thick,\AsColor] (x4) circle (.05);
\draw[rounded corners,dashed,cyan] (-.6,0.6) rectangle +(1,.8);
}
\arrow[dd,Rightarrow,"\alpha^\dag",cyan]
\\
&&&
\tikzmath{
\draw[thick,\AsColor] (0,0) coordinate (a1) -- +(0,.1) arc (180:90:.2) coordinate (m1) arc (90:0:.2) coordinate (m2);
\draw[thick,\AsColor] (m1) -- +(0,.2) coordinate (m3) arc (-90:0:.2) arc (180:90:.2) coordinate (m6);
\draw[thick,\AsColor] (m3) arc (-90:-180:.2) arc (180:90:.6) coordinate (m4);
\draw[thick,\MsColor] (m4) -- +(0,.2) coordinate (m5) -- +(0,-1.4) coordinate (x1);
\draw[thick,\MsColor] (m5) -- +(0,.5) coordinate (x2);
\draw[thick,\AsColor] (m5) arc (-90:-180:.2) -- +(0,.1) coordinate (a2);
\filldraw[thick,fill=white,draw=\AsColor] (a1) circle (.1);
\filldraw[thick,fill=white,draw=\AsColor] (a2) circle (.1);
\filldraw[thick,fill=\AsColor,draw=\AsColor] (m1) circle (.05);
\filldraw[thick,fill=\AsColor,draw=\AsColor] (m2) circle (.05);
\filldraw[thick,fill=\AsColor,draw=\AsColor] (m3) circle (.05);
\filldraw[thick,fill=\AsColor,draw=\AsColor] (m4) circle (.05);
\filldraw[thick,fill=\AsColor,draw=\AsColor] (m5) circle (.05);
\filldraw[thick,fill=\AsColor,draw=\AsColor] (m6) circle (.05);
\filldraw[thick,fill=white,draw=\MsColor] (x1) circle (.1);
\filldraw[thick,fill=white,draw=\MsColor] (x2) circle (.1);
\draw[rounded corners,dashed,blue] (-.2,0.1) rectangle +(.75,.3);
\draw[rounded corners,dashed,cyan] (-.2,0.6) rectangle +(1,.8);
}
\arrow[r,Rightarrow,"\rho",blue]
\arrow[d,Rightarrow,"\alpha",cyan]
&
\tikzmath{
\draw[thick,\AsColor] (0,0) coordinate (a1) -- +(0,.3) coordinate (m3) arc (-90:0:.2) arc (180:90:.2) coordinate (m6);
\draw[thick,\AsColor] (m3) arc (-90:-180:.2) arc (180:90:.6) coordinate (m4);
\draw[thick,\MsColor] (m4) -- +(0,.2) coordinate (m5) -- +(0,-1.2) coordinate (x1);
\draw[thick,\MsColor] (m5) -- +(0,.5) coordinate (x2);
\draw[thick,\AsColor] (m5) arc (-90:-180:.2) -- +(0,.1) coordinate (a2);
\filldraw[thick,fill=white,draw=\AsColor] (a1) circle (.1);
\filldraw[thick,fill=white,draw=\AsColor] (a2) circle (.1);
\filldraw[thick,fill=\AsColor,draw=\AsColor] (m3) circle (.05);
\filldraw[thick,fill=\AsColor,draw=\AsColor] (m4) circle (.05);
\filldraw[thick,fill=\AsColor,draw=\AsColor] (m5) circle (.05);
\filldraw[thick,fill=\AsColor,draw=\AsColor] (m6) circle (.05);
\filldraw[thick,fill=white,draw=\MsColor] (x1) circle (.1);
\filldraw[thick,fill=white,draw=\MsColor] (x2) circle (.1);
\draw[rounded corners,dashed,cyan] (-.4,0.4) rectangle +(1,.8);
}
\arrow[d,Rightarrow,"\alpha",cyan]
\\
&
\tikzmath{
\draw[thick,draw=\AsColor] (0,.9) coordinate (a1) -- +(0,-.1) arc (180:270:.2) coordinate (m1);
\draw[thick,draw=\AsColor] (.2,.4) coordinate (m2) arc (90:180:.4) -- +(0,-.1) coordinate (a2);
\draw[thick,draw=\AsColor] (0,0) coordinate (m3) arc (180:90:.2) coordinate (m4);
\draw[thick,draw=\MsColor] (.2,-.2) coordinate (x1) -- +(0,1.3) coordinate (x2);
\filldraw[thick,fill=white,draw=\AsColor] (a1) circle (.1);
\filldraw[thick,fill=white,draw=\AsColor] (a2) circle (.1);
\filldraw[thick,fill=\AsColor,draw=\AsColor] (m1) circle (.05);
\filldraw[thick,fill=\AsColor,draw=\AsColor] (m2) circle (.05);
\filldraw[thick,fill=\AsColor,draw=\AsColor] (m3) circle (.05);
\filldraw[thick,fill=\AsColor,draw=\AsColor] (m4) circle (.05);
\filldraw[thick,fill=white,draw=\MsColor] (x1) circle (.1);
\filldraw[thick,fill=white,draw=\MsColor] (x2) circle (.1);
\draw[rounded corners,dashed,orange] (-.35,0.1) rectangle +(.55,.25);
\draw[rounded corners,dashed,cyan] (-.4,0.05) rectangle +(.7,.45);
\draw[rounded corners,dashed,blue] (-.05,-0.1) rectangle +(.35,.35);
}
\arrow[r,Rightarrow,"\alpha",cyan]
\arrow[drrr,Rightarrow,"\lambda",swap,blue]
\arrow[urr,Rightarrow,"\eta_{\mu_\cA}",orange]
&
\tikzmath{
\draw[thick,\AsColor] (0,0) coordinate (a1) -- +(0,.1) arc (180:90:.2) coordinate (m1) arc (90:0:.2) coordinate (m2) arc (0:90:.2) arc (180:90:.3) coordinate (m3);
\draw[thick,\MsColor] (m3) -- +(0,.2) coordinate (m4) -- +(0,-.8) coordinate (x1);
\draw[thick,\MsColor] (m4) -- +(0,.5) coordinate (x2);
\draw[thick,\AsColor] (m4) arc (-90:-180:.2) coordinate (a2);
\filldraw[thick,fill=white,draw=\AsColor] (a1) circle (.1);
\filldraw[thick,fill=white,draw=\AsColor] (a2) circle (.1);
\filldraw[thick,fill=\AsColor,draw=\AsColor] (m1) circle (.05);
\filldraw[thick,fill=\AsColor,draw=\AsColor] (m2) circle (.05);
\filldraw[thick,fill=\AsColor,draw=\AsColor] (m3) circle (.05);
\filldraw[thick,fill=\AsColor,draw=\AsColor] (m4) circle (.05);
\filldraw[thick,fill=white,draw=\MsColor] (x1) circle (.1);
\filldraw[thick,fill=white,draw=\MsColor] (x2) circle (.1);
\draw[rounded corners,dashed,blue] (-.2,0.1) rectangle +(.75,.3);
\draw[rounded corners,dashed,orange] (-.15,0.15) rectangle +(.7,.2);
}
\arrow[drr,Rightarrow,"\rho",blue]
\arrow[r,bend left=10,Rightarrow,"\eta_{\mu_\cA}",orange]
&
\tikzmath{
\draw[thick,\AsColor] (0,0) coordinate (a1) -- +(0,.1) arc (180:90:.2) coordinate (m1) arc (90:0:.2) coordinate (m2) arc (0:90:.2) -- +(0,.2) coordinate (m5) arc (-90:450:.2) coordinate (m6) arc (180:90:.3) coordinate (m3);
\draw[thick,\MsColor] (m3) -- +(0,.2) coordinate (m4) -- +(0,-1.4) coordinate (x1);
\draw[thick,\MsColor] (m4) -- +(0,.5) coordinate (x2);
\draw[thick,\AsColor] (m4) arc (-90:-180:.2) coordinate (a2);
\filldraw[thick,fill=white,draw=\AsColor] (a1) circle (.1);
\filldraw[thick,fill=white,draw=\AsColor] (a2) circle (.1);
\filldraw[thick,fill=\AsColor,draw=\AsColor] (m1) circle (.05);
\filldraw[thick,fill=\AsColor,draw=\AsColor] (m2) circle (.05);
\filldraw[thick,fill=\AsColor,draw=\AsColor] (m3) circle (.05);
\filldraw[thick,fill=\AsColor,draw=\AsColor] (m4) circle (.05);
\filldraw[thick,fill=\AsColor,draw=\AsColor] (m5) circle (.05);
\filldraw[thick,fill=\AsColor,draw=\AsColor] (m6) circle (.05);
\filldraw[thick,fill=white,draw=\MsColor] (x1) circle (.1);
\filldraw[thick,fill=white,draw=\MsColor] (x2) circle (.1);
\draw[rounded corners,dashed,blue] (-.2,0.1) rectangle +(.75,.3);
\draw[rounded corners,dashed] (-.2,0.4) rectangle +(.75,.6);
}
\arrow[r,Rightarrow,"\rho",blue]
\arrow[l, bend left=10,Rightarrow,"\epsilon_{\mu_\cA}"]
&
\tikzmath{
\draw[thick,\AsColor] (0,0) coordinate (a1) -- +(0,.3) coordinate (m1) arc (-90:450:.2) coordinate (m2) arc (180:90:.3) coordinate (m3);
\draw[thick,\MsColor] (m3) -- +(0,.2) coordinate (m4) -- +(0,.7) coordinate (x1) -- +(0,-1.2) coordinate (x2);
\draw[thick,\AsColor] (m4) arc (-90:-180:.2) -- +(0,.1) coordinate (a2);
\filldraw[thick,fill=white,draw=\AsColor] (a1) circle (.1);
\filldraw[thick,fill=white,draw=\AsColor] (a2) circle (.1);
\filldraw[thick,fill=\AsColor,draw=\AsColor] (m1) circle (.05);
\filldraw[thick,fill=\AsColor,draw=\AsColor] (m2) circle (.05);
\filldraw[thick,fill=\AsColor,draw=\AsColor] (m3) circle (.05);
\filldraw[thick,fill=\AsColor,draw=\AsColor] (m4) circle (.05);
\filldraw[thick,fill=white,draw=\MsColor] (x1) circle (.1);
\filldraw[thick,fill=white,draw=\MsColor] (x2) circle (.1);
\draw[rounded corners,dashed] (-.4,0.2) rectangle +(.75,.6);
}
\arrow[d,Rightarrow,"\epsilon_{\mu_\cA}",swap]
&
\tikzmath{
\draw[thick,\AsColor] (0,0) coordinate (aa1) arc (180:90:.2) coordinate (m1) arc (90:0:.2) -- +(0,-.1) coordinate (mm2) -- +(0,0) arc (0:90:.2) arc (180:90:.4) coordinate (m3);
\draw[thick,\MsColor] (m3) -- +(0,.2) coordinate (m4) -- +(0,-.9) coordinate (x1);
\draw[thick,\MsColor] (m4) -- +(0,.5) coordinate (x2);
\draw[thick,\AsColor] (m4) arc (-90:-180:.2) coordinate (a2);
\filldraw[thick,fill=white,draw=\AsColor] (mm2) circle (.1);
\filldraw[thick,fill=white,draw=\AsColor] (a2) circle (.1);
\filldraw[thick,fill=\AsColor,draw=\AsColor] (m1) circle (.05);
\filldraw[thick,fill=\AsColor,draw=\AsColor] (aa1) circle (.05);
\filldraw[thick,fill=\AsColor,draw=\AsColor] (m3) circle (.05);
\filldraw[thick,fill=\AsColor,draw=\AsColor] (m4) circle (.05);
\filldraw[thick,fill=white,draw=\MsColor] (x1) circle (.1);
\filldraw[thick,fill=white,draw=\MsColor] (x2) circle (.1);
\draw[rounded corners,dashed,cyan] (-.1,.1) rectangle +(.8,.6);
}
\arrow[dr,Rightarrow,"\alpha^\dag",cyan]
&
\tikzmath{
\draw[thick,\AsColor] (0,0) coordinate (a1) -- +(0,.1) arc (180:90:.2) coordinate (x1);
\draw[thick,\MsColor] (x1) -- +(0,-.5) coordinate (m1) -- +(0,.8) coordinate (x2) -- +(0,1.1) coordinate (m2);
\draw[thick,\AsColor] (x2) arc (90:180:.2) arc (0:-90:.2) coordinate (x3) -- +(0,-.2) coordinate (x4) -- +(0,0) arc (-90:-180:.2) -- +(0,.3) coordinate (a2);
\filldraw[thick,fill=white,draw=\MsColor] (m1) circle (.1);
\filldraw[thick,fill=white,draw=\MsColor] (m2) circle (.1);
\filldraw[thick,fill=white,draw=\AsColor] (a1) circle (.1);
\filldraw[thick,fill=white,draw=\AsColor] (a2) circle (.1);
\filldraw[thick,\AsColor] (x1) circle (.05);
\filldraw[thick,\AsColor] (x2) circle (.05);
\filldraw[thick,\AsColor] (x3) circle (.05);
\filldraw[thick,\AsColor] (x4) circle (.05);
\draw[rounded corners,dashed] (-.5,.6) rectangle +(.85,.55);
}
\arrow[d,Rightarrow,"\kappa^{-1}"]
\\
\tikzmath{
\draw[thick,draw=\AsColor] (0,0) coordinate (a1) -- +(0,.4) coordinate (a2);
\draw[thick,draw=\MsColor] (.3,-.3) coordinate (m1) -- +(0,1) coordinate (m2);
\filldraw[thick,fill=white,draw=\AsColor] (a1) circle (.1);
\filldraw[thick,fill=white,draw=\AsColor] (a2) circle (.1);
\filldraw[thick,fill=white,draw=\MsColor] (m1) circle (.1);
\filldraw[thick,fill=white,draw=\MsColor] (m2) circle (.1);
\draw[rounded corners,dashed,blue] (.15,0) rectangle +(.3,.3);
}
\arrow[rrrr,Rightarrow,"\eta_{\mu_\cM}"]
\arrow[uuu,Rightarrow,"\lambda^\dag",blue]
&&&&
\tikzmath{
\draw[thick,draw=\AsColor] (0,0) coordinate (a1) -- +(0,.1) arc (180:90:.2) coordinate (m1);
\draw[thick,draw=\AsColor] (0,.8) coordinate (a2) -- +(0,-.1) arc (180:270:.2) coordinate (m2);
\draw[thick,draw=\MsColor] (.2,-.2) coordinate (x1) -- +(0,1.2) coordinate (x2);
\filldraw[thick,fill=white,draw=\AsColor] (a1) circle (.1);
\filldraw[thick,fill=white,draw=\AsColor] (a2) circle (.1);
\filldraw[thick,fill=\AsColor,draw=\AsColor] (m1) circle (.05);
\filldraw[thick,fill=\AsColor,draw=\AsColor] (m2) circle (.05);
\filldraw[thick,fill=white,draw=\MsColor] (x1) circle (.1);
\filldraw[thick,fill=white,draw=\MsColor] (x2) circle (.1);
\draw[rounded corners,dashed,blue] (-.1,.1) rectangle +(.3,.3);
}
\arrow[ur,Rightarrow,"\lambda^\dag",swap,blue]
&&
\tikzmath{
\draw[thick,draw=\AsColor] (0,0) coordinate (a1) -- +(0,.1) arc (180:90:.2) coordinate (x1);
\draw[thick,draw=\MsColor] (x1) -- +(0,.4) coordinate (x2) -- +(0,.6) coordinate (x3) -- +(0,1.1) coordinate (m1) -- +(0,-.5) coordinate (m2);
\draw[thick,draw=\AsColor] (x2) arc (90:180:.2) coordinate (x4);
\draw[thick,draw=\AsColor] (x3) arc (-90:-180:.2) -- +(0,.1) coordinate (a2);
\filldraw[thick,fill=white,draw=\AsColor] (a1) circle (.1);
\filldraw[thick,fill=white,draw=\AsColor] (a2) circle (.1);
\filldraw[thick,fill=white,draw=\MsColor] (m1) circle (.1);
\filldraw[thick,fill=white,draw=\MsColor] (m2) circle (.1);
\filldraw[thick,fill=\AsColor,draw=\AsColor] (x1) circle (.05);
\filldraw[thick,fill=\AsColor,draw=\AsColor] (x2) circle (.05);
\filldraw[thick,fill=\AsColor,draw=\AsColor] (x3) circle (.05);
\filldraw[thick,fill=\AsColor,draw=\AsColor] (x4) circle (.05);
\draw[rounded corners,dashed,blue] (-.1,0.4) rectangle +(.4,.4);
}
\arrow[ll,Rightarrow,"\lambda",blue]
\end{tikzcd}
\]
The result follows.
\end{proof}

\begin{cor}\label{cor:underlying-module-isometric}
When $\fC=2\Hilb$ and $(\cA,\Tr^\cA)\in 2\Hilb$ is an $\rmH^*$-algebra,
and $(\cM,\Tr^\cM)$ is a unital $\cA$-module, 
$(|\cM|, \Tr^{|\cM|})$
is a unital unitary $(|\cA|, \vee,\psi_{|\cA|})$-module equipped with a module trace.
Moreover, the canonical unitary $\cA\cong|\cA|$ module equivalence $\cM\cong |\cM|$ is isometric.
\end{cor}
\begin{proof}
We leave this proof to the reader, as it is almost identical to the proof of Corollary \ref{cor:underlying-isometric}.
\end{proof}


\begin{proof}[Proof of Theorem \ref{thmalpha:H*AlgebrasIn2Hilb}]
Proposition \ref{prop:hstarmfc-to-hstaralg} gives an assignment of objects of $\sH^*\mFC$ to objects of $\HstarAlg(\rmB2\Hilb)$.
Additionally, a bimodule category for the $\rmH^*$-multifusion category is naturally a bimodule for the underlying $\rmH^*$-algebra (just by forgetting the bimodule trace).
This assembles into a 3-functor $U:\sH^*\mFC \to \HstarAlg(\rmB2\Hilb)$.
By Corollaries \ref{cor:underlying-isometric} and \ref{cor:underlying-module-isometric}, 
the 3-functor $|-|$ lands in $\sH^*\mFC$
and the composite $|U(-)|$ is equivalent to the identity.
Moreover, this equivalence is isometric on objects and 1-cells.
(Here, we use the folding trick to apply the above module result to bimodules.)
It only remains to be seen that there is a natural equivalence of 3-functors $U(|-|) \cong \id_{\HstarAlg(\rmB2\Hilb)}$.
This follows from Corollary \ref{cor:Separable2AlgebrasIn2Vec}.
Notice that the unitarity conditions in Definition \ref{defn:Hstaralg-in-pre3hilb} are all properties, not structure, and so the equivalence $U(|\cA|) \cong \cA$ in Corollary \ref{cor:Separable2AlgebrasIn2Vec} is still an equivalence of $\rmH^*$-algebras.
\end{proof}


\bibliographystyle{alpha}
{\footnotesize{
\bibliography{../../bibliography/bibliography}
}}
\end{document}